\documentclass[12pt,twoside,openright]{ociamthesis} 
\usepackage[italian,british]{babel}
\usepackage[OT1]{fontenc}
\usepackage[latin1]{inputenc}
\usepackage{amsmath}
\allowdisplaybreaks[2]
\usepackage{amssymb}
\usepackage{mathtools}
\usepackage[amsmath,thmmarks]{ntheorem}
\usepackage{hyphenat}
\usepackage{url}
\usepackage{inconsolata}
\usepackage{verbatim}
\usepackage{fancyhdr}
\usepackage{faktor}
\usepackage{graphicx}
\usepackage[hang,small,bf]{caption}
\usepackage[pdftex]{color}
\usepackage[rigidchapters,explicit]{titlesec}
\usepackage{etoolbox}

\makeatletter
\patchcmd{\ttlh@hang}{\parindent\z@}{\parindent\z@\leavevmode}{}{}
\patchcmd{\ttlh@hang}{\noindent}{}{}{}
\makeatother

\usepackage{titletoc}
\usepackage{pictexwd,dcpic}
\usepackage{enumitem}\setlist{nolistsep}
\usepackage{bm}
\usepackage{bibspacing}
\title{Hyperbolic volume estimates\\ via train tracks}   %note \\[1ex] is a line break in the title

\author{Antonio De Capua}             %your name
\college{Merton College}  %your college

\degree{Doctor of Philosophy}     %the degree
\degreedate{Trinity 2016}         %the degree date

\newcommand{\cvd}{\hfill$\Box$\par\vspace{2.5ex}}
\renewcommand{\epsilon}{\varepsilon}
\newcommand{\R}{\mathbb R}
\newcommand{\sph}{\mathbb S}

\newcommand{\Hy}{\mathbb H}

\newcommand{\cc}{\mathbf C}
\newcommand{\pc}{\mathbf P}
\newcommand{\mc}{\mathbf M}
\newcommand{\ac}{\mathbf A}
\newcommand{\acc}{\mathbf{AC}}

\newcommand{\ce}{\mathbf C \mathrm E}
\newcommand{\cf}{\mathbf C \mathrm F}
\newcommand{\br}{\mathcal B}
\newcommand{\rot}{\mathrm{rot}}
\newcommand{\idx}{\mathrm{index}}
\newcommand{\seg}{\mathrm{seg}}
\newcommand{\hs}{\mathrm{hs}}
\newcommand{\hl}{\mathrm{hl}}
\newcommand{\pa}{{\mathbf{P}_+}}
\newcommand{\ma}{{\mathbf{M}_+}}
\newcommand{\cnr}{{\mathcal C}}
\newcommand{\rar}{{\mathcal R}}
\newcommand{\utw}{{\mathcal U}}
\newcommand{\trk}{{\mathcal T}}
\newcommand{\up}{\mathord{\uparrow}}
\newcommand{\dn}{\mathord{\downarrow}}

\newcommand{\Isom}{\mathrm{Isom}}

\newcommand{\vol}{\mathrm{vol}}
\newcommand{\inte}{\mathrm{int}}
\renewcommand{\emptyset}{\varnothing}
\newcommand{\nw}[1]{\textsf{\nohyphens{\textit{#1}}}}

\newcommand{\mcg}{\mathrm{Mod}}
\newcommand{\step}[1]{\vspace{1ex}\noindent\ul{Step #1:}}

\newcommand{\nei}{{\mathcal N}}
\newcommand{\ul}[1]{\underline{#1}}
\newcommand{\ol}[1]{\overline{#1}}
\newcommand{\core}{\mathrm{core}}
\newcommand{\E}{\mathcal{E}}
\newcommand{\F}{\mathcal{F}}
\newcommand{\e}{\mathrm{E}}
\newcommand{\f}{\mathrm{F}}

\newcommand{\tl}{\triangleleft}
\newcommand{\tr}{\triangleright}

\newlength{\larghnumeribox}
\newlength{\larghnumeri}
\newlength{\larghtitoli}
\newlength{\origparindent}
\newlength{\origparskip}
\theoremstyle{plain}
\theoremheaderfont{\normalfont\bfseries}
\theorembodyfont{\normalfont}
\theoremseparator{.}
\theoremnumbering{arabic}
\theoremsymbol{}
\newtheorem{defin}{Definition}[section]

\theorembodyfont{\normalfont\itshape}
\newtheorem{prop}[defin]{Proposition}
\newtheorem{theo}[defin]{Theorem}
\newtheorem{lemma}[defin]{Lemma}
\newtheorem{coroll}[defin]{Corollary}
\theoremheaderfont{\normalfont\sc}
\theorembodyfont{\normalfont}
\theoremprework{}
\theorempostwork{}
\newtheorem{rmk}[defin]{Remark}

\theoremstyle{nonumberplain}
\theoremheaderfont{\sffamily}
\theorembodyfont{\normalfont}
\theoremsymbol{\ensuremath{\Box}}
\newtheorem{proof}{Proof}
\theoremprework{\setlength{\parskip}{0ex}\setlength{\parindent}{0pt}}
\theorempostwork{\setlength{\parskip}{\origparskip}\setlength{\parindent}{\origparindent}}
\theoremstyle{empty}
\theorembodyfont{\normalfont\itshape}
\theoremsymbol{}
\newtheorem{claim}{Claim}
\theoremstyle{nonumberplain}
\newtheorem{theono}{Theorem}

\numberwithin{figure}{chapter}

\begin{document}

%this baselineskip gives sufficient line spacing for an examiner to easily
%markup the thesis with comments
\baselineskip=18pt plus1pt

%set the number of sectioning levels that get number and appear in the contents
\setcounter{secnumdepth}{3}
\setcounter{tocdepth}{3}

%\renewcommand{\submittedtext}{{Thesis Draft}}
%\frontmattter
\maketitle  

\begin{romanpages}
\begin{originality}
I declare that the work in this thesis is, to the best of my knowledge, original and my own work, except where otherwise indicated, cited, or commonly known. Much the external material employed is adapted from \cite{masurminskyii}, \cite{masurminskyq}, \cite{mosher} and, above all, \cite{mms}.
Roughly, the said material is concentrated as follows:
\begin{itemize}
\item Chapter \ref{cha:background}, \S \ref{sec:traintracks}, \S \ref{sec:traintracksmore}, and \S \ref{sub:goodbehaviour} include a large amount of previously existing definitions and fundamental results: their sources will be stated more or less individually;
\item part of \S \ref{sub:twistcurvebound} is inspired by an original idea found in \cite{masurminskyq};
\item \S \ref{sub:conclusion} revisits an argument introduced in \cite{mms} in a different setting;
\item the applications described in \S \ref{sec:dw} are based on technical tools introduced in \cite{dynnikovwiest}.
\end{itemize}

This thesis has not been submitted for a degree at another university.

\vspace{15ex}

Antonio De Capua\\
Oxford, 1 September 2016
\end{originality}

\begin{acknowledgementslong}
I am deeply grateful to  Prof.\ Marc Lackenby for having been an exceptional supervisor. I greatly appreciate how, back when I had just begun my DPhil, he took care of orienting me across the research environment at Oxford, and made sure that my background was solid enough before I started any research. During these years he has promptly aided me when I have been in need for new ideas, or hints, to carry on my work; and assisted me in the key steps of my course. On the other hand, his balanced supervision has been crucial for me to start learning how to carry out independent research.

I wish to thank all the people with whom I have had profitable conversations for my work and understanding of the topic, including my {\it viva} examiners: Saul Schleimer and Panos Papazoglou; people I met at academic events: Bert Wiest, Jeffrey Brock, Ian Agol, Javier Aramayona; my examiners in occasion of my Transfer and Confirmation of Status: Martin Bridson and Cornelia Drutu; and anyone who have contributed to broadening my horizons. Thanks to Prof.\ Minhyong Kim for being my advisor at Merton College. Also, I am very indebted to Prof.\ Bruno Martelli, for giving me the opportunity to speak in my old university, and to Prof.\ Roberto Frigerio, who has been my supervisor at the time of my Bachelor and left me enthusiastic about low-dimensional topology.

It has been an honor for me to collaborate with Professors Alexander Ritter, Ulrike Tillman, and the already mentioned Papazoglou, Lackenby, and Drutu in undergraduate teaching. My teaching experience at Exeter College with the latter has been particularly rewarding and educational for me, and I am looking forward to starting my job as a Lecturer there in few weeks.

I am grateful to Merton College and to the Mathematical Institute for having been two enviable environments where to work and spend my everyday life in Oxford: thanks to their staff, for their professionalism, their time and effort. Moreover, I wish to acknowledge the irreplaceable financial support of the Scatcherd European Fund, along with the aforementioned institutions.
\end{acknowledgementslong}
\begin{abstractlong}
We give a method of computing distances between certain points in the pants graph of a surface $S$, up to multiplicative and additive constants. More precisely, we consider splitting sequences of train tracks on $S$, such that the vertex set of each track in the sequence subdivides $S$ into pieces which are pairs of pants, or simpler than that. It is possible to regard the sequence given by the vertex set of each track as a path along the edges of a graph, which is naturally quasi-isometric to the pants graph of $S$: and we show how to estimate the distance between two points along this path.

The present work is inspired by a result of Masur, Mosher and Schleimer according to which, if the vertex sets along a splitting sequence fill $S$, then they give a quasi-geodesic path in the marking graph; and the distance between the extremes of this path is given, up to constants, by the number of \emph{splits} occurring in the sequence.

However, their result cannot hold for the pants graph: it may well be that a high number of splits in the splitting sequence make the vertex sets span a high distance in some annular projection; and, despite this, these sets cover no similarly high distances in the pants graph.

We work to treat this discrepancy: we describe a machinery that, given a train track splitting sequence, produces first a new one where the moves only contributing to annular distance are grouped altogether; and then a further one, the \emph{untwisted sequence}. This latter sequence resembles the former, but the distance it spans in any annular subsurface projection is controlled by the pants graph distance. After these constructions, we prove a distance formula by showing that the untwisted sequence is suitable for application of the same arguments conceived by Masur, Mosher and Schleimer.

Thanks to a result of J.F.~Brock, our distance estimates in the pants graph reflect into hyperbolic volume estimates for pseudo-Anosov mapping tori. We give a couple of results in this area: the first one uses I.~Agol's \emph{maximal splitting sequence}, the second one revisits I.~Dynnikov and B.~Wiest's \emph{interval identifications systems} and their \emph{transmission} to give an estimate of the hyperbolic volume for a solid torus minus a closed braid. We also sketch how one may regard this latter result as independent of the train track machinery.
\end{abstractlong}

\tableofcontents{}

\vfill
\noindent\textsf{Last edit on \today}
\end{romanpages}

\addcontentsline{toc}{chapter}{Preface}
\chapter*{Preface}

Possibly the most interesting questions related to the topology and geometry of surfaces are the ones arising when looking for bonds among geometric properties of the many spaces and groups which may be associated to a surface. In particular one may think of the \emph{Teichm\"uller space} (with either the Teichm\"uller or the Weil-Petersson metric), the \emph{mapping class group}, and a number of graphs: the \emph{curve graph}, the \emph{arc graph}, the \emph{pants graph} and the \emph{marking graph}. A summary of the network of natural maps and quasi-isometries between these objects may be found in \cite{duchin}, for instance. As this network of maps helps understanding one object's geometry through another one, the mentioned graphs give combinatorial, more manageable models for the Teichm\"uller space and the mapping class group. The pants graph, in particular, has been shown by J. Brock (\cite{brock1}) to be quasi-isometric to the Teichm\"uller space with the Weil-Petersson metric.

In the setting described above \emph{train tracks}, introduced by W.P. Thurston, have been employed for several steps towards a concrete comprehension of the surface-related graphs. A train track on a surface is a 1-complex which, just like a railway network, may be travelled smoothly in infinitely many different ways, making different choices when a switch is met. Travelling along a train track one may describe loops, or infinite paths with bounded curvature which can be straightened to geodesics. The simplest loops one may describe when travelling along a train track are finitely many, and they are called \emph{vertex cycles}.

When a train track is repeatedly altered via \emph{elementary moves} to get a so-called \emph{splitting sequence}, the change produced on the set of vertex cycles seems to `proceed towards a definite direction' in the surface-related graphs listed above. While a summary of the results in this area is given in \S \ref{sec:role_train_tracks}, there in one in particular, Theorem 6.1 in \cite{mms}, which motivates this work (here it is stated imprecisely):
\begin{theono}
Given a splitting sequence $\bm\tau$ of train tracks on a surface $S$, whose vertex sets fill $S$, each of these sets is a vertex of $\mc(S)$, the marking graph of $S$. There is a constant $Q=Q(S)$ such that the sequence of vertex sets of tracks in $\bm\tau$ moves along a $Q$-quasi-geodesic in $\mc(S)$.
\end{theono}

The definition of marking graph used in \cite{mms} is not the most common definition, to be found in \cite{masurminskyii}, but it gives a quasi-isometric graph, and is fitted for families of curves arising as vertex sets. We may apply a similar trick for the pants graph $\pc(S)$, so that the vertex sets of a large family of train tracks are vertices of this quasi-isometric, different version of $\pc(S)$: we denote it $\pa(S)$. Actually, the vertices of $\pc(S)$ will inject to the vertices of $\pa(S)$.

One may think, then, that a similar result as the one above may hold in $\pa(S)$. But there must be some phenomenon which obstruct the vertex sets along a train track splitting sequence from giving a quasi-geodesic in this graph.

The key to understanding this obstruction is the hierarchy machinery, and Theorem 6.12 in particular, of \cite{masurminskyii}. It applies for estimating distances both in the pants graph and in the marking graph of a fixed surface $S$; however, while distances in the marking graph are estimated via a summation over all distances induced in the curve complexes of subsurfaces of $S$, \emph{including annuli}, annuli are to be excluded when estimating distances in the pants graph. 

This gap between the two summations suggests that, in order to generalize the result from \cite{mms} to a statement valid for the pants graph, one has to control the contribution given by annuli in the summation. This is the idea behind the main theorem of this work i.e. Theorem \ref{thm:core}, here stated in a simplified way:

\begin{theono}
Let $\bm\tau=(\tau_j)_{j=0}^N$ be a splitting sequence of train tracks with their vertex sets $V(\tau_j)\in\pa(S)$ for all $0\leq j\leq N$. Then there is a number $A>1$, depending on $S$, such that
$$\frac{1}{A}|\utw(\rar\bm\tau)|-A\leq d_{\pa(S)}(V(\tau_0),V(\tau_N))\leq A|\utw(\rar\bm\tau)|+A.$$
\end{theono}

In this theorem, $\rar$ and $\utw$ are operations which turn a splitting sequence into a new one, and $|\cdot|$ denotes the number of elementary moves which are \emph{splits} (i.e. the non-invertible kind of elementary move). Loosely speaking, the difference between $\rar\bm\tau$ and $\bm\tau$ is that the elementary moves are performed in a different order: this way, every time there is annulus in $S$ such that $\bm\tau$ spans a high distance in the annulus' curve complex, in $\rar\bm\tau$ this distance appears as the result of a splitting sequence which realizes a series of many Dehn twists in a row.

$\utw(\rar\bm\tau)$ is obtained from $\rar\bm\tau$ via removal of the majority of these Dehn twists, in such a way that a splitting sequence is obtained anyway. This operation kills all overly high annulus contributions in the summation of Theorem 6.12 of \cite{masurminskyii} but produces a new splitting sequence which retains several properties of the old one.

Not only is our theorem similar in spirit to the aforementioned one from \cite{mms} but, once the operations $\rar$ and $\utw$ are defined, it may be proved using essentially the same line of proof. A major adaptation is necessary as the proof in \cite{mms} makes use of local finiteness in $\mc(S)$, while neither $\pc(S)$ nor $\pa(S)$ has this property; however, our sequence $\utw(\rar\bm\tau)$ has the property that, if a subsequence of it gives bounded distance in $\pa(S)$, it gives bounded distance in $\ma(S)$, too: this is the aim of our constructions.

It must be stressed that the three steps necessary to get the pants distance --- i.e. $\rar$, $\utw$ and $|\cdot|$ --- may be computed algorithmically from $\bm\tau$ using the definitions and proofs included in this work. So we give an effective way to compute the distance covered in the pants graph by a splitting sequence. More generally, given any two vertices of $\pc(S)$, one may always define a splitting sequence whose endpoints in $\pa(S)$ lie close to the selected vertices.

This way, the machinery in this work will be useful to get distances in the Weil-Petersson metric, too, as noted above; but we do not develop this aspect. We consider, instead, an application of the above for computation of volumes of hyperbolic 3-manifolds. According to a theorem proven in \cite{brock2}, the hyperbolic volume of a mapping torus over a surface $S$, defined by a pseudo-Anosov $\psi:S\rightarrow S$, is given (up to constants) by the minimum displacement induced in $\pc(S)$ by $\psi$.

Although that result requires extracting the \emph{minimum} displacement, which is a hard operation in general, it is possible to derive a couple of interesting corollaries from our machinery: a step towards effective computation of volume of hyperbolic mapping tori. In \cite{agol_pa}, given $\psi$, a standard method is described to get an infinite splitting sequence which is `preperiodic up to application of $\psi$', i.e. it may be subdivided into a `preperiod' followed by chunks of equal length, such that all entries in a given chunk are obtained from the previous one applying $\psi$. In Theorem \ref{thm:agol_volume} we prove that the `period' $\bm\rho$ in this sequence is a splitting sequence such that the distance in $\pa(S)$ between its extremes is close to the requested minimum. So the hyperbolic volume of the mapping torus is, up to additive and multiplicative constants, given by $|\utw(\rar\bm\rho)|$.

We also sketch a way of estimating hyperbolic volume for complements of closed braids in a solid torus, which are a special case of mapping tori. In this case there is another splitting sequence to be considered, the one deriving from the \emph{trasmission} and \emph{relaxing} technique for \emph{interval identification systems}, as defined in \cite{dynnikovwiest}. We explain briefly that their formalism is equivalent to train tracks except that, in order to compute pants distance, no rearrangement operation on the same line as $\rar$ is needed; and we give a few considerations about how hyperbolic volume relates with this formalism.

An outline of the contents of this work follows.

Chapter 1 traces the background the present work is built on. After giving some definitions about surfaces and 3-manifolds, aimed mainly at fixing the most basic terminology, in \S \ref{sec:graphs} we sum up all that is necessary to know about the curve graph, the marking graph and the pants graph. In particular we describe what is a subsurface projection and we give the statement of Theorem 6.12 of \cite{masurminskyii}. The section also includes some original work, as we define $\pa(S)$ and prove that it is quasi-isometric to $\pc(S)$.

In \S \ref{sec:hypvolumehistoric} we outline some previous work about estimates of hyperbolic 3-volume. In particular we focus on the \emph{guts} approach by Agol, and on the way it reflects on the volume estimates for link complements found by Lackenby and by Futer--Kalfagianni--Purcell. Then we switch to the setting of mapping tori, and give the full statement of the volume estimate in terms of distance in the pants graph, as found by Brock. Finally, in \ref{sec:role_train_tracks} we give a very short outline of the results about surface-related graphs which involve train tracks.

Chapter 2 is concerned with the main content of this work: the proof of Theorem \ref{thm:core}. After giving the basic definitions about train tracks, carried curves, and elementary moves, in \S \ref{sec:traintracksmore} we give some more involved, but still general constructions. The first one is a `subsurface projection' for train tracks, the \emph{induction} as defined in \cite{mms}. We also prove a list many basic properties of induced train tracks, in particular the ones induced on an annular subsurface. The second construction is a different viewpoint on elementary moves: an elementary move may be considered as the result of cutting the track open along a \emph{zipper}. This viewpoint will be convenient to describe a number of rearrangement procedures of elementary moves. In the third subsection we explain how to reduce train tracks to the kind treated in \cite{mms}, where each connected component of the track complement has a boundary with a corner on each component. In the fourth subsection we recover some lemmas concerning diagonal extensions of train tracks from the work of Masur and Minsky, revisiting them in terms of almost tracks or induced train tracks, according to our needs.

Once all the necessary terminology is given, in \S \ref{sub:goodbehaviour} we list the results about train tracks and geodicity of splitting sequences in the curve graph and in the marking graph, which are relevant to the present work.

In \S \ref{sec:twistcurves} we perform a deep analysis of \emph{twist curves}, in order to define the rearrangement $\rar$. We have already mentioned before that, given a splitting sequence $\bm\tau$, and the collections $V(\tau_j)$ of vertex sets for each track in the sequence, we wish to control the growth of $d_X\left(V(\tau_j),V(\tau_{j'})\right)$ for $X\subset S$ an annulus. It turns out that this distance may be high only if the core curve $\gamma$ of $X$ is a twist curve at some point of the splitting sequence. We consider twist curves as curves that, move after move, are able to produce high powers of a Dehn twist (or of a Dehn twist inverse). However, the moves producing these Dehn twists may be very sparse along $\bm\tau$, so it is convenient, for us to have a control on them, to group them consecutively. This requires a fair amount of work, because it is necessary to analyse minutely all the possible dynamics a splitting sequence may show around $\gamma$. In particular we set up some conventions to avoid the ambiguities caused by having every elementary move defined only up to isotopies. Then we show that twist moves are determined by \emph{twist modelling functions}, and we analyze how the evolution of $\bm\tau$ causes a movement in $\cc(X)$, the annulus' `curve complex', which proceeds always in the same direction, when $\gamma$ is a twist curve.

We define the \emph{rotation number} to quantify `the number of entire Dehn twists taking place about a twist curve', however sparse they may be, and finally we give a procedure to concentrate almost all the rotation number in a chunk of the splitting sequence where nothing occurs but Dehn twists about $\gamma$. We then define $\rar\bm\tau$ applying this procedure to all annuli $X$ which show a high distance $d_X\left(V(\tau_0),V(\tau_N)\right)$ between the extremes of the splitting sequence.

The last subsection bounds the number of these annuli in terms of the pants distance covered by the splitting sequence. It is inspired by an idea employed in \cite{masurminskyq}. This bound is necessary only as a technical step, because Theorem \ref{thm:core} would result into a better (linear) bound as an immediate corollary.

In \S \ref{sec:traintrackconclusion} we introduce the \emph{untwisted sequence} $\utw(\rar\bm\tau)$: the idea in its definition is that it shall closely mimic $\rar\bm\tau$, except that the chunks of sequence expressing Dehn twists will have a capped length. The splitting sequences $\utw(\rar\bm\tau)$ and $\rar\bm\tau$ begin with the same track and end with different ones; however, we prove that they share several properties, and the distances they cover in $\pa(S)$ are bounded in terms of one another. Crucially, $\utw(\rar\bm\tau)$ covers distances in $\ma(S)$ which are bounded in terms of the ones covered in $\pa(S)$.

This allows us to proceed to the proof of Theorem \ref{thm:core}: in general we need to subdivide a splitting sequence $\bm\tau$ into chunks such that each vertex set fills the same subsurface of $S$. Then, with an interplay between $\rar\bm\tau$ and $\utw(\rar\bm\tau)$, we are finally able to revisit the proof of Theorem 6.1 in \cite{mms} to suit the pants graph setting.

In Chapter 3 we connect Theorem 6.1 with the problem of estimating the hyperbolic volume of mapping tori. We first prove Theorem \ref{thm:agol_volume} about Agol's maximal splitting sequence, as mentioned before; then, in \S \ref{sec:dw}, we turn to the simpler case of punctured discs, and mapping tori which may be described as the complement of a closed braid in a solid torus. We describe how to turn Dynnikov and Wiest's transmissions of interval identification systems into a train track splitting sequence and prove Corollaries \ref{cor:dwgivesdistance} about pants distance and \ref{cor:dwgivesvolume} about hyperbolic volume.

Then we sketch some further properties, including Proposition \ref{prp:futervolume}, which is a volume formula in terms of words in the braid group $B_n$ under a suitable generating set, and of which David Futer has an independent proof.
\chapter{Background}\label{cha:background}

\section{Essentials on surfaces and 3-manifolds}

\subsection{Surfaces}\label{sub:surfaces}
These lines are conceived to fix once and for all the most basic objects treated in this work. When the term is used with no specification, a \nw{surface} $S$ is an oriented, connected, differentiable 2-manifold, possibly non-compact, without boundary unless otherwise specified. Even when any of these implicit specifications is not met (and in that case we will always make it explicit), the 2-manifolds we deal with are always of \nw{finite type}, i.e. homeomorphic to a \emph{compact} 2-manifold (possibly) with boundary, (possibly) with finitely many points removed (from its interior, if it has any boundary). For a surface $S$ with no extra specifications, we also require that the \nw{complexity} $\xi(S)\coloneqq 3(\text{genus}-1)+ (\#\text{punctures})$ is $\geq 1$.

\begin{defin} \label{def:hyperbolic}
Let $M$ be an orientable, connected, finite-type, smooth $n$-manifold ($n\geq 2$), with no boundary. A \nw{hyperbolic structure} on $M$ is an atlas of charts $\phi_i:U_i\rightarrow \Hy^n$, where each $U_i\subseteq M$ is an open subset, such that each change of chart $\phi_j\circ\phi_i^{-1}:\phi_i(U_i\cap U_j)\rightarrow \phi_j(U_i\cap U_j)$ is the restriction of an element of $\Isom^+(\Hy^n)$. This atlas is also required to cover the whole $M$, and to be maximal.
\end{defin}

Any surface will be understood to be endowed with a complete hyperbolic metric, \emph{not necessarily with finite area}, which makes it isometric to the quotient $\Hy^2/\Gamma$ for a subgroup $\Gamma<\mathrm{Isom}^+(\Hy^2)$, $\Gamma\cong \pi_1(S)$, acting freely and properly discontinuously. A \emph{puncture}, in this work, shall be considered as a purely topological concept, unrelated with the hyperbolic metric. Given a surface $S$ and one of its punctures, we say that $S$ has a \nw{cusp} there if the puncture has a neighbourhood with finite area. Note that the area of $S$ is infinite if and only if $S$ has a puncture which is not a cusp: then we say that $S$ has a \nw{funnel} at that puncture.

By \nw{curve} on a surface $S$ we mean an embedding ${\mathbb S}^1\hookrightarrow S$ which, unless otherwise specified, is defined \emph{only up to isotopies}. A curve is \nw{essential} if it can be homotoped neither into a disc, nor into a peripheral annulus. A \nw{multicurve} is a collection of pairwise disjoint and non-isotopic essential curves: it is a well-known fact that a multicurve comprises at most $\xi(S)$ different curves.

Given any two essential curves $\alpha_1,\alpha_2$ on a surface $S$, their \nw{intersection number} $i(\alpha_1,\alpha_2)$ is the minimum number of intersection points between $\alpha_1,\alpha_2$ attained when deforming both curves within their respective isotopy classes --- if $\alpha_1,\alpha_2$ are isotopic, then $i(\alpha_1,\alpha_2)=0$. If $A_1,A_2$ are two finite collections of essential curves on $S$, then $i(A_1,A_2)\coloneqq \sum_{\alpha_1\in A_1,\alpha_2\in A_2} i (\alpha_1,\alpha_2)$.

Given $S$ a surface, a \nw{subsurface} $X$ is either the entire $S$ or a connected 2-submanifold \emph{with boundary}, of finite type, with the following requests:
\begin{itemize}
\item each puncture of $X$ is also a puncture of $S$;
\item $\partial X$ consists of a collection of components of $\partial S$ and smooth essential curves in $S$;
\item $\mathrm{int}(X)$ is not homemorphic to a pair of pants $S_{0,3}$.
\end{itemize}
Note that this definition of subsurface includes closed annuli and complements of open annuli in $S$, as their boundary contains a pair of isotopic curves in $S$. In general, we allow two components of $\partial X$ to be isotopic curves (but, even in this case, we assume that distinct components have distinct realizations in $S$).

Subsurfaces, similarly as curves, are to be considered up to isotopies in $S$. In a few occasion we will drop the connectedness condition, and in those cases we make this explicit.

In sections 8.1--8.3, \cite{thurstonnotes}, the \emph{limit set} $L_\Gamma\subseteq \partial\ol{\Hy^2}$ of $\Gamma$ is defined; together with the \emph{domain of discontinuity} $D_\Gamma\coloneqq \partial\ol{\Hy^2}\setminus L_\Gamma$. Define the \nw{convex core} of $S$ as $\core(S)\coloneqq H(L_\Gamma)/\Gamma$, where $H(L_\Gamma)$ is the convex hull of $L_\Gamma$ in $\Hy^2$. It is a hyperbolic surface with totally geodesic boundary, and finite area. Also, define the \nw{funnel compactification} of $S$ as the quotient $\bar S\coloneqq (\Hy^2\cup D_\Gamma)/\Gamma$. Since it is shown that the action of $\Gamma$ on $\Hy^2\cup D_\Gamma$ is free and properly discontinuous (Proposition 8.2.3 in the cited work), $\bar S$ is a surface, with boundary unless $\bar S=S$. Note that $\bar S$ may still have punctures (the cusps of $S$), so it is \emph{not} compact in general.

There is a natural diffeomorphism $r:\bar S\rightarrow \core(S)$, isotopic to $\mathrm{id}_{\bar S}$. It restricts to a diffeomorphism $S\rightarrow \inte\left(\core(S)\right)$. Both $\bar S,\core(S)$ are not compact in general because they add a circular boundary to a given puncture of $S$ only if all its neighbourhoods have infinite area --- or, equivalently, if $S$ has a closed geodesic encircling the puncture; again equivalently, if and only if the puncture cannot be identified with the quotient of a point in $L_\Gamma$ (a parabolic one).

A \nw{peripheral annulus} of $S$ is a 2-submanifold diffeomorphic to $\mathbb S^1\times [0,1)$, bounded by an inessential curve, which serves as a closed neighbourhood for a puncture of $S$. It is \emph{not} a subsurface of $S$.

The definition of $\bar S$ depends, even topologically, on the hyperbolic metric on $S$. The \nw{compactification} of $S$, instead, is taken to be a compact surface with boundary $S_\bullet \coloneqq S\setminus \inte(P)$, where $P$ is a collection of disjoint peripheral annuli, one for each topological puncture of $S$. In particular, when a topological puncture is a funnel for the hyperbolic metric, one can choose a peripheral annulus bounded by an inessential closed geodesic: with this choice we get $\partial\core(S)\subseteq \partial S_\bullet$. So, $S_\bullet$ is a compact surface with boundary where `each puncture of $S$ is turned into a boundary component'. It is possible to identify $r_\bullet:S\rightarrow \inte(S_\bullet)$ via a diffeomorphism, homotopic to $\mathrm{id}_S$. No preferred metric is to be considered on $S_\bullet$.

A \nw{(properly embedded) arc} on $S$ is a smooth embedding $\rho:[0,1]\hookrightarrow S_\bullet$, with $\rho^{-1}(\partial S_\bullet)=\{0,1\}$; an arc is \nw{essential} if no homotopy relative to $\partial S_\bullet$ turns it into a point, or a path contained in $\partial S_\bullet$. We identify arcs with their restriction in $S$; more precisely, we identify $\rho$ with $r_\bullet^{-1}\circ \rho|_{(0,1)}:(0,1)\rightarrow S$. Similarly as curves, arcs are usually considered to be defined up to isotopies relative to $\partial S_\bullet$. A collection of pairwise disjoint and non-isotopic (rel. $\partial S_\bullet$) arcs in $S_\bullet$ (which is the same as saying, in $S$), consists of at most $3|\chi(S)|$ arcs (see Remark 1.1 in \cite{przytycki}).

With an abuse of notation, sometimes subsurfaces $X$ will be identified with the corresponding $\inte(X)$. A distinction between the two concepts will be done only when it is relevant and not self-evident.

Note that, for $X$ a subsurface of $S$, the natural map $\pi_1(X)\rightarrow \pi_1(S)$ is an injection. If $S\cong \Hy^2/\Gamma$, let $\Gamma_X<\Gamma$ be the subgroup which is identified with $\pi_1(X)$ under the isomorphism $\pi_1(S)\cong\Gamma$, and denote $S^X\coloneqq \Hy^2/\Gamma_X$, which is clearly a covering space for $S$: it is a surface if $X$ is non-annular, and diffeomorphic to $\R\times\mathbb S^1$ otherwise.

There is a natural, isometric inclusion $X\hookrightarrow S^X$, which will be enforced to consider $X$ both as a subsurface of $S$ and of $S^X$. If $\partial X\subseteq \core(S)$, then the subsurfaces $\core(S^X)$ and $X\cap\core(S)$ are isotopic in $S^X$. If $X$ has geodesic boundary (and still consisting of distinct curves), then the two actually coincide.

\subsection{Hyperbolic 3-manifolds}\label{sec:hyp3mflds}

By \nw{3-manifold} in this work we mean, unless explicitly stated otherwise, an oriented, connected, differentiable 3-manifold, possibly non-compact, without boundary unless otherwise specified, and always of \emph{finite type}, i.e.\ homeomorphic to a compact 3-manifold (possibly) with boundary, (possibly) with a 2-manifold (with or without boundary) removed from its boundary.

We refer to \cite{bonahon2002geometric} for a more precise discussion of most of the facts we are about to state. In general, (as found by H.~Kneser in 1929 and improved later), a 3-manifold $M$ [possibly with boundary, with some extra conditions: in the following sentences we use square brackets to refer to the adaptations due to cover the case of manifolds with boundary] admits a subdivision as a connect sum $M_1\#\ldots \# M_k$ of \nw{prime} 3-manifolds, i.e. 3-manifolds which do not admit any non-trivial further subdivision as a connect sum of 3-manifolds; and the factors of this subdivision are uniquely determined. Prime manifolds are, in particular, \nw{irreducible} i.e. they contain no embedded, essential 2-sphere.

Any irreducible manifold $M$, different from $\mathbb S^2\times\mathbb S^1$, contains a canonical collection $F$ of essentially embedded tori [and annuli with their boundary along $\partial M$] such that each connected component of $M\setminus F$ is one of the following (not pairwise exclusive): \nw{Seifert-fibred} (i.e. a $\mathbb S^1$-bundle over a 2-orbifold); [an interval bundle over a surface;] atoroidal [and anannular] (i.e. contains no essentially embedded torus [nor annulus with its boundary along $\partial M$]). This is called the \nw{JSJ decomposition} of $M$, after W.~Jaco, P.~Shalen and K.~Johannson.

W.P.~Thurston's Geometrization Conjecture, now a theorem by G. Perelman (see \cite{cao2006complete}; \cite{bonahon2002geometric}, Conjecture 4.1; \cite{friedl}, \S 6), implies that each connected component in the JSJ decomposition of an irreducible manifold admits a \emph{complete}, homogeneous Riemannian metric (i.e. a \emph{geometric structure}) with totally geodesic boundary.

In particular, one has:
\begin{claim}
Let $M$ be a 3-manifold which is closed, or is the interior of a compact 3-manifold whose boundary consists of tori. Suppose $M$ is irreducible, atoroidal, and not Seifert-fibered. Then $M$ admits a complete hyperbolic metric.
\end{claim}

A proof of this fact, with an extra hypothesis, was found by Thurston himself (\cite{thurstonhaken}). This overview should be sufficient to communicate the dominant role played by hyperbolic geometry in 3-manifolds, which Thurston correctly foresaw.

Also, in the statement of the Geometrization Theorem, one may require that each of the connected components of $M\setminus F$ has \emph{finite volume} under the specified geometric structure: there is only a handful of exceptional manifolds whose JSJ components cannot satisfy this further request; and none of them is hyperbolic, anyway. So, from now on, we say that a 3-manifold is \nw{hyperbolic} if it may be endowed with a finite-volume, complete hyperbolic structure.

Mostow's Rigidity Theorem in its classical form (Theorem 5.7.2 in \cite{thurstonnotes}) is a strong statement of uniqueness for these structures:
\begin{claim}
If $M_1^n$ and $M_2^n$ are two complete hyperbolic $n$-manifolds, for $n\geq 3$, with finite volume, and $\phi:\pi_1(M_1^n)\rightarrow \pi_1(M_2^n)$ is an isomorphism, then there exists a unique isometry $\psi: M_1^n\rightarrow M_2^n$ inducing $\phi$.
\end{claim}

\section{Graphs attached to a surface}\label{sec:graphs}
\subsection{Coarse geometry}
We are going to deal with several (in)equalities up to multiplicative and additive constants. Given four numbers $x,y\geq 0$, $Q\geq 1$ and $q\geq 0$, we write $x\leq_{(Q,q)} y$ (or $y\geq_{(Q,q)} x$) to mean $x\leq Qy+q$; and $x=_{(Q,q)} y$ to mean $x\leq_{(Q,q)} y \leq_{(Q,q)} x$. We also write $x\leq_Q y$, $x=_Q y$ to mean $x\leq_{(Q,Q)} y$, $x=_{(Q,Q)} y$ respectively.

Let $\mathbf G$ be a graph, with $\mathbf G^0$ the set of its vertices. $\mathbf G^0$ is turned into a geodesic metric space by assigning length one to each of its edges and defining, for each pair of vertices $x,y\in \mathbf G^0$, their distance $d_{\mathbf G}(x,y)$ to be the length of the shortest edge path connecting them. When $A,B \subseteq \mathbf G^0$ are non-empty, we define $d_{\mathbf G}(A,B)\coloneqq\mathrm{diam}_{\mathbf G} (A\cup B)$.

A map $g:I\rightarrow\mathbf G^0$, where $I=J\cap\mathbb Z$ for $J\subseteq \R$ an interval (possibly $J=\R$), is a \nw{$Q$-quasigeodesic} if $d_{\mathbf G}(g(a),g(b))=_Q |a-b|$ for all $a,b\in I$.

A map $g$ as above is a \nw{$Q$-unparametrized quasigeodesic} if there is an increasing map $\rho:I'\rightarrow I$, where $I'=J'\cap \mathbb Z$ for $J'\subseteq \R$ again an interval, such that: $\min \rho(I')=\min I,\max \rho(I')=\max I$; $g\circ \rho$ is a $Q$-quasigeodesic; if $a,b\in I$ are such that $\rho(\iota)\leq a\leq b\leq \rho(\iota+1)$ then $d_{\mathbf G}(g(a),g(b))\leq Q$.

In both definitions, $g$ is also allowed to be a \emph{multi-valued function}, i.e. a function with values in the power set, $I\rightarrow \mathcal P(\mathbf G^0)$; but in this case we require, in addition, $\mathrm{diam}_{\mathbf G}(g(i))\leq Q$ for all $i\in I$.

Note that in this work, when referring to a vertex $v$ of one of the graphs $\mathbf G$ we are going to define, we write $v\in \mathbf G$ (when we really mean $v\in\mathbf G^0$, as above). More generally, every time $\mathbf G$ is implicitly regarded as a set, we mean the set of its vertices, $\mathbf G^0$.

\subsection{The curve complex}

The parent of all graphs attached to a surface $S$ is the curve graph, extensively studied in \cite{masurminskyi}, \cite{masurminskyii} and further work. The graphs, albeit defined for $S$ a surface, can also be considered for surfaces with boundary, via $\cc(S)\coloneqq \cc\left(\mathrm{int}(S)\right)$ and similar identifications.

\begin{defin}
The \nw{curve complex} of $S$, denoted $\cc(S)$, is a simplicial complex whose vertices correspond to isotopy classes of essential curves in $S$.
\begin{itemize}
\item If $\xi(S)> 4$, there is an edge between two vertices if and only if the corresponding curves can be isotoped to be disjoint.
\item If $S\cong S_{1,1}$, there is an edge between two vertices if and only if the corresponding curves, when isotoped into minimal position, intersect in 1 point.
\item If $S\cong S_{0,4}$, there is an edge between two vertices if and only if the corresponding curves, when isotoped into minimal position, intersect in 2 points.
\end{itemize}
A collection of $n+1$ vertices in $\cc(S)$ spans a $n$-simplex if and only if any two of them are connected by an edge.

For $X$ a subsurface of $S$ which is not an annulus, the curve complex is defined just as if it were a stand-alone surface: $\cc(X)\coloneqq \cc\left(\inte(X)\right)$. The obvious map $\cc^0(X)\rightarrow \cc^0(S)$ is well defined and is an injection.

A special definition is needed when $X$ is an annular subsurface of $S$. Each vertex of $\cc(X)$ will represent an isotopy class, with fixed endpoints, of arcs (paths) properly embedded into $\ol{S^X}$ with an endpoint on each component of $\partial\ol{S^X}$. In this graph, too, there is an edge between two vertices if and only if the corresponding arcs can be isotoped, \emph{fixing their endpoints}, to have no intersection in $S^X$. The same construction as above applies for higher-dimensional skeleta of this graph.
\end{defin}

The most direct descendant of the curve graph is the \emph{arc graph}:

\begin{defin}
Suppose $S$ has punctures. The \nw{arc complex} of $S$, denoted $\ac(S)$, is a simplicial complex whose vertices correspond to classes of essential, properly embedded arcs in $S_\bullet$ under the equivalence relation given by isotopies fixing $\partial S_\bullet$ setwise; and there is an edge between any two vertices if the corresponding arcs can be isotoped to be disjoint. A collection of $n+1$ vertices in $\ac(S)$ spans a $n$-simplex if and only if any two of them are connected by an edge.

The \nw{arc and curve complex} $\acc(S)$ is defined as follows. Its $1$-skeleton is obtained from the $1$-skeleton of the disjoint union $\ac(S)\sqcup\cc(S)$ by adding an edge between a pair of vertices $\alpha\in\ac(S)$ and $\gamma\in\cc(S)$ every time $\alpha,\gamma$, may be seen as an arc and a curve, respectively, in $S_\bullet$, which can be isotoped (relatively to the boundary, resp.) to be disjoint. Again, a collection of $n+1$ vertices in $\acc(S)$ spans a $n$-simplex if and only if any two of them are connected by an edge.
\end{defin}

\begin{defin}\label{def:subsurfaceproj}
For $X\subseteq S$ a subsurface, the \nw{subsurface projection} $\pi_X:\cc(S)\rightarrow \mathcal P(\cc(X))$ is defined as follows (see \cite{masurminskyii}, \S 2.3, 2.4). 

\begin{itemize}
\item If $X$ is not an annulus: first of all, there is a natural map $\psi_X:\acc^0(X)\rightarrow \mathcal{P}\left(\cc^0(X)\right)$ which maps each $\alpha\in\cc^0(X)$ to $\{\alpha\}$ and each $\beta\in\ac^0(X)$ to the set of all isotopy classes of connected components of $\partial\nei(\beta\cup \partial X)$ which are essential in $X$ (they are 1 or 2). Here $\nei(\beta\cup \partial X)$ is a narrow, regular neighbourhood in $X$. One extends $\psi_X:\mathcal{P}\left(\acc^0(X)\right)\rightarrow \mathcal{P}\left(\cc^0(X)\right)$ naturally by defining $\psi_X(A)$ as the union of $\psi_X(a)$ over $a\in A$.

There is also a natural map $\pi'_X:\cc^0(S)\rightarrow \mathcal P\left(\acc^0(X)\right)$. Given $\alpha\in \cc^0(S)$: if $\alpha\in\cc^0(X)$ then, simply, $\pi'_X(\alpha)\coloneqq \{\alpha\}$. If $\alpha$ can be isotoped to lie completely out of $X$, then $\pi'_X(\alpha)\coloneqq\emptyset$. Otherwise $\alpha$ intersects $\partial X$ essentially; in this case, identify $\alpha$ with a representative of its isotopy class minimizing the number of intersection points with $\partial X$, and set $\pi'_X(\alpha)\coloneqq\alpha\cap X$, considered as a subset of $\ac^0(X)$ (minimality of the number of intersection points implies that $\alpha\cap X$ consists of essential arcs only).

Now, for $\alpha\in\cc^0(S)$, let $\pi_X(\alpha)\coloneqq \psi_X\left(\pi'_X(\alpha)\right)$. For $A\subseteq\cc^0(S)$ we define again $\pi_X(A)\coloneqq\bigcup_{\alpha\in A} \pi_X(\alpha)$.

\item If $X$ is an annulus: given $\alpha\in \cc^0(S)$, set $\pi_X(\alpha)=\emptyset$ if $\alpha$ does not intersect the core curve of $X$ essentially (including the case of $\alpha$ \emph{being} the core curve of $X$). Else, consider the preimage $\tilde\alpha$ of $\alpha$ in $S^X$ under the covering map $S^X\rightarrow S$: $\tilde\alpha$ consists of an infinite family of disjoint, quasi-geodesic paths. In particular each of them has two well-defined endpoints on $\partial\ol{S^X}$. Let then $\pi_X(\alpha)$ be the set of all connected components in $\tilde\alpha$ which connect the two connected components of $\partial\ol{S^X}$.
\end{itemize}

For $A,B\subseteq \cc^0(S)$ we use the shorthand notation $d_X(A,B)\coloneqq d_{\cc(X)}(\pi_X(A),\pi_X(B))$.
\end{defin}

Note that the subsurface projection of any curve onto any subsurface, if nonempty, has always diameter $\leq 1$.

Some literature, including \cite{mms} which will be employed many times in the present work, define the subsurface projection to be the map denoted $\pi'_X$ above. However, $\acc(X)$ includes $\cc(X)$ quasi-isometrically and these two versions of projection commute with that inclusion (again up to quasi-isometries): this is the only relevant aspect for us so we can stick with our definition, even if this might require changing some of the quasi-equality constants in the statements that will be quoted.

We recall a couple of useful lemmas about distances in curve complexes:

\begin{lemma}[Lemma 1.2, \cite{bowditch}]\label{lem:cc_distance}
If $S$ is a surface and $\alpha_1,\alpha_2\in \cc(S)$, then
$$
d_{\cc(S)}(\alpha_1,\alpha_2)\leq F\left(i(\alpha_1,\alpha_2)\right)
$$
for a function $F:\mathbb N\rightarrow\mathbb N$ with $F(n)=O(\log n)$, and independent of $S$.
\end{lemma}

\begin{rmk}\label{rmk:subsurf_inters_bound}
An observation which will be useful in several occasions when applying Lemma \ref{lem:cc_distance} above is the following: if $X\subseteq S$ is a non-annular subsurface and $\alpha_1,\alpha_2$ are curves in $\cc(S)$ with $\pi_X(\alpha_1),\pi_X(\alpha_2)\not=\emptyset$, then for any $\alpha'_1\in\pi_X(\alpha_1),\alpha'_2\in\pi_X(\alpha_2)$ the intersection number $i(\alpha'_1,\alpha'_2)\leq 4 i(\alpha_1,\alpha_2)+4$.

We may identify $\alpha_1,\alpha_2$ with two representatives that intersect transversely, realize the intersection number between their isotopy classes, and minimize the number of intersection points with $\partial X$. Let $\nei_1(\partial X)$, $\nei_2(\partial X)$ be two narrow, regular neighbourhoods of $\partial X$ in $S$, with $\bar\nei_1(\partial X) \subseteq \nei_2(\partial X)$. By definition of subsurface projection, for $i=1,2$, we may consider $\alpha'_i$ to be realized as the union of a set $a_i$ consisting of one or two parallel `copies' of a connected component of $(\alpha_i\cap X)\setminus\nei_i(\partial X)$; and a set $b_i$ which, if nonempty, consists of one or two segments of $\left(\partial\bar \nei_i(\partial X)\right)\cap X$. The set $b_i$ is empty (and $a_i=\alpha'_i$) exactly when $\alpha'_i=\alpha_i$. For the purposes of the following discussion, we may suppose that, for each choice of $i,j\in\{1,2\}$, each connected component of $a_i\cap \nei_j(\partial X)$ contains a point of $a_i\cap\partial X$.

Each intersection point between $\alpha'_1$ and $\alpha'_2$ is either:
\begin{itemize}
\item an intersection point between $a_1$ and $a_2$, which is to say, a `copy' of an intersection point between $\alpha_1$ and $\alpha_2$ which is contained in $X$; considering that each of $a_1$, $a_2$ consists of at most two parallel `copies' of an essential arc or curve in $X$, each intersection point between $\alpha_1$, $\alpha_2$ admits at most 4 `copies' among the intersection points between $a_1$, $a_2$.
\item an intersection point between $a_1$ and $b_2$: $a_1\cap b_2$ consists of at most $4$ points, each close to a different extremity of an arc constituting $a_1$.
\end{itemize}

Necessarily, $b_1\cap b_2=b_1\cap a_2=\emptyset$. Hence our claim.
\end{rmk}

\begin{lemma}[\S 2.4, \cite{masurminskyii}]\label{lem:annulus_distance}
If $X\subset S$ is an annulus and $\alpha_1,\alpha_2\in\cc^0(X)$ ($\alpha_1\not=\alpha_2$), then $d_{\cc(X)}(\alpha_1,\alpha_2)=1+i(\alpha_1,\alpha_2)$. Moreover $\cc(X)$ is quasi-isometric to $\mathbb Z$.
\end{lemma}
If $\alpha_1,\alpha_2\in \cc(X)$ for $X\subset S$ an annular subsurface, $i(\alpha_1,\alpha_2)$ is defined to be the minimum number of intersection points between $\alpha_1,\alpha_2$ attained when deforming both arcs within their respective isotopy class with fixed endpoints. Extreme points on $\partial\ol{S^X}$ do not count as intersection points.

Theorem 1.1 from \cite{masurminskyi} asserts that there is a $\delta=\delta(S)>0$ such that $\cc(S)$ is $\delta$-hyperbolic. It has been proved later (see e.g. \cite{bowditch2014uniform}, Theorem 1.1, or \cite{hensel20151}, Theorem 1.1) that there exists a universal value $\delta>0$ such that $\cc(S)$ is $\delta$-hyperbolic for all surfaces $S$. In Lemma \ref{lem:annulus_distance} above we have recalled that, for $X\subseteq S$ an annular subsurface, $\cc(X)$ is quasi-isometric to $\mathbb Z$. This implies, as recalled by Lemma 6.6 from \cite{mms}, that a reverse triangle inequality will hold for distances in it. Here we rephrase it in a way that refers to this setting only:
\begin{lemma}\label{lem:reversetriangle}
For any surface $S$ as above and any $Q>0$ there is a constant $\mathsf{R}_0=\mathsf{R}_0(S,Q)$ such that, if $X\subseteq S$ is a subsurface and $f:[l,m]\rightarrow \mathcal P\left(\cc^0(X)\right)$ is a $Q$-unparametrized quasi-geodesic, then for any $l\leq a\leq b\leq c\leq m$, if $\alpha=f(a),\beta=f(b),\gamma=f(c)$ then
$$
d(\alpha,\beta)+d(\beta,\gamma)\leq d(\alpha,\gamma)+\mathsf{R}_0.
$$
\end{lemma}
Uniform hyperbolicity of curve complexes implies that, in this statement, $\mathsf{R}_0$ may well be considered as depending on $Q$ only.

\subsection{Pants and marking graphs}

Two of the most immediate descendants of the curve complex are the pants graph and the marking graph. We define them as following \cite{masurminskyii}, \S 2, \S 6 and \S 8.

A \nw{pants decomposition} for a surface $S$ with $\xi(S)\geq 4$ is a maximal collection of essential, pairwise disjoint isotopy classes of curves $\{\alpha_1,\ldots,\alpha_n\}$ in $S$. Equivalently, once a set of pairwise disjoint representatives for $\{\alpha_1,\ldots,\alpha_n\}$ is chosen, its complement in $S$ will consist of a number of pairs of pants $\cong S_{0,3}$.

We say, instead, that a \emph{possibly infinite} collection $A\subseteq \cc(S)$ \nw{fills} $S$ if any other essential curve on $S$ intersects a curve in $A$. If $A=\{\alpha_1,\ldots,\alpha_n\}$ a finite set, then this is equivalent to saying that, once a set of representatives has been chosen for $\{\alpha_1,\ldots,\alpha_n\}$, with pairwise minimal number of intersection points, its complement in $S$ consists of a number of topological open discs $\cong S_{0,1}$ and 1-punctured discs $S_{0,2}$.

Given a collection of curves on $S$, either finite or infinite, there is always a subsurface of $S$ they fill; and it is unique \emph{up to isotopies in $S$}. So when we speak of the subsurface filled by the given collection, we mean its isotopy class or any representative of that class.

A \nw{complete marking} is a collection of pairs $\{p_1=(\alpha_1,t_1),\ldots,p_n=(\alpha_n,t_n)\}$ such that $\{\alpha_1,\ldots,\alpha_n\}$ is a pants decomposition of $S$ and, for each $i$, $t_i\subset \cc\left(\nei(\alpha_i)\right)$ (called \nw{transversal}) is a nonempty set of diameter $1$. A complete marking is \nw{clean} if, for each $i$, a curve $\beta_i\subset S$ exists such that: $\nei(\alpha_i\cup\beta_i)$ is either a 1-punctured torus or a 4-punctured sphere; $\pi_{\alpha_i}(\beta_i)=t_i$; and $\beta_i$ is disjoint from $\alpha_j$ for all $j\not=i$. If such $\beta_i$'s exists, they are unique.

If a complete marking is not clean, a clean marking which is \nw{compatible} with it is one whose base pants decomposition is the same, while each transversal set has the minimum distance possible from the original one, among all the clean complete markings with the same base pants decomposition. For a subdomain $Y\subset S$, the projection $\pi_Y(\mu)$ of a marking $\mu$ is defined as follows: when $Y$ is an annulus whose core is one of the $\alpha_i$, then $\pi_Y(\mu)=t_i$. When $Y$ is any other subsurface (including other annuli), $\pi_Y(\mu)=\bigcup_i \pi_Y(\alpha_i)$. A simplified version of this definition is given for pants decomposition.

\begin{defin}
The \nw{marking graph} $\mc(S)$ of a surface $S$ is a graph whose vertices correspond to all possible clean complete markings on $S$. There is a vertex between each couple of markings $\left((\alpha_1,\pi_{\alpha_1}(\beta_1)),\ldots,(\alpha_n,\pi_{\alpha_n}(\beta_n))\right)$ and\linebreak $\left((\alpha'_1,\pi_{\alpha'_1}(\beta'_1)),\ldots,(\alpha'_n,\pi_{\alpha'_n}(\beta'_n))\right)$ that are obtained from each other with one of these moves:
\begin{itemize}
\item twist: the only difference between the two markings is that $\beta'_i$ is obtained from $\beta_i$ by performing a twist or a half twist (depending on $\#\alpha_i\cap\beta_i$) around $\alpha_i$;
\item flip: all the couples are the same, except for one where $\alpha_i$ and $\beta_i$ have had their roles swapped, and the complete marking thus obtained has been then replaced with a compatible clean one.
\end{itemize}

The \nw{pants graph} $\pc(S)$ of $S$ is a graph whose vertices correspond to all possible pants decompositions $S$. Two vertices are joined by an edge if and only if the corresponding pants decompositions can be completed to clean complete markings which are obtained from each other with a flip move.
\end{defin}

Note that, if $S\cong S_{0,4}$ or $S_{1,1}$, then $\pc(S)$ coincides with the $1$-skeleton of $\cc(S)$.

Distances in the pants and the marking graphs are usually investigated via subsurface projections (Theorem 6.12 from \cite{masurminskyii}): consistently with Definition \ref{def:subsurfaceproj}, the subsurface projection of a pants decomposition is just the union of the projections of all curves constituting it.

\begin{theo}\label{thm:mmprojectiondist}
There exists a constant $M_6(S)$ such that, given $M>M_6$, there are constants $e_0, e_1$ only depending on $M$ and $S$ such that, for any pair of complete clean markings $\mu_I, \mu_T$ on $S$:
$$
\sum_{X\subseteq S} [d_X(\mu_I, \mu_T)]_M =_{(e_0,e_1)} d_{\mc}(\mu_I, \mu_T);
$$
and, for any pair of pants decompositions $p_I, p_T$,
$$
\sum_{\substack{X\subseteq S\\ X\text{ is not an annulus}}} [d_X(p_I, p_T)] =_{(e_0,e_1)} d_{\pc}(p_I, p_T).
$$

Here $[x]_M\coloneqq x$ if $x\geq M$, and $0$ otherwise. The summations shall be meant over $X\subseteq S$ subsurfaces, where each isotopy class of subsurfaces is counted only once.
\end{theo}

\subsection{The quasi-pants graph}
We will never really employ the given definition of marking graph: we will use a construction given in \cite{mms} instead. Let $k_1$ be the maximum self-intersection number of a complete clean marking on $S$, and let $\ell_1$ be the maximum intersection number between any two complete clean markings on $S$ obtained from each other via an elementary move. Fix any $k\geq k_1, \ell\geq \ell_1$. Then we can redefine $\ma(S)$ to be the graph whose vertices consist of collections of essential, distinct isotopy classes of curves on $S$ which fill the surface and have self-intersection number $\leq k$; and there is an edge between any two vertices corresponding to collections of curves intersecting in at most $\ell$ points. This graph, depending on the parameters $k,\ell$, is shown to be quasi-isometric to $\mc(S)$ (with constants depending on $k,\ell$). And the first formula in Theorem \ref{thm:mmprojectiondist} holds for estimating distances in $\ma(S)$ as well, if one chooses suitable $M\geq M_6(S,k,\ell)$; $e_j=e_j(S,M,k,\ell)$ ($j=0,1$) and the functions $M_6,e_0,e_1$ are suitably defined.

With the pants graph, we perform a similar construction:
\begin{defin}\label{def:quasipants}
A \nw{quasi-pants graph} $\pa(S)$ is defined as follows. Fix two parameters $k,\ell\geq 0$.
\begin{itemize}
\item Each vertex of the graph represents a collection $\{\alpha_1,\ldots,\alpha_m\}$ of essential, distinct isotopy classes of curves of $S$ such that, when a set of representatives minimizing the number of mutual intersection is chosen, this number is $\leq k$; and with one of the following, equivalent, properties:
\begin{enumerate}[label=\alph*)]
\item each isotopy class in $\cc(S)$ either is one of the $\alpha_j$, or intersects one of the $\alpha_j$ (however the representatives are chosen);
\item given a set of representatives for the $\{\alpha_1,\ldots,\alpha_m\}$, minimizing the number of mutual intersection, its complement in $S$ is a collection of topological open discs, 1-punctured discs, and pairs of pants;
\item Let $X$ be the subsurface of $S$ filled by $\{\alpha_1,\ldots,\alpha_m\}$ (possibly one which is not connected and with annular components); then $S\setminus X$ is a collection of (closed) pairs of pants.
\end{enumerate}
In particular, vertices include all pants decompositions of $S$.
\item There is an edge between two vertices $\mu$ and $\nu$ if and only if $i(\mu,\nu)\leq\ell$.
\end{itemize}
We only consider values of $\ell$ for which the graph is connected and each vertex has distance at most $1$ from a pants decomposition.
\end{defin}

\begin{rmk}\label{rmk:ell1}
There is a number $\ell_1\geq 0$ such that all values $\ell>\ell_1$ are acceptable for the last sentence of the above definition. Two pants decompositions which are adjacent in the pants graph have mutual intersection number $1$ or $2$: thus, for $\ell\geq 2$, the subgraph of $\pa(S)$ having pants decompositions as vertices is connected, because among its edges there are all the ones of $\pc(S)$, which is connected.

Moreover, note that orbits under the action of $\mcg(S)$ on $\pa(S)$ are finitely many, so there is a finite bound for $m=\max_{\mu\in\pa(S)}\min_{p\in\pc(S)} i(\mu,p)$. A suitable value for $\ell_1$ is then just $\max\{2,m\}$.
\end{rmk}

The parameters $k,\ell$ for $\ma(S)$ and $\pa(S)$ will be fixed once and for all after having introduced train tracks (see Remark \ref{rmk:pickparameters}). Meanwhile we prove the following:
\begin{lemma}\label{lem:pantsquasiisom}
The injection $\iota:\pc^0(S)\rightarrow\pa^0(S)$ of the vertex set of the first graph into the second one is a quasi-isometry, with constants depending on $S,k,\ell$. In particular, the second formula in Theorem \ref{thm:mmprojectiondist} holds for distance estimation in $\pa(S)$ as well, if we choose $M\geq M_6(S,k,\ell)$; $e_j=e_j(S,M,k,\ell)$ ($j=0,1$) for suitably defined functions $M_6,e_0,e_1$.
\end{lemma}

In order for the mentioned formula to make sense, we are extending naturally to vertices of $\pa(S)$ our notion of subsurface projection, and the notation $d_X$. A similar statement concerning the graphs we have denoted $\ma(S)$ and $\mc(S)$ is implicitly used in \cite{mms}, \S 6.

\begin{proof}
It suffices to show that $\iota$ is a quasi-isometric embedding because, by definition of $\pa(S)$, the $1$-neighbourhood of $\iota\left(\pc(S)\right)$ is the entire $\pa(S)$.

The inequality $d_\pa(\iota(p),\iota(q))\leq d_\pc(p,q)$ holds for any $p,q\in\pc(S)$: if two vertices of $\pc(S)$ are connected by an edge then so are their images in $\pa(S)$, by Remark \ref{rmk:ell1} above.

We only need to prove an inequality in the opposite direction. To do so, we build a map $\Phi:\pa(S)\rightarrow\pc(S)$ which will turn out to be a quasi-inverse of $\iota$: for $\mu\in \pa(S)$ let $X(\mu)$ be the (possibly disconnected) subsurface of $S$ filled by $\mu$. Set $\Phi(\mu)$ to be a pants decomposition including all curves of $\partial X(\mu)$ and chosen so that the intersection number between it and $\mu$ is minimal among all pants decompositions with the specified condition. We may suppose that $\Phi$ is $\mcg(S)$-equivariant: $\Phi(\psi\cdot\mu)=\psi\cdot\Phi(\mu)$ for all $\psi$ self-homeomorphism of $S$. As the orbits of $\mcg(S)$ in $\pa^0(S)$ are finitely many, there are two numbers $a,a'$ such that $i(\mu,\Phi(\mu))\leq a$ and $d_\pa(\mu,\Phi(\mu))\leq a'$ for all $\mu$. Note that $\Phi\circ\iota=\mathrm{id}_{\pc(S)}$.

Take any $\mu,\nu\in\pa(S)$ with $d_\pa(\mu,\nu)=1$; then $i(\mu,\nu)\leq\ell$. Let $X$ be the (possibly disconnected) subsurface of $S$ filled by $\mu$. The bound on the intersection number yields that there is a finite family $A(\mu)\subseteq \pa(S)$ such that $\nu$ is obtained from an element of $A(\mu)$ by Dehn twisting about some components of $\partial X$. The association $\mu\mapsto A(\mu)$ can be supposed to be equivariant under homeomorphisms of $S$. Again by finiteness of the number of orbits, there is a global upper bound $b$ for the size of $A(\mu)$, and also a bound $b'$ for $d_\pc(\Phi(\mu),\Phi(\nu))$ for any $\mu\in\pa(S),\nu\in A(\mu)$; and therefore for any $\mu,\nu$ at distance $1$ from each other.

This means that, more generally, $d_\pc(\Phi(\mu),\Phi(\nu))\leq b' d_\pa(\mu,\nu)$ for any $\mu,\nu$: given a geodesic connecting the two, we get an upper bound for the length of a path joining the images of the vertices under $\Phi$. In particular, if $p,q\in\pc(S)$, then $d_\pc(p,q)\leq b' d_\pa(\iota(p),\iota(q))$. The bound $b'$ depends only on $S,k,\ell$.

We now prove the second part of the lemma's statement. %For simplicity, we replace the previously defined map $\Phi$ with a new $\Phi':\pa(S)\rightarrow\pc(S)$ satisfying slightly different requests: $\Phi'$ shall be $\mcg(S)$-equivariant, fix all vertices of $\pa(S)$ which are pants decompositions, and map each vertex which is not a pants decomposition to a pants decomposition which is adjacent to it in $\pa(S)$.
Given any $\mu\in\pa(S)$ and $X\subseteq S$ a non-annular subsurface, $i(\mu,\Phi(\mu))\leq a$ implies that, given any $\alpha'_1\in \pi_X(\mu)$, $\alpha'_2\in \pi_X(\Phi(\mu))$, we have $i(\alpha'_1,\alpha'_2)\leq 4a+4$ by Remark \ref{rmk:subsurf_inters_bound}, therefore $d_X\left(\mu,\Phi(\mu)\right)\leq \max \left\{F(k),F(4a+4)\right\}$ by Lemma \ref{lem:cc_distance}. Hence there is number $\beta=\beta(S)$ such that, for any $\mu,\nu\in\pa(S)$, $|d_X(\mu,\nu)-d_X(\Phi(\mu),\Phi(\nu))|\leq \beta$.

Given $M\geq \max\{M_6(S),\beta(S)\}$, where $M_6$ is defined as in Theorem \ref{thm:mmprojectiondist}, we get the following chain of inequalities. We neglect indices on summations, as they will always range over all isotopy classes of non-annular subsurfaces $X\subset S$; and we write simply $d_X$ for $d_X(\mu,\nu)$ and $d_X\Phi$ for $d_X(\Phi(\mu),\Phi(\nu))$.
\begin{eqnarray*}
 & d_\pa(\mu,\nu)\leq d_\pa(\Phi(\mu),\Phi(\nu)) +2a' \leq d_\pc(\Phi(\mu),\Phi(\nu)) +2a'\leq  & \\
 & e_0(M)\sum [d_X\Phi]_M + e_1(M)+2a' \leq e_0(M)\sum [d_X+\beta]_M + e_1(M)+2a' \leq  & \\
 & e_0(M)\sum \left([d_X]_{M-\beta}+\beta_X\right) + e_1(M)+2a' \leq e_0(M)(1+\beta)\sum [d_X]_{M-\beta} + e_1(M)+2a'. & 
\end{eqnarray*}
Here we have denoted with $\beta_X$ a quantity which is equal to $\beta$ only if $[d_X]_{M-\beta}\not=0$; else it is also $0$. For the other inequality:
\begin{eqnarray*}
 & d_\pa(\mu,\nu) \geq d_\pa(\Phi(\mu),\Phi(\nu)) -2a' \geq (b')^{-1} d_\pc(\Phi(\mu),\Phi(\nu)) -2a' \geq  & \\
 & (b'e_0(M))^{-1}\sum [d_X\Phi]_M -e_1(M)-2a' \geq (b'e_0(M))^{-1}\sum [d_X-\beta]_M -e_1(M)-2a' \geq & \\
 & (b'e_0(M))^{-1} \frac{M}{M+\beta}\sum [d_X]_{M+\beta} -e_1(M)-2a'. & 
\end{eqnarray*}

So, we have that a formula as in Theorem \ref{thm:mmprojectiondist} holds for distances in $\pa(S)$, too. Denoting with a $'$ the quantities related with $\pa(S)$ rather than $\pc(S)$, we choose $M_6'(S)=\max\{M_6(S),\beta(S)\}+\beta(S)$; $e_0'(M)=e_0(M)\max\left\{1+\beta,b'\frac{M+\beta}{M}\right\}$, $e_1'(M)=e_1(M)+2$. Dependence of these new parameters from $k,\ell$ is hidden in the constants $a'$, $\beta$ and $b'$.
\end{proof}

\section{Hyperbolic volume estimates}\label{sec:hypvolumehistoric}

In theory, given a hyperbolic 3-manifold --- we know from Mostow Rigidity that no 3-manifold is hyperbolic in two different ways --- one may work out its volume by triangulating it with ideal tetrahedra, and then looking for a solution for Thurston's gluing consistency equations (see \cite{thurstonnotes}, Chapter 4). But, other than a precise computation of volume from a given, fixed hyperbolic 3-manifold, it is interesting to understand if one may work out the volume from some topological `parameter' in a given class of 3-manifold, albeit not precisely but only up to multiplicative and additive constants. Here we give two examples of this.

\subsection{Links complements in $\mathbb S^3$, and the guts approach}

Let $K\subset \mathbb S^3$ be a link (a smooth embedding of a number of disjoint copies of $\mathbb S^1$). It is a result of W.P. Thurston (see for instance Theorem 10.5.1, \cite{kawauchi}) that every \emph{knot} $K$ in $\mathbb S^3$ satisfies one and only one of the following: it is a torus knot; it is a satellite knot (including composite knots); its complement admits a complete hyperbolic metric with finite volume.

So, when a knot is prime, it is `likely' to fall in the third case (from now on: we say, simply, that the knot, or link, is \nw{hyperbolic}). There is no classification for \emph{links} that works as well as the one given above for knots, but from the underlying idea that a `majority' of link is hyperbolic is still true, via Thurston's Geometrization (see \S\ref{sec:hyp3mflds}). Many explicit classes of hyperbolic links are known (see e.g. \cite{adams} for an account), and it is a matter of interest how the hyperbolic volume of the complement is related with properties which be readable from a link diagram.

Results in this field have been established by Lackenby --- for links which admit an alternating diagram, see \cite{lackenby} --- and by Futer--Kalfagianni--Purcell --- for Montesinos links and for a subclass of the closed braids called \emph{positive}, see \cite{futerkalfapurcell}. They all descend from a theorem of Ian Agol (\cite{agol}, made sharper in \cite{agolimproved}), which reads as follows. If $S$ is a surface bounded by $K$, which `does not admit simplifications' in $\mathbb S^3\setminus K$, and $M\coloneqq \mathbb S^3\setminus \nei(S)$, then $\vol(\mathbb S^3\setminus K)\geq a\chi(\mathrm{guts}(M))$, for $a$ a constant and $\mathrm{guts}(M)$ being the union of the atoroidal and anannular components arising from the JSJ decomposition.

Lackenby's result is that, if a hyperbolic link $K$ has a prime, alternating diagram $D$ then $\vol(\mathbb S^3\setminus K)=_A t(D)$ for $A$ a constant and $t(D)$ the number of \emph{twist regions} in the diagram, i.e. maximal concatenations of bigons in $\mathbb S^2\setminus D$. A crossing of $D$ which is not adjacent to a bigon counts as a crossing. The theorem is established via a careful choice of a surface $S$ bounding $K$ and obtained from $D$ and then using Agol's theorem for the lower bound, and a clever triangulation for the upper bound. The result of Futer--Kalfagianni--Purcell is established via a heavy generalization of Lackenby's machinery.

\subsection{Mapping tori, and pants distance}\label{sub:mappingtori}

Given $S$ a surface and $\psi:S\rightarrow S$ is a suitable homeomorphism (or mapping class), the corresponding \nw{mapping torus} is a 3-manifold $M \cong S\times I/\sim_\psi$, where $I=[0,1]$, and $\sim_\psi$ is the equivalence relation that identifies each point $(x,0)$, $x\in S$, with $(\psi(x),1)$. A mapping torus is in particular a fibre bundle over $\mathbb S^1$, with fibre $S$. The map $\psi$ is called the \nw{monodromy} of the mapping torus.

We distinguish among three kinds of behaviour for $\psi$:
\begin{itemize}
\item $\psi$ has \nw{finite order} if some $n>0$ exists such that $\psi^n$ is isotopic to $\mathrm{id}_S$;
\item $\psi$ is \nw{reducible} if there is a multicurve on $S$ which is fixed by $\psi$, up to isotopy;
\item $\psi$ is \nw{pseudo-Anosov} if neither it has finite order nor it is reducible.
\end{itemize}

It is worth recalling some basic concepts: we follow the summary given in \cite{hyperbfibermfld}. Further details, with partially different conventions, may be found in \cite{cassonbleiler}, chapters 3--6 or in \cite{penner}, \S 1.6, \S 1.7, Chapter 3. A \nw{geodesic lamination} on a surface $S$ is a closed subset of $S$ which is a disjoint union of geodesics, called \nw{leaves} of the lamination.

Let $\lambda$ be a geodesic lamination; consider a function $\mu:T(\lambda)\rightarrow\R_{\geq 0}$, where $T(\lambda)$ is the set of all compact 1-manifolds embedded in $S$ and intersecting the leaves of $\lambda$ transversely (the 1-manifolds' boundaries, in particular, are disjoint from $\lambda$). We say that $\mu$ is a \nw{transverse measure} on $\lambda$ if it has the following properties. The function $\mu$ is $\sigma$-additive, meaning that, for each countable family $\{\alpha_i\}_{i\in\mathbb N}\subseteq T(\lambda)$ such that $(i\not=j\Rightarrow \alpha_i\cap\alpha_j=\partial\alpha_i\cap\partial\alpha_j)$ and that $\alpha\coloneqq \bigcup\alpha_i\in T(\lambda)$, we have $\mu(\alpha)=\sum \mu(\alpha_i)$. Given $\alpha_0,\alpha_1\in T(\lambda)$ two manifolds which are isotopic via a continuous family of $\alpha_t\in T(\lambda)$, $t\in [0,1]$, we have $\mu(\alpha_0)=\mu(\alpha_1)$. If $\alpha\in T(\lambda)$ is actually disjoint from $\lambda$, then $\mu(\alpha)=0$.

A \nw{measured lamination} is a pair $(\lambda,\mu)$ where $\lambda$ is a lamination and $\mu$ is a tranverse measure for $\lambda$, \emph{with full support} i.e. $\mu(\alpha)\not=0$ for all $\alpha\in T(\lambda)$ not disjoint from $\lambda$. We say that a lamination is \nw{minimal} if each half-leaf of $\lambda$ is dense in $\lambda$ and that it \nw{fills} $S$ if $S\setminus\lambda$ consists of a number of contractible connected components and peripheral annuli. Rephrasing a classical theorem of Thurston (cf. \cite{thurston_pa}, Theorem 4; \cite{flp}, Theorem 1.6; or \cite{hyperbfibermfld}, Theorem 2.5):

\begin{theo}
The map $\psi:S\rightarrow S$ is a pseudo-Anosov homeomorphism if and only if there exist a pair of minimal measured laminations $(\lambda^s,\mu^s)$, $(\lambda^u,\mu^u)$ filling $S$, a constant $c>1$ and a homeomorphism $\psi'$, isotopic to $\psi$, such that
$$
(\psi(\lambda^s),\psi_*\mu^s)=(\lambda^s,c^{-1}\mu^s)\quad\text{and}\quad
(\psi(\lambda^u),\psi_*\mu^u)=(\lambda^u,c\mu^u).
$$
The constant $c$ is unique, and the two measured laminations are unique up to scaling of the assigned transverse measure. They are called \nw{stable lamination} and \nw{unstable lamination}, respectively.
\end{theo}

\begin{rmk}\label{rmk:power_pa}
As an immediate consequence, if $\psi$ is pseudo-Anosov and $n\not=0$, then $\psi^n$ is also pseudo-Anosov.
\end{rmk}

In \cite{hyperbfibermfld} it is proved that each of the three possibilities listed above has a precise consequence on the JSJ characterization of a mapping torus:
\begin{theo}\label{thm:mappingtorushyperbolic}
Let $M$ be a mapping torus, equal to $\faktor{S\times I}{\sim_\psi}$ where $S$ is a surface, and $\psi$ is a self-homeomorphism of $S$. Then
\begin{itemize}
\item $M$ is Seifert-fibered if and only if $\psi$ has finite order;
\item $M$ contains an essential embedded torus if and only if $\psi$ is reducible;
\item $M$ is hyperbolic if and only if $\psi$ is pseudo-Anosov.
\end{itemize}
\end{theo}

In \cite{brock1}, \cite{brock2}, Brock proves (in particular) the following result.
\begin{theo}\label{thm:brockmappingtori}
Let $S$ be a surface. Two constants $e_2=e_2(S), e_3=e_3(S)$ exist such that, if $\psi:S\rightarrow S$ is a pseudo-Anosov homeomorphism and $M=S\times I/\sim_\psi$ is the corresponding mapping torus (which is hyperbolic), then
$$
\vol(M) =_{(e_2,e_3)} |\psi|.
$$
Here, $|\psi|$ denotes the translation distance of the map induced by $\psi$ on $\pc(S)$.
\end{theo}

Given a homeomorphism $\psi:S\rightarrow S$, the \nw{translation distance} induced by $\psi$ in $\pc(S)$ is the quantity $|\psi|_{\pc(S)}\coloneqq\min \{d(v,\psi\cdot v)|v\in\text{vertices of }\pc(S)\}$. The \nw{stable translation distance} induced by $\psi$ is, instead, $|\psi|_{\pc(S)}^{st}\coloneqq\lim_{n\rightarrow\infty} d(v,\psi^n\cdot v)/n$ where $v$ is any fixed vertex of $\pc(S)$. This latter quantity is well-defined and not depending on $v$, for general facts about metric spaces --- see \cite{bridson}, Chapter II.6, \S 6.6. Consequently, via triangle inequality,  one has $|\psi|_{\pc(S)}^{st}\leq |\psi|_{\pc(S)}$.

\begin{rmk}\label{rmk:stable_dist_pc}
We prove that there is a constant $e_4=e_4(S)$ such that, if $\psi\in\mcg(S)$ is pseudo-Anosov, then
$$|\psi|_{\pc(S)}=_{e_4} |\psi|_{\pc(S)}^{st}$$.

Theorem 3.2 in \cite{brock1} shows that there is a map $Q:\pc^0(S)\rightarrow \mathrm{Teich}(S)$, equivariant under the action of $\mcg(S)$ on the two metric spaces. Here $\mathrm{Teich}(S)$ is the Teichm\"uller space of $S$, equipped with the Weil-Petersson metric. Let $|\psi|_{WP}$, $|\psi|_{WP}^{st}$ be the translation distance and the stable translation distance, respectively, induced by $\psi$ in $\mathrm{Teich}(S)$; they are defined in an entirely similar way as in the pants graph. The theorem we have just mentioned implies that $|\psi|_{WP}$, $|\psi|_{\pc(S)}$ are equal up to multiplicative and additive errors; and the same is true of $|\psi|_{WP}^{st}$, $|\psi|_{\pc(S)}^{st}$.

In addition to what we have already noted earlier, we have to prove that $|\psi|_{\pc(S)}\leq_{e_4} |\psi|_{\pc(S)}^{st}$. We claim that $|\psi|_{WP}\leq |\psi|_{WP}^{st}$, which implies the desired inequality.

By Theorem 1.1 in \cite{daskalopoulos}, there is a unique $\psi$-equivariant, complete geodesic $g$ in $\mathrm{Teich}(S)$ i.e.\ an axis for the action of $\psi$. Let $x\in g$: then $|\psi|_{WP}\leq d_{WP}(x,\psi\cdot x)$, and note that $d_{WP}(x,\psi^n\cdot x)=n\cdot d_{WP}(g,\psi\cdot g)$ because the distance between $x$, $\psi^n\cdot x$ is realized as the segment of $g$ between these two points. Therefore also $|\psi|_{WP}^{st}=d_{WP}(g,\psi\cdot g)$ and the claim is proved.

An equality similar to the one just shown holds for translation distance and stable translation distance in $\cc(S)$, due to the reverse triangle inequality in Lemma \ref{lem:reversetriangle}.
\end{rmk}

\section{The role of train tracks}\label{sec:role_train_tracks}

\emph{Train tracks} on a surface $S$ will be properly defined in \S \ref{sec:traintracks}. Informally, a train track is a 1-complex on $S$ with the property that each vertex is `smoothed out' i.e. there is one selected edge incident to the vertex, such that one may proceed smoothly from this edge to any other one. Train tracks were introduced by Thurston (see \cite{thurstonnotes}, \S 8.9) to study geodesic laminations: informally, given a lamination, it is always possible to `squeeze' bands consisting of parallel segments of its leaves, and turn the lamination into a train track (\cite{penner}, Theorem 1.6.5);  as a particular case, one may do this for a (multi)curve. We say, then, that the train track \emph{carries} a lamination or a (multi)curve when one may draw a family of smooth paths along the track which is isotopic to the lamination or (multi)curve.

A track $\tau$, in general, will carry a huge quantity of different laminations and (multi)curves. The set $\cc(\tau)\subseteq \cc(S)$ of the curves carried by $\tau$, in particular, includes ones which travel along the edges of $\tau$ a high number of times: but there is a finite family $V(\tau)\subseteq\cc(\tau)$ collecting the simplest ones. A \emph{split} is a move on a train track which turns it into a new one, $\tau'$, with $\cc(\tau')\subseteq\cc(\tau)$; and in general, each curve in $\cc(\tau')$ traverses the branches of $\tau'$ fewer times than the ones in $\tau$.

What makes train tracks particularly interesting, then, is that a \emph{splitting sequence} of train tracks i.e. a sequence $(\tau_j)_{j\geq 0}$ of iterated splits on a train track, will change the set $V(\tau_j)$ so that, by following them, we move through $\cc(S)$ keeping, roughly, always the same direction. To start with, splitting sequences were used to prove the hyperbolicity of the curve complex in \cite{masurminskyi}. Some formal statements are given in \S \ref{sub:goodbehaviour}, but we stress here that Masur and Minsky proved in \cite{masurminskyq} that $\left(V(\tau_j)\right)_{j\geq 0}$ is an unparametrized quasi-geodesic in the curve complex. This property is complemented by the fact that $\cc(S)\setminus\cc(\tau_0)$ is quasi-convex, as shown in \cite{notcarried}. Recent work \cite{mms} of Masur, Mosher and Schleimer --- the one that motivates this thesis --- has shown that if some extra, mild hypotheses are met, then this sequence is a \emph{true} quasi-geodesic in $\ma(S)$. 

Furthermore, it is worth pointing out that the reason which motivated the introduction of train tracks is also reflected into the behaviour at infinity of a splitting sequence. The sets of all laminations carried by $\tau_j$ are also a decreasing family in $j$. The monograph \cite{mosher} of Mosher explores how, and in what circumstances, the splitting sequence `converges' to a lamination (Mosher actually uses the language of \emph{foliations} instead). Also notably, the paper \cite{hamenstadt} of Hamenst\"adt uses train track splitting sequences to construct a natural identification of the Gromov boundary of $\cc(S)$ with a suitable subspace of the space of geodesic laminations.

Train tracks and their splitting sequences, then, provide a combinatorial, concrete way of understanding some aspects of the geometry of the surface-related graphs defined in \S \ref{sec:graphs} and of their closest relatives, e.g. Teichm\"uller spaces. The present thesis follows this current.
\chapter{Train tracks and pants distance}

\section{Train tracks: basics}\label{sec:traintracks}

\subsection{Definition. Tie neighbourhoods}\label{sub:traintrackdefin}

Our basic definitions are largely inspired by the ones of \cite{mms} and of \cite{mosher}, but they will not coincide entirely with those.
\begin{defin}\label{def:pretrack}
A \nw{pretrack} on a surface $S$ (resp. on $S^X$ where $X\subseteq S$ is a non-peripheral annulus) is a 1-complex $\tau$ smoothly, properly embedded in $S$ (resp. $S^X$) such that, for each of its vertices $v$, there are a tangent line $L\subset T_vS$ (resp. $T_vS^X$) and a compact neighbourhood $S$ (resp. $S^X$) $\supset U\ni v$, such that the following is true. $U$ is homeomorphic to a disc; the boundary of $U$ is piecewise smooth; $U$ includes no other vertex of $\tau$ and intersects no edge which is not incident with $v$; $\tau\cap\bar U$ is a union of smooth, properly embedded paths in $U$, with their endpoints on $\partial\bar U$, each passing through $v$ and such that its tangent line at $v$ is $L$. Vertices of $\tau$ are called \nw{switches}, and edges are called \nw{branches}. 

A pretrack is \nw{semigeneric} if, for any switch $v$, there is a neighbourhood $U$ as above, with the following extra property: there is a point $x\in\partial\bar U$ such that, for each of the aforementioned smooth paths which make up $\tau\cap U$, $x$ is one of its endpoints. A pretrack is \nw{generic} if each switch (vertex) is $3$-valent.

\ul{All pretracks in the present work are (at least) semigeneric}: we will not specify it again. We will use the adjective `semigeneric' only when we wish to stress that a pretrack is not necessarily generic.

Fix a branch $b$ of a pretrack $\tau$, and consider the family $F$ of all closed segments contained in $b$ such that exactly one of their endpoints is a switch of $\tau$. Moreover, consider the smallest equivalence relation on $F$ which contains the inclusion relation $\subseteq$: it partitions $F$ into two equivalence classes that, in a self-descriptive manner, are called \nw{branch ends}. Occasionally, we use the term `branch end' also in reference to a fixed element of $F$. 

Consider a switch $v$ of $\tau$ and the above construction of the neighbourhood $U$ and the point $x$. There is only one branch end that reaches $v$ from the direction of $x$: it will be called \nw{large}. All other branch ends at $v$ are called \nw{small}. A branch is large or small if both its ends are; it is \nw{mixed} if its ends are of opposite kinds.

We denote $\br(\tau)$ the set of all branches of $\tau$. A pretrack $\sigma$ is a \nw{subtrack} of $\tau$ if $\sigma\subseteq \tau$ as sets.
\end{defin}

Again, let $X\subseteq S$ be a non-peripheral annulus. Occasionally, we will use the term \nw{generalized pretrack} for $\tau=\bar \tau\cap S^X$ , where $\bar\tau$ is a 1-complex smoothly, properly embedded in $\ol{S^X}$ such that all its vertices $v$ \emph{which lie in $S^X$} have a tangent line $L$ and a compact neighbourhood $U$ as in the above definition. In a generalized pretrack, only vertices lying in $S^X$ will be called \emph{switches}, while all edges are still called \emph{branches}. The concept of \emph{subtrack} is easily extended: in particular, among the subtracks of a non-compact pretrack, there may be some which are only generalized pretracks.

Any semigeneric pretrack $\tau$ can be endowed with a \emph{tie neighbourhood}: the following construction is a variation of the one given in \cite{mms}.
 
\begin{defin}
Fix $\epsilon>0$ small. Given any representative $e$ of a branch end in a semigeneric pretrack $\tau$, let $v$ be the switch which serves as an endpoint of $e$: a \nw{branch end rectangle} is a smooth, orientation-preserving embedding $R_e: [a,1+\epsilon]\times[-1,+1]\rightarrow S$ (resp. $S^X$), where $-1<a<1$, such that:
\begin{itemize}
\item $R_b([a,1]\times\{0\})=e$, $R_b(1,0)=v$;
\item $R_b([1-\epsilon,1+\epsilon]\times[-1,1])$ is a compact neighbourhood of $v$ testifying (like the $U$ used in Definition \ref{def:pretrack}) that the graph $\tau$ satisfies the condition defining a semigeneric pretrack at the vertex $v$;
\item a branch of $\tau$ intersects the image of $R_e$ if and only if $v$ is one of its switches;
\item for all $a\leq t\leq 1+\epsilon$, the arc $\alpha_t^e\coloneqq R_e(\{t\}\times[-1,1])$ is transverse to any smooth path embedded in $\tau$ and intersecting $\alpha_t^e$.
\end{itemize}

Let now $b$ be a branch of $\tau$. A \nw{branch rectangle} for $b$ is a map $R_b:[-1-\epsilon,1+\epsilon]\times[-1,1]\rightarrow S$ (resp. $S^X$) such that, for two suitable numbers $-1<a_1<a_2<1$, the maps $R_{e_1}:[a_1,1+\epsilon]\times[-1,1]\rightarrow S$ (resp. $S^X$) defined by restriction of $R_b$, and $R_{e_2}:[-a_2,1+\epsilon]\times[-1,1]\rightarrow S$ (resp. $S^X$) defined by $R_{e_2}(x,y)\coloneqq R_b(-x,-y)$ are branch end rectangles relative to the branch end representatives $e_1\coloneqq R_b([a_1,1+\epsilon]\times\{0\})$ and $e_2\coloneqq R_b([-1-\epsilon,a_2]\times\{0\})$, respectively. A branch rectangle is \emph{not} an embedding exactly when the two switches that delimit the branch $b$ coincide: in this case, its image is not really diffeomorphic to a rectangle.

\begin{figure}
\def\svgwidth{.65\textwidth}
\begin{center}
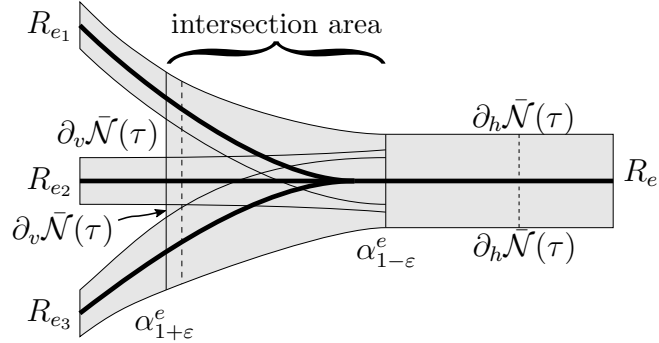
\end{center}
\caption{\label{fig:tienbh}This is the local picture of how the images of the functions $R_e$ overlap in the neighbourhood of a switch of $\tau$. The dashed vertical lines are two examples of ties. The tie neighbourhood $\bar\nei(\tau)$ is coloured in grey.}
\end{figure}

A \nw{tie neighbourhood} for $\tau$, denoted $\bar\nei(\tau)$, is specified by a family of branch end rectangles $\{R_e\}_{e\in E}$ where $E$ is a family of branch end representatives such that $\tau=\bigcup_{e\in E} e$, no branch end has two distinct representatives in $E$ and, if the union of $e_1,e_2\in E$ is a branch, then $e_1\cap e_2$ consists of more than one point. We list now a number of `consistency' conditions that the branch end rectangles $R_e$ are required to meet. Taking care of avoiding confusion, here and in the rest of this work we use $R_e$ to denote both a branch end rectangle and its image; we do the same for branch rectangles.

\begin{itemize}
\item The intersection of any two branch end rectangles $R_e, R_{e'}$ ($e,e'\in E$) is either connected or empty. More precisely, if $e\cup e'$ is a branch of $\tau$, or $e,e'$ meet at a switch of $\tau$, then $R_e\cap R_{e'}\not=\emptyset$ (so this set is connected); if neither of these is true, we require $R_e\cap R_{e'}=\emptyset$.
\item When $R_e\cap R_{e'}\not=\emptyset$, we require that the set $\alpha_t\coloneqq R_e(\{t\}\times [-1,1])\cap R_{e'}$ is connected for all $t$ for which it is defined and not empty. Same for $\alpha'_t\coloneqq R_{e'}(\{t\}\times [-1,1])\cap R_{e}$. Furthermore, the families $\{\alpha_t\}_t$, $\{\alpha'_t\}_t$ are required to define the same foliation of $R_e\cap R_{e'}$, possibly with some leaves degenerating to single points.
\end{itemize}
Suppose now that, at a switch $v$ of $\tau$, a large branch end representative $e\in E$ is meeting a collection of small branch end representatives $e_1,\ldots,	e_k\in E$. The following requests ensure that the rectangles get assembled as shown in Figure \ref{fig:tienbh}.
\begin{itemize}
\item $R_e(\{1\pm\epsilon\}\times[-1,1])\supseteq R_{e_i}(\{1\mp\epsilon\}\times[-1,1])$ for all $i=1,\ldots, k$ (here the signs $\pm,\mp$ mean that this expression summarizes two equalities).
\item For $i\not=j$, the segments $R_{e_i}(\{1-\epsilon\}\times[-1,1])$ and $R_{e_j}(\{1-\epsilon\}\times[-1,1])$ are disjoint.
\item There are two indices $1\leq i+,i-\leq k$ such that $R_e([1-\epsilon,1+\epsilon]\times \{1\})=R_{e_{i+}}([1-\epsilon,1+\epsilon]\times \{-1\})$ and $R_e([1-\epsilon,1+\epsilon]\times \{-1\})=R_{e_{i-}}([1-\epsilon,1+\epsilon]\times \{1\})$.
\end{itemize}
Suppose, finally, that the union of $e_1,e_2\in E$ is a branch $b$ of $\tau$.
\begin{itemize}
\item There must exist a branch rectangle $R_b$ that restricts to $R_{e_1},R_{e_2}$ in the sense specified by the above definition.
\end{itemize}

Usually we identify $\bar\nei(\tau)$ with the union of all branch end rectangles that constitute it. We denote with $\nei(\tau)$ the interior of $\bar\nei(\tau)$: it is an open, regular neighbourhood of $\tau$.

An arc $\alpha \subseteq \partial\bar\nei(\tau)$ that, when intersected with any $R_e$ ($e\in E$), is the image of a vertical segment $\{t\}\times[-1,1]$ is called a \nw{tie}. All ties intersect $\tau$ transversally and together they specify a foliation of $\bar\nei(\tau)$.

The boundary $\partial\bar\nei(\tau)$ can be subdivided into $\partial_v\bar\nei(\tau)$ which consists of the smooth segments of boundary which are also segments of ties; and $\partial_h\bar\nei(\tau)$ which consists of the remaining segments.
\end{defin}

A tie neighbourhood for a generalized pretrack $\tau$ on $S^X$ may be defined with a slight generalization of this construction. In this case, if $b$ is a branch of $\tau$ which is not compact and is obtained as the intersection of an edge $\bar b$ of $\bar\tau$ with $S^X$, we define a branch rectangle for $b$ as an embedding $R_b:[a,1+\epsilon]\times[-1,1]\rightarrow \bar S^X$, with $R_b([a,1+\epsilon]\times\{0\})=\bar b$, satisfying similar conditions as branch end rectangles, plus the extra request that $R_b\cap \partial\ol{S^X}=R_b(\{a\}\times[-1,1])$. The generalized definition of tie neighbourhood continues by revisiting the above constructions in a natural way.

We now define a different version of `neighbourhood' of a pretrack $\tau$: rather than right-angled corners, this time we want all corners in the boundary to be cusps, and coinciding with the vertices of $\tau$. Note that the following definition does not give a genuine neighbourhood of $\tau$; moreover, if $\tau$ is not generic, the interior of this `neighbourhood' may have more connected components than $\tau$.

First of all, when $\tau$ is a generic pretrack, there is a natural bijection between connected components of $\partial_v\bar\nei(\tau)$ and switches of $\tau$. If $\tau$ is not generic, instead, for each connected component of $\partial_v\bar\nei(\tau)$ there is a canonical choice of an associated switch of $\tau$, but this choice is not injective.

Given any component $V$ of $\partial_v\bar\nei(\tau)$, associated to a switch $v$ of $\tau$, let $\eta_1(V),\eta_2(V)$ be the two components of $\partial_h\bar\nei(\tau)$ which share an endpoint with $V$ (they may coincide). Let $\theta_1(V),\theta_2(V)$ be two smooth arcs which connect each endpoint of $V$ to $v$, are transverse to all ties they encounter, meet $\tau$ only at $v$, and are chosen so that $\eta_j(V)\cup\theta_j(V)$ is a smooth arc for $j=1,2$. There is a triangle $T_V\subseteq\bar\nei(\tau)$ whose edges are $V,\theta_1(V),\theta_2(V)$. 

We define $\nei_0(\tau)\coloneqq \nei(\tau)\setminus\left(\bigcup_V T_V\right)$, where the union is over all the connected components $V$ of $\partial_v\bar\nei(\tau)$. So $\nei_0(\tau)$ includes the interiors of all branches of $\tau$ but not the switches, which are the corners of $\partial\bar\nei_0(\tau)$. The smooth segments of $\partial\bar\nei_0(\tau)$ biject naturally with the connected components of $\partial_h\bar\nei(\tau)$.

One may also define a retraction $c_\tau:\bar\nei(\tau)\rightarrow \tau$ as follows. If $p\in \bar\nei_0(\tau)$, define $c_\tau(p)$ to be the only point of $\tau$ contained in the tie along $p$. If $p\in T_V$ for a component $V$ of $\partial_v\bar\nei(\tau)$ as above, instead, we define $c_\tau(p)$ as $c_\tau(r_V(p))$: here $r_V: T_V\rightarrow \theta_1(V)\cup\theta_2(V)$ is a retraction whose fibres are transverse to the ties of $\bar\nei(\tau)$. The map $c_\tau$ is called a \nw{tie collapse} for $\tau$ (even though its actual definition is a bit more involved than what the name suggests).

Given a tie neighbourhood $\bar\nei(\tau)$ of a pretrack $\tau$ in $S$ (resp. $S^X$ where $X$ is a non-peripheral annulus), and $\sigma$ a subtrack of $\tau$, there exists a tie neighbourhood $\bar\nei(\sigma)\subset \bar\nei(\tau)$ (not unique) such that each tie of $\bar\nei(\sigma)$ is a sub-tie of $\bar\nei(\tau)$, with the property that: if $b\in\br(\tau)$ is contained in $\sigma$ (note that $b$ may not be a branch there) and has a branch rectangle $R_b$ in $\bar\nei(\tau)$, then $R_b([-1+\epsilon ,1-\epsilon]\times[-1,1])\subseteq \bar\nei(\sigma)$ and $R_b([-1+\epsilon,1-\epsilon]\times\{-1,1\})\subseteq \partial_h\bar\nei(\sigma)$: informally, this property means that $\partial \bar\nei(\sigma)$ contains (roughly) as much as possible of $\partial\bar\nei(\tau)$.

\begin{figure}
\def\svgwidth{.65\textwidth}
\begin{center}
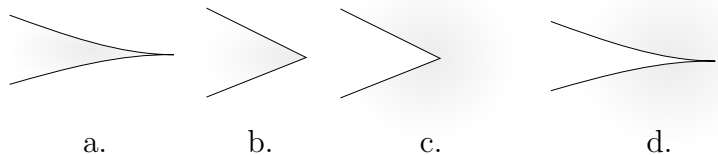
\end{center}
\caption{\label{fig:cornertypes} The possible local pictures for a 2-submanifold $C\subseteq S$ around a point of $\partial C$ which is not smooth. $C$ is shaded in grey. We need to distinguish whether the angle $\partial C$ forms at that point is convex (a. and b.) or concave (c. and d.); and whether the two segments delimiting the angle meet tangentially (a. and d.) or transversely (b. and c.). These are the only two pieces of information from these local pictures which are invariant under smooth isotopies of $S$.}
\end{figure}

\begin{defin}
Let $C$ be a 2-submanifold of a surface (possibly with boundary) $S$, such that $\partial C$ is piecewise smooth. With reference to Figure \ref{fig:cornertypes}, we define the \nw{index} of $C$ as
\begin{eqnarray*}
\idx(C) & \coloneqq & \chi(C) - (\#\text{corners of type a})/2 - (\#\text{corners of type b})/4\\
 & & + (\#\text{corners of type c})/4 + (\#\text{corners of type d})/2.
\end{eqnarray*}
\end{defin}

The index is additive: given two submanifolds as above $C',C''$, which only share a portion of their boundaries, $\idx(C'\cup C'')=\idx(C')+\idx(C'')$. Note that the 2-submanifolds considered here include, for any pretrack $\tau$, all compact connected components of $\bar\nei(\tau)$ and of $S\setminus \nei(\tau)$, plus all closures of connected components of $\nei_0(\tau)$, and of $S\setminus\bar \nei_0(\tau)$, which are compact. The possibility of foliating $\bar\nei(\tau)$ into ties implies that the index of any of its connected components is $0$. Closures of connected components of $\nei_0(\tau)$ have zero index, too.

\begin{defin}
A \emph{compact} pretrack $\tau$ on a surface $S$, entirely contained in the compactification $S_\bullet$, is:
\begin{itemize}
\item an \nw{almost track} if each connected component of $S_\bullet\setminus\nei(\tau)$ either has negative index or is a peripheral annulus; and, for each peripheral annulus $P$ among these, $\partial P$ is isotopic to a closed geodesic of $S$ (i.e. $P$ gives a funnel in $S$);
\item a \nw{train track} if each connected component of $S_\bullet\setminus\nei(\tau)$ has negative index;
\item a \nw{cornered train track} if it is a train track and each connected component of $\partial\left(S\setminus\nei(\tau)\right)$ has a corner.
\end{itemize}
\end{defin}

Note that, on a surface $S$ with finite hyperbolic area, $\tau$ is an almost track if and only if it is a train track, because $S$ has no closed geodesic encircling a puncture. The reason why we require peripheral annuli in almost tracks to be funnels is to prevent lifts of a train track from behaving pathologically at infinity (see \S \ref{sub:induced}), and is the only reason why the hyperbolic structure on $S$ is relevant for us.

Here is another definition which is, in some sense, symmetrical:
\begin{defin}
Let $S$ be a surface, even one with boundary. When each connected component $C\subseteq \inte(S)\setminus\nei(\tau)$, for $\tau$ a pretrack is \emph{homeomorphic} to a disc or a once-punctured disc, we say that the pretrack \nw{fills $S$}.
\end{defin}

\subsection{Carrying}

We will have a particular care for arcs and curves `contained' in pretracks and train tracks:
\begin{defin}
A \nw{(bounded, infinite, biinfinite, periodic) train path} along a pretrack $\tau$ is a smooth immersion $f:A\rightarrow \tau$, where:
\begin{itemize}
\item $A=[m_1,m_2]$, $[0,+\infty)$, $\R$ or $\faktor{\mathbb R}{m_3\mathbb Z}$, respectively, according to the adjective\linebreak ($m_1,m_2,m_3\in\mathbb Z$, with $m_1<m_2$ and $m_3>0$);
\item $f^{-1}(\text{switches})=A\cap\mathbb Z$.
\end{itemize}
\end{defin}

\begin{defin}\label{def:carried}
Let $\tau$ be a pretrack on a surface $S$, with $\bar\nei(\tau)$ a tie neighbourhood; let $\sigma$ be another pretrack; let $\beta$ be a curve or a multicurve; let $\delta$ be a properly embedded arc in $\bar\nei(\tau)$, with $\partial\delta=\delta\cap \partial_v\bar\nei(\tau)$, to be considered up to isotopy leaving the endpoints fixed.

An inclusion map $f:\sigma$ (resp. $\beta,\delta$) $\hookrightarrow\bar\nei(\tau)$, with its image transverse to each tie it encounters, and ambient isotopic to $\sigma$ (resp. to $\beta$; or isotopic to $\delta$ with fixed endpoints) in $S$ is called a \nw{carried realization}.

The pretrack $\sigma$ (resp. the (multi)curve $\beta$ or arc $\delta$) is \nw{carried} by $\tau$ if it admits a carried realization. 

We will often talk, more loosely, of carried realization referring simply to the image of $f$.

The pretrack $\sigma$ (resp. curve/multicurve $\beta$ or arc $\delta$) \nw{traverses} a branch $b$ if $\mathrm{im}(f)\cap \alpha_t^b\not=\emptyset$ for all $t\in[-1,1]$. If a (multi)curve or arc $\beta$ traverses a branch $b$, the \nw{multiplicity} of the traversing is the number of points in $\mathrm{im}(f)\cap \alpha_t^b$ for any $t\in[-1,1]$ (we may use expressions like \emph{traverses once, twice\ldots}). The \nw{carrying image} of a carrying injection is the union of the branches of $\tau$ which are traversed by $\sigma$ (resp. $\beta$, $\delta$). We will denote it with $\tau.\sigma$ ($\tau.\beta$, $\tau.\delta$ resp.).

A pretrack $\sigma$ is \nw{fully carried} if it is carried and $\tau.\sigma=\tau$. It is \nw{suited} to $\tau$ if it is carried and $f:\sigma\hookrightarrow \bar\nei(\tau)$ is a homotopy equivalence.

We denote $\cc(\tau)\subseteq\cc(S)$ the set of isotopy classes of curves carried by $\tau$.
\end{defin}

\begin{defin}\label{def:trainpathrealization}
Let $f:\beta\hookrightarrow\bar\nei(\tau)$ be a carried realization of a (multi)curve $\beta\in \cc(S)$ in a pretrack $\tau$, as specified in Definition \ref{def:carried} above. Then $f$ may be homotoped, keeping each point along the same tie of $\bar\nei(\tau)$, to a map $f'$ whose image is entirely contained in $\tau$: this new map is not injective anymore, but it is still an immersion. A suitable reparametrization of $f'$, then, defines a (collection of) periodic train path(s), which we call a \nw{train path realization} of $\beta$. The image of a train path realization of $\beta$ is the carrying image $\tau.\beta$.

Let now $f:\delta\hookrightarrow\bar\nei(\tau)$ be a carried realization of an arc $\delta$ in a pretrack $\tau$. Let $V,W$ be the (possibly coinciding) components of $\partial_v\bar\nei(\tau)$ where the endpoints of $\delta$ lie; $v,w$ be the switches of $\tau$ associated with $V,W$ respectively; and $\alpha_v,\alpha_w$ be the ties of $\bar\nei(\tau)$ through $v,w$ respectively. Define $\delta_{trim}$ as the longest segment of $\delta$ such that $f$ maps its extremes to points of $\alpha_v$, $\alpha_w$ respectively; and let $f_{trim}$ be the restriction of $f$ to $\delta_{trim}$.

Similarly as above, $f_{trim}$ may be homotoped, keeping each point along the same tie of $\bar\nei(\tau)$, to a map $f'_{trim}$ whose image is entirely contained in $\tau$ and can be reparametrized to get a bounded train path along $\tau$. We call the latter, again, a \nw{train path realization} of $\delta$; its image is $\tau.\delta$.
\end{defin}

\begin{rmk}
If $\tau$ is a \emph{generic} pretrack, then $\tau$ is a deformation retract of $\bar\nei_0(\tau)$, which is in turn a deformation retract of $\bar\nei(\tau)$. This makes it possible to ask more --- and we will --- from a carrying injection $f:\sigma\hookrightarrow \bar\nei(\tau)$ (or $f:\beta\hookrightarrow \bar\nei(\tau)$, $f:\delta\hookrightarrow \bar\nei(\tau)$), up to altering $f$ via isotopies which still map each point of $\sigma$ (resp. $\beta$, $\delta$) along the same tie as $f$ does.
\begin{itemize}
\item For a pretrack $\sigma$ and for a (multi)curve $\beta$, we require the image of $f$ to be contained in $\bar\nei_0(\tau)$. This means that a train path realization of $\beta$ is obtained just by reparametrizing $c_\tau\circ f$, while for a pretrack $\sigma$ we have $\tau.\sigma=c_\tau\circ f(\sigma)$.
\item For an embedded arc $\delta$, we require $\mathrm{im}(f)\cap \bar\nei_0(\tau)$ to be connected. This implies that, given $f_{trim}$ as defined above, one gets a train path realization of $\delta$ by reparametrizing $c_\tau\circ f_{trim}$.
\end{itemize}
\end{rmk}

\begin{rmk}\label{rmk:idx_of_nei_diff}
Let $\tau,\sigma$ be pretracks such that the injection $\sigma\hookrightarrow\nei(\tau)$ is a carrying map; choose a tie neighbourhood $\bar\nei(\sigma)\subseteq \bar\nei(\tau)$. Then each compact component of $\bar\nei(\tau)\setminus\nei(\sigma)$ has zero index.

This is because $\bar\nei(\tau)\setminus\nei(\sigma)$ can be subdivided into a family of rectangles and triangles with two (outward) right corners and a cusp; and they both have zero index. The triangles, in particular, may arise when $\partial_h\bar\nei(\sigma)$ and $\partial_h\bar\nei(\tau)$ have segments in common.
\end{rmk}

In the case of a train track, there is little space available in deciding \emph{how} to get something carried: the following statement summarizes Propositions 3.5.2 and 3.6.2 from \cite{mosher}:
\begin{prop}\label{prp:carryingunique}
Let $\sigma,\tau$ be two train tracks on a surface $S$, with $\sigma$ determined up to isotopies of $S$. Then, given any two carrying injections $f_1,f_2:\sigma\hookrightarrow\bar\nei(\tau)$, $f_1$ and $f_2$ are homotopic through carrying injections. The same is true for two carrying injections of a curve $\alpha\in\cc(S)$ carried by $\tau$. In particular the carrying image $\tau.\sigma$ or $\tau.\alpha$ is uniquely determined, and so is the train path corresponding to $\alpha$, up to reparametrization.
\end{prop}

We will need a slight generalization:
\begin{coroll}\label{cor:carryingunique}
Let $\tau$ be an almost track on a surface $S$. Then any two carrying injections of another almost track $\sigma$, or of a curve $\alpha\in\cc(S)$, in $\bar\nei(\tau)$ are homotopic through carrying injections.
\end{coroll}
\begin{proof}
Fix two carrying injections $f_1,f_2:\sigma$ (or $\alpha$) $\rightarrow\bar\nei(\tau)$; then there is a map $h: \sigma \times [0,1]\rightarrow S$ such that $h|_{\sigma\times\{0\}}=f_1$ and $h|_{\sigma\times\{1\}}=f_2$. As the domain of $h$ is compact, its image also is. Hence there is a union $P$ of peripheral annuli for $S$, one for each puncture, disjoint from each other and from $\mathrm{im}(h)\cup\bar\nei(\tau)$.

Let $\Sigma \in P$ be a finite set of points, one for each connected component; let then $S'\coloneqq S\setminus \Sigma$. This is a surface on its own, with a hyperbolic metric which is not related with the one of $S$. However $\tau\subseteq S'$ is a train track, as this is a property which is independent of the metric.

For the case of a carried almost track $\sigma$: $f_1(\sigma)\subseteq S'$ is a train track because, if we pick a tie neighbourhood $\nei(f_1(\sigma))\subseteq\nei(\tau)$, then any connected component of $S'\setminus\nei(f_1(\sigma))$ is a gluing of some connected components of $S'\setminus\nei(\tau)$ (at least one of them) with some of $\bar\nei(\tau)\setminus \nei(f_1(\sigma))$; the latter have zero index because of Remark \ref{rmk:idx_of_nei_diff}. Hence $S'\setminus\nei(f_1(\sigma))$ has negative index.

For the case of a carried curve $\alpha$, $f_1(\alpha)\subset S'$ is still essential.

The map $h$ serves as a homotopy between $f_1$ and $f_2$ also in $S'$. So, with an application of the above proposition, $f_1$ and $f_2$ are actually homotopic through carrying injections; hence the same property is true in $S$.
\end{proof}

The set $\cc(\tau)$, for $\tau$ a train track, is `generated' by few curves. To understand this we need the following notion:
\begin{defin}\label{def:transversemeasure}
Let $\tau$ be an almost track on a surface $S$. A \nw{transverse measure} on $\tau$ is a map $\mu:\br(\tau)\rightarrow \R_{\geq0}$ with the following property. For each switch $v$ of $\tau$, if $b$ is the branch having a large end at $v$, and $b_1,\ldots,b_m$ are the branches having a small end there (we list any of those branches \emph{twice} if it has both ends there), then $\mu(b)=\sum_{i=1}^m \mu(b_i)$. We denote the set of such measures with $\mathcal M(\tau)$. A transverse measure can be equally seen as an element of $\R^{|\br(\tau)|}$. More precisely, the subset $C$ of $\R^{|\br(\tau)|}$ consisting of transverse measures is a cone with its summit at the origin. Also define ${\mathcal M}_{\mathbb Q}(\tau)\coloneqq {\mathcal M}(\tau)\cap \mathbb Q^{|\br(\tau)|}$, i.e. the set of transverse measures which assign a rational weight to each branch.

Given $\alpha\in\cc(\tau)$ and a train path realization $f:\alpha\rightarrow \tau$, we can define the transverse measure $\mu_\alpha$ by setting, for each branch $b$ of $\tau$, $\mu_\alpha(b)=$ number of connected components of $f^{-1}(b)$ (i.e. the number of times a carried realization of $\alpha$ traverses $b$). It is an almost immediate consequence of Proposition \ref{prp:carryingunique} that this measure depends only on the isotopy class of $\alpha$.
\end{defin}

The following is a simplified version of Theorem 3.7.1 from \cite{mosher}, or Theorem 1.7.7 from \cite{penner}.
\begin{prop}\label{prp:measurecurvecorresp}
Let $\tau$ be a train track on $S$. Define\footnote{$WMC$ stands for \emph{weighted multicurves}.}
$$WMC(\tau)\coloneqq \left\{\left((\gamma_1,a_1),\ldots,(\gamma_m, a_m)\right)\left|\begin{array}{l}
m\in\mathbb N; \\
\gamma_j\in\cc(\tau)\text{ for all } j \text{ and are pairwise disjoint};\\
a_j\in \mathbb Q_{>0}\text{ for all } j
\end{array}\right.\right\}.$$
Let also 

Then the map
\begin{eqnarray*}
WMC(\tau) & \longrightarrow & {\mathcal M}_{\mathbb Q}(\tau)\setminus \{0_\tau\} \\
\left\{(\gamma_1,a_1),\ldots,(\gamma_m, a_m)\right\} & \longmapsto & \sum_{j=1}^m a_j\mu_{\gamma_j}
\end{eqnarray*}
is a bijection ($0_\tau$ denotes the zero transverse measure).
\end{prop}
\begin{coroll}\label{cor:measurecurvecorresp}
Let $\tau$ be an almost track on $S$. Then the above map is injective.
\end{coroll}
\begin{proof}
If the given map is not injective, one may find two distinct collections\linebreak $\left\{(\gamma_1,a_1),\ldots,(\gamma_m, a_m)\right\}$ and $\left\{(\gamma'_1,a'_1),\ldots,(\gamma'_{m'}, a_{m'})\right\}$ such that $\sum_{j=1}^m a_j\mu_{\gamma_j}=\linebreak \sum_{j=1}^{m'} a'_j\mu_{\gamma'_j}$. Let $\ul{\gamma_1},\ldots,\ul{\gamma_m},\ul{\gamma'_1},\ldots,\ul{\gamma'_{m'}}$ be carried realizations of the isotopy classes $\gamma_1,\ldots,\gamma_m$, $\gamma'_1,\ldots,\gamma'_{m'}$. We repeat the construction seen in the proof of Corollary \ref{cor:carryingunique}: let $\Sigma\subset S$ be a finite set consisting of a point for each peripheral annulus among the components of $S\setminus\bar\nei(\tau)$, and let $S'\coloneqq S\setminus \Sigma$. Then $\tau$ is a train track in $S'$ and the curves $\ul{\gamma_i},\ul{\gamma'_j}$ are all essential in $S'$, and carried by $\tau$: we identify them with their respective classes in $\cc_{S'}(\tau)$ (i.e. the subset of $\cc^0(S')$ consisting of all curves carried by $\tau$).

For $\alpha\in \cc_{S'}(\tau)$, let $\mu'_\alpha$ be the measure it induces on $\tau$ as a train track in $S'$. Then $\sum_{j=1}^m a_j\mu'_{\ul{\gamma_j}}=\sum_{j=1}^{m'} a'_j\mu'_{\ul{\gamma'_j}}$, but this contradicts the above proposition.
\end{proof}

\begin{defin}
Let $\tau$ be a train track on a surface $S$. It is shown that the cone $C$ specified above is the convex hull of a bounded number of rays. Pick the smallest such set $\{r_1,\ldots,r_l\}$ of rays: each $r_j$ will correspond to the real multiples of a single $\mu_{\gamma_j}$, for a $\gamma_j\in\cc(\tau)$. We denote $V(\tau)=\{\gamma_1,\ldots,\gamma_l\}$ the \nw{vertex set} of $\tau$; an element of this set is called a \nw{vertex cycle}.

If $\tau$ is an almost track on $S$, then remove an extra point from each of the components of $S\setminus\nei(\tau)$ which are peripheral annuli, to get a surface $S'$. Now $\tau\subset S'$ is a train track, and has vertex set $\{\gamma_1,\ldots,\gamma_l\}\subseteq \cc(S')$. Up to changing their order, we may suppose that for an index $0\leq k\leq l$ the curves $\gamma_{k+1},\ldots \gamma_l$ are inessential in $S$, while the remaining ones define isotopy classes $[\gamma_1],\ldots,[\gamma_k]\in\cc(S)$. We define then $V(\tau)\coloneqq\{[\gamma_1],\ldots,[\gamma_k]\}$.
\end{defin}

\subsection{More about tracks and curves}

\begin{defin}
Let $\tau$ be any pretrack on a surface $S$. A carried curve or arc $\gamma$ is \nw{wide} if its carried realization may be given an orientation such that each branch $b$ is either: traversed by $\gamma$ at most once; or traversed twice, in such a way that each segment of $\gamma\cap R_b$ appears to the right of the other.

We denote $W(\tau)\subset \cc(\tau)$ the set of wide carried curves of $\tau$.
\end{defin}

Note that, for $\tau$ an almost track, $V(\tau)\subseteq W(\tau)$ because it is quite easy to decompose $\mu_\gamma$ into a sum of measures represented by simpler curves if $\gamma$ is not wide (Lemma 2.8 in \cite{mms}).

\begin{lemma}\label{lem:vertexsetbounds}
There are bounds $N_0, N_1, N_2, C_0$ depending on $S$ such that, for any almost track $\tau$, the set $W(\tau)$ has no more than $N_0$ elements and its diameter is no larger than $C_0$, $\tau$ has at most $N_1$ branches, and $S\setminus\nei(\tau)$ consists at most $N_2$ connected components.

For a matter of convenience, we suppose all these bounds to be increasing with $\xi(S)$ (by enlarging them when necessary).
\end{lemma}
\begin{proof}
Compactness of almost tracks implies that the number of switches and branches in any of them is finite; and so is the number of connected components of $S \setminus \nei(\tau)$. This means that $\idx(S_\bullet)$ is the sum of indices of the connected components of $S_\bullet\setminus \nei(\tau)$ (because $\idx\left(\bar\nei(\tau)\right)=0$, as it is a finite gluing of rectangles). These components all have negative index, except for some peripheral annuli (they are at most as many as the punctures of $S$). But this poses an upper bound first of all on their number, and then also on the number of smooth edges in their boundary. Eventually, this yields that there is a \emph{uniform} bound on the number of branches and of switches of $\tau$, in terms of the topology of $S$.

As a consequence, there are finitely many distinct almost tracks up to diffeomorphisms of $S$; and the sets of wide curves are equivariant under diffeomorphisms. The definition of wide curve poses combinatorial constraints yielding that, for each almost track $\tau$, they are finitely many isotopy classes of them.

These remarks are enough to prove all claims in the statement.
\end{proof}

Given a curve $\gamma$, carried by a pretrack $\tau$, a \nw{wide collar} $A_\gamma$ for $\gamma$ is an open annulus $A_\gamma\subseteq S$, with compact closure, such that $\tau.\gamma$ constitutes one of the two components of $\partial\bar A_\gamma$, and the other component is an embedding of $\mathbb S^1$ belonging to the isotopy class $\gamma$. On top of this, we can require that $A_\gamma$ does not contain any switch of $\tau$.

A curve $\gamma\in\cc(\tau)$ has a wide collar if and only if it is wide carried. Suppose first that $\gamma$ admits a wide collar $A_\gamma$. Then a suitable realization of the core of $A_\gamma$, close to $\tau.\gamma$, is a carried realization of $\gamma$ and shows that $\gamma$ is wide carried. Conversely, if $\gamma$ is wide carried, one can always arrange for a carried realization, $\ul\gamma$, to have the property that, for each $b\in\br(\tau)$ traversed by $\gamma$, the segment $\ul\gamma\cap R_b([-1,1]\times[-1,1])$ has $b$ to its right. Then $\tau.\gamma$ and $\ul\gamma$ bound together a wide collar, obtained as a union of sub-ties in $\nei(\tau)$. 

There is a notion of canonical realization even for curves which are not carried, as pointed out in \cite{mms}:
\begin{defin}\label{def:efficientposition}
A multicurve $\gamma$ on $S$ is in \nw{efficient position} with respect to a train track $\tau$ if $\gamma$ is embedded in $S$, transversely to $\tau$, with the following restrictions. Let $\nei(\gamma)$ be a regular neighbourhood of $\gamma$ containing no corners of $\partial \bar\nei(\tau)$.
\begin{itemize}
\item the chosen embedding of $\gamma$ does not include any switch of $\tau$;
\item for each branch $b$, the intersection of $\gamma\cap R_b$ (when nonempty) is a collection of ties, or it is transverse to all ties (in this latter case we say that $b$ is \nw{traversed} by $\gamma$, similarly as for carried multicurves);
\item for each connected component $C$ of $S_\bullet\setminus\nei(\tau)$, each connected component of $\ol{C\setminus\nei(\gamma)}$ either has negative index or is a rectangle.
\end{itemize}

Such a multicurve is \nw{wide} if, moreover, there is a choice for the orientation of its components with the following properties.
\begin{itemize}
\item If $b$ is a branch such that $\gamma$ traverses $R_b$ transversally to the ties, $\gamma\cap R_b$ is either a single segment, or two segments each appearing to the right of the other.
\item Fix any component $C$ of $S_\bullet\setminus\nei(\tau)$, and any two points $P_1,P_2$ of $\gamma\cap \partial C$ which are consecutive along $\partial C$. Let $\alpha_1,\alpha_2$ be the connected components of $\gamma\cap C$ the two points belong to. Then either $\alpha_1=\alpha_2$ or, with the given orientations on $\gamma$, $\alpha_1,\alpha_2$ appear to the right of each other.
\end{itemize}
\end{defin}

Carried realizations of curves and arcs are a particular case of efficient position; the opposite extreme case is the following. An arc or curve $\gamma\in\acc(S)$ is \nw{dual} to a pretrack $\tau$ if is (isotopic to one) in efficient position, with no branch $b$ such that $\mathrm{im}(R_b)$ is traversed transversally to the ties. We denote $\cc^*(\tau)\subset \cc(S)$ the set of curves in $S$ which are dual to $\tau$.

\begin{defin}\label{def:recurrent}
An almost track $\tau$ is \nw{recurrent} if for each branch $b$ of $\tau$ is traversed by a periodic train path in $\tau$. Note that, for $\tau$ a train track, this is equivalent to saying that each branch is traversed by some carried curve of $\tau$.

The track $\tau$ is \nw{transversely recurrent} if if for each branch $b$ of $\tau$ there is a dual curve $\gamma$ intersecting $b$.

The track is \nw{birecurrent} if it both recurrent and transversely recurrent.
\end{defin}

It can be shown (see \cite{penner}, \S 1.3) that a train track $\tau$ is recurrent if and only if there is a carried multicurve whose carrying image is the entire $\tau$. Similarly, $\tau$ is transversely recurrent if and only if there is a family of pairwise disjoint simple closed loops in $S$, dual to $\tau$, possibly including isotopic pairs, and such that, for each branch $b\in\br(\tau)$, one of the specified realizations intersects $b$.

One of the main theorems in \cite{mms} is that efficient position always exists under some mild conditions on the train track (Theorem 4.1):
\begin{theo}\label{thm:efficientpositionexists}
Let $\tau$ be a birecurrent, cornered train track on $S$; and let $\gamma\subset \cc(S)$ be a multicurve. Then $\gamma$ has an isotopy representative in efficient position with $\tau$.

Given two different realizations $\gamma_1,\gamma_2$ of $\gamma$ in efficient position, they can be transformed into each other via a finite sequence of elementary operations (described in the original statement).
\end{theo}

\subsection{Elementary moves}

Some alterations of pretracks are considered to be \nw{elementary moves}. We describe them in Figures \ref{fig:ttcombing} and \ref{fig:ttsplitting}, referring to \S 3.12 and 3.13 of \cite{mosher} for formal definitions.

\begin{figure}
\begin{center}
\def\svgwidth{.7\textwidth}
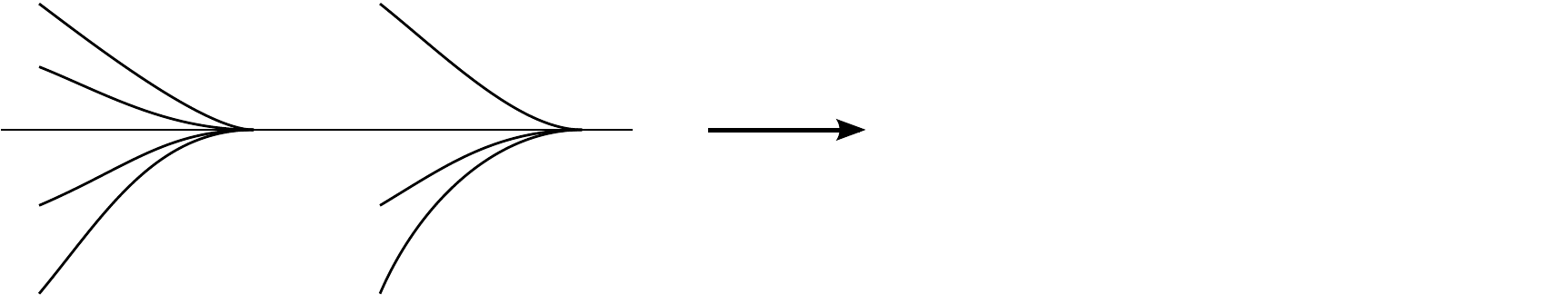

\vspace{1ex}\includegraphics[width=.7\textwidth]{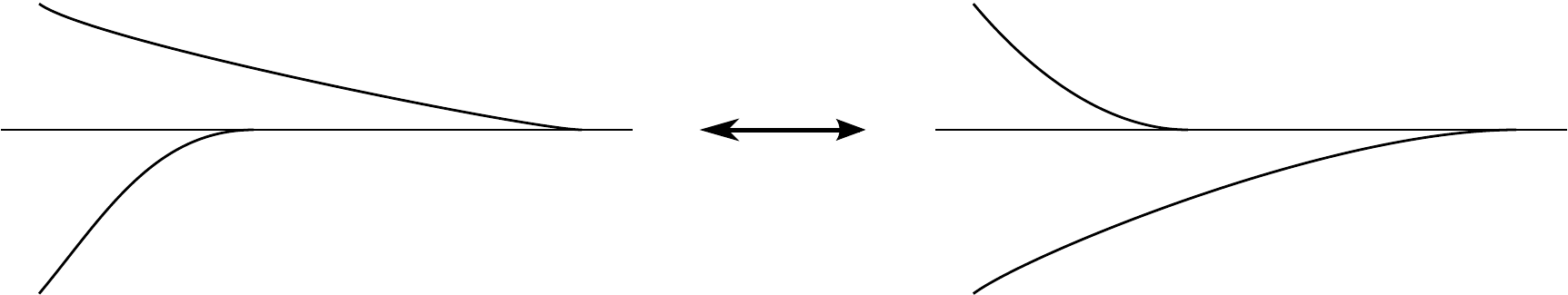}
\end{center}
\caption{\label{fig:ttcombing}The upper picture represents a \nw{comb move} in a semigeneric pretrack; this move, in particular, shrinks a mixed branch to a point. The inverse of a comb move will be called \nw{uncomb}. The notion of comb move does not make sense in the setting of generic train track, where it is replaced by the one of \nw{slide}, depicted in the lower picture: given a mixed branch, its large end `moves past' the small end, and in so doing it replaces the old mixed branch with a new one. The effect of a slide can be cancelled, up to isotopy, with a further slide.}
\end{figure}
\begin{figure}
\begin{center}
\def\svgwidth{.7\textwidth}
\vspace{1ex}\includegraphics[width=.7\textwidth]{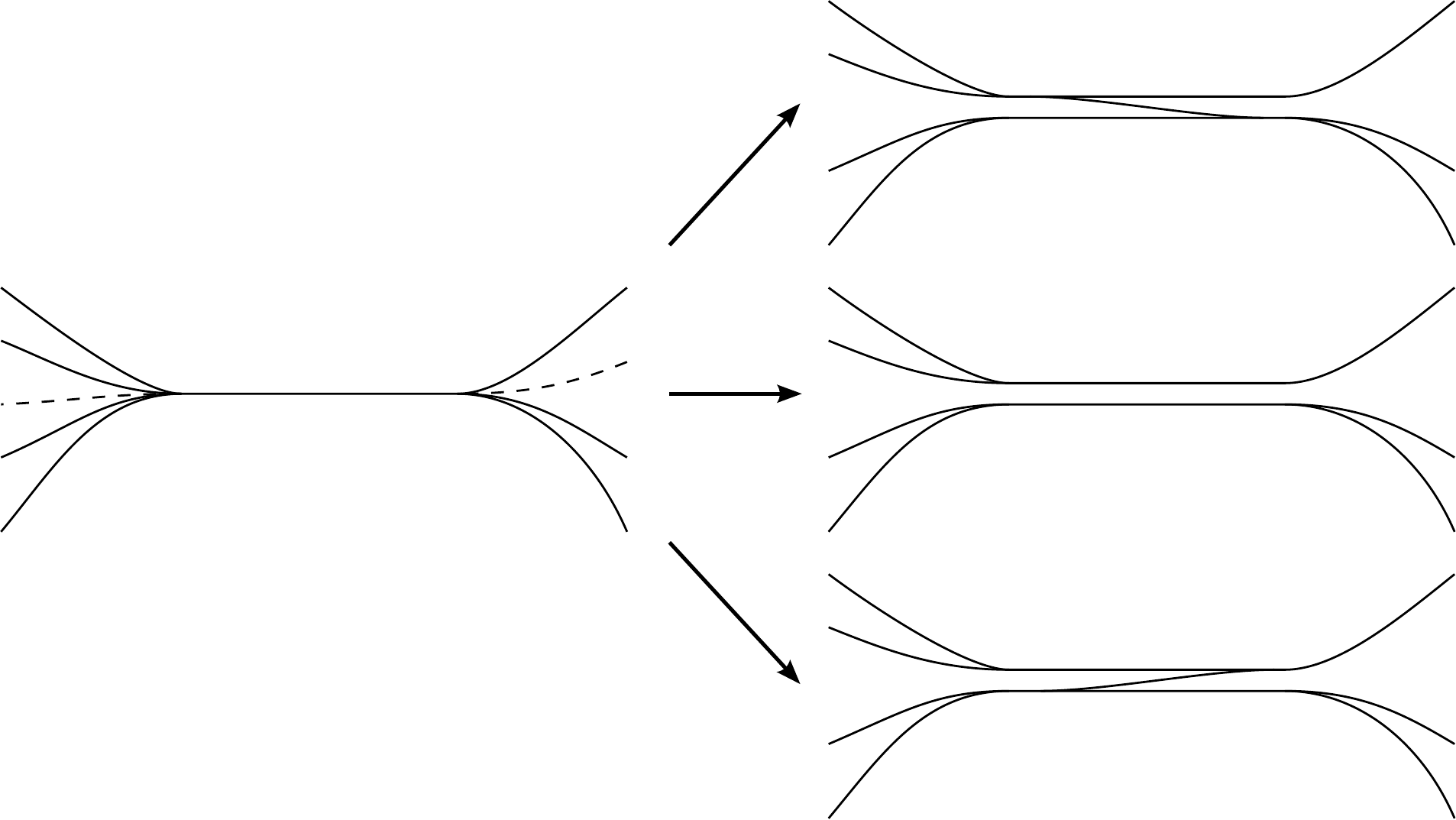}
\end{center}
\caption{\label{fig:ttsplitting} A \nw{split} is an elementary move acting on a large branch. The easiest case is the one of the \nw{central split} (middle arrow) which is intuitively understood as cutting along the large branch with a pair of scissor; the scissors enter a branch end from the gap between two small branches, and exit on the opposite side through a similar gap. The two copies of the split branch may be connected via a `diagonal' new branch: if the latter is placed as shown in the top arrow, we speak of a \nw{right split}; if it is as in the bottom arrow, it will be called a \nw{left split}. Left and right splits are collectively called \nw{parity splits}. In a more neutral way, we will refer to the property of a split being left, right or central also as its \nw{parity}; but we will not generate confusion.}
\end{figure}

\begin{defin}
A sequence of almost tracks $\bm\tau=(\tau_j)_{j=0}^N$, $N\in\mathbb N$, or $(\tau_j)_{j=0}^{+\infty}$, with all $\tau_j$ lying on the same surface $S$ is a \nw{splitting sequence} if, for all $0\leq j <N$, $\tau_{j+1}$ is obtained from $\tau_j$ via a comb, uncomb or split move. In the case of generic almost tracks, instead, we require each element of the sequence to be obtained from the previous one via a slide or a split (we may use the adjective \emph{generic} in reference to splitting sequences, too, in order to mark this difference). We denote with $|\bm\tau|$ the number of splits the splitting sequence $\bm\tau$ includes.

Given two indices $0\leq k\leq l$ ($\leq N$), we denote by $\bm\tau(k,l)=(\tau_j)_{j=k}^l$ the splitting sequence obtained from $\bm\tau$ by considering only the entries indexed by the indices between $k$ and $l$. When referring to splitting sequences, we use the term \nw{subsequence} only to refer to a sequence $\bm\tau(k,l)$ i.e. a subsequence never skips intermediate entries.

Given two splitting sequences $\bm\sigma$ and $\bm\tau$, with the last entry of $\bm\sigma$ isotopic to the first entry of $\bm\tau$, we denote by $\bm\sigma*\bm\tau$ the splitting sequence obtained by adjoining the entries of $\bm\tau$ at the end of $\bm\sigma$, in the given order and excluding the first one (so as not to have a repetition).
\end{defin}

Note that any semigeneric train track can be converted into a generic one via some uncomb moves. Actually, a splitting sequence of semigeneric train tracks can be converted into a splitting sequence of generic ones (i.e. not only each track may be made generic, but comb and uncomb moves may be replaced with slides in a suitable manner), and vice versa.

Combs, uncombs and slides are to be regarded as invertible and unessential moves, whereas splits in some sense `downgrade' the train track; this is made precise by the following statement, which collects Propositions 3.12.2 and 3.14.1 from \cite{mosher}:
\begin{prop}\label{prp:carriediffsplit}
Given two train tracks $\tau,\tau'$ on a surface $S$, $\tau'$ is fully carried by $\tau$ if and only if $\tau$ and $\tau'$ are, respectively, the beginning and the end of a splitting sequence.

Moreover, $\tau'$ is suited to $\tau$ if and only if $\tau$ and $\tau'$ are, respectively, the beginning and the end of a splitting sequence with no central splits.

Finally, $\tau'$ and $\tau$ carry each other if and only if $\tau$ and $\tau'$ are the extremes of a splitting sequence with no splits.
\end{prop}

\begin{rmk}\label{rmk:decreasingmeasures}
If a pretrack $\tau'$ is carried by another pretrack $\tau$, one can pick tie neighbourhoods $\nei(\tau')\subseteq \nei(\tau)$, with the ties of the former obtained as restriction of the ones of the latter. As a consequence, any curve that is carried by $\tau'$ is also carried by $\tau$: $\cc(\tau')\subseteq \cc(\tau)$. In the case of almost tracks, where uniqueness of carrying holds, given any $\alpha\in\cc(\tau')$ and the transverse measures $\mu'_\alpha$, $\mu_\alpha$ induced by it on $\tau'$, $\tau$ respectively, we have $\max_{b\in\br(\tau')}\left(\mu'_\alpha(b)\right)\leq \max_{b\in\br(\tau)}\left(\mu_\alpha(b)\right)$ and $\min_{b\in\br(\tau')}\left(\mu'_\alpha(b)\right)\leq \min_{b\in\br(\tau)}\left(\mu_\alpha(b)\right)$.

The above proposition, then, implies that the sequence of sets $\left(\cc(\tau_j)\right)_j$ along a splitting sequence is decreasing; and the equality $\cc(\tau_j)=\cc(\tau_{j+1})$ holds every time the two tracks are obtained from each other with an elementary move which is not a split.
\end{rmk}

\begin{rmk}\label{rmk:recurrence_at_extremes}
Let $\bm\tau=(\tau_j)_{j=0}^N$ be a train track splitting sequence.
\begin{itemize}
\item If $\tau_N$ is recurrent then so are all entries of $\bm\tau$.
\item If $\tau_0$ is transversely recurrent then so are all entries of $\bm\tau$ (cfr. \cite{penner}, Lemma 1.3.3(b)).
\end{itemize}
\end{rmk}

\begin{rmk}\label{rmk:pickparameters}
After introducing almost tracks and elementary moves, we can finally fix the parameters $k=k(S),\ell=\ell(S)$ employed in the definitions of $\ma(S)$ and the quasi-pants graph $\pa(S)$, as promised after Remark \ref{rmk:ell1}.

Here are the choices in detail for $\pa(S)$. We let $k'$ be the maximum self-intersection number of a collection $V(\tau)$ over all almost tracks $\tau$ on $S$ (which is finite because almost tracks are finitely many different ones up to diffeomorphisms of $S$; see also Lemma \ref{lem:vertexsetbounds}), and we let $k(S)\coloneqq k'$. This means that an almost track vertex set is a vertex of $\pa(S)$ if and only if it cuts $S$ as specified in either of the conditions a), b), c) of Definition \ref{def:quasipants}.

As for the choice of $\ell$, let $\ell_0$ be the maximum intersection number between $V(\tau)$ and $V(\sigma)$, for $\tau$ any almost track and $\sigma$ obtained from $\tau$ with a split or taking a subtrack (this is again finite because all possible pairs are finitely many up to diffeomorphism). Then define $\ell\coloneqq\max\{\ell_0,\ell_1\}$ for $\ell_1$ the one defined in Remark \ref{rmk:ell1}: this means that a splitting of almost tracks induces a displacement along an edge in $\pa(S)$.

When $X\subseteq S$ is a non-annular subsurface, we would like that, every time $V(\tau)$ is a vertex of $\pa(S)$, $\pi_X V(\tau)$ is a vertex of $\pa(X)$: so it will be understood that, when $X$ is regarded as a subsurface of $S$, then the bound on self-intersection number for vertices of $\pa(X)$ is not taken to be $k(X)$ but $k(X,S)\coloneqq \max\{k(X),4k(S)+4\}$ instead (this works, because of Remark \ref{rmk:subsurf_inters_bound}).

The choices for $\ma(S)$ shall be made with the same spirit: given any train track, if its vertex set fills $S$ then it must be a vertex of the graph. The vertex sets of any pair of train tracks related via an elementary move must be represented by vertices at distance $\leq 1$; and any vertex must have distance $\leq 1$ from a complete clear marking.

Again, we wish also that, when $X\subseteq S$ is a non-annular subsurface and $V(\tau)$ is a vertex of $\ma(S)$, then also $\pi_X V(\tau)$ is a vertex of $\ma(X)$. So, when $X$ is regarded as a subsurface of $S$, we pick a bound for the self-intersection number differently from the case in which $X$ is a stand-alone surface.

These choices for $\ma(S)$ and $\ma(X)$ do not coincide with the ones given at the beginning of \S 6 in \cite{mms} as they give higher constants, but they do not alter the validity of the results in that work, and in particular of their Theorem 6.1, as we will see.
\end{rmk}

\section{Train tracks: more constructions}\label{sec:traintracksmore}
\subsection{Lifting and inducing an almost track}\label{sub:induced}

If $\tau$ is an almost track on $S$, and $X\subseteq S$ is a subsurface, we denote with $\tau^X$ the pre-image of $\tau$ in $S^X$; it is a pretrack. The universal cover of $S$ will be denoted $\tilde S$; similarly, the notation $\tilde\tau$ will denote the pre-image of $\tau$ in $\tilde S$ via the covering map $\tilde S\rightarrow S$.

\begin{rmk}\label{rmk:negativeindexincover}
\textit{All connected components of $S^X\setminus \nei(\tau^X)$ or of $\tilde S\setminus\tilde\tau$ which are compact have negative index.}

The best way to see this (focusing on $S^X\setminus \nei(\tau^X)$, as the case of the universal cover is identical) is to take $\sigma$ an almost track which fills $S$, has $\tau$ as a subtrack, and has the property that each connected component of $S\setminus \nei(\sigma)$ which is a \emph{smooth} peripheral annulus is also a connected component of $S\setminus \nei(\tau)$.

Each compact connected component of $S^X\setminus \nei(\sigma^X)$ which is not a smooth peripheral annulus, then, is diffeomorphic to a suitable connected component of $S\setminus \nei(\sigma)$; as such, it has negative index.

So the compact connected components of $S^X\setminus \nei(\tau^X)$ are obtained from gluing a handful of negative index components of $S^X\setminus \nei(\sigma^X)$ with, possibly, some components as in Remark \ref{rmk:idx_of_nei_diff}. So they have negative index.
\end{rmk}

Some adaptations of the results from \S 3.3 in \cite{mosher} will be now given, to rule the behaviour of $\tilde\tau$ and $\tau^X$ at infinity.

\begin{prop}\label{prp:paths_in_univ_cover}
Let $\tau$ be an almost track on a surface $S$. Then
\begin{itemize}
\item all train paths along $\tilde\tau$ are embedded;
\item the images of any two train paths in $\tilde\tau$ intersect in a connected set.
\end{itemize}
\end{prop}

\begin{prop}
In the same setting as the above proposition,
\begin{itemize}
\item if $\rho, \rho'$ are two infinite train paths which have finite Haudorff distance in $\tilde S$, then they eventually coincide;
\item if $\rho, \rho'$ are two biinfinite train paths which have finite Haudorff distance in $\tilde S$, then they coincide entirely.
\end{itemize}
\end{prop}

The proof of these two propositions are exactly the same as Propositions 3.3.1 and 3.3.2 in \cite{mosher}.

\begin{prop}\label{prp:paths_quasi_geod}
Let $\tau$ be an almost track on a surface $S$. Then there exists $\lambda\geq 1$ (depending on $\tau$) such that every train path in $\tilde\tau$ is a $\lambda$-quasigeodesic in $\tilde S\cong\Hy^2$ with the hyperbolic metric.
\end{prop}
\begin{proof}[Sketch]
The proof of this statement is substantially the same as the proof of Proposition 3.3.3 in \cite{mosher}. Rather than repeating their proof, we give some instructions to adapt it to our setting.

Replace $\tau$ with an almost track which fills $S$, has the original $\tau$ as a subtrack, and chosen so that the collection of funnels with smooth boundary appearing as connected components of $S\setminus \nei(\tau)$ remains unvaried. By isotoping $\tau$, make sure that $\tau\subseteq \core(S)$ and that the closure of each connected component of $\core(S)\setminus \tau$ is a closed disc. This request is equivalent to demanding that $\tau\cup \partial\core(S)$ is a connected pretrack on $S$, entirely contained in $\core(S)$.

Let $\tilde K$ be the preimage of $\core(S)$ in the universal cover $\tilde S$: it is connected and convex, by definition of $\core(S)$. The same arguments as in the proof of Proposition 3.3.3 in \cite{mosher} can be used to show that:
\begin{itemize}
\item if $\core(S)$ is compact, then $\tilde\tau$ is quasi-isometric to $\tilde K$;
\item if $\core(S)$ has punctures, then there exists a pretrack $\sigma^\infty\subseteq \tilde K$ which contains $\tilde\tau$ as a subtrack and is quasi-isometric to $\tilde K$.
\end{itemize}

After that, one proves that any train path $\delta$ in $\tilde\tau$ is a quasigeodesic, following the original proof verbatim.
\end{proof}

As a consequence:
\begin{coroll}\label{cor:distinctends}
Let $\tau$ be an almost track on a surface $S$ and let $X\subseteq S$ be a subsurface. Every infinite train path in $\tau^X$ with noncompact image, and every infinite train path in $\tilde\tau$, has a unique limit point on $\partial\ol{S^X}$ ($\partial\ol{\tilde S}$, resp.).

If two paths as above have the same limit point, then they eventually coincide.

A biinfinite train path in $\tilde\tau$ has distinct endpoints on $\partial\ol{\tilde S}$, and there are no two distinct paths with the same pair of endpoints.
\end{coroll}

\begin{proof}
All the statements concerning train paths in $\tilde\tau$ are easy consequences of the above propositions.

As for an infinite path $\rho$ in $\tau^X$ with noncompact image: let $\pi_1(X)\cong\Gamma_X<\Isom(\tilde S)$ be the group such that $S^X=\tilde S/\Gamma_X$. Let $P$ be a disjoint, $\Gamma_X$-equivariant union of an open horoball for each point of $\partial\tilde S$ which is parabolic under the action of $\Gamma_X$, and with the property that $\tau^X\cap P=\emptyset$: $\left(H(L_{\Gamma_X})\setminus P\right)/\Gamma_X$ is a compact surface with boundary. 

Let $\tilde\rho$ be any lift of $\rho$ to $\tilde S$: $\tilde\rho$ is a quasigeodesic in $\tilde S$, as seen in the previous proposition. If its endpoint at infinity belongs to the limit set $L_{\Gamma_X}$ then there is an $\epsilon>0$ such that $\tilde\rho$, with the exception of an initial segment, falls entirely in the closed neighbourhood $N\coloneqq\bar\nei_\epsilon\left(H(L_{\Gamma_X})\right)$. Therefore $\rho$ lies eventually in the compact set $(N\setminus P)/ \Gamma_X$, which means that it is compact itself.

So the endpoint at infinity of $\tilde\rho$ belongs to the domain of discontinuity $D_{\Gamma_X}$; let $\tilde A$ be an open neighbourhood of this point such that $\tilde A\cap \tilde S$ gets mapped isometrically via the covering map $\tilde S\cup D_{\Gamma_X}\rightarrow \ol{S^X}$; and let $A$ be its image in $\ol{S^X}$. Then $\rho\cap A$, which excludes only an initial segment of $\rho$, is a quasigeodesic in $S^X$, and has a unique limit point on $\partial\ol{S^X}$.

If two infinite paths $\rho,\rho'$ in $\tau^X$ with noncompact image approach the same endpoint in $\partial\ol{S^X}$, then one may choose lifts $\tilde\rho,\tilde\rho'$ to $\tilde S$ which also approach the same point in $\partial\ol{\tilde S}$, and so they eventually coincide. But then, $\rho, \rho'$ also do.
\end{proof}

\begin{rmk}\label{rmk:limitsetconsistency}
If $\tau$ carries $\tau'$, then all train paths in $(\tau')^X$ are carried by $\tau^X$, too.

In particular, the set of points of $\partial S^X$ which are endpoints for some infinite, noncompact train path in $(\tau')^X$ is a subset of the similarly defined set for $\tau^X$.
\end{rmk}

The concept of subsurface projection finds a version for train tracks in \cite{mms}. Before we introduce it, we need a slight generalization of Definition \ref{def:carried}. If $X$ is an annular subsurface of a surface $S$ and $\beta\in \cc(X)$, a \emph{carried realization} $f:\beta \hookrightarrow \bar\nei(\tau^X)$ is specified by the same hypotheses as in the definition of carried realization of a curve, plus the extra one that the endpoints of $f(\beta)$ in $\partial\ol{S^X}$ are the same as the ones of $\beta$. This gives sense to the adjective \emph{carried}, hence to the notation $\cc(\tau^X)$ for the subset of $\cc^0(X)$ consisting of all arcs carried by $\tau^X$. We can adapt similarly the content of Definition \ref{def:trainpathrealization} to speak of a \emph{train path realization} of an arc $\beta$ carried by $\tau^X$.

\begin{defin}
Given $\tau$ an almost track on a surface $S$ and $X\subseteq S$ a subsurface, the \nw{track induced by $\tau$} on $X$ is $\tau|X$, the subtrack of $\tau^X$ consisting of those branches belonging to the carrying image of some element of $\cc(\tau^X)$.
\end{defin}

Note that $\cc(\tau|X)=\cc(\tau^X)$ and that, if $X$ is not an annulus, then $\cc(\tau|X)=\cc(\tau)\cap\cc^0(S^X)= \cc(\tau)\cap\cc^0(X)$. If $X$ is an annulus, then each element of $\cc(\tau)$ which intersects $X$ essentially lifts to an element of $\cc(\tau|X)$, but it is not true in general that an element of $\cc(\tau|X)$ projects to an element of $\cc(\tau)$.

Following \cite{mms}, we define $V(\tau|X)\subset \cc(\tau|X)$ to be the set of \emph{wide} carried arcs. Again, the definition of wide carried arc here is not any different from the definition of wide carried curve in an almost track.

\begin{lemma}\label{lem:inducedisonsurface}
If $X$ is not an annulus, then the induced track $\tau|X$ is an almost track on $S^X$. Moreover, if $r:\ol{S^X}\rightarrow \core(S^X)=X\cap\core(S)$ is a retraction, then $r(\tau|X)\eqqcolon\rho$ is a pretrack contained in $\inte(X)$ which is an almost track in $S$, too.
\end{lemma}

Let $p:S^X\rightarrow S$ be the covering map. In the statement we have already implicitly identified $p^{-1}(X)$ with $X$, and $\core(S^X)$ with $X\cap \core(S)$, using the fact that these pairs are isometric via $p$. As it is noted in \cite{mms}, Lemma 3.1 and subsequent remarks, $\tau|X$ need not be a train track when $\tau$ is.

\begin{proof}
We prove directly that $\rho$ is an almost track in $S^X$; clearly this will mean that $\tau|X$ is, too. Also, we may suppose that $X$ has geodesic boundary\footnote{If $X$ has two isotopic boundary components in $S$, their geodesic representatives end up coinciding. But this is no cause for concern, as the boundary components of $X\subseteq S^X$ are pairwise not isotopic.}. 

We claim that the compact connected components of $\core(S^X)\setminus\nei(\rho)$ either have negative index or are compact peripheral annuli in $X$. If one of these components, $C$ say, were not either of these two alternatives, then $\idx(C)\geq 0$ and $C$ would lie away from $\partial \core(S^X)$. But then, consider $r^{-1}(C)$: it would be contained in $S^X$, hence it would be a compact connected component of $S^X\setminus\nei(\tau^X)$. So $\idx\left(r^{-1}(C)\right)<0$, because it would be the gluing of patches as in Remark \ref{rmk:negativeindexincover} plus other ones ruled by Remark \ref{rmk:idx_of_nei_diff}. This is a contradiction.

We will now exclude that $\core(S^X) \setminus \nei(\rho)$ includes a connected component which is a cusp $P$ with smooth boundary $\partial P$. If there is one, then there are a periodic train path in $\rho$, and one in $\tau|X$, which (up to the due reparametrization) are isotopic to $\partial P$; necessarily, there is no closed geodesic in $S^X$ which is isotopic to $\partial P$. 

There is a chain of inclusions of closed subsets $p(P)\subseteq p\left(\core(S^X)\right)\subseteq \core(S)\subseteq S$ (again, remember that $p(P)\cong P$ and $p\left(\core(S^X)\right)\cong \core(S^X)$); note that there cannot be any closed geodesic in $S$ which is isotopic to $\partial p(P)$, either.

On the other hand, $\partial P=p(\partial P)$ is also an embedded loop in $S$, isotopic to a (reparametrized) train path along $p(\tau^X)\subseteq \tau$. However, since $\tau$ is an almost track, is does not admit periodic train paths encircling cusps, hence a contradiction.

The argument carried out so far proves that any smooth connected component of $\partial\bar\nei(\rho)$ which is inessential in $\core(S^X)$ (which is the same as saying, inessential in $S^X$) is parallel to a closed geodesic, one of the connected components of $\partial\core(S^X)$. This means that $\rho$ is an almost track in $S^X$, and so is $\tau|X$.

Now we prove that $p(\rho)$ (which can be identified with $\rho$ itself) is an almost track on $S$. Each connected component of $S\setminus p\left(\nei(\rho)\right)$ is either contained in $X$ and, in this case, it is a copy of a connected component of $p^{-1}(X)\setminus \nei(\rho)$; or it contains entirely a closed geodesic, a connected component of $\partial X$. None of these complementary components may transgress the conditions that make $p(\rho)$ an almost track in $S$.
\end{proof}

\begin{rmk}\label{rmk:annulusinducedbasics}
Let $\gamma\in\cc(S)$, let $X$ be a regular neighbourhood of $\gamma$ and let $\tau$ be an almost track on $S$. We list here a few basic facts about the induced track $\tau|X$.

\begin{enumerate}
\item If $\gamma$ admits a carried realization in $\tau|X$ then it admits one in $\tau$. Note, though, that the first one is certainly embedded as a train path (see below), whereas the latter need not be. The converse is not true in general: if $\gamma$ is carried by $\tau$ but $\cc(\tau^X)=\emptyset$ then $\gamma$ is not carried by $\tau|X$. On the other hand, as we note in point \ref{itm:gammacarriedininduced} below, $\gamma$ is carried by $\tau|X$ if $\cc(\tau^X)\not=\emptyset$.

\item\label{itm:embeddedcore} \textit{When $\gamma$ is carried by $\tau^X$, it is embedded as a train path. The realization as a train path is unique up to obvious reparametrizations.}

Uniqueness of carrying of $\gamma$ provided the existence of a carried realization is just a consequence of the uniqueness of carrying of $\gamma$ in $\tau$ (Corollary \ref{cor:carryingunique}). 

We prove that the train path realization of $\gamma$ is embedded by contradiction. Suppose first that $\tau^X.\gamma$ is not a simple curve: then (forgetting about it being a pretrack) it is a compact 1-complex with two or more distinct simple cycles. As $S^X$ is a planar surface, this means that $\ol{S^X}\setminus\nei(\tau^X.\gamma)$ consists of at least $3$ connected components. Considering that $\tau^X.\gamma$ keeps away from $\partial\ol{S^X}$, and that $\gamma$ is not nullhomotopic in $S^X$, two of these components include each one component of $\partial\ol{S^X}$ and are topological annuli, possibly with a bunch of outward right angles in their boundaries. So their index is not positive.

Any other connected component $C$ is also a compact connected component of $S^X\setminus\nei(\tau^X.\gamma)$, and it must have $\idx(C)<0$ because it is a gluing of regions as in Remarks \ref{rmk:negativeindexincover} and \ref{rmk:idx_of_nei_diff}. This means that $\ol{S^X}$ is a gluing of regions with index $\leq 0$, and one of them has certainly index $<0$. But this is impossible since $\idx(\ol{S^X})=0$.

So $\tau^X.\gamma$ is a simple curve, not nullhomotopic in $S^X$. If $\gamma$ traversed any branch there more than once, it would traverse at least twice all of them, and always keeping the same orientation. But this means that $\gamma$ would not be a generator for $\pi_1(S^X)$, thus it contradicts the definition of $X$.

Before we continue with more observations, we give a definition:
\begin{defin}\label{def:eorientation}
Let $\tau$ be an almost track, let $\gamma\in \cc(S)$ and let $X$ be a regular neighbourhood of $\gamma$. Suppose that $\gamma$ is carried by $\tau$ (and by $\tau^X)$ and that $e$ is a branch end of $\tau^X$ such that $e\cap \tau^X.\gamma=\{v\}$, where $v$ is the switch of $\tau^X$ which serves as endpoint of $e$. Consider a train path $\beta$ in $\tau^X$ which traverses the branch containing $e$ pointing towards $v$, and then follows $\gamma$ until it reaches $v$ again. The \nw{$e$-orientation} on $\gamma$ is the orientation specified by the path $\beta$.
\end{defin}

\item\label{itm:uniquecarrying} \textit{Uniqueness of carrying on $\tau^X$ for an element $\alpha\in \cc(\tau^X)$ holds, similarly as for train tracks and almost tracks.}

It is enough to show that there is only one train path realization of $\alpha$, up to reparametrization. Let then $\ul\alpha_1,\ul\alpha_2$ be two train path realizations of $\alpha$: they have the same endpoints on $\partial\ol{S^X}$. Consider each of these two realizations as the union of two infinite train paths approaching opposite components of $\partial\ol{S^X}$: then, by Corollary \ref{cor:distinctends}, they coincide outside a compact subset of $S^X$.

Since $\ul\alpha_1,\ul\alpha_2$ are homotopic arcs, there are lifts $\tilde{\ul\alpha}_1,\tilde{\ul\alpha}_2$ of the two on $\tilde S$ which approach the same pair of points at infinity. By Corollary \ref{cor:distinctends} $\tilde{\ul\alpha}_1,\tilde{\ul\alpha}_2$ must coincide, and so do $\ul\alpha_1,\ul\alpha_2$.

\item\label{itm:windaboutgamma} \textit{Let $\alpha\in\cc(\tau^X)$. If $\alpha$ traverses at least one of the branches of $\tau^X$ twice, then $\gamma$ is carried by $\tau^X$ and all branches traversed more than once by $\alpha$ must belong to $\tau^X.\gamma$.}

The carrying image $\tau^X.\alpha$ is a generalized pretrack. Among the connected components of $S^X\setminus\nei(\tau^X.\alpha)$ there is no topological closed disc: if there is one, which we call $D$, then lift $\alpha$ to an arc $\tilde\alpha$ in $\tilde S$. This is homotopic, with fixed extremes, to a biinfinite train path along $\tilde\tau$: let $\tilde\tau.\tilde\alpha$ be its image. $D$ lifts homeomorpically to a connected component of $\tilde S\setminus\nei(\tilde\tau.\tilde\alpha)$; but the presence of a topological disc among them is a contradiction to Proposition \ref{prp:paths_in_univ_cover}, according to which train paths along $\tilde\tau$ are embedded.

Encircle each component of $\partial\ol{S^X}$ with a smooth loop which intersects $\nei(\tau^X.\alpha)$ in exactly one tie. Let $Y$ be the compact annulus delimited by these two loops in $S^X$. As a consequence of the previous paragraph, each connected component of $Y\setminus \nei(\tau^X.\alpha)$ includes part of $\partial Y\setminus\nei(\tau^X.\alpha)$, which consists of exactly 2 components. The connected components of $Y\setminus \nei(\tau^X.\alpha)$, then, are at most 2.

If $Y\setminus \nei(\tau^X.\alpha)$ is connected, then $\idx\left(Y\setminus \nei(\tau^X.\alpha)\right)=0$, by index additivity since $\idx(Y)=\idx\left(\bar\nei(\tau^X.\alpha)\right)=0$. This means that $Y\setminus \nei(\tau^X.\alpha)$ is a rectangle i.e. $\partial\bar\nei(\tau^X.\alpha)$ consists of two smooth components, and they are both isotopic to $\alpha$: so $\alpha$ will be embedded.

If, instead, the components of $Y\setminus \nei(\tau^X.\alpha)$ are two, then each of them is necessarily homeomorphic to a disc. So $\left((\tau^X.\alpha)\cap Y\right)\cup\partial Y$ is a 3-valent graph, hence with the property that $2\#\text{(edges)}=3\#\text{(vertices)}$, which cuts the annulus $Y$ into two cells. 

Computation of the Euler characteristic yields that the graph has $6$ edges and $4$ vertices. This implies that $\tau^X.\alpha$ is a generalized pretrack with two switches, two compact branches, and two more branches heading towards the two opposite components of $\partial\ol{S^X}$.

\begin{figure}[h]
\centering
\includegraphics[width=.25\textwidth]{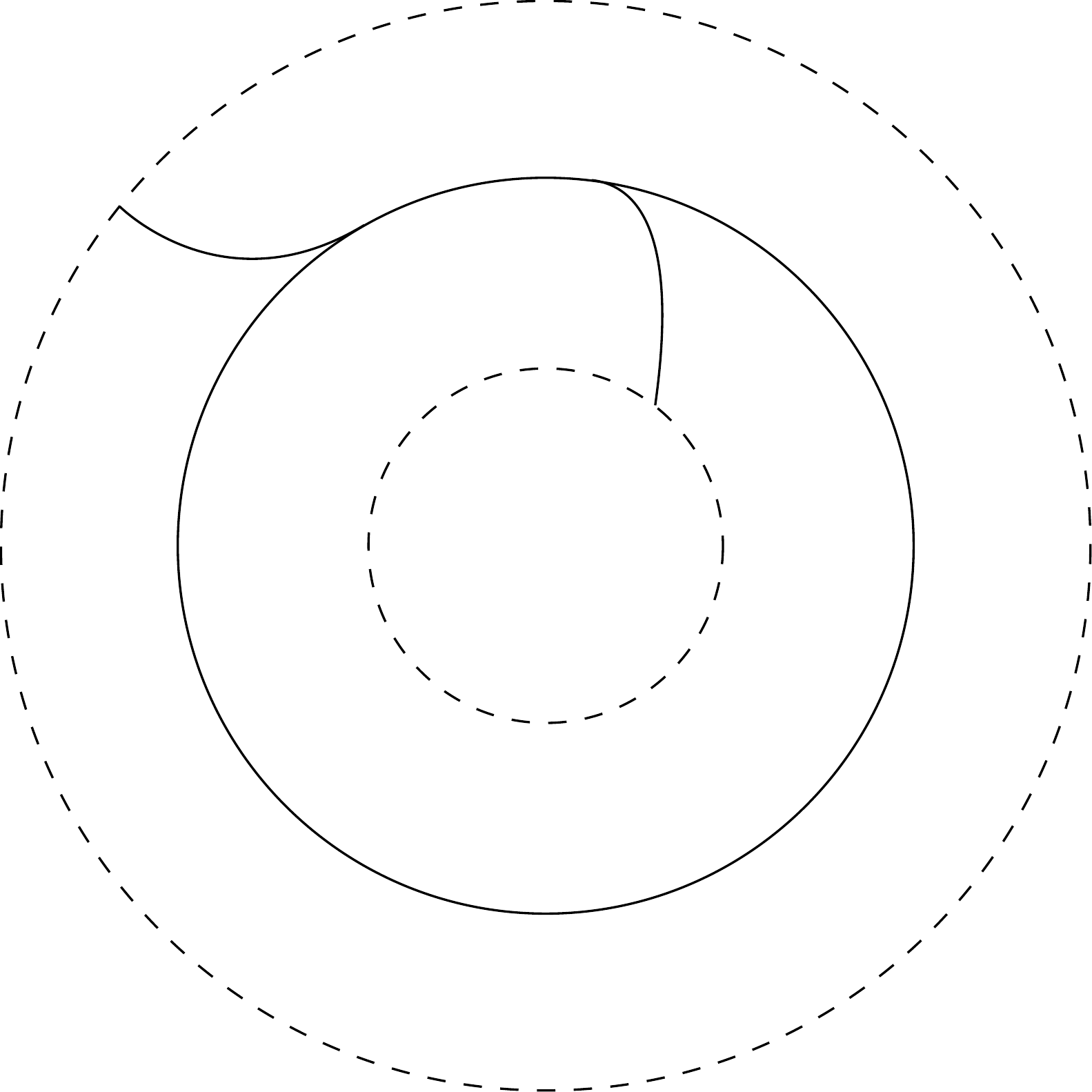}
\caption{\label{fig:arcinannularcover}The only possibility for $\tau^X.\alpha$ (Remark \ref{rmk:annulusinducedbasics}, point \ref{itm:windaboutgamma}) provided that $\alpha$ does not have an embedded realization as a train path, up to a diffeomorphism of $S^X$ (possibly an orientation-reversing one). The arc $\alpha$ may wind around the core curve an arbitrarily high number of times.}
\end{figure}

There is only one diffeomorphism type of generalized pretrack in $S^X$ with these properties and that carries a path from one component of $\partial\ol{S^X}$ to the other, traversing all branches of the pretrack: $\tau^X.\alpha$ must belong to this class (Figure \ref{fig:arcinannularcover}). So $\tau^X.\alpha$, in particular, carries $\gamma$.

Necessarily $\alpha$, after traversing the first infinite branch of $\tau^X.\alpha$, winds around $\tau^X.\gamma$ traversing each branch always in the same direction, and eventually leaves it to traverse the remaining infinite branch of $\tau^X.\alpha$.

The structure of an arc $\alpha\in \cc(\tau^X)$ when $\gamma$ is carried can therefore be summarized as follows.

\item \label{itm:horizontalstretch} \textit{Suppose that $\tau^X$ carries $\gamma$. Let $\alpha\in\cc(\tau^X)$ and let $\ul\alpha$ be a realization of $\alpha$ as a train path along $\tau^X$. Then $\ul\alpha$ consists of the concatenation of three segments $\rho_1,\beta,\rho_2$ such that: $\beta$ traverses branches of $(\tau|X).\gamma$, always in the same direction; $\rho_1$ and $\rho_2$ are embedded train paths, connecting $\tau^X.\gamma$ with distinct components of $\partial\ol{S^X}$ and not traversing any branch in $\tau^X.\gamma$. The branch ends of $\rho_1\rho_2$ which are adjacent to $\tau^X.\gamma$ induce opposite orientations on $\gamma$.}

We will denote $\beta=\hs(\tau,\alpha)$ and call it the \nw{horizontal stretch} of $\alpha$ in $\tau$. We have that $\alpha\in V(\tau^X)$ if and only if $\hs(\tau,\alpha)$ is embedded. In this case we may abusively identify $\hs(\tau,\alpha)$ with its image $\tau^X.\gamma\cap\tau^X.\alpha$.

The paths $\rho_1,\rho_2$ also deserve a name. When $\tau^X$ carries $\gamma$, we call an \nw{outgoing ramp} a train path $\rho:[0,+\infty)\rightarrow \tau^X$ such that $\rho(0)\in\tau^X.\gamma$, and converging to a point in $\partial\ol{S^X}$ without traversing any branch in $\tau^X.\gamma$. An \nw{ingoing ramp} is a train path $\rho:(-\infty,0]\rightarrow \tau^X$ such that the map $\rho'(x)\coloneqq \rho(-x)$ is an outgoing ramp.

\item\label{itm:dataforramp} \textit{Suppose two distinct outgoing ramps for $\tau^X$ diverge towards the same point of $\partial\ol{S^X}$. Then their initial branch ends give opposite orientations to $\gamma$.}

Let $\rho_1,\rho_2:[0,+\infty)\rightarrow \tau^X$ be the two ramps. By Corollary \ref{cor:distinctends} they eventually coincide so the subtrack $\sigma=(\tau^X.\gamma)\cup\mathrm{im}(\rho_1)\cup\mathrm{im}(\rho_2)$ has, among the components of $S^X\setminus\nei(\sigma)$, a compact component which is a compact topological disc. It has at least two outward corners in the boundary --- at the confluence of $\rho_1,\rho_2$ --- but, as it is a union of components as in Remarks \ref{rmk:idx_of_nei_diff} and \ref{rmk:negativeindexincover}, its index is negative and the corners in the boundary must be at least three more. This is possible only if it has corners both at $\rho_1(0)$ and at $\rho_2(0)$, which means that their initial branch ends induce opposite orientations on $\gamma$.

\item \textit{The elements of $V(\tau^X)$ are exactly the ones of $\cc(\tau^X)$ which are embedded as train paths and, if $\tau|X$ does not carry $\gamma$, then $V(\tau|X)=\cc(\tau|X)$} (cfr. Lemma 3.5 in \cite{mms}). \textit{If $\tau|X$ carries $\gamma$ then, for any $\alpha\in\cc(\tau|X)$, there exist $\beta\in V(\tau|X)$, $m\in\mathbb Z$, such that $\alpha=D_X^m(\beta)$. Here $D_X$ is the self-diffeomorphism \emph{of $S^X$} given by the Dehn twist about its core curve $\gamma$.}

The first sentence is just a direct consequence of point \ref{itm:windaboutgamma} above. As for the second one: consider the decomposition given in the above point. So $\alpha$ fails to be wide carried if and only if $\hs(\tau,\alpha)$ is not an embedded train path, and this occurs if and only if it traverses all branches of $\tau^X.\gamma$. But then one between $D_X(\ul\alpha)$ and $D_X^{-1}(\ul\alpha)$ traverses each branch of $\tau^X.\gamma$ once less than $\ul\alpha$. Repeat Dehn twisting until the arc thus obtained has an embedded horizontal stretch.

\item \label{itm:gammacarriedininduced} A similar argument as in the point above yields a weaker form for the missing implication in point 1: if $\gamma$ is carried by $\tau$ (hence by $\tau^X$, too) and $\cc(\tau^X)\not=\emptyset$ then $\gamma$ is carried by $\tau|X$. Given any $\alpha\in\cc(\tau^X)$ then, up to replacing $\alpha$ with the correct $D_X^{\pm 1}(\alpha)\in\cc(\tau^X)$, we can suppose that it traverses all branches of $\tau^X.\gamma$. So those branches are part of $\tau|X$.

\item \label{itm:diambound}\textit{Given any two $\alpha_1,\alpha_2\in V(\tau|X)$, $i\coloneqq i(\alpha_1,\alpha_2)\leq 4$ --- hence $d_{\cc(X)}(\alpha_1,\alpha_2)\leq 5$.}

Consider any two carried realizations $\ul\alpha_1,\ul\alpha_2$ which realize the intersection number $i$ between their isotopy classes; we identify them with their images in $S^X$. We may suppose that $i>0$. Perform a series of isotopies on $\ul\alpha_1,\ul\alpha_2$, transversely to the ties of $\nei(\tau|X)$, resulting in two new realizations $\ul\alpha'_1,\ul\alpha'_2$. These isotopies shall be performed in such a way that each transverse intersection between $\ul\alpha_1$ and $\ul\alpha_2$ turns into an entire \emph{segment} of intersection between $\ul\alpha'_1$ and $\ul\alpha'_2$, without introducing new connected components of intersection. Informally, we may say that the local picture around each component of $\ul\alpha'_1\cap\ul\alpha'_2$ is the same as the one around a large branch in a pretrack.

Let then $\sigma\coloneqq \ul\alpha'_1\cup\ul\alpha'_2$: it is a generalized pretrack, carried by $\tau|X$. It has no mixed branches: the existence of one would imply that one of the two arcs traverses some branch of $\tau|X$ more than once and this contradicts the fact that they are embedded as train paths. Fix a tie neighbourhood $\bar\nei(\sigma)\subseteq \nei(\tau|X)$: then each compact connected component $C$ of $S^X\setminus \nei(\sigma)$ has negative index because, as $\alpha_1,\alpha_2$ traverse each branch of $\tau|X$ at most once and we have chosen representatives which intersect minimally, $C$ includes at least one compact component of $S^X\setminus\nei(\tau|X)$; Remarks \ref{rmk:idx_of_nei_diff} and \ref{rmk:negativeindexincover} apply. So this means that, if a component of $S^X\setminus \sigma$ has compact closure, then it is a topological disc and contains at least three outward cusps in its boundary.

Now, $(\ul\alpha'_1\cup\ul \alpha'_2\cup \partial Y)\cap Y$ as a graph has $2i+4$ vertices (out of which $4$ lie along $\partial Y$ and the others are switches), and they are all trivalent. The edges of the graph are $i+2(i+1)+4$ (here they are counted as large branches + small branches + edges $\partial Y$ is cut into). By an Euler characteristic computation, this graph cuts $Y$ into $i+2$ regions, each homeomorphic to a disc. Out of these, exactly $4$ have an edge along $\partial Y$; so the remaining $i-2$ ones are identifiable with the connected components of $S^X\setminus\sigma$ with compact closure.

We have noted that each of these latter regions has at least $3$ outward cusps in its boundary; moreover, out of the four regions which are adjacent to $\partial Y$, exactly two must contain an outward cusp. In total, the number of outward cusps in $\bar\nei_0(\sigma)$ is $2i$, because each large branch of $\sigma$ provides two of them. Hence $2i\geq 3(i-2)+2$ and $i\leq 4$.
\end{enumerate}
\end{rmk}

\begin{lemma}\label{lem:decreasingfilling}
Let $\tau$ be an almost track on $S$. Then the subsurface $S'$ of $S$ --- not necessarily a connected one, and possibly including annular components --- filled by the curves in $V(\tau)$ is also the subsurface filled by $\cc(\tau)$.

Moreover, if $C^1,\ldots, C^k$ are the connected components of $S'$ then, for each $i$ such that $C^i$ is non-annular, $V(\tau|C^i)=V(\tau)\cap\cc(C^i)$.
\end{lemma}

\begin{proof}
\step{1} case of $\tau$ a train track.

Use the notation $\tau\|i$ to mean $\tau|C^i$ if $C^i$ is not an annulus, and the core curve of $S^{C^i}$ otherwise.

We identify each $C^i$ with its homeomorphic copy in $S^{C^i}$: this copy is a deformation retract of $\ol{S^{C^i}}$. With this convention, for each $v\in V(\tau)$ there is exactly one $i=i(v)$ such that $C^i$ contains a curve isotopic to $v$, and this curve either is essential there (for $C^i$ not an annulus) or is the core curve (for $C^i$ an annulus). Also, $v$ is carried by the respective $\tau\|i(v)$, enforcing the identification $\cc(S^{C^{i(v)}})=\cc(C^{i(v)})$.

For each $1\leq i\leq k$, let $Y^i\subset S^{C^i}$ be a surface (\emph{not} an essential subsurface of $S^{C^i}$ according to the definition used so far), defined as follows:
\begin{itemize}
\item if $C^i$ is not an annulus, fix a collection of peripheral annuli in $S^{C^i}$, one for each topological puncture, disjoint from each other and from $\bar\nei(\tau\|i)$; and remove one point from each of them (similarly as it is done in Corollary \ref{cor:carryingunique});
\item if $C^i$ is an annulus, remove one point from each connected component of $C^i\setminus \tau\|i$.
\end{itemize}
 
Let $\tau^i=\tau\|i$ setwise, but considered as a train track on $Y^i$. Then we may identify ${\mathcal M}_{\mathbb Q}(\tau^i)={\mathcal M}_{\mathbb Q}(\tau\|i)$ and there is a natural, linear map $f^i:{\mathcal M}_{\mathbb Q}(\tau^i)\rightarrow {\mathcal M}_{\mathbb Q}(\tau)$.

For each element $v \in V(\tau)$, $\tau^{i(v)}$ carries a unique curve $\tilde v$, essential in $Y^{i(v)}$, such that the chain of natural maps $Y^{i(v)}\hookrightarrow S^{C^{i(v)}} \twoheadrightarrow \inte(C^{i(v)})\hookrightarrow S$ sends $\tilde v$ to a loop homotopic to $v$.

For each $\alpha\in\cc(\tau)$ we denote $\mu_\alpha$ the corresponding measure in ${\mathcal M}_{\mathbb Q}(\tau)$; while, for $\beta\in\cc(\tau^i)$, we denote $\nu^i_\beta$ the corresponding measure in ${\mathcal M}_{\mathbb Q}(\tau^i)$ (see Definition \ref{def:transversemeasure}). For all $v\in V(\tau)$, $f^{i(v)}(\nu^{i(v)}_{\tilde v})=\mu_v$, implying in particular that $\tilde v\in V(\tau^{i(v)})$: if $\nu^{i(v)}_{\tilde v}$ could be written nontrivially as a sum in ${\mathcal M}_{\mathbb Q}(\tau^i)$, then also $\mu_v$ could be written nontrivially as a sum in in ${\mathcal M}_{\mathbb Q}(\tau)$. If $C^{i(v)}$ is not an annulus then, by definition of $i(v)$, $v\in \cc(\tau\|i(v))$ and $\nu^{i(v)}_v=\nu^{i(v)}_{\tilde v}$ under the identification ${\mathcal M}_{\mathbb Q}(\tau^i)={\mathcal M}_{\mathbb Q}(\tau\|i)$, so $v\in V(\tau\|i(v))$. In particular this proves the inclusion $V(\tau\|i)\supseteq V(\tau)\cap \cc(C^i)$ for all $1\leq i\leq k$ such that $C^i$ is not an annulus.

Fix now $\alpha\in \cc(\tau)$; then in ${\mathcal M}_{\mathbb Q}(\tau)$ there is an equality $\mu_\alpha=\sum_{v\in V(\tau)} a_v\mu_v$ for $a_v\in {\mathbb Q}_{\geq 0}$; note that $\mu_\alpha=\sum_{i=1}^k f^i(\nu^i)$ where $\nu^i\in {\mathcal M}_{\mathbb Q}(\tau^i)$, $\nu^i= \sum_{v\in V(\tau) \text{ s.t. } i(v)=i} a_v\nu_{\tilde v}$.

As a consequence of Proposition \ref{prp:measurecurvecorresp}, for each $1\leq i\leq k$ it is also possible to write $\nu^i=\sum_{j=1}^{s(i)} c^i_j \nu_{\beta^i_j}$ for $c^i_j\in{\mathbb Q}_{> 0}$ and $\{\beta^i_1,\ldots,\beta^i_s\}\subset \cc(\tau^i)$ a family of pairwise disjoint curves. But each train path realization $\ul\beta^i_j$ of $\beta^i_j$ in $\tau^i$ is also a train path in $\tau\|i$; as such, it is homotopic to an embedded loop $\hat \beta^i_j$, entirely contained in $C^i$, and possibly homotopic into a puncture of $C^i$.

The covering map $\hat p^i:S^{C^i}\rightarrow S$ turns this into a homotopy in $S$. So each $\hat p^i(\hat \beta^i_j)$, which can be considered to coincide with $\hat \beta^i_j$, is a simple closed curve in $S$, homotopic to the periodic train path $\hat p^i(\ul\beta^i_j)$ along $\tau^i$. Also, $\hat p^i(\hat \beta^i_j)$ is an essential curve in $S$, as no train path in $\tau$, which is a train track, may represent a null-homotopic or puncture-homotopic curve.

For each $i$, $j$, let $\alpha^i_j$ be the isotopy class of $\hat p(\hat\beta^i_j)$. This defines a family of isotopy classes of curves in $S$ which are admit pairwise disjoint realizations, and each of them may be supposed to lie entirely in the respective $C^i$. So $\mu_\alpha=\sum_{i=1}^k f^i(\nu^i) = \sum_{i=1}^k \sum_{j=1}^{s(i)} c^i_j \mu_{\alpha^i_j}$. The bijection guaranteed by Proposition \ref{prp:measurecurvecorresp}, then, implies that this double summation has only one entry, identical to the original $\mu_\alpha$. This implies that $\alpha$ is contained in $S'$, so $\cc(\tau)$ fills $S'$ as well as $V(\tau)$.

If $\alpha\in V(\tau\|i)$ for $C^i$ not an annulus, suppose that $\alpha\not\in V(\tau)$. Anyway $\alpha\in\cc(\tau)$, and the argument above proves, in particular, that in the expression $\mu_\alpha=\sum_{v\in V(\tau)} a_v\mu_v$ the coefficients $a_v\not=0$ all have the same $i(v)\eqqcolon i$. But then $\nu^i_\alpha=\sum_{v\in V(\tau)\cap \cc(C^i)} a_v\nu_v$ implying that only one $a_v\not=0$ because $\alpha$ is a vertex cycle for $\tau\|i$. So $V(\tau\|i)\subseteq V(\tau)\cap \cc(C^i)$.

\step{2} $\tau$ is only an almost track on $S$.

Similarly as above, let $T$ be the surface obtained from $S$ by picking a collection $P$ of peripheral annuli, one for each topological puncture, disjoint from each other and from $\bar\nei(\tau)$, and removing one point from each of them. Then $\tau_T=\tau$ is a train track on $T$, and $\cc(\tau_T)$ and $V(\tau_T)$ fill the same, possibly disconnected, subsurface $T'$ of $T$; up to isotopies, we may suppose that, every time a connected component of $\partial T'$ is isotopic to one of $\partial P$, they actually coincide. Let then $T''$ be the subsurface of $S$ consisting of all connected components of $T'\cup P$ which are not peripheral annuli. There is a natural bijection between the connected components of $T'$ and the ones of $T''$, and we claim that $T''$ is the subsurface of $S$ filled by both $V(\tau)$ and $\cc(\tau)$.

By definition, $V(\tau)$ may be identified with the subset of $V(\tau_T)$ consisting of all curves in this latter set which are not homotopically trivial in $S$; i.e. the ones that are not homotopic to a connected component of $\partial P$. An application of the same idea as in the proof of Corollary \ref{cor:carryingunique} guarantees that, among these elements of $V(\tau_T)$, no two distinct ones become isotopic in $S$. Fix an embedding in $T'$ for all elements in $V(\tau_T)$.

We prove first that, given any curve $\alpha\in\cc(T'')\subseteq \cc(S)$ there is an element of $V(\tau)$ that intersects it essentially. Isotope $\alpha$ so that it lies entirely in $T'$, and to be disjoint from all the chosen embeddings in $T'$ of elements in $V(\tau_T)\setminus V(\tau)$. Consider what happens in the surface $T$: since $V(\tau_T)$ fills $T'$, there exists an element of $V(\tau_T)$ which intersects $\alpha$ however $\alpha$ is isotoped within $T'$. This element is necessarily an element of $V(\tau)$, because we have already excluded the remaining elements of $V(\tau_T)$.

Now note that each element of $\cc(\tau)$ admits an embedding contained in $T'$, and is essential in $S$: so it has actually an embedding in $T''$. This means that $\cc(\tau)$ fills $T''$ again, because it cannot fill any higher complexity surface properly containing $T''$.

It has been proved above that, for any $Z'$ non-annular connected component of $T'$, $V(\tau_T|Z')=V(\tau_T)\cap \cc(Z')$. Let $Z''$ be the corresponding connected component of $T''$: then $(T^{Z'},\tau_T|Z')$ and $(S^{Z''},\tau|Z'')$ are easily identified. Discarding, in both sides of the given equality, any curve homotopic to a component of $\partial P$, we have $V(\tau|Z'')=V(\tau)\cap \cc(Z'')$ as required.
\end{proof}

\begin{rmk}
For $Y\subseteq X$ nested subsurfaces of $S$, $\tau$ an almost track on $S$, $\bm\tau=(\tau_j)_{j=1}^N$ a splitting sequence of almost tracks, the following facts, following partly from the above lemma, hold. Recall that the parameters for $\pa$ and $\ma$, chosen in Remark \ref{rmk:pickparameters}, are suited to `work well' with subsurface projections.

\begin{itemize}
\item If the collection of curves $\pi_X\left(V(\tau)\right)$ fills $X$, or the collection $V(\tau)$ does, then $\pi_Y\left(V(\tau)\right)$ fills $Y$. Also, if $V(\tau)$ or $\pi_X\left(V(\tau)\right)$ is a vertex of $\pa(X)$, then $\pi_Y\left(V(\tau)\right)$ is a vertex of $\pa(Y)$. These two facts are seen more easily using the version a) in Definition \ref{def:quasipants} and the definition of filling in terms of constraints on the intersection pattern between curves; the self-intersection number of the considered families of curves is no cause for concern, thanks to the choices made in Remark \ref{rmk:pickparameters}.
\item For $1\leq j\leq N$, the sets $\cc(\tau_j)$ make up a decreasing family (Remark \ref{rmk:decreasingmeasures}): so the lemma yields that, along the sequence $\bm\tau$, the subsurfaces of $S$ filled by $V(\tau_j)$ are also a decreasing family with respect to inclusion. If $X$ is a fixed subsurface of $S$, the same is true of the subsurfaces of $S^X$ (or of $X$ itself, via Lemma \ref{lem:inducedisonsurface}) filled by $V(\tau_j|X)$.
\item If, along a splitting sequence, for a fixed subsurface $X\subseteq S$, $\pi_X V(\tau_j)$ is a vertex of $\ma(X)$, then all $\pi_X V(\tau_{j'})$, for $j'<j$, are. And similarly with $\pa(S)$, using the version c) in Definition \ref{def:quasipants}.
\item The same as above is true for $V(\tau_j|X)$ instead of $\pi_X V(\tau_j)$. In particular the indices $j$ such that $V(\tau_j|X)$ is a vertex of $\pa(X)$ and of $\ma(X)$ include all the ones in the accessible interval $I_X$ (see \S \ref{sub:goodbehaviour} below).
\item The surface filled by $V(\tau|X)$ is a subsurface of the one filled by $\pi_X V(\tau_j)$.
\end{itemize}
\end{rmk}

The second statement of Lemma \ref{lem:decreasingfilling} may be slightly generalized: let $X\subseteq S$ be a subsurface such that $\partial X$ is essentially disjoint from $S'$. Then $V(\tau|X)=V(\tau)\cap\cc(X)$. The proof is the same as the one developed above, except that $X$ replaces all connected components of $S'$ which it contains.

\begin{lemma}\label{lem:induction_vertices_commute}
Let $X\subset S$ be a subsurface, and let $\tau$ be an almost track on $S$.
\begin{enumerate}
\item For $X$ not an annulus, let $\gamma\in W(\tau|X)$, $\delta\in \pi_X W(\tau)$, and suppose that $Y\subseteq X$ is a non-annular subsurface such that $\pi_Y(\gamma),\pi_Y(\delta)\not=\emptyset$: then 
$$d_{\cc(X)}(\gamma,\delta)\leq F(8 N_1(S^X))\quad\text{and}\quad d_Y(\gamma,\delta)\leq F\left(32 N_1(S^X) + 4\right)$$
where the function $F$ is the one defined by Lemma \ref{lem:cc_distance} and $N_1$ is defined in Lemma \ref{lem:vertexsetbounds}.
\item If $X$ is an annulus, then $\pi_X W(\tau)\subseteq V(\tau|X)$.
\item There is a bound $C_1=C_1(S)$ such that the following is true, for $X$ not an annulus. If $V(\tau|X)$ is a vertex of $\pa(X)$ then also $\pi_X V(\tau)$ is; and\linebreak $d_{\pa(X)}\left(\pi_X V(\tau),V(\tau|X)\right)\leq C_1$. Similarly for $\ma(X)$ instead of $\pa(X)$.
\end{enumerate}
\end{lemma}
In the statement of this lemma we are implicitly using the identification of $\cc(X)$ with $\cc(S^X)$. Moreover, we are identifying $\pa(S^X)=\pa(X)$ using the constants for $X$ a subsurface of $S$ (see Remark \ref{rmk:pickparameters}).

\begin{proof}
In order to prove Claim 1, let $X_0\subset S^X$ be a compact connected 2-submanifold with boundary such that $\bar\nei(\tau|X)\cup X\subset \inte(X_0)$ and the inclusions $X\subseteq X_0\subseteq S^X$ are homotopy equivalences. Let $p:S^X\rightarrow S$ be the covering map. Let $\gamma\in W(\tau|X), \beta\in W(\tau)$, and fix a carried realization of each of these two curves in the respective tie neighbourhoods (which are subsets of $S^X$ and of $S$, respectively).

One way to get a representative of $\pi_X(\beta)$ in $\cc(S^X)$ is as follows. We suppose that $\beta$ is not essentially disjoint from $X$, else Claim 1 is trivial: under this assumption, let $\beta^X$ be the union of all connected components of $p^{-1}(\beta)$ that intersect $X$ essentially: $\beta^X$ is either a single curve or a collection of essential arcs in $S^X$.

Then we can realize $\pi_X(\beta)$ as a family of curves in $X_0$: it will consist of all curves which arise as connected components of $X_0\cap\partial\left(\bar\nei(\beta')\cup \bar\nei(\partial X_0)\right)$ for $\beta'$ a connected component of $\beta^X$, and are essential in $S^X$. If the regular neighbourhoods specified in this expression is narrow enough, we have $\partial\bar\nei(\beta^X)\subseteq \bar\nei(\tau^X)$ and transverse to all ties it encounters; and $\bar\nei(\partial X_0)\cap \bar\nei(\tau|X)=\emptyset$. 

As a result, let $\delta\in\pi_X(\beta)$: the above realization of $\pi_X(\beta)$ gives a representative of $\delta$ which is transverse to all ties of $\bar\nei(\tau|X)$ it encounters; however, it is possible that $\delta$ has nonempty intersection with $\partial_h\bar\nei(\tau|X)$. Since $\beta$ is wide in $\tau$, a component of $\beta^X$ intersects at most twice any tie of $\bar\nei(\tau^X)$, and then, $\delta$ intersects at most $4$ times any tie in $\bar\nei(\tau|X)$.

Recall that $\gamma$ is wide carried in $\bar\nei(\tau|X)$: so we may suppose that, within each branch rectangle $\mathrm{im}(R_b)$ of $\bar\nei(\tau|X)$, there are no more than $8$ intersection points between $\gamma$ and $\delta$. No intersection point can be found outside $\bar\nei(\tau|X)$, so $i(\gamma,\delta)\leq 8N_1(S^X) \leq 8N_1(S)$. Recalling Lemma \ref{lem:cc_distance}, $d_{\cc(S^X)}(\gamma,\delta)\leq F\left( i(\gamma,\delta)\right)$.

By Remark \ref{rmk:subsurf_inters_bound}, for any $Y$ nonannular subsurface of $X$, if $\xi_1\in \pi_Y\left(V(\tau)\right)=\pi_Y\left(\pi_X\left(V(\tau)\right)\right)$ and $\xi_2\in \pi_Y\left(V(\tau|X)\right)$, then $i(\xi_1, \xi_2)\leq 32 N_1(S^X) + 4$, so
$$d_{\cc(Y)}\left(\pi_X V(\tau),V(\tau|X)\right)\leq F\left(32 N_1(S^X) + 4\right)\leq F\left(32 N_1(S) + 4\right)\eqqcolon f(S).$$

For Claim 2, in which $X$ is an annulus instead: just note that, every time $\beta$ is wide carried by $\tau$, $\pi_X(\beta)$ is wide carried by $\tau^X$. 

For Claim 3: apply the identification of $\tau|X$ with an almost track in $X$ as specified by Lemma \ref{lem:inducedisonsurface}, for ease of notation. If $V(\tau|X)$ is a vertex of $\pa(X)$ then $V(\tau|X)$ fills a --- possibly disconnected --- subsurface $X'$ of $X$ such that $X\setminus X'$ is a collection of pairs of pants (Definition \ref{def:quasipants}); but $V(\tau|X)\subseteq \cc(\tau)$, so there is a surface $X\subseteq W\subseteq S$ such that $\cc(\tau)$ fills a possibly disconnected subsurface $W'$ of $W$ with $W\setminus W'$ a collection of pairs of pants. $V(\tau)$ fills the same $W'$ (Lemma \ref{lem:decreasingfilling}), so it is a vertex of $\pa(W)$; and then, $\pi_X V(\tau)$ is a vertex of $\pa(X)$ (see Lemma \ref{lem:decreasingfilling} and the Remark following it).

Let $M> \max\{M_6(S),f(S)\}$. Then, by Theorem \ref{thm:mmprojectiondist} and Lemma \ref{lem:pantsquasiisom},
$$
d_{\pa(X)}\left(\pi_X V(\tau),V(\tau|X)\right)\leq 
e_0 \sum_{\substack{Y\subseteq X\\ Y\text{ not an annulus}}} \left[d_Y\left(\pi_X V(\tau), V(\tau|X)\right)\right]_M + e_1
$$
where $e_0,e_1$ are both functions of $M$ and $S$; but the summation is empty. This proves the existence of the claimed constant $C_1(S)$ holding for the case of $\pa(X)$.

One deals with $\ma(S)$ in an entirely similar way. One shows that $V(\tau|X)$ is a vertex of $\ma(X)$ similarly, and the distance estimate goes also along the same lines, except that the annular subsurface contributions also have to be taken into account, and this is done with the following observation.

For $Y$ an annular subsurface of $X$, making use of Claim 2, $\pi_Y\left(\pi_X V(\tau)\right)=\pi_Y V(\tau) \subseteq V(\tau|Y)$; and $\pi_Y V(\tau|X)\subseteq V((\tau|X)|Y)$. Therefore $d_Y\left(\pi_X V(\tau), V(\tau|X)\right)\leq d_{\cc(Y)} \left( V(\tau|Y), V((\tau|X)|Y)\right)$. However $(\tau|X)|Y$ is a subtrack of $\tau|Y$, so we can simply apply point \ref{itm:diambound} in Remark \ref{rmk:annulusinducedbasics}, which gives $d_{\cc(Y)} \left( V(\tau|Y), V((\tau|X)|Y)\right)\leq 5$.

Theorem \ref{thm:mmprojectiondist}, applied with $M> \max\{5,M_6(S),f(S)\}$, concludes.
\end{proof}

\begin{defin}
Let $\sigma,\tau$ be train tracks on a surface $S$, and let $X\subseteq S$ be a subsurface. Then the following notation will be employed:
\begin{eqnarray*}
d_{\cc}(\sigma,\tau) & \coloneqq & d_{\cc(S)}\left(V(\sigma),V(\tau)\right) \\
d_{X}(\sigma,\tau) & \coloneqq & d_X\left(V(\sigma),V(\tau)\right) = d_{\cc(X)}\left(\pi_X V(\sigma),\pi_X V(\tau)\right)
\end{eqnarray*}
provided that neither of the collections involved in the distance measurement is empty, and
\begin{eqnarray*}
d_{\ma}(\sigma,\tau) & \coloneqq & d_{\ma(S)}\left(V(\sigma),V(\tau)\right) \\
d_{\pa}(\sigma,\tau) & \coloneqq & d_{\pa(S)}\left(V(\sigma),V(\tau)\right) \\
d_{\ma(X)}(\sigma,\tau) & \coloneqq & d_{\ma(X)}\left(V(\sigma|X),V(\tau|X)\right) \\
d_{\pa(X)}(\sigma,\tau) & \coloneqq & d_{\pa(X)}\left(V(\sigma|X),V(\tau|X)\right)
\end{eqnarray*}
provided that the sets considered on the right hand side are actual vertices of the considered graphs.
\end{defin}

A closing remark for this subsection: in general, a train train track split reflects into an operation on induced train tracks which is either taking a subtrack, or performing splits (maybe more than one). This behaviour is controlled slightly better with a theorem from \cite{mms} which will be stated in a few pages.

\subsection{Multiple moves specified by an arc}
We now describe an operation on an almost track called an \emph{unzip}. Informally, an unzip for an almost track $\tau$ consists of cutting it open along a path that begins at a point of $\partial_v\bar\nei(\tau)$ and, proceeding transversally to the ties, ends at an interior point of $\nei(\tau)$. An example of unzip can be seen in Figure \ref{fig:multiplesplit}, e.g. from the first to the second picture. The effect on an unzip can always be expressed in terms of elementary moves, but a number of arguments in this work will be explained more conveniently by usage of this alternative formalism.

\begin{defin}\label{def:zipper}
Let $\tau$ be an almost track on a surface $S$. A \nw{zipper} in $\tau$ is a smooth embedding $\kappa:[-\epsilon,t]\rightarrow \bar\nei(\tau)$, for $t>0$ and $\epsilon>0$ small, such that the following are true.
\begin{itemize}
\item $\kappa(-\epsilon)$ lies along a connected component of $\partial_v \bar\nei(\tau)$ --- call this component $Z$.
\item $\kappa$ is transverse to all ties of $\bar\nei(\tau)$ it encounters.
\item $\kappa^{-1}(\text{ties through switches of }\tau)=\mathbb Z\cap [-\epsilon,t]$.
\item Let $\kappa_P\coloneqq c_\tau\circ\kappa|_{[0,t]}$; it is a smooth path along $\tau$, not necessarily embedded. If $\kappa(t)$ lies on the same tie as a switch of $\tau$, then the branch end $\kappa_P\left((t-\epsilon,t]\right)$ is small.
\end{itemize}

In order to describe the effect of an unzip note that, if $\kappa_P$ is an embedding itself, then there is an open set $\nei(\kappa)$, with $\mathrm{im}(\kappa_P)\subseteq \bar\nei(\kappa)\subseteq \bar\nei(\tau)$ and with the following properties:
\begin{itemize}
\item $\bar\nei(\kappa)$ is diffeomorphic to a triangle;
\item one edge of the triangle is contained in $Z$, and the opposite vertex is $\kappa_P(t)$;
\item the remaining two edges are transverse to all ties they encounter;
\item on the other hand, the interiors $s_1,s_2$ of each of these two edges intersects $\tau$ tangentially in a number of closed segments: i.e. in a neighbourhood of each intersection segment between $\tau$ and $s_i$, $\tau\cup s_i$ looks like a neighbourhood of a large branch in a pretrack.
\end{itemize}

These properties imply the following. Let $B$ be the connected component of $\partial\bar\nei(\kappa)\setminus \tau$ which intersects $Z$, and let $C\coloneqq\partial\bar\nei(\kappa)\setminus B$ --- in other words, $C$ is the maximum connected subset of $\partial\bar\nei(\kappa)$ which contains $\kappa_P(t)$ and has its extremes along $\tau$. Then, the tangential intersection property implies that $\tau\cup C$ is a pretrack whose complementary bigon regions are all adjacent to the image of $\kappa_P$.

So, again if $\kappa_P$ is an embedding, we define the \nw{unzip along $\kappa$} as $\tau''=(\tau\cup C)\setminus \nei(\kappa)$. It is an almost track and has a tie neighbourhood $\bar\nei(\tau'')\subseteq \bar\nei(\tau)$ with $\partial\bar\nei(\tau)\setminus Z\subseteq \partial\bar\nei(\tau'')$. Note, however, that there is no direct correlation between $\nei(\kappa)$ and $\nei(\tau'')$.

If $\kappa_P$ is not an embedding, the unzip is defined inductively on $\lceil t\rceil$. Let $t'=\zeta-2\epsilon$, for $\zeta\in \mathbb Z_+$ chosen such that $\kappa_P|_{[-\epsilon,t']}$ is an embedding instead. One can suppose, up to tie-transverse isotopies, that $\kappa|_{[-\epsilon,t']}=\kappa_P|_{[-\epsilon,t']}$. Let $\tau'$ be the unzip of $\tau$ along $\kappa|_{[-\epsilon,t']}$, and let $\kappa'=\kappa\cap\bar\nei(\tau')$, reparametrized so as to be a zipper for $\tau'$: the unzip $\tau''$ of $\tau$ along $\kappa$ is then defined as the unzip of $\tau'$ along $\kappa'$.
\end{defin}

\begin{defin}\label{def:multiplesplit}
A \nw{large multibranch} for an almost track $\tau$ in a surface $S$ is a carried realization $\beta$ of a carried arc (i.e. $\beta$ is \emph{not} to be considered up to isotopy) in $\bar\nei(\tau)$, whose endpoints belong to distinct components of $\partial_v\bar\nei(\tau)$.

A \nw{splitting arc} is an embedded large multibranch traversing exactly one large branch --- equivalently it does not traverse any small branch. It may traverse any number of mixed branches.

Let $\beta$ be a large multibranch. Specify a splitting parity (left, right, or center); if the parity is not center, also specify two distinct points $P_1,P_2$ lying along $\beta_{trim}$ (see Definition \ref{def:trainpathrealization}), but not contained in the same ties as a switch of $\tau$; they will be called \nw{anchors}. 

The \nw{multiple split} with respect to the specified data is given by the following construction.

\begin{itemize}
\item In case the specified parity is not center (see Figure \ref{fig:multiplesplit}): let $\kappa_j:[-\epsilon,t_j]\rightarrow \bar\nei(\tau)$, $j=1,2$, be two zippers, beginning at the opposite endpoints of $\beta$, following $\beta$ and ending at $P_1,P_2$ respectively; the pairing of the endpoints of $\beta$ with the anchors is to be done so that each the images of $(\kappa_1)_P$ and $(\kappa_2)_P$ overlap. 

Unzip $\tau$ along $\kappa_1$ to get a new almost track that we call $\tau_1$. Note that $P_1$ is a switch of $\tau_1$. We may suppose that the tie $\alpha_2$ of $\bar\nei(\tau)$ through $P_2$ is not entirely contained in $\bar\nei(\tau_1)$. 

Then it is possible to find two zippers $\kappa_{21},\kappa_{22}$ in $\bar\nei(\tau_1)$ with the following property. On the one hand, $\kappa_{21},\kappa_{22}$ are both isotopic to $\kappa_2$ via an isotopy which keeps each point along the same tie of $\bar\nei(\tau)$ (equivalently, $c_\tau\circ\kappa_{21}|_{[0,t_2]},c_\tau\circ\kappa_{22}|_{[0,t_2]}$ are both reparametrizations of $(\kappa_2)_P$); on the other hand, the two are \emph{not} isotopic to each other via a similar tie-transverse isotopy in $\bar\nei(\tau_1)$.

Finally, if the parity specified for the multiple split is right, let $\kappa'_2$ be the one between $\kappa_{21},\kappa_{22}$ that, after intersecting the tie through $P_1$, traverses the small branch end to the right of $P_1$. If the parity is left instead, let $\kappa'_2$ be the other of those two zippers. The multiple split of $\tau$ with the aforementioned data is, then, defined to be the unzip of $\tau_1$ along $\kappa'_2$.

\item In case the specified parity is center: $\beta$ traverses at least one large branch: choose one, $b$. Let $P_1,P_2$ be distinct points (also called \emph{anchors}, even if in this case they are part of the specified data) belonging to the same component of $\beta\cap R_b([-1,1]\times[-1,1])$. Let $\kappa_1,\kappa_2$ be two zippers, with each of them following $\beta$ from an endpoint to $P_1$ and $P_2$, respectively; this time, however, we choose the two zippers to be \emph{disjoint}. Unzip $\tau$ along $\kappa_1$ and along $\kappa_2$ to get a new track which we call $\tau'$. Here, $\beta\cap\bar\nei(\tau')$ traverses a single branch, which is large. The central multiple split of $\tau$ along $\beta$ is, then, the central split of $\tau'$ along this large branch.
\end{itemize}

If $\beta$ is a splitting arc such that $b$ is the only large branch $\beta$ traverses, the \nw{wide split} along $\beta$ with a given parity is defined to be the multiple split along $\beta$, according to the given parity, choosing any two anchors $P_1,P_2$ which lie along ties in $R_b([-1,1]\times[-1,1])$.
\end{defin}

\begin{rmk}
We have defined a large multibranch to be a fixed carried realization, rather than an isotopy class, of a carried arc. This has been convenient for the definitions of anchors, and for the description of what a multiple split is supposed to do. However, if one perturbs a large multibranch under an isotopy keeping each point along the same tie, and the anchors in particular, the result of the multiple split is the same.
\end{rmk}

\begin{figure}[htbp]
\begin{center}
\def\svgwidth{\textwidth}
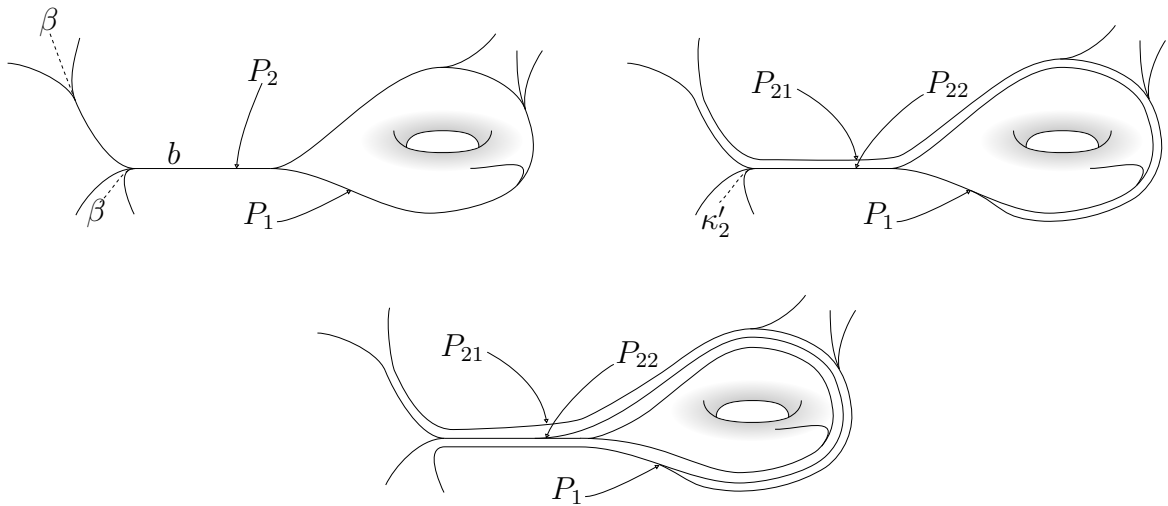
\end{center}
\caption{\label{fig:multiplesplit}An example of a left multiple split. Here the large multibranch $\beta$ is the one with endpoints specified in the first picture and traversing twice the branch $b$ (this information suffices to figure out what it looks like in $\bar\nei(\tau)$, up to isotopies which are inconsequential for the construction). The anchors are $P_1,P_2$; $P_2$ in the picture is meant to lie along the upper segment in $\beta\cap R_b$. In the second picture, the zipper $\kappa_1$ following $\beta$ from a suitable endpoint to $P_1$ has been unzipped, resulting in the definition of two points $P_{21},P_{22}$ in place of the old $P_2$. The points $P_{21}, P_{22}$ are endpoints for two substantially different zippers $\kappa_{21},\kappa_{22}$ which both begin at the remaining endpoint of $\beta$ and project to the old zipper $\kappa_2$ for $\tau$. In the third picture, the zipper $\kappa'_2=\kappa_{22}$ has been unzipped, too. The choice of $\kappa_{22}$ rather than $\kappa_{21}$ is determined by the specified parity.}
\end{figure}

\begin{rmk}\label{rmk:generic_move_as_unzip}
The given definition of wide split includes, as a special case, the one of a `standard' split as defined by Figure \ref{fig:ttsplitting}. If $b$ is the branch to split, one may choose a suitable splitting arc which is entirely contained in $R_b$.

The result of a wide split is always the same as a splitting sequence involving exactly one split.

If $\tau$ is a \emph{generic} almost track, then any elementary move on $\tau$ which is not a central split is also the result of an unzip along a \emph{single} zipper $\kappa: [-\epsilon, 1+\epsilon]\rightarrow \bar\nei(\tau)$: a \emph{large} zipper defined on this domain will produce a parity split, and a \emph{small} one will produce a slide (see below for the definitions).
\end{rmk}

\begin{rmk}\label{rmk:zipvssplit}
If $\tau'$ is the unzip of an almost track $\tau$ along a zipper $\kappa:[-\epsilon,t]\rightarrow \bar\nei(\tau)$, then $\tau'$ is itself an almost track, clearly carried by $\tau$ and therefore obtained from it via a splitting sequence (Proposition \ref{prp:carriediffsplit}). Some remarks will ease the incoming estimates for the number of splits necessary to get $\tau'$ from $\tau$. If $t\in (0,1)$ then $\tau'$ is isotopic to $\tau$; whereas, if $t\in[1,2)$, then:
\begin{itemize}
\item if $\kappa_P\left((1-\epsilon,1]\right)$ is a small branch end, then $\tau'$ is comb equivalent to $\tau$;
\item if $\kappa_P\left((1-\epsilon,1]\right)$ is a large branch end, and all small branch ends of $\tau$ incident to $\kappa_P(1)$ lie on the same side of $\kappa_P$, then $\tau'$ is obtained from $\tau$ with a parity split, plus possibly some comb/uncomb moves;
\item if $\kappa_P\left((1-\epsilon,1]\right)$ is a large branch end, and small branch ends of $\tau$ incident to $\kappa_P(1)$ lie on both sides of $\kappa_P$, then $\tau'$ is obtained from $\tau$ with two parity splits, plus possibly some comb/uncomb moves.
\end{itemize}

More generally, suppose that $\kappa:[-\epsilon,t]\rightarrow \bar\nei(\tau)$ is a zipper with no restriction on its length, but with the property that $\kappa_P\left((\zeta-\epsilon,\zeta]\right)$ is a small branch end for all integers $1\leq \zeta\leq t$ (we call it a \nw{small zipper}). Then the result of the unzipping is comb equivalent to $\tau$.

If one supposes instead that $\kappa_P\left((\zeta-\epsilon,\zeta]\right)$ is a large branch end for all integers $1\leq \zeta\leq t$ (we call it a \nw{large zipper}), then necessarily $\kappa_P$ is an embedding. A branch can be traversed twice only if, for some value of $\zeta$, $\kappa_P\left((\zeta-\epsilon,\zeta]\right)$ is a small branch end.

Similarly as before, there is a difference in behaviour depending on whether the branches incident to $\kappa_P([1,t])$ all lie on the same side of it or not. In the first case, the unzip along $\kappa$ is comb equivalent to an almost track obtained from $\tau$ with one parity wide split; in the second case, the number of wide splits needed is $2$. A particular case of the first behaviour is when there is only 1 branch sharing a switch with $\tau([1,t])$.

Given a zipper $\kappa:[-\epsilon,t]\rightarrow\bar\nei(\tau)$ which is neither large nor small itself, it is always possible to take the maximal $\zeta\in \mathbb N$ such that $\kappa|_{[-\epsilon,\zeta+\epsilon]}$ is either a large or a small zipper. Let $\tau'$ be the unzip of $\tau$ along $\kappa|_{[-\epsilon,\zeta+\epsilon]}$; $\kappa|_{[\zeta+\epsilon, t]}$ admits a natural reparametrization as $\kappa':[-\epsilon,t']$, for some $t'\leq t-\zeta$ which makes it into a zipper for $\tau'$: unzipping $\tau'$ along $\kappa'$ is the same as unzipping $\tau$ along $\kappa$. 

Recursively, one can decompose the unzip along $\kappa$ into a sequence of unzips along zippers that are alternatively small and large. An upper bound for the number of large zippers involved is given by the number of $\zeta\in \mathbb N$ such that $\kappa_P([\zeta,\zeta+1])$ is a large branch of $\tau$. We call this decomposition \nw{canonical}.
\end{rmk}

Our notion of wide split agrees with the one defined in \cite{mosher}, \S 3.13. We are interested in Proposition 3.13.3 in that work: it uses the fact that, given two comb equivalent train tracks, there is a canonical identification between the respective families of splitting arcs. This correspondence may be found with our definitions as well, so the said Proposition holds for us:

\begin{prop}\label{prp:combpersistence}
Let $\tau,\sigma$ be two comb equivalent train tracks, and $\alpha,\beta$ be splitting arcs, for $\tau$ and for $\sigma$ respectively, which correspond to each other under the equivalence of the two tracks. Let $\tau'$ be the train track obtained from $\tau$ by splitting along $\alpha$ according to a given parity; and let $\sigma'$ be the train track obtained from $\sigma$ by splitting along $\beta$ according to the same parity.

Then $\tau',\sigma'$ are comb equivalent.
\end{prop}

Also, a recursive application of Proposition 3.13.4 in that monograph gives the following:
\begin{prop}\label{prp:deleteslidings}
Given a train track splitting sequence $\bm\tau$, let $(j_r)_{r=1}^R$ be the indices such that the move from $\tau_{j_r-1}$ to $\tau_{j_r}$ is a split. Then there is a sequence $\bm\sigma=(\sigma_r)_{r=0}^R$ where $\sigma_0=\tau_0$, $\sigma_r$ is comb equivalent to $\tau_{j_r}$ and $\sigma_r$ is obtained from $\sigma_{r-1}$ via a wide split, of the same parity as the split between $\tau_{j_r-1}$ and $\tau_{j_r}$.
\end{prop}
A sequence of wide splits will be called a \nw{wide splitting sequence}.

\subsection{Cornerization of train tracks}
The purpose of the following discussion is to connect the most common version of train track found in the literature (which corresponds to the definition we have given) to the notion of cornered train track, which is the notion of train track given and developed in \cite{mms}.

\begin{defin}
Given a train track $\tau$ on a surface $S$, a train track $\tau'$ is a \nw{cornerization} of $\tau$ if:
\begin{itemize}
\item $\tau'$ is cornered;
\item $\tau$ is obtained from $\tau'$ via a sequence of central splits and comb equivalences;
\item if $\tau$ is transversely recurrent, then $\tau'$ also is.
\end{itemize}
\end{defin}
We have not required explicitly that if $\tau$ is recurrent, then $\tau'$ also has to be: this is indeed a direct consequence of the second bullet.

\begin{lemma}\label{lem:cornerization}
Any train track $\tau$ on $S$ admits a cornerization.
\end{lemma}
\begin{proof}
We focus on the case when $\tau$ is transversely recurrent. Let $\ul\gamma=\{\ul\gamma_1,\ldots,\ul\gamma_m\}$ be a set of essential curves in $S$ as specified after Definition \ref{def:recurrent}: pairwise disjoint, dual to $\tau$, such that for each branch $b$ of $\tau$ there is a $\ul\gamma_i$ intersecting $b$. It is not required that two of them are not isotopic.

Suppose that $\tau$ is not cornered. We will build a transversely recurrent train track $\tau'$ such that $\partial\bar\nei(\tau')$ has fewer smooth components than $\partial\bar\nei(\tau)$, and $\tau$ is obtained from $\tau'$ with a central split. This will prove our claim by induction.

We deal with a special case first: $\tau$ contains a connected component $\lambda$ which is an embedded loop and all curves $\ul\gamma_1,\ldots,\ul\gamma_m$ intersect $\lambda$ in at most one point. Then necessarily one of them, $\ul\gamma_1$ say, intersects $\lambda$ in exactly one point. In that case, $\lambda$ and $\ul\gamma_1$ together fill a handle $H$ of $S$. In $H$, replace $\lambda$ with a new component $\lambda'$ and $\ul\gamma_1$ with a new dual curve $\ul\gamma'_1$ as shown in Figure \ref{fig:cornerizehandle}.

\begin{figure}[h]
\centering \def\svgwidth{.8\textwidth}
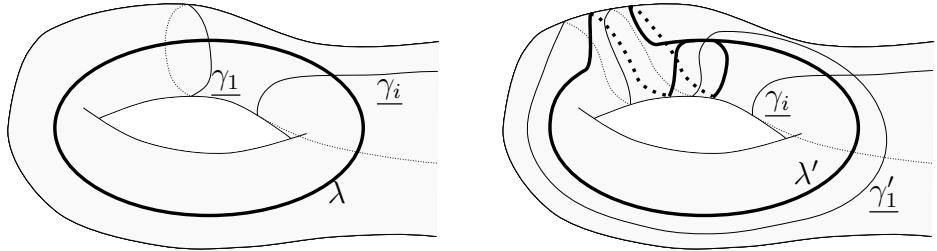
\caption{\label{fig:cornerizehandle}How to get rid of a loop component $\lambda$ of $\tau$ with a dual curve $\ul\gamma_1$ which intersects $\tau$ in a single point, lying along $\lambda$. The two lie in a handle $H$ of $S$. Note: $\lambda'$ shall differ from $\lambda$ only in a neighbourhood of $\ul\gamma_1$ whose intersection with any other $\ul\gamma_i$ is empty.}
\end{figure}

Define $\tau'\coloneqq(\tau\setminus\lambda)\cup\lambda'$. It is true that $\partial\bar\nei(\tau')$ has one smooth component less than $\partial\bar\nei(\tau)$ and that $\tau'$ is a train track. Furthermore it is tranversely recurrent: the new curve $\ul\gamma'_1$ may be isotoped to keep dual to $\tau'$ and intersect either of the two branches in $\lambda'$; while the other curves $\ul\gamma_2,\ldots,\ul\gamma_m$ are still in efficient position with respect to $\tau'$, and therefore dual: as it was done in Definition \ref{def:efficientposition}, fix a regular neighbourhood $\nei(\gamma_i')$ which contains no corners of $\partial\bar\nei(\tau')$. As one may see from the figure, $H\setminus\left(\nei(\tau')\cup \nei(\gamma_i')\right)$ will only consist of connected components with $\geq 4$ corners each, so $S\setminus\left(\nei(\tau')\cup \nei(\gamma_i')\right)$ consists only of negative index components or rectangles. For each branch $b$ of $\tau'\setminus\lambda'$ there is one curve among $\ul\gamma_2,\ldots,\ul\gamma_m$ which intersects it.

Now suppose that a component as above does not exist in $\tau$. Let $\alpha\subset\partial\bar\nei(\tau)$ be a smooth connected component: in particular $\alpha$ is isotopic to a wide carried curve of $\tau$, so it makes sense to talk about the carrying image $\tau.\alpha$. Let $C$ be the connected component of $\bar S\setminus\nei(\tau)$ such that $\alpha\subset\partial C$.

We wish to make sure that there is a curve $\ul\gamma_i$ intersecting $\alpha$, and intersecting $\tau$ in at least two points. Suppose such a curve does not exist: i.e. all curves $\underline{\gamma_i}$ intersecting $\alpha$ in fact intersect $\tau$ in only $1$ point. By our hypothesis, this excludes the possibility that $\tau.\alpha$ is a smooth loop consisting of an entire connected component of $\tau$.

Thus, in case $\tau.\alpha$ is a smooth loop (i.e. $\alpha$ has an embedded train path realization), we have that $\tau.\alpha$ includes a mixed or large branch $b$ of $\tau$. In case $\tau.\alpha$ is not a smooth loop, $\tau.\alpha$ surely includes a large branch $b$ of $\tau$.

In both cases, there is a $\ul{\gamma_i}$ such that the single point $\ul{\gamma_i}\cap \tau.\alpha$ lies along $b$. With a small isotopy on $\ul{\gamma_i}$, slide it towards one of the large ends/the large end (resp.) of $b$, and then past the relative switch. The curve thus obtained, $\ul{\gamma'_i}$, is still dual to $\tau$; moreover the single intersection of $\ul{\gamma_i}$ with $b$ has been replaced with an intersection point of $\ul{\gamma'_i}$ with each of the small branch ends which are incident at that switch.

So, it is always possible to assume, up to enlarging the collection $\ul\gamma$ of curves dual to $\tau$, that there exists a curve $\underline{\gamma_i}$ intersecting $\alpha$, and having more than $1$ intersection point with $\tau$. Fix a regular neighbourhood $\nei(\gamma_i)$ which contains no corners of $\partial\bar\nei(\tau)$. Let $p\in \ul{\gamma_i}\cap \tau.\alpha$, let $g$ be an arc of $\gamma$ which begins at $p$, proceeds towards $C$ through a half-tie, then through $C$, and finally follows a tie of $\tau$ so as to arrive at the nearest intersection point $q\in \ul{\gamma_i}\cap\tau$. Pinch a neighbourhood of $p$ in $\tau$ along $g$ and fold it with a neighbourhood of $q$, so as to create a new large branch $f$: this pretrack will be our $\tau'$ (Figure \ref{fig:cornerization}).

\begin{figure}
\begin{center}
\def\svgwidth{.9\textwidth}
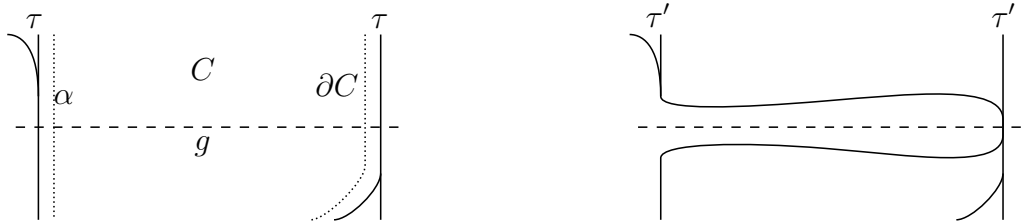
\end{center}
\caption{\label{fig:cornerization}The elementary step to reduce the number of smooth components of $\partial\bar\nei(\tau)$ is depicted above. The dual curve $\ul{\gamma_i}$ shows the way to \emph{fold} $\tau$ (this is how the inverse of a split move is usually called) so as to reduce the total number of smooth boundary components.}
\end{figure}

We claim, first of all, that $\tau'$ is a train track. Note that the complementary regions of $\nei(\tau')$ in $S$ can be identified with the ones of $\nei(\tau)$, except $C$ which is replaced by either a single complementary component $C'$ or by two $C',C''$. In either case what we obtain is diffeomorphic to $C\setminus (\nei(\ul{\gamma_i})|_g)$, where by $\nei(\ul{\gamma_i})|_g$ we mean the component of $\nei(\ul{\gamma_i})\cap C$ which contains $g$. Since, by efficient position of $\ul{\gamma_i}$, components of $C\setminus\nei(\ul{\gamma_i})$ either have negative index or are rectangles, $C',C''$ will have the same property, because they are obtained as the gluing of some of such components, plus some rectangles, along boundary edges. But $g$ cannot cut a rectangle out of $C$: that would mean that $\alpha$ is not a smooth boundary component. This proves the claim.

Also, we claim that $\tau'$ is transversely recurrent. The given curves $\ul{\gamma_1},\ldots,\ul{\gamma_k}$ may not be enough to intersect all branches of $\tau'$ because $\ul{\gamma_i}$, in general, will intersect the new large branch $f$ born from folding, but not the four branches which are incident to its ends. However, if $\ul{\gamma_i}$ is altered via an isotopy that slides it past either switch of $f$, it turns into a curve, again dual to $\tau'$, and intersecting two of the four branch ends sharing their switch with $f$. So we may add two more curves, isotopic to $\ul\gamma_i$, to the collection $\ul\gamma$, to intersect all branches of $\tau'$.

In case $\tau$ is \emph{not} transversely recurrent, we apply a similar idea. We do not select a family of curves $\ul\gamma$; after choosing a connected component $C$ of $\bar S\setminus\nei(\tau)$ with a smooth boundary component $\alpha$, we may pick a properly embedded arc $g$ in $C$ such that: the endpoints of $g$ are distinct; at least of them lies along $\alpha$; and $g$ does not cut out a region of $C$ with nonnegative index. One builds $\tau'$ similarly as above and then shows that is a train track.
\end{proof}

We now describe how to cornerize consistently an entire splitting sequence: the first step is to make it cleaner by turning it into a wide splitting sequence.

\begin{rmk}\label{rmk:centralsplitbound}
Any (wide) splitting sequence involves at most $N_1$ central (wide) splits: given a train track $\tau$, each branch contributes at most $4$ corners in $\partial\bar\nei(\tau)$; and each central split decreases the total number of corners by $4$, whereas any other move keeps it unvaried.
\end{rmk}

\begin{lemma}[Postponing central splits]\label{lem:postpone}
Given a wide splitting sequence $\bm\sigma=(\sigma_0\ldots,\sigma_i,\sigma_{i+1},\sigma_{i+2},\ldots,\sigma_N)$, call $\beta_j$ (for $j=0,1$) the splitting arc for $\sigma_{i+j}$ that shall be split along to get $\sigma_{i+j+1}$. Suppose that the split along $\beta_0$ is central whereas the one along $\beta_1$ is parity. Then there is a wide splitting sequence $\bm{\sigma'}=(\ldots,\sigma'_i,\ldots,\sigma'_{i+k},\ldots)$, with $2\leq k\leq 7$, such that
\begin{itemize}
\item $\sigma_j=\sigma'_j$ for all $j\leq i$;
\item the wide split(s) between $\sigma'_i$ and $\sigma'_{i+k-1}$ is/are parity;
\item the wide split between $\sigma'_{i+k-1}$ and $\sigma'_{i+k}$ is central;
\item $\sigma_j$ is comb equivalent to $\sigma'_{j-2+k}$ for all $j\geq i+2$. (and in particular the all wide splits in the subsequence $(\sigma'_j)_{j\geq i+k}$ have the same parity as the corresponding ones in $(\sigma_j)_{j\geq i+2}$). 
\end{itemize}
\end{lemma}
\begin{proof}
We first build a splitting sequence $\bm\rho$ turning $\sigma_i$ into $\sigma_{i+2}$, and involving $1$ to $6$ parity splits, followed by $1$ central split.

Since the split along $\beta_0$ is central, a tie neighbourhood $\bar\nei(\sigma_{i+1})$ is contained in $\bar\nei(\sigma_i)\setminus \beta_0$, with ties obtained by restriction of the old ones; and $\partial_v\bar\nei(\sigma_{i+1})$ consists of $\partial_v\bar\nei(\sigma_i)$, deprived of two connected components. So we may suppose that, in $\bar\nei(\sigma_i)$, the arcs $\beta_0,\beta_1$ are properly embedded and disjoint. The endpoints of $\beta_0$ lie on distinct components of $\partial_v\bar\nei(\sigma_i)$, and so do the ones of $\beta_1$ because this property must hold when considering it as an arc in $\bar\nei(\sigma_{i+1})$. It is also impossible that the two splitting arcs have an extreme each on the same connected component of $\partial_v\bar\nei(\sigma_i)$, because the two connected components where $\partial\beta_0$ lie are not part of $\partial_v\bar\nei(\sigma_{i+1})$.

The arc $\beta_1$, in general, is not a splitting arc for $\sigma_i$, but it is a large multibranch. Also, since the wide central split along $\beta_0$ shall turn it into a splitting arc, the two arcs are embedded in $\bar\nei(\sigma_i)$ with this property: $\beta_1$ traverses each branch at most twice, and if $b$ is any branch traversed twice then $R_b\setminus \beta_0$ consists of two connected components, and the two segments of $\beta_1\cap R_b$ each lie in one.

Let $\delta\coloneqq \bigcup_{b\text{ traversed by }\beta_0} \beta_1\cap R_b$. If $B_0$ is the large branch of $\sigma_i$ traversed by $\beta_0$, then each component of $\delta$ traverses $B_0$ (here we use the term `traverse' in an obvious generalized setting). Let indeed $e$ be the first or last branch end traversed by a connected component $\delta_0$ of $\delta$: either $e$ is the first/last branch end traversed by $\beta_0$ or $\beta_1$; or $\beta_1$ and $\beta_0$ give two different train paths originating from $e$. So the branch ends at both extremities of $\delta_0$ are large, implying that $\delta_0$ must traverse a large branch: necessarily $B_0$. In particular, $\delta$ shall consist of $0$, $1$ or $2$ connected components.

Let $B_1$ be the large branch traversed by $\beta_1$ in $\sigma_{i+1}$: $c_{\sigma_i}(B_1)\subseteq \sigma_i$ is a union of branches, including a large one $b_1$. Place the two anchors $P_1,P_2$ for the wide split of $\sigma_{i+1}$ along $\beta_1$ within the same segment of $\beta_1\cap R_{b_1}([-1,1]\times[-1,1])$. This is not important for the wide split, as it only affects its result up to isotopy; but is relevant for the following argument.

Let $\rho''_i$ be the result of the multiple split of $\sigma_i$ along $\beta_1$, with the same parity as the wide split that $\sigma_{i+1}$ undergoes in the sequence $\bm\sigma$, and with anchors $P_1,P_2$. Call $\kappa_{11}$ the first zipper that gets unzipped, and $\rho'_i$ the result of the unzip; and call $\kappa_{12}'$ the second zipper, with $P'_2\in\rho'_i$ the endpoint of $\kappa_{12}'$. The arc $\beta_0$ is a large multibranch in $\rho''_i$, and the central multiple split along it gives back $\sigma_{i+2}$. This multiple split will be described as the unzip of a zipper $\kappa_{01}$ to obtain an almost track $\rho'''_i$, followed by the unzip along another zipper $\kappa_{02}$ to get $\rho''''_i$, and finally a central split.

According to Remarks \ref{rmk:generic_move_as_unzip} and \ref{rmk:zipvssplit}, the multiple split along $\beta_1$ can be canonically decomposed into unzips along small and large zippers. While unzips along small zippers give the same result as a series of comb/uncomb moves, the ones along large zippers give a track which is comb equivalent to the one obtained with one or two wide splits so their effect is obtained with a splitting sequence which involves one or two splits, even if \emph{only up to isotopy}. This specifies how to build a splitting sequence $\bm\rho^1$, turning $\sigma_i$ progressively into $\rho'_i$ and $\rho''_i$. Similarly, the subsequent multiple split along $\beta_0$ is the result of a splitting sequence $\bm\rho^0$ turning $\rho''_i$ into $\rho'''_i,\rho''''_i$, and finally $\sigma_{i+2}$. Also, let $\bm\rho=\bm\rho^1*\bm\rho^0$. The last split in this sequence is central, the others are parity.

We will now estimate how many splits are involved, in total, in the sequence $\bm\rho^1$, by analysing how many unzips along large zippers take place when canonically decomposing the unzips along $\kappa_{11}$ and $\kappa'_{12}$. There are several cases to consider, but the following consideration is always true: let $\eta$ be the second-to-last almost track obtained when applying the unzips in the canonical decomposition, and let $\kappa_l:[-\epsilon,t]\rightarrow \bar\nei(\eta)$ be the last zipper in the decomposition, which is a large one. Then $t\in(1,2)$, $P_1$ belongs to the tie through $\kappa_l(1)$, and the branch end $(\kappa_l)_P([1,t])$ shares its switch only with another branch end: so the unzip along $\kappa_l$ contributes only one (parity) split in the sequence $\bm\rho^1$.

Suppose that $P_1,P_2$ lie along $\delta$, and therefore along one same connected component $\delta_0$; in particular they lie along $\delta_0\cap R_{B_0}$. The first case to consider is the one when this supposition holds, and $\delta$ has only 1 connected component. In this case there are no large branches contained in the image of $(\kappa_{11})_P$, and the one of $(\kappa'_{12})_P$ contains only one. So the zipper $\kappa_{11}$ is small, and either the zipper $\kappa'_{12}$ is large, or the unzip along it can be subdivided into the unzip along a small zipper followed by one along a large zipper. In either case, with this configuration $\bm\rho^1$ only involves 1 split (due to the consideration in the above paragraph).

The second case is the one of $P_1,P_2$ lying again along $\delta$, but $\delta$ having two connected components. Then the component different from $\delta_0$ is included entirely in $\beta_1\setminus b_1$ --- thus it traverses branches of $\sigma_i$ that are completely contained in the image of either $(\kappa_{11})_P$, or $c_{\sigma_i}\circ(\kappa'_{12})_P$. In the first sub-case the unzip along $\kappa_{11}$ is canonically decomposed into at most $3$ unzips, with one occurring along a large branch and the other, or others, along one or two small ones; and the unzip along $\kappa'_{12}$, instead, is decomposed similarly as above; the large zippers involved are one or two, implying that the splits required in $\bm\rho^0$ are $1$ to $3$ ($4$ is to be excluded because of the previous consideration; in any case, they are all parity splits). In the second sub-case the zipper $\kappa_{11}$ is small, while the canonical decomposition of the unzip along $\kappa'_{12}$ involves $2$ unzips along large zippers, the second of which requires only $1$ split: there are again at most $3$ (parity) splits in $\bm\rho^1$.

The third and last case is the one of $P_1,P_2$ not lying along $\delta$. In this case $\delta$ (if nonempty) is entirely contained in the images of $(\kappa_{11})_P$ and $c_{\sigma_i}\circ(\kappa'_{12})_P$. Several cases are possible as $\delta$ may consist of up to $2$ components, and each of them may be traverse branches belonging to each of the images of the specified maps. But a similar analysis as above shows that the canonical decompositions of the two unzips together involve at most three unzips along large zippers, thus there are at most $5$ (parity) splits in $\bm\rho^1$.

Turning to $\bm\rho^0$: $\beta_0$ traverses each branch of $\rho''_i$ at most once. The restriction of the tie collapse $c_{\sigma_i}:\rho''_i\rightarrow \sigma_i$ maps all branch ends of $\rho''_i$ sharing a switch with $\rho''_i.\beta_0$ to branch ends of $\sigma_i$ sharing a switch with $\sigma_i.\beta_0$, except possibly for the one, $e$, whose endpoint $v$ lies along the same tie of $\bar\nei(\sigma_i)$ as $P_1$ or $P_2$: it may be the case that $c_{\sigma_i}(e)\subseteq \sigma_i.\beta_0$.

The multiple central split along $\beta_0$ is a wide central split, if $\beta_0$ traverses only one large branch of $\rho''_i$. If this is not true, the only possibility is that $\beta_0$ traverses $2$ large branches of $\rho''_i$, with one of them delimited by $v$; call the other one $b$. Without affecting the result of the multiple central split, one can suppose that the anchors $Q_1,Q_2$ (with $Q_j$ endpoint of $\kappa_{0j}$) lie in $\beta_0\cap R_b$, and that $Q_1$ is the one closer to $v$.

We can now apply a similar argument as before. The unzip along $\kappa_{01}$ can be canonically decomposed into at most $3$ unzips, only one of which takes place along a large zipper $\kappa_l$, defined on an appropriate interval $[-\epsilon,t]$. One sees that, necessarily, $t\in(1,2)$; $(\kappa_l)_P([1,t])$ contains no switches other than $v$ and is only incident to the branch end $e$. The zipper $\kappa_{02}$ is small instead. So the splitting sequence $\bm\rho^1$ requires at most $1$ parity split further than the last, central split.

Let $k$ be the total number of splits in $\bm\rho$: it has been proved above that $2\leq k\leq 7$, and that only the last split is central anyway.

Now that $\bm\rho$ has been defined, we build $\bm\sigma'$ from $\bm\sigma$ with the following steps. Call $\bm\sigma_-=\bm\sigma(0,i)$. The wide splitting sequence $\bm\sigma(i+2,N)$ can be turned into a (regular) splitting sequence, $\bm\sigma''_+$. According to Proposition \ref{prp:deleteslidings} there is a wide splitting sequence $\bm\sigma'_{0+}$, starting with $\sigma_i$, whose elements are orderedly comb equivalent to the ones of $\bm\rho*\bm\sigma''_+$. This sequence begins with with $k-1$ wide parity splits followed by a central one, and $(\sigma'_{0+})_k$ is comb equivalent to $\sigma_{i+2}$. Define $\bm\sigma'=\bm\sigma_-*(\bm\sigma'_{0+})$. 
\end{proof}

\begin{defin}\label{def:cornerizedseq}
Let $\tau=(\tau_j)_{j=0}^N$ be a train track splitting sequence. We build a splitting sequence $P\bm\tau$ (not uniquely defined) in the following way.

Let $\hat\tau_0$ be a cornerization of $\tau_0$. Let $\bm{\hat\tau}$ be a splitting sequence from $\hat\tau_0$ to $\tau_0$, and then continuing as $\bm\tau$.

If $\bm{\hat\tau}$ involves no central split, or all central splits are followed only by central splits and slides, the process ends here. Else replace $\bm{\hat\tau}$ with a wide splitting sequence $\bm\sigma^0$ according to Proposition \ref{prp:deleteslidings}, and define a list of wide splitting sequences $\bm\sigma^i=(\sigma^i_j)_j$ recursively this way: within the sequence $\bm\sigma^i$, let $\ell(i)$ be the highest index $j$ such that $\sigma^i_j,\sigma^i_{j+1},\sigma^i_{j+2}$ are a central wide split followed by a parity one: postpone the central split according to Lemma \ref{lem:postpone} above, and call $\bm\sigma^{i+1}$ the new sequence. Repeat the process as far as it is possible to define $\ell(i)$. Finally, replace the last constructed $\bm\sigma^R$ back with a (non-wide) splitting sequence: this will be $P\bm\tau$.

Let $\cnr\bm\tau$ be the initial subsequence of $P\bm\tau$ obtained by truncating just after the last parity split. It is called a \nw{cornerization} of $\bm\tau$.
\end{defin}

We summarize the properties of $\cnr\bm\tau$ in the below:
\begin{lemma}\label{lem:ctauproperties}
Given a splitting sequence $\bm\tau=(\tau_j)_{j=0}^N$ on a surface $S$, let $\cnr\bm\tau=(\cnr\tau_j)_{j=0}^M$ be a cornerization of it defined from $P\bm\tau=(P\tau_j)_{j=0}^{M+M'}$. Then the following properties hold:
\begin{enumerate}
\item if all elements of $\bm\tau$ are recurrent, then so are the ones of $\cnr\bm\tau$;
\item if all elements of $\bm\tau$ are transversely recurrent, then so are the ones of $\cnr\bm\tau$;
\item all elements of $\cnr\bm\tau$ are cornered;
\item there is an increasing function $f:[0,N]\cap\mathbb Z\rightarrow [0,M+M']\cap\mathbb Z$, with $f(0)=0$, such that $\cnr\tau_{f(j)}$ is comb equivalent to a cornerization of $\tau_j$ for all $0\leq j\leq N$, and if $N_0$ is the lowest index such that no parity split occurs in $\bm\tau$ after $\tau_{N_0}$ then $f(N_0)=M$;
\item $|\bm\tau|-N_1\leq|\cnr\bm\tau|\leq 6^{N_1}|\bm\tau|$, for $N_1$ defined as in Remark \ref{rmk:centralsplitbound}.
\end{enumerate}
\end{lemma}
\begin{proof}
\begin{enumerate}
\item $P\tau_{M+M'}$ is comb equivalent to $\tau_N$. If the latter is recurrent, then so is the former; but if the last element of a splitting sequence is recurrent then so are all the previous ones.

\item $\cnr\tau_0$ is defined to be a cornerization of $\tau_0$. If the latter is transversely recurrent, then so is the former; but if the first element of a splitting sequence is transversely recurrent then so are all the following ones.

\item $\cnr\tau_0$ is cornered and, as all wide splits in the sequence are parity splits, the complementary regions of each $\cnr\tau_j$ are diffeomorphic to the ones of $\cnr\tau_0$.

\item Recall the notation for the construction of $P\bm\tau$ set up in Definition \ref{def:cornerizedseq}. Suppose that $\bm{\hat\tau}=(\hat\tau_j)_{j=0}^{N'+N}$ and, for each $i$, $\bm\sigma^i=(\sigma^i_j)_{j=0}^{m_i}$. Let $h:[0,N]\rightarrow [0,N'+N]$ (here and onwards it is understood that intervals are always in $\mathbb Z$) be defined as $h(j)=N'+j$. Let $\omega:[0,N'+N]\rightarrow [0,m_1]$ be defined inductively as follows: $\omega(0)=0$; $\omega(j+1)=\omega(j)+1$ if $\hat\tau_{j+1}$ is obtained from $\hat\tau_j$ with a split, else $\omega(j+1)=\omega(j)$ (i.e. $\omega$ establishes the natural correspondence between the splits of $\bm{\hat\tau}$ and the corresponding wide splits of $\bm\sigma^0$). And let also $\rho: [0,m_R]\rightarrow [0,M+M']$ be defined inductively with these conditions: $\rho(0)=0$; $\rho(j+1)$ is the least index $s>\rho(j)$ such that $P\tau_{s+1}$ is obtained from $P\tau_s$ with a split (so $\rho$ establishes the natural correspondence between $\bm\sigma^R$ and $P\bm\tau$).

For $0\leq i<R-1$, let $f_i:[0,m_i]\rightarrow [0,m_{i+1}]$ be defined as follows. Suppose $\bm\sigma^i=(\ldots,\sigma_s,\sigma_{s+1},\sigma_{s+2},\ldots)$ and $\bm\sigma^{i+1}=(\ldots,\sigma'_s,\ldots,\sigma'_{s+k_i},\ldots)$, with a notation derived from the one used in the statement of Lemma \ref{lem:postpone}. Then we define $f_i(j)=j$ for $j\leq s$; $f_i(j)=j+k_i-2$ for $j\geq s+3$; $f_i(s+1)=s$; $f_i(s+2)=s+k_i-1$ (i.e. $f_i$ maps the indices before and after the parity wide split involved in the postponement to the indices before and after the parity wide split(s) introduced by the postponement, respectively). Note that, for all $j$, $\sigma^i_j$ is either obtained from $\sigma^{i+1}_{f_i(j)}$ with comb moves and a central wide split, or is  comb equivalent to it; whereas $\omega$ and $\rho$ establish a correspondence between tracks which are comb equivalent. So $f\coloneqq\rho\circ f_{R-1}\circ\cdots\circ f_0\circ\omega$ is such that $f(0)=0$, and also $f(m_0)=M$ because $M$ is the lowest index such that the following moves in $P\bm\tau$ are all combs or central splits, which is what one must get, due to the descriptions of the maps given. The map $f$ is increasing because all the composed maps are. Also, $\tau_j$ is obtained from $\cnr\tau_{f(j)}$ with a sequence of central splits and so, by properties 2. and 3., the latter is a cornerization of the former.

\item Lemma \ref{lem:postpone} and Remark \ref{rmk:centralsplitbound} together yield that $P\tau$ involves at least as many splits as $\tau$, and the number of central ones is unvaried: hence the first inequality. For the second one, consider for each $i$ the numbers $p(i)$ counting how many parity wide splits occur between $\sigma^i_{\ell(i)+1}$ and the end of the sequence $\bm\sigma^i$; and $c(i)$ counting the number of consecutive wide central splits at the end of the sequence $\bm\sigma^i$. For each $0\leq i <R$ one of the following is true: either $p(i+1)=p(i)-1$, and in this case $c(i)=c(i+1)$; or $p(i)=1,p(i+1)\geq 1$, and in this case $c(i+1)=c(i)+1$. Let $i_1,\ldots,i_s$ be the sequence of indices $i$ such that the second scenario occurs: necessarily $s\leq N_1$. For notational convenience, denote also $i_0=0$ and $i_{s+1}=R$. Between any $i_j,i_{j+1}$ the number $p(i)$ is strictly decreasing, and is undefined for $i=R$, yielding that $|\bm\sigma^{i_{j+1}}|\leq 6|\bm\sigma^{i_j}|$ for all $0\leq j\leq s$, because each old parity split has been replaced with at most six ones. Hence $|\bm\sigma^R|\leq 6^{N_1}|\bm\sigma^0|$.

Switching between wide and regular split sequences leaves the number of splits unaltered.
\end{enumerate}
\end{proof}

\subsection{Diagonal extensions}
\label{sub:diagext}
We will list some technical lemmas by Masur and Minsky, which require the following definitions:
\begin{defin}\label{def:diagext}
Let $\tau$ be a recurrent almost track that fills $S$. A \nw{diagonal extension} of $\tau$ is an almost track $\sigma$ of which $\tau$ is a subtrack, with the properties that
\begin{itemize}
\item if a branch $b$ of $\sigma$ is not included in $\tau$, then its interior lies entirely in a connected component of $S\setminus \bar\nei_0(\tau)$ (and its endpoints, in particular, are switches of $\tau$);
\item each peripheral annulus component in $S\setminus\nei(\sigma)$ is also a peripheral annulus component in $S\setminus\nei(\tau)$.
\end{itemize}

The set of all \emph{recurrent} diagonal extensions of $\tau$ will be denoted $\e(\tau)$. We will, furthermore, denote
$$\f(\tau)\coloneqq\bigcup_{\substack{\rho\text{ recurrent subtrack of }\tau\\ \rho\text{ fills }S}}\e(\rho)$$
(this differs from the notation employed in \cite{masurminskyi}, \cite{masurminskyq} where the same set is called $\mathrm{N}(\tau)$). Consequently it is possible to define $\ce(\tau)=\bigcup_{\rho\in \e(\tau)} \cc(\rho)$ and $\cf(\tau)=\bigcup_{\substack{\rho\text{ recurrent subtrack of }\tau\\ \rho\text{ fills }S}}\ce(\rho)$.

For $k\in\mathbb N$ and $\tau$ a recurrent almost track, let 
$$\ce_k(\tau)\coloneqq \bigcup_{\delta\in \e(\tau)} \{\gamma\in \cc(\delta)\mid \mu_\gamma(b)\geq k\text{ for each } b\in\br(\tau)\};$$
where, for each $\delta\in \e(\tau)$, $\mu_\gamma$ is understood to be the element corresponding to $\gamma$ in ${\mathcal M}_{\mathbb Q}(\delta)$: so $\ce_k(\tau)$ is the set of the curves $\gamma$ in $\ce(\tau)$ with the property that, among the diagonal extensions which carry it, there is one where $\gamma$ traverses each branch in the subtrack $\tau$ at least $k$ times.

Let also $\cf_k(\tau)\coloneqq \bigcup_{\substack{\rho\text{ recurrent subtrack of }\tau\\ \rho\text{ fills }S}} \ce_k(\rho)$.
\end{defin}

\begin{lemma}[\cite{masurminskyi}, Lemma 4.2]\label{lem:cf_decreasing}
Let $\tau,\tau'$ be recurrent train tracks on a surface $S$, and let $X\subseteq S$ be a subsurface. Suppose that $\tau'|X$ is carried by $\tau|X$, and that they both fill $S^X$ (or, equivalently, $X$). Then $\cf(\tau'|X)\subseteq \cf(\tau|X)$ and, if $\tau'|X$ is fully carried by $\tau|X$, then $\ce(\tau'|X)\subseteq \ce(\tau|X)$.

More specifically, if $\delta'\in \f(\tau'|X)$ is a diagonal extension of a subtrack $\sigma'$ of $\tau'|X$ which fills $X$, then $\delta'$ is fully carried by  $\delta\f(\tau|X)$ a diagonal extension of a subtrack $\sigma$ of $\tau|X$ which again fills $X$; and $\sigma'$ is fully carried by $\sigma$.
\end{lemma}
The original lemma has no reference to subsurfaces and induced tracks, but its proof works equally well to prove the above statement replacing occurrences of $S,\sigma,\tau$ there with $S^X,\tau'|X,\tau|X$ respectively, even if induced train tracks, being almost tracks, may have paths encircling some peripheral annuli.

Note that the notion of tie neighbourhood in \cite{masurminskyi} is slightly different from ours, as it is similar to our notion of $\bar\nei_0(\tau)$. Some terminology is different, too: note in particular that train tracks filling the surface they lie on are called \emph{large}, while the sets we denote $\ce,\cf$ correspond to the ones called $PE,PN$ there, respectively (they are, more generically, sets of \emph{transverse measures}, but it is not important). The last sentence in the statement above is not mentioned in the original statement, but is deducible from its proof.

\begin{lemma}[\cite{masurminskyi}, Lemma 4.5]\label{lem:diag_inters_control}
Let $\tau$ be a recurrent train track on a surface $S$, and let $X\subseteq S$ be a subsurface such that $\tau|X$ fills $S^X$. If $\alpha\in\ce(\tau|X)$ and $\beta\in\cc(X)$ is not carried by any diagonal extension of $\tau|X$, then $i(\alpha,\beta)\geq \min_{b\in\br(\tau)}\mu_\alpha(b)$.
\end{lemma}
Similarly as for the previous lemma, the original statement is not meant for induced train tracks. Again the proof goes through with the same modifications pointed out above, but it is worth highlighting some further points. In particular, from the beginning of the proof, we need that all biinfinite train paths along the complete lift of $\tau|X$ in $\Hy^2$ are uniformly quasi-geodesic: this is the content of Proposition \ref{prp:paths_quasi_geod}.

Another point that is worth marking concerns what sort of shapes $\Hy^2$ is cut into by $\widetilde{\tau|X}$. Since $\tau|X$ fills $S^X$, each of its complementary components in $S^X$ either has a polygon as its closure, or is peripheral (meaning that its closure in $\ol{S^X}$ includes an arc of $\partial\ol{S^X}$). Since polygons lift diffeomorphically, a similar property is true for $\widetilde{\tau|X}$ in $\Hy^2$. Any peripheral component $P$ of $\Hy^2\setminus \widetilde{\tau|X}$ must be part of $\mathcal H_\pm$, because the endpoints of $\tilde\beta$ cannot lie in $\partial\bar\Hy^2\cap P$: that would mean that $\beta$ is not a closed curve in $S^X$ but a properly embedded, and non-compact arc. These remarks ensure that the original proof works for our generalized statement.

\begin{lemma}[\cite{masurminskyi}, Lemma 4.4]\label{lem:ccnesting}
Let $\tau$ be a recurrent train track on a surface $S$, and let $X\subseteq S$ be a subsurface such that $\tau|X$ fills $X$. Then:
\begin{itemize}
\item if $X\cong S_{0,4}$ then $\nei_1(\ce_3(\tau|X))\subset \ce(\tau|X)$;
\item if $X\cong S_{1,1}$ then $\nei_1(\ce_2(\tau|X))\subset \ce(\tau|X)$;
\item if $X$ is of any other topological type, then $\nei_1(\ce_1(\tau|X))\subset \ce(\tau|X)$.
\end{itemize}
\end{lemma}
Here, with $\nei_k(\cdot)$ we mean the set of points at distance $\leq k$ from the given subset of $\cc(X)$.

\begin{proof}
Let us focus with the first case. According to the above lemma, for each essential closed curves $\alpha\in \ce_3(\tau|X)$ and $\beta\not\in \ce(\tau|X)$ we must have $i(\alpha,\beta)\geq 3$. So $d_\cc(\alpha,\beta)\geq 2$, and this proves the claim. The same argument proves the claim in the other cases.
\end{proof}

Recall that in \cite{masurminskyi}, for two train tracks $\sigma,\tau$ on a same surface, the authors define $d_\cc(\tau,\sigma)\coloneqq \min_{\beta\in V(\tau),\alpha\in V(\sigma)} d_\cc(\beta,\alpha)$.
\begin{lemma}[\cite{masurminskyi}, Lemma 4.7 (Nesting Lemma)]
There is a constant $D=D(S)$ such that if $\sigma$ and $\tau$ are two recurrent train tracks filling $S$, with $\sigma$ carried by $\tau$ and $d_\cc\left(V(\tau),V(\sigma)\right)>D$, then:
\begin{itemize}
\item for $S\cong S_{0,4}$, $\cf(\sigma)\subset \cf_3(\tau)$;
\item for $S\cong S_{1,1}$, $\cf(\sigma)\subset \cf_2(\tau)$;
\item for all other topological types, $\cf(\sigma)\subset \cf_1(\tau)$.
\end{itemize}
\end{lemma}
\begin{proof}
The third statement is the actual content of the lemma in \cite{masurminskyi}. For what concerns the other two statements, a simplified proof is possible. Recall, first of all, that $\cc(S_{0,4})$ and $\cc(S_{1,1})$ are isomorphic to the Farey graph (see for instance \cite{farb}, \S 4.1.1). In \cite{ibaraki}, Figure 10, there is a classification of the diffeomorphism types of train tracks on $S_{0,4}$ which are \emph{complete} i.e. are not subtracks of any other train track; whereas the only one in $S_{1,1}$ is given in Figure 4 there. This means that any other train track on these surfaces is diffeomorphic to a subtrack of the given ones; but it is easy to note that none of those fills the surface. So, for each track $\tau$ that fills $S_{0,4}$ or $S_{1,1}$ respectively, $\{\tau\}=\e(\tau)=\f(\tau)$.

For each of the given train tracks $\tau$ in $S_{0,4}$ (resp. in the only track given in $S_{1,1}$), the two vertex cycles $\alpha_1,\alpha_2$ intersect in 2 (resp. 1) points, so they are connected by an edge in $\cc(S_{0,4})$ (resp. $\cc(S_{1,1}$). It is therefore possible to enforce an identification of the curve complex with the Farey graph such that $\alpha_1,\alpha_2$ correspond to $0/1$ and $1/0$, respectively.

Let $\beta\in\cc(S)$ (where $S=S_{0,4}$ or $S_{1,1}$ accordingly) be a curve carried by $\sigma$, and for $j=1,2$ let $b_j$ the weight that $\beta$ assigns to a branch of $\tau$ which is traversed by $\alpha_j$ but not $\alpha_{3-j}$ (it will be the same for any such branch). Then, as an element of the Farey graph, $\beta$ corresponds to $\pm b_2/b_1$. If $\beta\not\in \cf_3(\tau)$ then either $b_1$ or $b_2$ is $\leq 2$. Suppose, without loss of generality, that $b_2\leq 2$. If $b_2=1$ then $d_{\mathrm{Farey}}(\pm b_1/b_2,1/0)=1$. If $b_2=2$ then $\lfloor b_1/2\rfloor/1$ lies at distance $1$ from both $\pm b_1/b_2$, $1/0$. This proves the claim.
\end{proof}

\begin{lemma}[\cite{masurminskyq}, Lemma 3.5]\label{lem:vertexnotinterior}
Let $\tau$ be a recurrent train track on a surface $S$, and let $X\subseteq S$ be a subsurface such that $\tau|X$ fills $X$. Then, if $\alpha \in V(\tau|X)$ then $\alpha\not\in \cf_1(\tau|X)$.
\end{lemma}
Again the original given proof works for induced train tracks, with the same replacement as the one pointed out for Lemma \ref{lem:cf_decreasing}. But, again, it is worth noting some points: first of all, $V(\tau|X)$ must consist of at least $2$ elements else it cannot fill $S^X$; but then, $\cc(\tau|X)$ does neither (by Lemma \ref{lem:decreasingfilling}). This replaces the argument used at the beginning of the original proof to show that the set of projective transverse measures has dimension $\geq 2$.

Secondly, the proof appeals to the injectivity of a map $P(\omega)\rightarrow\mathcal{ML}(S)$, basically the inverse map of the one mentioned in Proposition \ref{prp:measurecurvecorresp}. In order to suit better our setting, it may be more straightforward to say, to obtain the contradiction constructed in the original proof, that $v\in\cc(X)$ not having a unique carrying image in $\omega$ is a behaviour that transgresses Corollary \ref{cor:carryingunique}.

The assertion about positive generalized Euler characteristics (i.e. index) of the disc $D''$ is still a contradiction, even if an almost track gives a larger variety of possible complementary regions than in the original setting: this is because the original disc $D$ cannot intersect any lift of a peripheral annulus component of $S^X\setminus (\tau|X)$, anyway.

\section{Good behaviour of splitting sequences}\label{sub:goodbehaviour}

In the following sections we aim to prove the following statement.
\begin{theo}\label{thm:core}
Let $\bm\tau=(\tau_j)_{j=0}^N$ be a splitting sequence of semigeneric, birecurrent train tracks with their vertex sets $V(\tau_j)\in\pa(S)$ for all $0\leq j\leq N$. Then there is a constant $A>1$, depending only on $S$, such that
$$d_{\pa(S)}(V(\tau_0),V(\tau_N))=_A |\utw(\rar(\cnr\bm\tau))|.$$
\end{theo}
Here, $\rar$ and $\utw$ denote, respectively, a \emph{rearrangement} and an \emph{untwisting} of the splitting sequence $\bm\tau$, which we define in \S \ref{sub:rearrang} and in \S \ref{sec:traintrackconclusion}, respectively.

Our result about splitting sequences in a graph which is quasi-isometric to the pants graph comes after several other ones on the same line. The first one is a \emph{structure theorem} for cornered train track splitting sequences. We state it in a slighty altered way. Given $X\subseteq S$ a subsurface, and $\bm\tau=(\tau_j)_{j=0}^N$ a train track splitting sequence indexed by the interval $[0,N]$, let the \nw{accessible interval} for $X$ in $[0,N]$ be defined as follows.
\begin{itemize}
\item Suppose $X$ is not an annulus. Then define 
\begin{eqnarray*}
  m_X & \coloneqq &\min\left\{i\in [0,N]\mid \mathrm{diam}_X\left(\cc^*(\tau_i|X)\right)\geq 3\right\} \\
\text{and } n_X & \coloneqq & \max\left\{i\in [0,N]\mid \mathrm{diam}_X\left(\cc(\tau_i|X)\right)\geq 3\right\}.
\end{eqnarray*}
 If $m_X,n_X$ are both finite and $m_X\leq n_X$, then let $I_X=[m_X,n_X]$. Else set $I_X=\emptyset$.
\item Suppose $X$ is an annulus, with $\gamma$ its core curve. Then define $I_X$ to be the set of indices $i$ such that $\gamma$ is carried and is a twist curve in $\tau_i$. That the set $I_X$ is an actual interval is proved in Lemma \ref{lem:twistcurvebasics}. This definition of $I_X$ for annuli is different from the one of \cite{mms} thus the meaning of our statement is slightly different from the one in the in the original paper; but this new version is easily derived through Lemma \ref{lem:twistininduced}, as we will see. We will also use the notation $I_\gamma$ to mean $I_{\nei(\gamma)}$ where $\gamma\in\cc(S)$ and $\nei(\gamma)$ is a regular neighbourhood of a curve in the isotopy class $\gamma$.
\end{itemize}

\begin{theo}[\cite{mms}, Theorem 5.3]\label{thm:mmsstructure}
Given a splitting sequence $\bm\tau=(\tau_j)_{j=0}^N$ of generic, cornered, birecurrent train tracks on a surface $S$, there is a constant $\mathsf K_0=\mathsf K_0(S)$ such that, for each subsurface $X\subseteq S$, the following properties hold.

\begin{enumerate}
\item Let $[a,b]\subseteq [0,N]$ be an interval of indices, disjoint from $I_X$ except for, possibly, a single point. Then, if $\pi_X(\tau_b)\not=\emptyset$, then $d_X(V(\tau_a),V(\tau_b))\leq \mathsf K_0$. If $X$ is an annulus, also $d_X(\tau_a|X,\tau_b|X)\leq \mathsf K_0$.
\item If $X$ is an annulus, and $i\in I_X$, then $\tau_{i+1}|X$ is obtained from $\tau_i|X$ with slides (if this is the case between $\tau_{i+1}$ and $\tau_i$); or with at most 2 splits and/or taking a subtrack (if there is a split between $\tau_{i}$ and $\tau_{i+1}$).
\item If $X$ is not an annulus and $i\in I_X$, then $\partial X$, when put in efficient position with respect to $\tau_i$, is wide. Moreover $V(\tau_i|X)$ fills $S^X$. Finally, if $\tau_{i+1}|X\not=\tau_i|X$, then $\tau_{i+1}|X$ is obtained from $\tau_i|X$ with slides (if this is the case between $\tau_{i+1}$ and $\tau_i$); or with a split or taking a subtrack (if there is a split between $\tau_{i}$ and $\tau_{i+1}$).
\end{enumerate}
\end{theo}

The differences between this statement and the original one may be summarized into three points. The first one is the definition of accessible interval for an annulus, and this is cared after in Lemma \ref{lem:twistininduced} and the following observation. The other ones, instead, concern the first statement: the interval $[a,b]$ here may share an endpoint with $I_X$, which was not allowed in the original version. But if that statement is true then this one also is, possibly picking a larger value for $\mathsf K_0(S)$ --- see also Remark \ref{rmk:pickparameters} for the existence of a universal bound for the distance induced by a single split. A larger value of $\mathsf K_0$ is necessary also for the last sentence in this statement to be true; but it is possible to choose one, due to Lemmas \ref{lem:vertexsetbounds} and \ref{lem:induction_vertices_commute}.

\begin{rmk}
When a multicurve $\gamma$ is wide (not carried) with respect to a given train track $\tau$, in general, it is not necessarily true that all possible efficient positions for $\gamma$ comply with the conditions defining a wide curve in Definition \ref{def:efficientposition}: the definition of wide curve only asks for an efficient position as such to exist.

When $\gamma=\partial X$ as in the above theorem, anyway, \emph{any} efficient position with respect to $\tau_i$, $i\in I_X$, shows that $\partial X$ is wide. This is a consequence of Lemma 5.2 in \cite{mms}, where the proof works by contradiction, supposing that there exists any efficient position for $\partial X$ which does not show that it is wide.
\end{rmk}

A number of theorems prove that, in different measures, train track splitting sequences induce quasigeodesic in more than one of the graphs previously introduced.

\begin{theo}[\cite{masurminskyq}, Theorem 1.3]\label{thm:mm_cc_geodicity}
Given a train track splitting sequence $\bm\tau=(\tau_j)_{j=0}^N$ on a surface $S$, there is a constant $Q=Q(S)$ such that the set $\left(V(\tau_j)\right)_j$ is a $Q$-unparametrized quasi-geodesic in $\cc(S)$.
\end{theo}
The proof of this theorem employs all the lemmas listed in \S \ref{sub:diagext} above. We spend just a few words about a secondary issue: that is, such proof requires some slight adaptations in order to hold for $S_{0,4}$ and $S_{1,1}$, whose curve graph have specific definitions. In order to cope with those surfaces, we may employ the lemmas in \S \ref{sub:diagext}: i.e. read the theorem's proof replacing the lemmas employed there with the corresponding ones in \S\ref{sub:diagext}. Also, we need to replace each occurrence of  $\mathcal PN(\cdot)$ with $\cf(\cdot)$; and each occurrence of $int(\mathcal PN(\cdot)$ with $\cf_3(\cdot)$ or $\cf_2(\cdot)$ depending on whether we aim to prove the statement for $S_{0,4}$ or $S_{1,1}$, respectively.

The same theorem is true for subsurface projections (adding some technical hypotheses):
\begin{theo}[\cite{mms}, Theorem 5.5]\label{thm:mms_cc_geodicity}
Given a splitting sequence $\bm\tau=(\tau_j)_{j=0}^N$ of cornered, birecurrent train tracks on a surface $S$, there is a constant $Q=Q(S)$ such that, for each subsurface $X\subseteq S$ such that $\pi_X(V(\tau_N))\not=\emptyset$, the sequence $\left(\pi_XV(\tau_j)\right)_{j=0}^N$ is a $Q$-unparametrized quasi-geodesic in $\cc(X)$.
\end{theo}

And a stronger statement is true for the marking graph: it is the starting point for our Theorem \ref{thm:core}.
\begin{theo}[\cite{mms}, Theorem 6.1]\label{thm:mms_main}
Given a splitting sequence $\bm\tau=(\tau_j)_{j=0}^N$ of cornered, birecurrent train tracks on a surface $S$, whose vertex sets each fill $S$, there is a constant $Q=Q(S)$ such that the set $\left(V(\tau_j)\right)_{j\in J}$ is a $Q$-quasi-geodesic in $\ma(S)$.

Here $J\subseteq [0,N-1]$ is the set of indices $j$ such that $\tau_j$ splits to $\tau_{j+1}$
\end{theo}

Theorems \ref{thm:mmsstructure} and \ref{thm:mm_cc_geodicity} are employed in the proof of \ref{thm:mms_cc_geodicity}, and \ref{thm:mms_main} depends on all the previous three.
\section{All about twist curves}\label{sec:twistcurves}

\subsection{Terminology and basics}

When a train track splitting sequence spans long distances in annulus subsurface projections we find that, morally, they are caused by application of high powers of Dehn twist; they may be produced by elementary moves which are hidden and sparse in the sequence, and the key to track them down are \emph{twist curves}. We need some work to make this sentence precise.

First of all, suppose that $\tau$ is an almost track and that $\gamma$ is a \emph{wide carried} curve for $\tau$. Let $A_\gamma$ be a wide collar and let $X$ be a regular neighbourhood of $\gamma$ in $S$. Let $\hat p:S^X\rightarrow S$ be the covering map. We have seen in point \ref{itm:embeddedcore} of Remark \ref{rmk:annulusinducedbasics} that the core curve of $S^X$ (which we identify with $\gamma$ itself) is carried by $\tau^X$ and is embedded as a train path; $A_\gamma$ has a homeomorphic lift to $S^X$, which is an annulus having $\tau^X.\gamma$ as a boundary component.

\begin{defin}
Let $e$ be a small branch end of $\tau^X$ such that $e\cap\tau^X.\gamma$ is a switch of $\tau^X$. If $e\cap A_\gamma\not=\emptyset$, we say that $e$ \nw{hits} $A_\gamma$; else, $e$ \nw{avoids} $A_\gamma$. This terminology does not apply to branch ends of $\tau^X$ such that $e\cap(\tau^X.\gamma)$ is empty, or consists of more than one point.

Given a small branch end $e'$ of $\tau$, we say that it \emph{hits} $A_\gamma$ if $e'=\hat p(e)$ for $e$ a branch end of $\tau^X$ hitting $A_\gamma$; we say that $e'$ \emph{avoids} $A_\gamma$ if $e'$ does not hit $A_\gamma$ and $e'=\hat p(e)$ for $e$ a branch end of $\tau^X$ avoiding $A_\gamma$. With this definition we are including, among the small branch ends of $\tau$ avoiding $A_\gamma$, the ones which are part of $\tau.\gamma$, too. However, this terminology does not apply to branch ends of $\tau$ which are disjoint from $\tau.\gamma$.

The set of all branch ends hitting $A_\gamma$ and the set of all branches avoiding it will be called the two \nw{sides} of $\gamma$. 
\end{defin}

\begin{defin}\label{def:twistcurve}
Let $\tau$ be an almost track on a surface $S$, let $\gamma\in\cc(S)$ be wide carried by $\tau$ and let $X$ be a regular neighbourhood of $\gamma$. The curve $\gamma$ is a \nw{twist curve} for $\tau^X$ and for $\tau$ if the two following conditions hold.
\begin{itemize}
\item There is a wide collar $A_\gamma$ (to be viewed in $S^X$) such that, for each branch end $e$ of $\tau^X$ hitting $A_\gamma$, the $e$-orientation (see Definition \ref{def:eorientation}) on $\gamma$ is always the same. The collar $A_\gamma$ (regarded as either $\subseteq S$ or $\subseteq S^X$) will be called a \nw{twist collar}, and the orientation given to $\gamma$ the \nw{$A_\gamma$-orientation}.
\item Among the branches of $\tau^X$ included in $\tau^X.\gamma$, there is at least a large one.
\end{itemize}
A curve $\gamma$ whose realization as a train path in $\tau$ is embedded, and which may be regarded as a twist curve after picking a collar on \emph{either} of its sides, is called \nw{combed}.
\end{defin}

An example of twist curve for an almost track $\tau$ is given in Figure \ref{fig:twistcurve}. The second bullet above is equivalent to saying that there is a branch end $f$ in $\tau^X$ which avoids $A_\gamma$, and such that the $f$-orientation is opposite to the $A_\gamma$-orientation. In particular, if $\gamma$ is combed, the orientations given to $\gamma$ by twist collars on opposite sides are opposite. So, this condition implies that $\cc(\tau^X)\not=\emptyset$ and that $\gamma$ is carried by $\tau|X$, by point \ref{itm:gammacarriedininduced} in Remark \ref{rmk:annulusinducedbasics}.

\begin{figure}
\centering{\includegraphics[width=.9\textwidth]{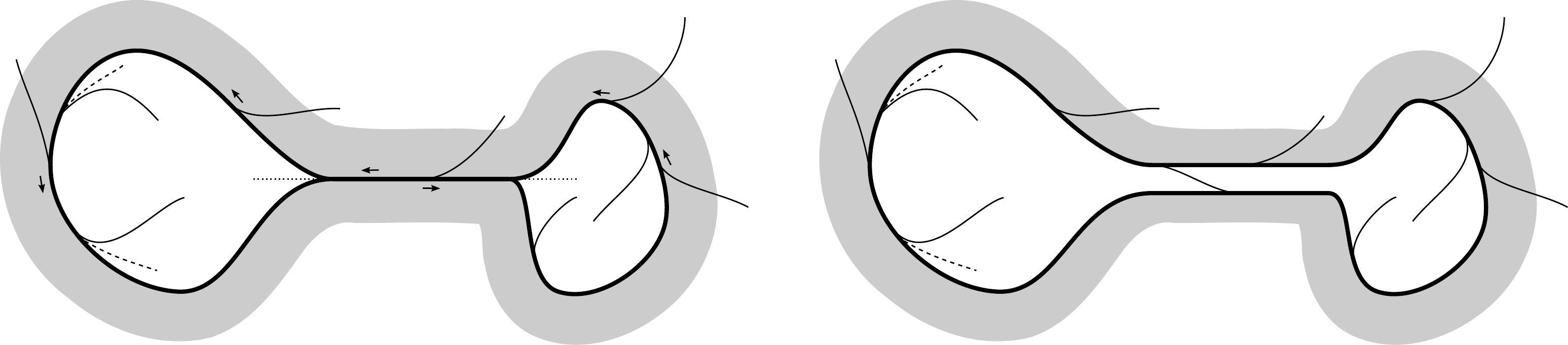}}
\caption{\label{fig:twistcurve}\emph{Left:} An example of twist curve which is not embedded as a train path. A possible choice for a twist collar $A_\gamma$ is marked in grey, and the arrows describe the $A_\gamma$-orientation. The extremities of splitting arcs for a wide spurious and a wide bispurious split (see Definition \ref{def:twistmoves}) are drawn with a dashed and a dotted line, respectively. \emph{Right:} The effect of a wide bispurious split. The twist curve has a twist collar that necessarily differs from the previous one. In Remark \ref{rmk:permanenceconventions} we show how to make the two collars correspond under a suitable homotopy equivalence $\E$ based on the tie collapse.}
\end{figure}

\begin{defin}
Let $\tau$ be an almost track on a surface $S$, let $\gamma\in\cc(S)$ be a twist curve for $\tau$ with a fixed regular neighbourhood $X$ and a fixed twist collar $A_\gamma$. Let $e$ be a branch end of $\tau^X$ sharing a switch with $\tau^X.\gamma$ and avoiding $A_\gamma$. If the $e$-orientation on $\gamma$ is opposite to the $A_\gamma$-orientation we call $e$ an \nw{$A_\gamma$-favourable} branch end; else we say that $e$ is \nw{$A_\gamma$-adverse}.
\end{defin}

For $\tau$, $\gamma$, $X$ as specified above, we say that $\gamma$ is a \emph{twist curve}, or it is \emph{combed}, in $\tau|X$ if it is a carried curve there and satisfies similar hypotheses as the ones that have been laid out for $\gamma$ in $\tau^X$. Lemma \ref{lem:twistininduced} specifies how being a twist curve in $\tau$ relates with being one in $\tau|X$.

\vspace{1ex}
\ul{Note}: all pretracks and splitting sequences in the present \S \ref{sec:twistcurves} will be \emph{generic}, as we will specify in our definitions and statements. Only in \S \ref{sub:twistcurvebound}, for technical reasons, we need to employ some semigeneric ones (see the note at the beginning of that subsection).

\vspace{1ex}
Note first of all that, if $\tau$ is a generic almost track with a wide curve $\gamma$, which has $A_\gamma$ as a wide collar and $X$ as a regular neighbourhood of the latter, then a small branch end $e'$ that hits (resp. avoids) $A_\gamma$ in $\tau$ has only one lift to $S^X$ which also hits (resp. avoids) $A_\gamma$\footnote{This is not true in the semigeneric setting: a branch end of $\tau$ which avoids $A_\gamma$ and shares its switch with two small branch ends in $\tau.\gamma$ has two distinct lifts that give opposite orientations to $\gamma$.}. So some of the definitions given above for $\tau^X$ descend to $\tau$.
\begin{itemize}
\item If $e'$ is a small branch end of $\tau$ which hits (resp. avoids) $A_\gamma$, let $e$ be the lift of $e'$ to $\tau^X$ which hits (resp. avoids) $A_\gamma$, and define the \emph{$e'$-orientation} on $\gamma$ to be the $e$-orientation.
\item If $\gamma$ is a twist curve with $A_\gamma$ a twist collar, and $e'$ is a small branch end of $\tau$ which avoids $A_\gamma$, let $e$ be the lift of $e'$ which avoids $A_\gamma$. We say that $e'$ is \emph{$A_\gamma$-favourable} (resp. \emph{adverse}) if $e$ is $A_\gamma$-favourable (resp. adverse).
\end{itemize}

\begin{rmk}\label{rmk:twistparam}
We set up here several pieces of notation referring to regular neighbourhoods of twist curves, to be used in the present section. Given $\gamma$ a twist curve of a generic almost track $\tau$, let $A_\gamma$ be a twist collar and $X$ be a regular neighbourhood of $\gamma$. For simplicity, we will suppose that $A_\gamma$ is a connected component of $\nei(\tau.\gamma)\setminus \tau.\gamma$ which is suitable to serve as twist collar. Also, we will identify it with its diffeomorphic lift to $S^X$, which serves as a twist collar for $\gamma\subset S^X$, regarded as a twist curve of $\tau^X$. Call $\Hy^2\stackrel{p}{\rightarrow} S^X\stackrel{\hat p}{\rightarrow} S$ the (metric) covering maps between these surfaces. 

Extend $p$ to the metric universal cover $\bar\Hy^2_X\stackrel{p}{\rightarrow}\ol{S^X}$ of $\ol{S^X}$ ($\bar\Hy^2_X$ is homeomorphic to a closed strip). We set up a parametrization as specified by the following commutative diagram:

\begin{equation*}
\begindc{\commdiag}[260]
\obj(1,3)[strip]{$\R\times[-2,2]$}
\obj(4,3)[hyclosed]{$\bar\Hy^2_X$}
\obj(1,1)[cylinder]{$\sph^1\times[-2,2]$}
\obj(4,1)[annulusclosed]{$\ol{S^X}$}
\mor{strip}{hyclosed}{$\tilde\upsilon$}
\mor{cylinder}{annulusclosed}{$\upsilon$}
\mor{hyclosed}{annulusclosed}{$p$}[\atleft,\surjectivearrow]
\mor{strip}{cylinder}{$e$}[\atleft,\surjectivearrow]
\mor{strip}{annulusclosed}{$q$}[\atleft,\surjectivearrow]
\enddc
\end{equation*}

Here, $e(x,t)=(e^{2\pi i x},t)$ and $\upsilon,\tilde\upsilon$ are built to be two diffeomorphisms, which are not required to be orientation-preserving (see below for a determination of whether they preserve orientations), and are not unique; however, we specify some requests below. To start with, we require that $\tau^X.\gamma=\upsilon(\sph^1\times\{0\})$ and $A_\gamma=\upsilon\left(\sph^1\times(0,1)\right)$.

Define $q\coloneqq\upsilon\circ e$; $\tilde\gamma\coloneqq \tilde\upsilon(\R\times\{0\})$; and $\tilde A_\gamma\coloneqq\tilde\upsilon\left(\R\times(0,1)\right)$. Also, orient the two components of $\partial\ol{S^X}$ consistently with the $A_\gamma$-orientation on $\gamma$; their lifts are to be oriented accordingly.

We require that the map $q$ complies with the following request. Let $b$ be a branch of $\tau^X$ hitting $A_\gamma$, and consider a connected component $\tilde b$ of $q^{-1}(b)\cap (\R\times [0,1])$: we require that, for all $t\in [0,1]$, $\tilde b\cap (\R\times\{t\})$ is a unique point that we denote $\tilde b(t)$, and that, if $\tilde b:[0,1]\rightarrow \R\times [0,1]$ is parametrized so that $\tilde b(1)$ lies along $\R\times\{0\}$, then the second coordinate of $\tilde b$ is decreasing.

This request means that, in $\R\times[0,1]$, the preimages of branches hitting $A_\gamma$ are directed from the upper left to the lower right. The $A_\gamma$-orientation on $\gamma$ together with the above request on $q$, then, determine whether $\upsilon$ is orientation-preserving or -reversing. In the first case we say that the twist curve $\gamma$ has \nw{positive sign}, in the second one that it has \nw{negative} one. This property does not depend on the twist collar $A_\gamma$ chosen; even in case $\gamma$ is combed, and we take two different twist collar which intersect opposite sides of $\gamma$, the sign of $\gamma$ with respect to either collar is the same.

A natural map $\R\times[0,1]\rightarrow \bar A_\gamma$, where the closure of $A_\gamma$ is to be meant in $S$, is obtained as $\hat p\circ q$; but, with an abuse of notation, we will denote it again as $q$, taking care of avoiding misunderstandings.

A number $x\in\R$ will be called an \nw{upper obstacle} for $\tau^X$ if $(x,0)\in q^{-1}(v)$ where $v$ is a switch of $\tau^X$ along $\tau^X.\gamma$, and incident to a branch end hitting $A_\gamma$; a \nw{lower obstacle} if $(x,0)\in q^{-1}(v)$ where $v$ is a switch of $\tau^X$ along $\tau^X.\gamma$, and incident to a branch end avoiding $A_\gamma$ and favourable; a \nw{fake obstacle} if $(x,0)\in q^{-1}(v)$ where $v$ is a switch of $\tau^X$ along $\tau^X.\gamma$, and incident to a branch end avoiding $A_\gamma$ and adverse. 

A ramp $\rho$ in $\tau^X$ will be said to be \nw{hitting} or \nw{avoiding} $A_\gamma$, to be \nw{$A_\gamma$-favourable} or \nw{adverse} in accordance with the properties of the only branch end adjacent to $\tau^X.\gamma$ and traversed by $\rho$.
\end{rmk}

\begin{rmk}\label{rmk:sameorientationforarcs}
Let $\gamma$ be a twist curve for a generic almost track $\tau$, with a twist collar $A_\gamma$; and let $\alpha\in\cc\left(\nei(\gamma)\right)$. 

The decomposition of a train path realization of $\alpha$ into $\rho_1,\beta,\rho_2$ given in point \ref{itm:horizontalstretch} of Remark \ref{rmk:annulusinducedbasics} gets a further property in this case: since one of between $\rho_1,\rho_2$ (say $\rho_1$) terminates with a branch end hitting $A_\gamma$, the last branch end of $\rho_2$, which avoids $A_\gamma$ must be favourable: it has been already noted, indeed, that the two branch ends must give opposite orientations to $A_\gamma$. If $\alpha$ is oriented so that $\rho_1$ is its first segment, then the segment $\beta$ is swept according to the $A_\gamma$-orientation.
\end{rmk}

\begin{lemma}\label{lem:twistininduced}
Let $\tau$ be a generic, recurrent almost track on a surface $S$, $\gamma\in \cc(S)$ be a curve carried by $\tau$, and $X$ be a regular neighbourhood of $\gamma$. Then $\gamma$ is a twist curve for $\tau$ if and only if it is a curve carried by, and combed in, $\tau|X$.

More specifically, if $\gamma$ is a twist curve, then $\tau|X$ contains all branches of $\tau^X$ hitting $A_\gamma$ in $S^X$ (which biject naturally with the ones of $\tau$ hitting $A_\gamma$ in $S$) and branch ends which avoid $A_\gamma$ and are favourable. Moreover, for each pair $e,e'$ where $e$ is a branch end hitting $A_\gamma$ and $e'$ is one which avoids $A_\gamma$ and is favourable, there is an element of $V(\tau^X)$ which traverses both.

The branches avoiding $A_\gamma$ and adverse to it do not belong to $\tau|X$.
\end{lemma}
\begin{proof}
Recall points \ref{itm:embeddedcore} and \ref{itm:uniquecarrying} in Remark \ref{rmk:annulusinducedbasics}, in particular that $\gamma$ is embedded as a train path in $\tau^X$.

No matter whether $\gamma$ is a twist curve for $\tau$, if it is one for $\tau|X$ then it is actually combed: for in Remark \ref{rmk:sameorientationforarcs} it is shown that, if $e$ is a branch end traversed by some $\alpha\in\cc(X)$ and avoiding $A_\gamma$, then $e$ is favourable.

Note moreover the following behaviour. Let $e,e'$ be two branch ends sharing a switch, $v,v'$, respectively, with $\tau^X.\gamma$, located on opposite sides of $\tau^X.\gamma$, and such that the $e$- and the $e'$-orientations on $\gamma$ are opposite. One can always find an incoming ramp $\rho$ traversing $e$ and an outgoing one $\rho'$ traversing $e'$. The switches $v,v'$ will cut $\tau^X.\gamma$ into two segments, only one of which, $\beta$, has at its extremes two large branch ends. The train route obtained from the concatenation of $\rho,\beta,\rho'$ is then a properly embedded arc, hence it is a realization of some element of $V(\tau^X)$ which traverses both $e, e'$.

Suppose now that $\gamma$ is a twist curve with a twist collar $A_\gamma$: then it is a twist curve in $\tau^X$ too and, since it is embedded there, the branches it traverses in $\tau^X$ include a large one only if a favourable branch end $e'$ exists. Pick a branch end $e$ of $\tau^X$ hitting $A_\gamma$: the pair $e,e'$, however the two are chosen, complies with the hypotheses above; thus they are traversed by one same element of $V(\tau^X)$, and in particular they are part of $\tau|X$; so all branches hitting $A_\gamma$ and all favourable ones belong to $\tau|X$.

Let $\alpha\in V(\tau|X)$ be any of the wide arcs constructed in the way specified above: it certifies that $\cc(\tau^X)\not=\emptyset$ and therefore, by point \ref{itm:gammacarriedininduced} in Remark \ref{rmk:annulusinducedbasics}, that $\tau^X.\gamma\subset \tau|X$. Moreover $(\tau|X).\gamma$ must necessarily include a large branch of $\tau|X$, else it is impossible for the arc $\alpha$ to enter $(\tau|X).\gamma$ and leave it after travelling through some of its branches.

Clearly, there is a side of $\gamma$ such that all branch ends of $\tau|X$ sharing a switch with $(\tau|X).\gamma$ and approaching it from that side will give $\gamma$ the same orientation, so $\gamma$ is a twist curve in $\tau|X$. Due to the argument at the beginning of this proof, it is combed; equivalently, no branch ends avoiding $A_\gamma$ and adverse belong to $\tau|X$. One implication of the first part the lemma's statement is thus proved, together with the part of the statement which specifies which branch ends around $\tau^X.\gamma$ belong to $\tau|X$.

We only have left to prove the remaining implication. Suppose that $\gamma$ is combed in $\tau|X$ --- in particular it is carried, thus so it is by $\tau^X$ and by $\tau$. The presence of a large branch of $\tau^X$ within $\tau^X.\gamma$ is proved with the same argument as in the paragraph above; hence there must be a large branch of $\tau$ within $\tau.\gamma$. 

We claim now that $\gamma$ is wide in $\tau$. Let $\ul\gamma$ be a carried realization: if it is not wide then --- regardless of the orientation that we put on it --- there are two branches $b_1,b_2$ of $\tau$, traversed at least twice, such that in $R_{b_1}$ one segment $s_1$ of $\ul\gamma$ sees another segment $t_1$ to its left, and in $R_{b_2}$ one segment $s_2$ of $\ul\gamma$ sees another segment $t_2$ to its right. Let $\ul{\hat\gamma}$ be the homeomorphic lift of $\ul\gamma$ to $S^X$. Call $\hat s_1,\hat s_2$ the lifts of $s_1,s_2$ which are contained in $\ul{\hat\gamma}$, and $\hat t_1,\hat t_2$ the lifts of $t_1,t_2$ which are located in the same branch rectangles as $\hat s_1,\hat s_2$. The segments $\hat t_1,\hat t_2$ belong to other lifts $\ul{\hat\gamma}_1,\ul{\hat\gamma}_2$ of $\ul\gamma$; they do not cross $\ul{\hat\gamma}$. So $\ul{\hat\gamma}_1$ must traverse two branch ends $e_{11},e_{12}$, each sharing a switch with $\tau^X.\gamma$, located on the same side of $\gamma$, and impressing on it opposite orientations. Similarly $\ul{\hat\gamma}_2$ must traverse branch ends $e_{21},e_{22}$ with the same properties. Suppose that the $e_{11},e_{21}$-orientations on $\gamma$ are opposite: then, by the observation at the beginning of the present proof, there is an element $\alpha_j\in V(\tau^X)$ traversing both $e_{1j},e_{2j}$ for $j=1,2$. Thus $e_{11},e_{12}$ both belong to $\tau|X$ contradicting the assumption that $\gamma$ is combed there.

Now we prove that, among the wide collars of $\gamma$ in $\tau^X$, there must be at least one twist collar (meaning, one such that all branch ends intersecting it give the same orientation to $\gamma$). Suppose not: then, on each side of $\gamma$ in $\tau^X$, we can find a pair of branch ends sharing a switch with $\tau^X.\gamma$ and impressing opposite orientations on $\gamma$. Call $e_{11},e_{12}$ the ones on one given side and $e_{21},e_{22}$ the others on the other one. The same contradiction as above occurs.

Note that a twist collar $A_\gamma$ for $\gamma$ in $\tau^X$ projects to a wide collar for $\gamma$ in $\tau$; and it will be a twist collar, too. This completes the proof.
\end{proof}

The above lemma is the reason why we have been able to state Theorem \ref{thm:mmsstructure} without problems due to having employed a definition of accessible interval for annular subsurfaces which is different from the one given in \cite{mms}. As for the first statement: by looking at the original statement, we see that our definition comprises an interval larger than the original one, and this does not affect the statement's validity.

The second statement of Theorem \ref{thm:mmsstructure} may instead seem stronger than what is proved in the original paper. But it is not in reality, because the proof only uses the property that the induced track under exam is combed.

\begin{lemma}\label{lem:twistcurvetrees}
Let $\gamma$ be a twist curve for a generic, recurrent almost track $\tau$, and let $X$ be a regular neighbourhood of $\gamma$. Then $\ol{\tau|X\setminus (\tau|X).\gamma}$ (the closure is meant in $\ol{S^X}$ here) is a disjoint union of trees: each of them intersects $(\tau|X).\gamma$ in a single point which serves as a root, and consists only of mixed branches of $\tau|X$. 

In particular, every train path between two endpoints of such a component is a ramp. If such a ramp $\rho$ is outgoing, then $\rho$ enters each traversed branch from its small end.
\end{lemma}
\begin{proof}
\step{1} given any two distinct outgoing ramps $\rho_1,\rho_2$ in $\tau|X$, their images either are disjoint or their intersection is a bounded, initial sub-train path of both.

Suppose that the images of the two ramps are not disjoint. Then, by Lemma \ref{lem:twistininduced}, either they both hit $A_\gamma$ or both avoid it and are favourable. Suppose the first alternative holds, the other being entirely similar. As a consequence of point \ref{itm:windaboutgamma} in Remark \ref{rmk:annulusinducedbasics}, both $\rho_1,\rho_2$ are embedded train paths. Even if $\rho_1,\rho_2$ do not begin at the same switch along $\tau^X.\gamma$, then (up to reversing indices) it is possible to have an embedded train path $\rho'_1$ which begins at the same switch as $\rho_2$, follows part of $(\tau|X).\gamma$ and then continues as $\rho_1$. The claim is proved if we show that the intersection between the images of $\rho'_1$ and $\rho_2$ is connected.

Suppose it is not. Then necessarily $\rho_1,\rho_2$ together bound a topological disc $B$, with 1 or 2 cusps along its boundary. Since $\tau|X\subseteq \tau^X$, $B$ is a union, along their respective boundaries, of regions as in Remarks \ref{rmk:idx_of_nei_diff} and \ref{rmk:negativeindexincover}, so it must have negative index: and this is a contradiction.

\step{2} the bounded branches of $\tau|X$ not belonging to $(\tau|X).\gamma$ are all mixed.

An immediate consequence of the previous step is that no branch in $\tau|X$ is traversed by distinct ramps in opposite directions.

If $b\in\br(\tau|X)$ is a large branch, not contained in $(\tau|X).\gamma$, then fix $\rho$ an outgoing ramp which traverses $\rho$, and let $\rho'$ be any other ramp which also does so. Then, since $b$ belongs to the intersection between the images of $\rho$ and $\rho'$, $\rho$ and $\rho'$ are not only traversing $b$ in the same direction but they are also coming from the same branch end adjacent to $b$. This means that one of the branch ends adjacent to $b$, is actually not traversed by any outgoing ramp in $\tau|X$: and this contradicts the definition itself of $\tau|X$.

If $b\in\br(\tau|X)$ is a small branch, not contained in $(\tau|X).\gamma$, then fix $\rho$ an outgoing ramp which traverses $\rho$, and let $e$ be the branch end $\rho$ traverses just after leaving $b$. Let $\rho'$ be an outgoing ramp traversing $e$ but not $b$. Then $\rho,\rho'$ traverse $e$ in the same direction, and their routes must coincide up to the point when they enter $e$. But this would require $\rho'$ to traverse $b$, too: and this is our contradiction.

\step{3} $\ol{\tau|X\setminus (\tau|X).\gamma}$ is a graph with no cycles.

If $C\subseteq \ol{\tau|X\setminus (\tau|X).\gamma}$ is a cycle, it consists entirely of mixed branches in $\tau|X$. But then it is smooth i.e. it represents a curve carried by $\tau|X$. This curve cannot be nullhomotopic, so it is necessarily homotopic to $\gamma$. But then, $C$ and $(\tau|X).\gamma$ project to $S$ and give two distinct carried realizations of $\gamma$, a contradiction.

Note that any outgoing ramp $\rho$ in $\tau|X$ will enter its first branch from a small end. Since the branch is mixed, $\rho$ leaves it through its large end; necessarily, $\rho$ shall traverse all its branches in a similar fashion. This ends the proof.
\end{proof}

\subsection{Splitting sequences seen at a twist curve}
\begin{defin}\label{def:twistmoves}
Let $\tau$ be a generic almost track on $S$ and let $\gamma$ be a twist curve for $\tau$. A \nw{twist split} about $\gamma$ is a parity split of a large $b\in\br(\tau)$ included in $\tau.\gamma$ and sharing one switch with a branch end hitting $A_\gamma$, where the parity is chosen so that $\gamma$ is still carried and stays a twist curve after the split.

A \nw{twist slide} about $\gamma$ is a slide as the one explicated by the second drawing in Figure \ref{fig:ttcombing}, provided that the horizontal line consists of branches in $\tau.\gamma$; and either the collar $A_\gamma$ is under such line and the slide is represented by the rightward arrow; or the collar $A_\gamma$ is above such line and the slide is represented by the leftward arrow. Note that the inverse move of a twist slide is not a twist slide.

A slide, split, or wide split on $\tau$ is \nw{far from $\gamma$} if there is a regular neighbourhood $\nei(\tau.\gamma)$ such that $\tau\cap\nei(\tau.\gamma)$ is not altered by the move. 

A move that is neither a twist one nor a far one from $\gamma$, but keeps $\gamma$ a twist curve, is called a \nw{spurious} move. A spurious (wide) split affecting a large branch traversed twice by $\gamma$ will be called \nw{bispurious}.

A \nw{wide twist split} about $\gamma$ is a wide split corresponding to a sequence of elementary moves where the only split is a twist split about $\gamma$. Similar definitions hold for a \nw{wide far}, \nw{wide spurious} or \nw{wide bispurious} split. Note that a wide split is far from $\gamma$ if and only if $\nei(\tau.\gamma)$ is not altered by the wide split as a whole.

We will say that a splitting sequence $\bm\tau=(\tau_j)_{j=0}^k$ has \nw{twist nature} about $\gamma$ if it consists only of twist splits and twist slides (\emph{not} wide twist splits) about $\gamma$.
\end{defin}

Note that the property of being a twist curve is never altered by a slide. A twist move always involves a branch $b$ along $\tau.\gamma$ with exactly one switch shared with a branch end hitting $A_\gamma$. Informally, one may view a twist move as moving this latter branch end forward along $\tau.\gamma$, according to the $A_\gamma$-orientation, so that its endpoint moves past the other switch of $b$. This idea is exploited with the introduction of \emph{modelling functions} for sequences of twist nature, which will be explained in Remark \ref{rmk:twistnaturemodelling}. However, note that the definition of twist nature is not easily adapted to a wide split setting: when realizing a wide twist split as a sequence of elementary moves, in particular, it may be necessary to employ slides which are not twist.

\begin{rmk}\label{rmk:spurious}
Let $\gamma$ be a twist curve for a generic almost track $\tau$. Suppose that a splitting arc $\alpha$ or a zipper $\kappa$ describe a move that leaves $\gamma$ carried. Then how is $\alpha$, or $\kappa$, placed with respect to $\tau.\gamma$? And with respect to a carried realization $\ul\gamma$ of $\gamma$? The following considerations are true also for `standard' (non-wide) splits, as a special case.

If the wide split along $\alpha$ is far from $\gamma$, then $\alpha$ is disjoint from both $\tau.\gamma$ and $\ul\gamma$. If the wide split is twist, then, however one picks a lift $\hat\alpha$ of $\alpha$ to $S^X$, its extremes lie in distinct components of $S^X\setminus\tau^X.\gamma$.

If the given wide split is spurious, then $\alpha$ and $\gamma$ admit disjoint carried realizations in $\nei(\tau)$. The examples of spurious and bispurious splitting arcs given in Figure \ref{fig:twistcurve}, left, should be able to communicate intuitively why this is true. Note that, when $\gamma$ is an embedded in $\tau$, $\ul\gamma$ can be taken to be a train path realization of $\gamma$. When it is not, there is no guarantee that $\alpha$ can be realized disjointly from $\tau.\gamma$: it cannot if and only if the wide spurious split is actually bispurious.

As for the case of a zipper $\kappa$, similar consideration apply. Namely, if all elementary moves specified by unzipping $\kappa$ are far, then $\kappa$ is disjoint from both $\tau.\gamma$ and $\ul\gamma$. More generally, $\kappa$ can always be made disjoint from $\ul\gamma$ (not from $\tau.\gamma$). Also, if $A_\gamma$ is a twist collar for $\gamma$ in $\tau$, the following alternative can be assumed: either the last point of $\kappa$ lies along $\tau.\gamma$, or the image of $\kappa$ does not intersect $A_\gamma$. In other words, the unzip $\kappa$ will never create a switch inside $A_\gamma$. This assumption does not pose any actual limitation as to the possible result of an unzip: the only restriction following from it concerns the realization of the result of the unzip within its isotopy class.

Note that a spurious split or wide split is possible only if $\gamma$ is not combed; and that the splitting arc corresponding to a spurious split must traverse one of the large branches in $\tau.\gamma$.
\end{rmk}

\begin{rmk}\label{rmk:consistentcollars}
Suppose that $\gamma$ is a twist curve in a generic almost track $\tau$ with a twist collar $A_\gamma$.

Let $\tau$ undergo a (wide) split which turns it into $\tau'$, but keeps $\gamma$ a twist curve: then, once a carried realization $\tau'\hookrightarrow \nei(\tau)$ is fixed, let $c:\tau'\rightarrow \tau$ be the restriction to $\tau'$ of the tie collapse $c_\tau:\bar\nei(\tau)\rightarrow \tau$. Then $c(\tau'.\gamma)=\tau.\gamma$; moreover one may find an extension $c:S\rightarrow S$, isotopic to $\mathrm{id}_S$, and a twist collar $A'_\gamma$ for $\gamma$ in $\tau'$, such that $A_\gamma= c(A'_\gamma)$. A more concrete way of choosing twist collars, consistently with a splitting sequence, will be described in Remark \ref{rmk:permanenceconventions}.

Meanwhile we give the following definition: let $T$ be a subset of the family of almost tracks fully carried by $\tau$, with the property that $\gamma$ is a twist curve for each $\sigma\in T$ (e.g. $T$ may be the family of the entries of a splitting sequence preserving $\gamma$). For each $\sigma\in T$, let $A_\gamma(\sigma)$ be a twist collar for $\gamma$. We say that the twist collars $\{A_\gamma(\sigma)\mid\sigma\in T\}$ (more appropriately, the pairs $\{(\sigma,A_\gamma(\sigma))\mid \sigma\in T\}$) form an \nw{$A_\gamma$-family} if, for each pair $\sigma,\sigma'\in T$ such that $\sigma$ carries $\sigma'$, the two twist collars $A_\gamma,A_\gamma'$ are related in the way described above.

The $A_\gamma(\sigma)$-orientations for $\gamma$ in the respective almost tracks are all the same, and we may as well just call then the $A_\gamma$-orientation. Similarly we may speak of a branch end of $\sigma$ which hits or avoids $A_\gamma$, when we actually mean that it hits or avoids $A_\gamma(\sigma)$.
\end{rmk}

\begin{lemma}\label{lem:twistcurvebasics}
Given a generic splitting sequence of almost tracks $\bm\tau$ and a curve $\gamma$, let $\Sigma$ be the set of indices $j$ such that $\gamma$ is a twist curve for $\tau_j$: then $\Sigma$ is an interval. If all entries in the splitting sequence are recurrent, then $\gamma$ can only cease to be a twist curve by ceasing to be carried.

Given a fixed $J\in\Sigma$ and $A_\gamma$ a twist collar for $\tau_J$, there is an $A_\gamma$-family of twist collars for $\bm\tau(J,\max \Sigma)$.
\end{lemma}

This means that, along a \emph{recurrent} splitting sequence, an element of $\cc(S)$ evolves through a subsequence of these stages:
\begin{enumerate}[label=\alph*.]
\item carried, not wide;
\item wide, not twist;
\item twist, not combed;
\item combed;
\item not carried.
\end{enumerate}

\begin{proof}
Let $j_0$ be the lowest of the indices $j$ such that $\gamma$ is a twist curve for $\tau_j$, with a twist collar $A_\gamma(j_0)$, and let $j_1$ be the highest index such that $\gamma$ is carried by $\tau_{j_1}$. Then $\gamma$ is carried by all $\bm\tau(j_0,j_1)$ because the set of carried curves can only shrink along a splitting sequence (Remark \ref{rmk:decreasingmeasures}). Also, $\gamma$ is wide in this sequence, as a family of wide collars $A_\gamma(j)$ for $\tau_j.\gamma$ ($j\in[j_0,j_1]$) can be defined recursively, with the condition that a suitable homotopy equivalence $S\rightarrow S$ mapping $\tau_j$ to $\tau_{j-1}$, as in Remark \ref{rmk:consistentcollars}, will also map $A_\gamma(j)$ to $A_\gamma(j-1)$.

Suppose that $\gamma$ is a twist curve in $\tau_j$ for a fixed index $j_0\leq j< j_1$, with $A_\gamma(j)$ a twist collar; so $\gamma$ is carried by $\tau_{j+1}$. If the elementary move on $\tau_j$ is a slide or there is neighbourhood of $\tau_j.\gamma$ not affected by the move, then clearly $\gamma$ is a twist curve in $\tau_{j+1}$ with $A_\gamma(j+1)$ a twist collar. So suppose that the elementary move operated on $\tau_j$ is a split which affects the train track close to $\tau_j.\gamma$. Then the branch being split is traversed by $\gamma$; the split can be considered to be a wide split along a splitting arc $\alpha$ that traverses one branch only.

Since $\gamma$ is a twist curve, at least one of the ends of $\alpha$ must lie on the side of $\tau_j.\gamma$ opposite to $A_\gamma(j)$. If exactly one does, then the only parity that keeps $\gamma$ carried after the split is the one specifying a twist split. If both endpoints lie opposite $A_\gamma(j)$, then any splitting parity keeps $\gamma$ carried, and the branch ends of $\tau_{j+1}$ hitting $A_\gamma(j+1)$ all give $\gamma$ the same orientation, similarly as before. So the only way $\gamma$ may fail being a twist curve in this last case is the absence of any large branch of $\tau_{j+1}$ among the ones traversed by $\gamma$. But if this is the case, then no further split may affect all neighbourhoods of $\tau_{j'}.\gamma$ for $j'\geq j$: therefore $\gamma$ will not return a twist curve at any later stage in the splitting sequence $\bm\tau$ --- but it may still be carried.

This proves that $\Sigma$, as defined in the statement, is an interval; also, $\{A_\gamma(j)\}_{j\in\Sigma}$ is an $A_\gamma(0)$-family. Moreover, if $\bm\tau$ is recurrent, the last kind of split described above, which keeps $\gamma$ carried but does not keep it a twist curve, cannot occur. In that event, indeed, $\tau_{j+1}$ is not recurrent, because any train path which traverses a branch end $e$ hitting $A_\gamma(j+1)$ is forced to remain within $\tau_{j+1}.\gamma$ and cannot get back to traversing $e$. So, in case $\bm\tau$ is recurrent, $\Sigma=[j_0,j_1]$ i.e. $\gamma$ ceases being a twist curve by ceasing being carried.

As for the last claim the lemma's statement, just restrict the arguments in this proof to the subsequence $\bm\tau(J,\max \Sigma)$: let $A_\gamma(J)\coloneqq A_\gamma$ given in the statement, and recursively construct the $A_\gamma$-family of twist collars.
\end{proof}

\begin{rmk}\label{rmk:permanenceconventions}
In the present Remark we describe how to choose consistent representatives of train tracks and twist collars in a splitting sequence such that a fixed curve is a twist curve for all tracks in the sequence.

Train tracks and almost tracks are are generally understood to be regarded up to isotopy, but it is convenient to fix some conventions concerning the way we regard almost tracks to change along a generic (possibly wide) splitting sequence $\bm\tau=(\tau_j)_{j=0}^N$, when our focus is on a curve $\gamma$ which stays a twist curve all along the sequence. This step is necessary because most of this section analyses how to discern the effect of twist moves from the other ones. Since all the following conventions are only a limitation when choosing an almost track within its isotopy class, they will be used in the proofs, but do not affect the validity of the statements. For a matter of convenience, we allow the sequence $\bm\tau$ to include moves which consist of isotopies only.

For all $0\leq j\leq N$, we will use the notation $A_\gamma(j)$ to mean a twist collar for $\gamma$ in $\tau_j$ such that $\{A_\gamma(j)\}_{j=0}^N$ is an $A_\gamma(0)$-family; and we denote $q_j$ the map $q:\R\times[-2,2]\rightarrow S^X$ (or $\rightarrow S$) built as specified in Remark \ref{rmk:twistparam}, but with specific reference to the almost track $\tau_j$.

It is convenient to communicate the idea behind the adopted conventions before they are described in detail. For each entry in the sequence $\left(q_j^{-1}(\tau_j^X)\right)_j$, i.e. the sequence $\bm\tau$ lifted to $\R\times[-2,2]$, the following entry $q_{j+1}^{-1}(\tau_{j+1}^X)$ shall not only include the equator $\R\times\{0\}$, but also appear naturally as a carried realization in $q_j^{-1}\left(\nei(\tau_j^X)\right)$.

Unfortunately it is not possible to keep the same map $q$ for all entries in the sequence, because one has to take into account the change of $\tau_j.\gamma$ under a bispurious split, for instance (see Figure \ref{fig:twistcurve}, right): in that case, adjusting $\tau_j$ within its own isotopy class will not suffice for us to keep the same map $q$ for all entries, and expect it to behave in compliance with Remark \ref{rmk:twistparam}. And the same problem arises when, albeit $\tau_j.\gamma,\tau_{j+1}.\gamma$ are isotopic, the isotopy between the two does not keep each point anchored along its tie in $\nei(\tau_j)$.

\begin{figure}
\centering
\def\svgwidth{\textwidth}
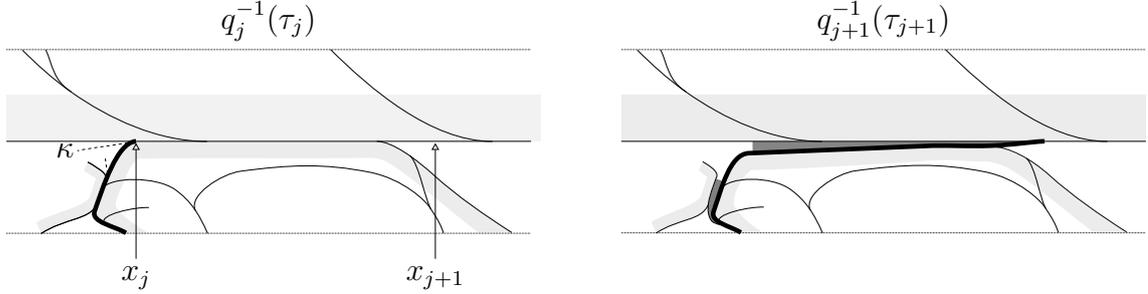
\caption{\label{fig:permanenceconventions} An illustration of the conventions set up in Remark \ref{rmk:permanenceconventions}, for a bispurious parity split, and in particular of its effects seen in $\R\times[-2,2]$. The pictures show only the changes for a region of $q_j^{-1}(\tau_j)$ and are not meant to be accurate. Three pieces of $q_j^{-1}(A_\gamma(j))$ and of $q_{j+1}^{-1}(A_\gamma(j+1))$ are shown in light grey. The split in $S$ is equivalent to unzipping a zipper, that lifts to an infinite family including two ones along $\R\times\{0\}$: call $\kappa$ one of these two. The main purpose of the set of established conventions is to make sure that, along a splitting sequence, for all $j$, $\tau_j.\gamma=q_j(\R\times\{0\})$ and $A_\gamma(j)=q_j\left(\R\times(0,1)\right)$ while each elementary move leaves the new pretrack naturally carried by the old one. Note indeed that, in these pictures, as $\kappa$ is unzipped, the result of the operation is not embedded as Definition \ref{def:zipper} would describe: it is slightly altered with an isotopy so that $\R\times\{0\}$ stays a lift for $\tau_{j+1}.\gamma$. No particular care is needed for any other unzips, far from $\R\times\{0\}$. The map $\E=\E_j$ defined in Remark \ref{rmk:permanenceconventions} lifts, via $q_j,q_{j+1}$, to a map of $\R\times[-2,2]$ which crushes the dark grey regions to a single line, and is a diffeomorphism on the complement of the grey regions (including the ones which are not drawn). If the split were only spurious, no region would be crushed: $\E_j$ would give a diffeomorphism of $\R\times[-2,2]$ --- and we may assume that it would give the identity. Finally, in these pictures two lifts of ramps, $\rho_j,\rho_{j+1}$, have been highlighted with a thicker line: they begin at the same point along $\R\times\{-2\}$ and are both $A_\gamma$-adverse. Note that the extremity $(x_j,0)$ of $\rho_j$ along $\R\times\{0\}$ --- here: a fake obstacle --- moves along $\R\times\{0\}$ in the direction suggested by the last branch end $\rho_j$ traverses; and reaches $(x_{j+1},0)$, which is the last point of the zipper $\kappa$. This is because any lift of $c_{\tau_j^X}\circ\rho_{j+1}$ via $q_j$ includes a segment along $\R\times\{0\}$, which is not part of $\rho_j$. Should $\rho'_{j+1}$ be such that $c_{\tau_j^X}\circ\rho'_{j+1}$ does not include a segment as above, the corresponding $\rho'_j$ would end at the same point of $\R\times\{0\}$ as $\rho'_{j+1}$. Note that, if a lift of an incoming ramp for $\tau_{j+1}^X$ begins at a given point $a$ of $\R\times\{\pm 2\}$, then necessarily $\tau_j^X$ also has an incoming ramp with one of its lifts beginning at $a$; but the inverse implication does not hold.
}
\end{figure}

We start describing the conventions from the following basic case (see also Figure \ref{fig:permanenceconventions}): let $\tau$ be an almost track where $\gamma$ is a twist curve with twist collar $A_\gamma$; and let $\tau''$, which carries $\gamma$, be obtained via the unzip of $\tau$ along a zipper $\kappa: [-\epsilon, t)\rightarrow \bar\nei(\tau)$ such that $\kappa_P$ is embedded.

Recall the notation used in Definition \ref{def:zipper}. Isotope $\tau''$ so that, if a switch lies in the set called $C$ there, then there is a switch of $\tau$ lying along the same tie of $\nei(\tau)$.

Let $R_\kappa$ be the union of the ties of $\nei(\tau)$ intersecting $\kappa$. Consider the restriction $c_\tau|_{\tau''.\gamma}$ of the tie collapse $c_\tau$ (see \S \ref{sub:traintrackdefin}), and extend it to a continuous map $\E:(S,\tau''.\gamma)\rightarrow (S,\tau.\gamma)$ with the following properties:
\begin{itemize}
\item $\E$ is surjective;
\item $\E|_{S\setminus R_\kappa}=\mathrm{id}_{S\setminus R_\kappa}$;
\item $\E$ maps each tie segment in $\bar\nei(\tau)$ to a segment or point of the same tie;
\item if $C\cap\tau_{j+1}.\gamma$ consists of exactly two \emph{smooth} components\footnote{Either two connected components which are smooth, or a single connected component consisting of two smooth paths joined at a cusp. It is exactly when one of these two scenarios occurs that we cannot keep the parametrization $q$ constant along the sequence $\bm\tau$: and in this case $\E$ shall describe the `crush' of the two pieces of $C$ to a single one.} note that, with the conditions set up so far, $\E(\bar\nei(\kappa))=\mathrm{im}(\kappa_P)$: then take $\E$ to be a diffeomorphism between $(S\setminus \bar\nei(\kappa),(\tau''.\gamma)\setminus C)$ and $(S\setminus\mathrm{im}(\kappa_P),(\tau.\gamma) \setminus\mathrm{im}(\kappa_P))$;
\item if not, then take $\E$ to be a diffeomorphism $(S,\tau''.\gamma)\rightarrow (S,\tau.\gamma)$.
\end{itemize}

Another equally basic case is the one of two almost tracks $\tau,\tau''$, both carrying $\gamma$, with $\tau''$ obtained from $\tau$ via a central split, along a branch $b$. Identify $R_b$ with its image. Also, identify $\tau''$ with a carried realization of it in $\nei_0(\tau)$, with the property that $\tau\setminus R_b=\tau''\setminus R_b$. Similarly as above, extend $c_\tau|_{\tau''.\gamma}$ to a map $\E$ with the same properties as above, except that we replace $R_\kappa$ with $R_b$, define $C$ as the union of the two `copies' of $b$ produced with the central split, and replace $\nei(\kappa)$ with $K$ the union of the open tie segments in $R_b$ each delimited by two points in $C$.

In all the above described scenarios, call $\F$ the inverse \emph{diffeomorphism} of $\E$ --- its domain and image will depend on the particular case as described above. The map $\E$ gives rise to a unique lift $\hat\E:S^X\rightarrow S^X$, via the covering map $\hat p$, with the property that it extends to the identity map on $\partial\ol{S^X}$.

The twist collar for $\gamma$ in $\tau''$ will be supposed to be $A_\gamma''= \F(A_\gamma)$: this is possible if the map $\E$ is realized appropriately, and/or $\nei(\tau'')$ is chosen wisely (recall indeed that we want $A_\gamma''$ to be one of the components of $\nei(\tau''.\gamma)\setminus(\tau''.\gamma)$. Given the parametrization $q:\R\times [-2,2]\rightarrow S^X$ for $\tau^X$, the corresponding parametrization $q''$ for ${\tau''}^X$ shall be taken to comply with the equality $q=\hat\E\circ q''$.

If all the elementary moves produced by unzipping $\kappa$ are far from $\gamma$, it is possible to take $\E=\F=\mathrm{id}_S$, $A_\gamma''=A_\gamma$, $q=q''$, and suppose that $\tau\cap\nei(\tau.\gamma)=\tau''\cap\nei(\tau.\gamma)$.

If $\kappa$ decomposes into two shorter zippers $\kappa_1,\kappa_1$, then the respective maps relate via $\E(\kappa)=\E(\kappa_2)\circ \E(\kappa_1)$ and $\F(\kappa)=\F(\kappa_1)\circ \F(\kappa_2)$. This allows us to give sense to the maps $\E,\F$ and to $A_\gamma''$ even if $\kappa_P$ is not an embedded path.

Back to the case of a (possibly wide) splitting sequence $\bm\tau$, each move can be regarded as the combination of one or more unzips. Remark \ref{rmk:generic_move_as_unzip}, and Remark \ref{def:multiplesplit} for wide splitting sequences, give some guidelines to convert a splitting sequence into a sequence of unzips and central splits, albeit not in a unique way. Exceptionally we will also allow sequences with trivial moves, i.e. unzips that only result in isotopies.

For each index $j$, then, a continuous surjection ${\mathcal E}_j:(S,\tau_{j+1}.\gamma)\rightarrow(S,\tau_j.\gamma)$ is defined by composing the maps $\E$ related with each of the unzips and central splits used to perform the move turning $\tau_j$ into $\tau_{j+1}$. Let $\F_j$ be its inverse, defined as the composition of the maps $\F$ seen above where they are defined.

If $b$ is the large branch, mixed branch or carried realization of the splitting arc involved in the move, then $\E_j$, restricted to $S\setminus \E_j^{-1}(b)$, is a diffeomorphism with its image; so $S\setminus b$ is always included in the domain of $\F_j$. Every time the elementary move or wide split is far from $\gamma$, $\E_j=\F_j=\mathrm{id}_S$. 

Denote $\hat{\E}_j:S^X\rightarrow S^X$ the lift of $\E_j$ as above. As a direct consequence of what established, the twist collars correspond under $A_\gamma(j+1)=\F_j\left(A_\gamma(j)\right)$, and in particular $\{A_\gamma(j)\}_{J=0}^N$ is an $A_\gamma(0)$-family. Meanwhile, the parametrization maps are be subject to the condition $q_j=\hat\E_j\circ q_{j+1}$. If the elementary move/wide split between $\tau_j$ and $\tau_{j+1}$ is far from $\gamma$, then $\tau_j\cap\nei(\tau_j.\gamma)=\tau_{j+1}\cap\nei(\tau_j.\gamma)$.

Remark \ref{rmk:limitsetconsistency} will be used multiple times. In particular, combined with the above conventions, it gives a constraint on the following construction. Fix $i<j$, and let $\rho_j:(-\infty,0]\rightarrow \tau_j^X$ be any incoming ramp; then there is a unique incoming ramp $\rho_i:(-\infty,0]\rightarrow \tau_i^X$ coming from the same point of $\partial\ol{S^X}$ and of the same kind as $\rho_j$ (by `kind' we mean hitting, favourable, or adverse; for the uniqueness, see point \ref{itm:dataforramp} in Remark \ref{rmk:annulusinducedbasics}). Actually, as a consequence of the conventions established here, $\rho_i$ is an initial segment of (a reparametrization of) $c_{\tau_i^X} \circ \rho_j$. Note that this last expression would not even be well defined, if a carried realization of $\tau_j$ in $\nei(\tau_i)$ were not fixed as we have done in the present Remark.

So, let $\hat\rho_j: (-\infty,0]\rightarrow \R\times[-2,2]$ be any lift of $\rho_j$ via $q_j$: it begins at a point $a\in\R\times\{-2,2\}$ and ends at an obstacle $x_j\in \R\times\{0\}$. There will be a connected component of $q_i^{-1}(\mathrm{im}(\rho_i))$ starting at the same point $a$; it ends at an obstacle $x_i$ for $\tau_i^X$, of the same kind (upper, lower, or fake) as $x_j$ for $\tau_j^X$.

The constraint, then, is that $x_j\geq x_i$ if they are upper or fake obstacles i.e. if $\rho_i,\rho_j$ are ramps hitting $A_\gamma$ or $A_\gamma$-adverse; and $x_j\leq x_i$ if they are lower obstacles i.e. if $\rho_i,\rho_j$ are ramps avoiding $A_\gamma$ and favourable. The inequalities are strict exactly when the image of $c_{\tau_i} \circ \rho_j$ has more than one point along $\tau_i^X.\gamma$, implying that one needs to trim it in order to get $\rho_j$.

Figure \ref{fig:permanenceconventions} summarizes part of the ideas behind these conventions. Another, simpler and more restrictive convention, which concerns sequences of twist nature only, will be described in Remark \ref{rmk:twistnaturemodelling}, and it clashes with parts of the assumptions made above. Therefore in the rest of this section, whenever necessary, we will clarify which of the two conventions we are about to use.
\end{rmk}

\begin{defin}
A \nw{twist modelling function} is a smooth map $h:\R\times[0,1]\rightarrow \R$ with the following properties. For all $x,t$, it holds that
$$
h(x+2\pi,t)=h(x,t)+2\pi\text{; }\frac{\partial}{\partial x}h(x,t)>0 \text{ and }\frac{\partial}{\partial t} h(x,t)\leq 0.$$
Moreover there exists and $\epsilon>0$ such that, for all $x$ and for $1-\epsilon \leq t<1$, the map satisfies $h(x,t)=x$.
\end{defin}

\begin{rmk}\label{rmk:twistnaturemodelling}
In the present Remark we describe how to choose consistent representatives of train tracks in splitting sequences of twist nature.

Consider a generic splitting sequence of almost tracks $\bm\tau=(\tau_j)_{j=0}^N$, of twist nature about a twist curve $\gamma$, and let $X$ be a regular neighbourhood of the latter in $S$. We use here the notations given in Remark \ref{rmk:twistparam}. 

As announced in Remark \ref{rmk:permanenceconventions} above, one may apply a different convention from the one seen there, when it comes to find a concrete realization of $\tau''$ obtained by applying a twist split or a twist slide on an almost track $\tau$, about a twist curve $\gamma$. Under this restriction, indeed, it is always possible to make $\tau.\gamma$ and $\tau''.\gamma$ coincide under an isotopy that keeps each point along its tie.

If $b\in\br(\tau)$, with $b\subseteq \tau.\gamma$, is the large/mixed branch that is being split/slid, then exactly one of its endpoints is adjacent to a branch end $e$ hitting $A_\gamma$: as it has been already noted, the result of the elementary move may be visualized as a shift of this branch end beyond the other endpoint of $b$.

In order to attain this visualization, the move may be realized with an unzip along a zipper $\kappa:[-\epsilon, t]\rightarrow \bar\nei(\tau)$ such that $\kappa\left([0,t)\right)$ is contained in $A_\gamma$. In this scenario, the map $\E$ defined in Remark \ref{rmk:permanenceconventions} is a diffeomorphism (isotopic to $\mathrm{id}_S$). But in this set of conventions, rather than keeping track of the map $\E$, the new track $\tau''$ is immediately replaced with $\F(\tau'')$. This is convenient as $A_\gamma$ is a twist collar for $\gamma$ in $\tau''$, too; and $\tau\setminus A_\gamma=\tau''\setminus A_\gamma$.

In other words, it becomes useless to appeal to $A_\gamma$-families, because there is no need to adapt the twist collar $A_\gamma$. Similarly, in this context the maps $\tilde\upsilon$, $\upsilon$, $q$ of Remark \ref{rmk:twistparam} are taken to be the same when applying the construction to $\tau$ or to $\tau''$. 

One may apply this convention for an entire sequence of twist nature $\bm\tau$. However, rather than consider each elementary move as the result of a \emph{single} unzip, it is better to keep the model less rigid: an elementary move may as well be the result of unzipping \emph{more} than a zipper as described above, posing the condition that all unzips but one result into isotopies of the almost track. This makes sure that constructions like the one in Lemma \ref{lem:functiongivestwist} work properly.

%Once a family of realizations $\bm\tau^{twist}$ for the entries of $\bm\tau$ has been defined, in compliance with this convention, it may be easily converted to a family of realizations $\bm\tau^{general}$ as in Remark \ref{rmk:permanenceconventions}. It suffices, indeed, \emph{not} to apply the isotopies which keep the sets $\tau_j.\gamma$ fixed as $j$ varies, and keep trace of the maps $\E_j$ instead. This means also that the covering maps $q_j$ for entries of $\bm\tau^{general}$ will be distinct: this compensates the fact that, when lifting the sequences $\bm\tau^{twist}$ to $\R\times[-2,2]$ via $q$, and each entry of $\bm\tau^{general}$ via the respective $q_j$, the two sequences of pretracks in $\R\times[-2,2]$ coincide.
The argument concerning endpoints of ramps, developed in Remark \ref{rmk:permanenceconventions}, applies in this set of conventions, too; but the twist nature of the sequence limits the variety of possible behaviours. Given an almost track $\tau$ where we unzip a zipper $\kappa$ as specified above, there is exactly one coset $x+2\pi\mathbb Z\subset \R$, where $x$ is an \emph{upper} obstacle, such that all obstacles (whether upper, lower or fake) of $\tau^X$ which do not belong to this coset are found also as obstacles of $(\tau'')^X$.

Define the obstacles $x_i,x_j$ similarly as in that paragraph. Then $x_j\geq x_i$ if they are upper obstacles, and $x_j=x_i$ if they are lower or fake obstacles. The first inequality is strict only when $x_i$ belongs to one of the families of upper obstacles $x+\mathbb Z$ which are affected by the unzips along $\bm\tau(i,j)$.

Moreover, if one defines similarly two pairs $(x_i,x_j)$ and $(x'_i,x'_j)$, then $x_i=x'_i$ if and only if $x_j=x'_j$. In words, two ramps in $\tau_i^X$ have their images partly overlapping if and only if their counterparts in $\tau_j^X$ partly overlap (possibly they only have one common endpoint).

This establishes a correspondence between sequences of twist nature and twist modelling functions, as we explain below.

Given a twist modelling function $h$, one can use it to define a self-diffeomorphism $\tilde h$ of $\R\times[0,1]$, by sending $(x,t)\mapsto(h(x,t),t)$; and it can also be extended to a map $\tilde h:\R\times [-2,2]\rightarrow \R\times [-2,2]$, by setting it to the identity outside $\R\times[0,1]$. This map $\tilde h$ is not continuous, as it fails along $\R\times \{0\}$. We call $\tilde H\coloneqq \tilde\upsilon\circ \tilde h \circ\tilde\upsilon^{-1}$ the corresponding self-bijection of $\bar\Hy^2_X\rightarrow \bar\Hy^2_X$.

Since $\tilde h$ is equivariant under horizontal translation by $2\pi$, it descends to a self-bijection of $\sph^1\times[-2,2]$ on $\ol{S^X}$ and, via conjugation by $\upsilon$, the latter is turned into a self-bijection $\hat H:\ol{S^X}\rightarrow \hat H$.

Finally, a map $H:S\rightarrow S$ is obtained by setting $H|_{A_\gamma}=\hat H|_{A_\gamma}$ and $H|_{S\setminus A_\gamma}\coloneqq \mathrm{id}_{S\setminus A_\gamma}$. This is again not continuous, but it is a self-diffeomorphism of $S\setminus \tau_0.\gamma$ which does not permute its connected components. Note a detail: $H$ fixes $\tau_0.\gamma$ pointwise, but $\hat H$ and $\tilde H$ fix $\tau_0^X.\gamma$ and $\tilde\gamma$, respectively, only setwise.

One can see (cfr. \cite{mosher}, p. 215) that, after modelling $\bm\tau$ under the conventions given above, for each pair of indices $0\leq k\leq l\leq N$, there exists a twist modelling function $h_k^l:\R\times[0,1]\rightarrow \R$ such that $\tau_l= H_k^l(\tau_k)$, where $H_k^l$ is defined with the process above; and we can also suppose that, given three indices $0\leq k\leq l \leq r \leq N$, it holds that $h_k^r(x,t)=h_l^r\left(k_k^l(x,t),t\right)$ i.e. $\tilde h_k^r=\tilde h_l^r\circ\tilde h_k^l$ and similar composition rules hold for the maps defined above. We will say that $h_k^l$ is a twist modelling function \nw{associated with} the splitting sequence $\bm\tau(k,l)$.
\end{rmk}

\begin{lemma}\label{lem:functiongivestwist}
Let $\tau_0$ be a generic almost track on a surface $S$ with $\gamma$ a twist curve, and $A_\gamma$ a fixed twist collar. Let $h:\R\times[0,1]\rightarrow \R$ be a twist modelling function. Define a bijection $H:S\rightarrow S$ as in Remark \ref{rmk:twistnaturemodelling}.

Suppose that $\tau_1\coloneqq H(\tau_0)$ is a (generic) almost track; equivalently that, for any $x$ upper obstacle for $\tau_0^X$, $h(x,0)$ is not a lower or fake obstacle of $\tau_0^X$. Then there exists a splitting sequence $\bm\tau$ turning $\tau_0$ into $\tau_1$, and having twist nature about $\gamma$ with twist collar $A_\gamma$. Moreover, $\bm\tau$ has $h$ as an associated twist modelling function.
\end{lemma}
\begin{proof}
Consider the smooth map $\Phi:\R\times[0,1]\times[0,1]\rightarrow \R\times [0,1]$ defined by the linear combination $\Phi(x,t,u)=(1-u)x+ u h(x,t)$. For all $u\in[0,1]$ the map $h_u(x,t)\coloneqq\Phi(x,t,u)$ is a twist modelling function; let $\tilde h_u$ be the corresponding self-map of $\R\times[-2,2]$, defined as prescribed in Remark \ref{rmk:twistnaturemodelling}. We have $\frac{\partial}{\partial u}\Phi(x,t,u)\geq 0$ for all $x,t,u$ and $\frac{\partial}{\partial u}\Phi(x,t,u)= 0$ if and only if $h_1(x,t)=x$.

For all $u$ let $H_u$ be the self-map induced by $h_u$ on $S$. Up to small perturbations of $\Phi$ not affecting the previously listed properties, one can also suppose that for every fixed $u\in[0,1]$, one of the following is true:
\begin{itemize}
\item there is one coset $x_0(u)+2\pi\mathbb Z$ such that, if $x$ is an upper obstacle for $\tau_0^X$ while $h_u(x,0)$ is a lower or fake obstacle for $\tau_0^X$, then $x\in x_0(u)+2\pi\mathbb Z$;
\item if $x$ is an upper obstacle for $\tau_0^X$, then $h_u(x,0)$ is neither a lower nor fake obstacle for $\tau_0^X$.
\end{itemize}

Let $0<u_1<\ldots<u_k<1$ be the values of $u$ as in the first bullet. Define, for $u\in[0,1]$, $\tau_u=H_u(\tau_0)$: it is a generic almost track exactly when $u$ is none of the aforementioned values.

Given two values $0\leq u<u'\leq 1$: if there is a $j$ such that $u_j<u<u'<u_{j+1}$ then $\tau_u,\tau_{u'}$ are turned into each other with an isotopy of $S$. If there is exactly one value of $j$ such that $u<u_j<u'$ then $\tau_{u'}$ is obtained from $\tau_u$ by shifting exactly one branch end hitting $A_\gamma$ beyond a branch end that avoids it (plus an isotopy); which is the same as saying: via a twist slide or a twist splitting.

So if we define another sequence $v_0,\ldots,v_k$ such that $0\leq v_0<u_1<v_1<u_2<\ldots<v_{k-1}<u_k<v_k\leq 1$, then $\bm\tau=(\tau_{v_j})_{j=0}^k$ is the desired splitting sequence, with twist nature.
\end{proof}

\begin{rmk}\label{rmk:twistfunctionsflexibility}
\begin{enumerate}
\item Given any increasing function $\eta:\R\rightarrow \R$ such that $\eta(x)\geq x$ for all $x$, there is a twist modelling function $h: \R\times[0,1]\rightarrow \R$ such that $h|_{\R\times\{0\}}=\eta$.

\item Fix a generic almost track $\tau$ carrying a twist curve $\gamma$; fix $A_\gamma$ and the other complementary constructions given in Remark \ref{rmk:twistparam}. Given any two twist modelling functions $h,h'$ and the relative maps $H,H':S\rightarrow S$ built from them, suppose that $h|_{\R\times\{0\}}=h'|_{\R\times\{0\}}$ and that $H(\tau)$ is a train track. Then $H'(\tau)$ is also a train track, isotopic to the former one: the isotopy is made explicit by the $1$-parameter family $H_u(\tau)$, where $u\in[0,1]$ and $H_u$ is the self-map of $S$ obtained from the twist modelling function $h_u=(1-u)h+uh'$.

\item A particular case of twist modelling function is given by the ones giving $h(x,1)=x+2\pi m$ for some fixed $m\in\mathbb N$. We consider these maps in the setting specified in Remark \ref{rmk:twistnaturemodelling}. Due to the $2\pi$-periodicity in the $x$ variable, the corresponding maps $\hat H$ and $H$ are continuous; and, due to the point above, one may as well suppose that they are smooth, up to isotopy. If $\epsilon$ is the sign of $\gamma$ as a twist curve, then $\hat H$ is isotopic to $D_X^{\epsilon m}$, where $D_X$ is the Dehn twist of $\ol{S^X}$ about its core curve; and $H$ is isotopic to $D_\gamma^{\epsilon m}$, where $D_\gamma$ is the Dehn twist of $S$ about $\gamma$.
\end{enumerate}
\end{rmk}

\subsection{The wide arc set at a twist curve}
We introduce a few more notations and immediate remarks about the arcs carried by an induced track in an annulus. We use here the parametrization of Remark \ref{rmk:twistparam}.

This definition integrates the definition of horizontal stretch given in point \ref{itm:horizontalstretch} in Remark \ref{rmk:annulusinducedbasics}:
\begin{defin}\label{def:horizontallength}
Let $\tau$ be a generic almost track where a specified curve $\gamma$ is a twist curve with a specified twist collar $A_\gamma$, and let $X$ be a regular neighbourhood of $\gamma$. For $\alpha\in\cc(\tau^X)$, let $\ul\alpha_\tau$ be a realization of $\alpha$ as a train path along $\tau^X$, and let $\ul{\tilde\alpha}_\tau$ be a lift of it to the universal cover $\bar\Hy^2_X$ of $\ol{S^X}$. The train path $\ul{\tilde\alpha}_\tau$ is necessarily embedded in $\tilde\tau$, and traverses one or more branches belonging to $\tilde\tau.\tilde\gamma$, which we already know to be consecutive.

An \nw{interval stretch} is an interval $[x,y]$ such that $[x,y]\times\{0\}= \tilde \upsilon^{-1}(\tilde\gamma\cap\ul{\tilde\alpha}_\tau)= (\R\times\{0\})\cap \tilde\upsilon^{-1}(\ul{\tilde\alpha}_\tau)$; there is one for each choice of $\ul{\tilde\alpha}_\tau$, so they are a family of intervals obtained from one another via addition of a $\mathbb Z$-multiple of $2\pi$. Necessarily, $x$ is an upper obstacle for $\tau^X$ and $y$ is a lower one (see Remark \ref{rmk:sameorientationforarcs}).

Finally, we call the \nw{horizontal length} of $\alpha$ the length of any of these intervals: $\hl(\tau,\alpha)\coloneqq y-x$ (this is independent of the choice of $\ul{\tilde\alpha}_\tau$).

If $\bm\tau=(\tau_j)_{j=0}^N$ is a splitting sequence of almost tracks such that $\gamma$ stays a twist curve throughout, and $\alpha\in\cc(\tau_i^X)$ we may denote, leaving that sequence implicit, $\hl_\alpha(i)= \hl(\tau_i,\alpha)$; $\ul\alpha_i=\ul\alpha_{\tau_i}$. In that case we choose the family $\ul{\tilde\alpha}_j$, for $0\leq j\leq i$, to be consisting of arcs all having the same pair of endpoints along $\partial\bar\Hy^2$. We denote $[x_j,y_j]$ the corresponding interval stretches, and $\hl_\alpha(j)=\hl(\tau_j,\alpha)$.
\end{defin}

We list some basic properties of the horizontal length:
\begin{enumerate}
\item With $x,y$ as in the above definition, consider the two segments of $\tilde\upsilon^{-1}(\ul{\tilde\alpha}_\tau)\setminus \left((x,y)\times\{0\}\right)$, i.e. the two lifts of ramps of $\ul\alpha_\tau$. The one having an endpoint at $(x,0)$ is contained in $\R\times[0,2]$ i.e. it projects to a ramp hitting $A_\gamma$ in $S^X$; while the other one is contained in $\R\times[-2,0]$ so it projects to a ramp which avoids $A_\gamma$ and is favourable.

\item \textit{For any $\alpha\in\cc(\tau)$, $\hl(\tau,\alpha)$ is never a multiple of $2\pi$.} If it were, that would mean that $\hs(\tau,\alpha)$ begins and ends at the same switch along $\tau^X.\gamma$. But this is impossible, because one of the two ramps composing $\tau^X.\alpha\setminus\tau^X.\gamma$ hits $A_\gamma$ and the other one avoids it: by genericness of $\tau$, they do not meet $\tau^X.\gamma$ at the same switch.

\item\label{itm:hl_vs_multiplicity} \textit{$\lfloor \hl(\tau,\alpha)/2\pi\rfloor$ is the minimum number of times $\alpha$ travels along a fixed branch within $\tau^X.\gamma$.} In particular this quantity is independent of metric properties and realization of $\tau$ within a given isotopy class. Moreover there is at least one branch traversed $\lfloor \hl(\tau,\alpha)/2\pi\rfloor+1$ times.

\item\label{itm:hlvertex} \textit{An arc $\alpha\in\cc(\tau^X)$ is wide carried if and only if $\hl(\tau,\alpha)<2\pi$}: this bound on horizontal length is equivalent to saying that $\hs(\tau,\alpha)$ does not traverse all branches in $\tau^X.\gamma$. Remark \ref{rmk:annulusinducedbasics}, point \ref{itm:windaboutgamma}, concludes our argument.

\item\label{itm:farisininfluent} \textit{Let $\bm\tau$ be a splitting sequence as in Definition \ref{def:horizontallength}, and fix $0\leq j<N$. Suppose the move between $\tau_j$ and $\tau_{j+1}$ is far from $\gamma$. Then, if the conventions of Remark \ref{rmk:permanenceconventions} are used then, for any $\alpha\in\cc(\tau_{j+1}^X)$, $\hs_\alpha(j+1)\subseteq \hs_\alpha(j)$. If the move between $\tau_j$ and $\tau_{j+1}$ is far from $\gamma$, then $\hs_\alpha(j)=\hs_\alpha(j+1)$.}

As a consequence of point \ref{itm:horizontalstretch}, a train path realization $\ul\alpha_{j+1}$ of $\alpha$ in $\tau_{j+1}^X$ is the concatenation of $\rho^h_{j+1},\hs_\alpha(j+1),\rho^f_{j+1}$ for $\rho^h_{j+1}$ an incoming ramp for $\tau_{j+1}^X$ hitting $A_\gamma$, and $\rho^f_{j+1}$ an outgoing ramp avoiding $A_\gamma$ and favourable.

A train path realization $\ul\alpha_j=c_{\tau_j^X}(\ul\alpha_{j+1})$ --- here we are neglecting any necessary reparametrization --- is the concatenation of $c_{\tau_j^X}(\rho^h_{j+1})$, $c_{\tau_j^X}(\hs_\alpha(j+1))=\linebreak\hs_\alpha(j+1)$, $c_{\tau_j^X}(\rho^f_{j+1})$: hence $\hs_\alpha(j+1)\subseteq \hs_\alpha(j)$.

When the move that takes place between $\tau_j,\tau_{j+1}$ is far from $\gamma$, the conventions in Remark \ref{rmk:permanenceconventions} imply that neither $c_{\tau_j^X}(\rho^h_{j+1})$ nor $c_{\tau_j^X}(\rho^f_{j+1})$ include any segment along $\tau_j^X.\gamma$. So $\hs_\alpha(j)=\hs_\alpha(j+1)$.

\item\label{itm:twistnaturehl} \textit{Let $\bm\tau$ be a splitting sequence as in Definition \ref{def:horizontallength}, with twist nature about $\gamma$, and fix two indices $0\leq k\leq l\leq N$. Model $\bm\tau$ as said in Remark \ref{rmk:twistnaturemodelling}, and let $h$ be the twist modelling function $h$ associated with $\bm\tau(k,l)$. Let $\alpha\in V(\tau_l)$, and let $[x_k,y_k],[x_l,y_l]$ be two consistent choices of interval stretches for $\alpha$ in $\tau_k$ and $\tau_l$, respectively; then $y_k=y_l$ and $x_l=h(x_k,0)\geq x_k$; and, in particular, $\hl_\alpha(l)\leq \hl_\alpha(k)$.}

Under the conventions of Remark \ref{rmk:twistnaturemodelling}, $\ul{\tilde\alpha}_l\setminus p^{-1}(\bar A_\gamma)= \ul{\tilde\alpha}_k\setminus p^{-1}(\bar A_\gamma)$, and this implies that $\ul{\tilde\alpha}_l\setminus(\R\times\{0\})=\tilde H\left(\ul{\tilde\alpha}_l\setminus(\R\times\{0\})\right)$. This gives immediately $y_k=y_l$, and $x_l=h(x_k,0)$.
\end{enumerate}

\begin{lemma}\label{lem:onerollingdirection}
Let $\tau$ be a generic almost track, and let $\gamma$ be a twist curve with sign $\epsilon$ and $X$ a regular neighbourhood of it. Then the following are true. 
\begin{itemize}
\item For each $m\in\mathbb N$,
$$
D_X^{\epsilon m}\cdot V(\tau^X)=\{\alpha\in\cc(\tau^X)\mid 2\pi m<\hl(\tau,\alpha)<2\pi(m+1)\}
$$
where $D_X$ is the Dehn twist about $\gamma$ in $S^X$. In particular, for all $\alpha\in V(\tau^X)$, $\hl(\tau,D_X^{\epsilon m}(\alpha))=\hl(\tau,\alpha)+2\pi m$; and
$$
\cc(\tau^X)=\bigcup_{j\geq 0} D_X^{\epsilon j}\cdot V(\tau^X).
$$
\item If $m\in\mathbb N,\alpha\in V(\tau^X), \beta\in D_X^{\epsilon m}\cdot V(\tau^X)$ then $m-1\leq i(\alpha,\beta)\leq m+1$. In particular, $\mathrm{diam} \left(V(\tau^X)\right)\leq 2$.
\end{itemize}
\end{lemma}
\begin{proof}
The first equality in the first bullet has been shown, in the particular case $m=0$, in point \ref{itm:hlvertex} of the above list. For $m>0$: given $\alpha\in V(\tau^X)$, construct $\ul{\tilde\alpha}_\tau$ and the corresponding interval stretch $[x,y]$ as prescribed in Definition \ref{def:horizontallength}.

Let $\tilde\rho_+=\ul{\tilde\alpha}_\tau\cap(\R\times(0,2])$ and $\tilde\rho_-=\ul{\tilde\alpha}_\tau\cap(\R\times[-2,0))$. Define then $\rho'_-\coloneqq\{(x'+2\pi m,y')\in\R\times[-2,0)|(x',y')\in \rho_-\}$, and $\ul{\tilde\beta}_\tau\coloneqq\rho_+\cup [x,y+2\pi m]\times\{0\}\cup \rho'_-$. The path specified by $\ul\beta=q(\ul{\tilde\beta}_\tau)\subset\ol{S^X}$ is then a train path along $\tau$, realizing some $\beta\in \cc(\tau^X)$. By construction, $\hl(\tau,\beta)=\hl(\tau,\alpha)+2\pi m\in (2\pi m, 2\pi(m+1))$; and $\beta=D^{\epsilon m}(\alpha)$. Since all $\beta\in D_X^{\epsilon m}\cdot V(\tau^X)$ may be realized this way, this construction proves the $\subseteq$ inclusion in the given equality.

To show the opposite inclusion, we can pick any $\alpha\in\cc(\tau^X)$ with $2\pi m<\hl(\tau,\alpha)<2\pi(m+1)$ and define $\tilde\rho_+,\tilde\rho_-$ as above. Then we define $\rho_-'=\{(x'-2\pi m,y')\in\R\times[-2,0)|(x',y')\in \rho_-\}$, and  $\ul{\tilde\beta}_\tau=\rho_+\cup [x,y-2\pi m]\times\{0\}\cup \rho'_-$ (note that, due to our hypothesis on $\alpha$, $y>x+2\pi m$). Then, similarly as above, this gives $\beta\in \cc(\tau^X)$ with $\alpha=D_X^{\epsilon m}(\beta)$. Moreover $0<\hl(\beta)<2\pi$, hence $\beta\in V(\tau^X)$.

The second equality of the first bullet is immediate.

A preliminary discussion for the second bullet: we have seen in Lemma \ref{lem:twistcurvetrees} that $\tau|X$ consists of $\tau^X.\gamma$ plus a forest $\sigma$ of trees each having its root along $\tau^X.\gamma$. Let $P\subset \R$ be the set such that $P\times \{0\}=q^{-1}(\text{switches of }\tau^X\text{ along }\tau^X.\gamma)$; and let $\nu>0$ be small, to be constrained more precisely in a bit. In the conventions set up in Remark \ref{rmk:twistparam}, we may add the further request that
$$\tilde\upsilon^{-1}(\sigma)\subseteq \bigcup_{a\in P} (a-\nu,a+\nu)\times[-2,2].$$
Then, let $\alpha\in\cc(\tau^X)$ and $\ul{\tilde\alpha}_\tau$ be defined as above. Also, let $\alpha_-=\ul{\tilde\alpha}_\tau\cap \R\times\{-2\}, \alpha_+=\ul{\tilde\alpha}_\tau\cap \R\times\{2\}$ be the two endpoints of $\ul{\tilde\alpha}_\tau$. Let $\seg(\alpha)$ be defined as the straight line segment joining $\alpha_-$ with $\alpha_+$: since $\seg(\alpha)$ and $\ul{\tilde\alpha}_\tau$ are isotopic relatively to their endpoints, $q\left(\seg(\alpha)\right)$ is a representative of the isotopy class $\alpha\in\cc(X)$.

Given any $X\subset \R\times [-2,2], j\in \mathbb Z$, denote $X+j=\{(a+2\pi j,b)|(a,b)\in X\}$; and $\mathcal O(X)=\bigcup_{j\in\mathbb Z}(X+j)$.

However we pick another $\alpha\not=\beta\in \cc(\tau^X)$, the collection of segments $\mathcal O\left(\seg(\alpha)\right)\cup \mathcal O\left(\seg(\beta)\right)$ does not bound any bigon, so $q\left(\seg(\alpha)\cup\seg(\beta)\right)$ does neither.  Hence, $i(\alpha,\beta)$ is the number of points of $\mathcal O\left(\seg(\alpha)\right)\cap \seg(\beta)$; which is the same as saying, the number of $j\in\mathbb Z$ such that $\seg(\alpha)+j$ intersects $\seg(\beta)$.

Now take $m\in\mathbb N,\alpha\in V(\tau^X), \beta\in D_X^{\epsilon m}\cdot V(\tau^X)$ as in the second bullet in the statement. Let $[x,y]$ be the interval stretch for $\alpha$ corresponding to a definite choice of $\ul{\tilde\alpha}_\tau$, and $[z,w]$ the one for a choice of a lift $\ul{\tilde\beta}_\tau$. Also, let $\alpha_+=(x',2),\alpha_-=(y',-2), \beta_+=(z',2),\beta_-=(w',-2)$: we have $|x'-x|<\nu$ and similarly for the other letters. If $\nu$ is small, then $y'>x'$ and $z'>w'$.

We claim that, without loss of generality, one may suppose $y'-x'<w'-z'$. Recall indeed that $0<y-x<2\pi$ and $2\pi m<z-w<2\pi(m+1)$. So, if $m>0$ and $\nu$ is small enough, then also $y'-x'<w'-z'$; whereas, if $m=0$, the roles of $\alpha$ and $\beta$ can be swapped, so the same supposition can be done by symmetry after discarding the only case left out: $y'-x'=w'-z'$. In this special case, indeed $\seg(\alpha)$ and $\seg(\beta)+j$ are parallel for any $j$, so $i(\alpha,\beta)=0$ consistently with our statement.

Now, $\seg(\alpha)+j$ intersects $\seg(\beta)$ if and only if their respective endpoints along $\R\times\{2\}$ come in the reverse order with respect to the ones on $\R\times\{-2\}$. And, as $y'-x'<w'-z'$, this condition is verified if and only if $[x'+2\pi j,y'+2\pi j]\subseteq (z',w')$ i.e. $j\in \frac{1}{2\pi}(z'-x',w'-y')$: we estimate how many integer $j$ lie within this interval.

From an algebraic manipulation of the above inequalities involving $x,y,z,w$ we get, on the one hand, $w-y>z-x +2\pi(m-1)$. If $\nu$ is small enough, we have also $w-y-2\nu>z-x+2\pi(m-1)+2\nu$, hence $(z'-x',w'-y')\supseteq (z-x+2\nu,w-y-2\nu)\supseteq (z-x+2\nu,z-x+2\nu+2\pi(m-1))$ --- we agree that an interval is empty if its supremum is lower than its infimum. In this last interval there are exactly $m-1$ elements of $2\pi\mathbb Z$, so the possible values of $j$ are at least as many: $i(\alpha,\beta)\geq m-1$.

On the other hand, a similar manipulation gives $w-y<z-x +2\pi(m+1)$ and, again for $\nu$ small enough, $w-y+2\nu<z-x+2\pi(m-1)-2\nu$. Hence $(z'-x',w'-y')\subseteq (z-x-2\nu,w-y+2\nu)\subseteq (z-x-2\nu,z-x-2\nu+2\pi(m+1))$. Since in this last interval there are exactly $m+1$ elements of $2\pi\mathbb Z$, we get $i(\alpha,\beta)\leq m+1$.

The last claim in the statement is a direct application of Lemma \ref{lem:annulus_distance}.
\end{proof}

We introduce the following definition to measure how many \emph{complete} Dehn twists are induced by a twist splitting sequence:

\begin{defin}\label{def:rot}
Let $\bm\tau=(\tau_j)_{j=0}^N$ be a generic splitting sequence of almost tracks such that a twist curve $\gamma$ stays a twist curve throughout, with a given $A_\gamma$-family of twist collars.

We define the \nw{rotation number} about $\gamma$ of a subsequence $\bm\tau(k,l)$ as
$$
\rot_{\bm\tau}(\gamma,k,l)\coloneqq \min_{\alpha\in\cc(\tau_l|X)}\left(\left\lfloor\frac{\hl_\alpha(k)}{2\pi}\right\rfloor-\left\lfloor\frac{\hl_\alpha(l)}{2\pi}\right\rfloor\right).
$$
When some of the data is understood from the context we will use a lighter notation such as $\rot(\gamma,k,l)$ or $\rot(k,l)$; or $\rot_{\bm\tau}(\gamma),\rot_{\bm\tau}$ when the rotation number is computed on the entire splitting sequence rather than a subsequence.
\end{defin}

\begin{rmk}\label{rmk:rotbasics}
We give some basic facts related with the rotation number, using the same notations already set up in the definition above. 

\begin{enumerate}
\item \label{itm:concatrot_below} \textit{If $0\leq k\leq r\leq l\leq N$, then $\rot(k,l)\geq\rot(k,r)+\rot(r,l)$.}
This is just because $\rot(k,l)\geq \min_{\alpha\in\cc(\tau_l|X)}\left(\left\lfloor\frac{\hl_\alpha(k)}{2\pi}\right\rfloor-\left\lfloor\frac{\hl_\alpha(r)}{2\pi}\right\rfloor\right) + \min_{\alpha\in\cc(\tau_l|X)}\left(\left\lfloor\frac{\hl_\alpha(r)}{2\pi}\right\rfloor-\left\lfloor\frac{\hl_\alpha(l)}{2\pi}\right\rfloor\right)\linebreak \geq \rot(k,r)+\rot(r,l)$. The second inequality is explained by replacing the minimum in the first summand with the one taken on the larger set $\cc(\tau_r|X)$.

\item\label{itm:rotwithvertexonly} \textit{An alternative definition is $\rot_{\bm\tau}(\gamma,k,l)= \min_{\alpha\in V(\tau_l|X)}\left\lfloor\frac{\hl_\alpha(k)}{2\pi}\right\rfloor$. In particular the rotation number is always nonnegative.}

We can replace the minimum over $\cc(\tau_l|X)$ with the one over $V(\tau_l|X)$ because of Lemma \ref{lem:onerollingdirection}.

\item \label{itm:rotofdehn} \textit{If $\bm\tau(k,l)$ is a splitting sequence of twist nature, associated with a twist modelling function $h$ having $h(x,0)=x+2\pi m$ for some $m\in\mathbb N$ (i.e. a neighbourhood of $\gamma$ in $S$ differs, from $\tau_k$ to $\tau_l$, according to the self-map $H=D_\gamma^{\epsilon m}:S\rightarrow S$) then $\rot(k,l)=m$.} This is a simple consequence of point \ref{itm:twistnaturehl} after Definition \ref{def:horizontallength}.

\item \textit{If $0\leq k\leq r\leq l\leq N$ and $\bm\tau(r,l)$ is a sequence consisting of slides only then $\rot(k,r)=\rot(k,l)$.} Since the inverse of a slide move is also a slide move, from point \ref{itm:concatrot_below} in this list we get $\rot(k,r)\leq\rot(k,l)\leq \rot(k,r)$.

\item \textit{The rotation number is not affected by changing the train tracks within their isotopy class and changing the choice for the parametrizations defined in Remark \ref{rmk:twistparam}.}

\item \label{itm:rot_vs_dt_vertices} \textit{If $m=\rot(k,l)$, then
$$V(\tau_l^X)\subseteq \left(D_X^{\epsilon m}\cdot V(\tau_k^X)\right)\cup \left(D_X^{\epsilon(m+1)}\cdot V(\tau_k^X)\right)\cup \left(D_X^{\epsilon(m+2)}\cdot V(\tau_k^X)\right).$$}
This is a consequence of Lemma \ref{lem:onerollingdirection}. By definition of rotation number, $\hl_\alpha(k)>2\pi m$ for all $\alpha \in V(\tau_l^X)$, and so $V(\tau_l^X)\subseteq \bigcup_{j\geq m} D_X^{\epsilon j}\cdot V(\tau_k^X)$. Fix $\alpha$ realizing the minimum in the definition of $\rot(k,l)$: then $\alpha= D_X^{\epsilon m}(\alpha')$ for some $\alpha'\in V(\tau_k^X)$.

Given any $\beta\in V(\tau_l^X)\cap \left(D_X^{\epsilon(m+j)}\cdot V(\tau_k^X)\right)$ for a fixed $j\geq 0$, write correspondingly $\beta=D_X^{\epsilon m}(\beta')$ for $\beta'\in D_X^{\epsilon j}\cdot V(\tau_k^X)$. Then said Lemma implies that $i(\alpha,\beta)=i(\alpha',\beta')\geq j-1$ while, on the other hand, $i(\alpha,\beta)\leq 1$ because these two arcs both belong to $V(\tau_l^X)$. So $j\leq 2$, as required.

\item \label{itm:rot_vs_dist}\textit{If $m=\rot(k,l)$, then $m\leq d_{\cc(X)}\left(V(\tau_k^X),V(\tau_l^X)\right)\leq m+4$.}

This is a consequence of the point above and Lemma \ref{lem:onerollingdirection}. The core of the argument is that, if $\alpha\in V(\tau_k^X)$ and $\beta\in V(\tau_l^X)$ are distinct, then $m-1\leq i(\alpha,\beta)\leq m+3$ and Lemma \ref{lem:annulus_distance} gives $m\leq d_{\cc(X)}(\alpha,\beta)\leq m+4$. 

\item \textit{If $\alpha\in\cc(\tau_l^X)$ then $\left\lfloor \frac{\hl_\alpha(k)}{2\pi}\right\rfloor-\left\lfloor \frac{\hl_\alpha(l)}{2\pi}\right\rfloor\leq  \rot(k,l)+2$.}

Point \ref{itm:rot_vs_dt_vertices} together with Lemma \ref{lem:onerollingdirection} give this for $\alpha\in V(\tau_l^X)$ first, and then for all $\alpha\in\cc(\tau_l^X)$.

\item \label{itm:concatrot_above} \textit{If $0\leq k\leq r\leq l\leq N$, then $\rot(k,l)\leq \rot(k,r)+\rot(r,l)+2$.}
If $\alpha\in \cc(\tau_l^X)$ realizes the minimum defining $\rot(r,l)$ then
$\rot(k,l)\leq \left(\left\lfloor \frac{\hl_\alpha(k)}{2\pi}\right\rfloor-\left\lfloor \frac{\hl_\alpha(r)}{2\pi}\right\rfloor\right) + \rot(r,l)$
and we use the point above.

\item \label{itm:tmfbeyondrot} \textit{If $\bm\tau(k,l)$ is a splitting sequence of twist nature, associated with a twist modelling function $h$, then $h(x,0)<x+2\pi(\rot(k,l)+3)$ for all $x\in\R$ which are upper obstacles for $\tau_k^X$.}

For each $x$ upper obstacle for $\tau_k^X$ there are some $\alpha\in V(\tau_l^X)$ and $y>x$ such that $[x,y]$ is an interval stretch for $\alpha$ with respect to $\tau_k$; from point \ref{itm:rot_vs_dt_vertices} above, $y-x=\hl_\alpha(k)<2\pi(\rot(k,l)+3)$; on the other hand, because of point \ref{itm:twistnaturehl} after Definition \ref{def:horizontallength}, $y-h(x)=\hl_\alpha(l)\in(0,2\pi)$; therefore $h(x)-x\leq 2\pi(\rot(k,l)+3)$.
\end{enumerate}
\end{rmk}

\begin{lemma}[Dehn twists with remainder]\label{lem:dehn+remainder}
Let $\bm\tau=(\tau_j)_{j=0}^N$ be a generic splitting sequence of almost tracks on a surface $S$ which has twist nature about a fixed curve $\gamma$ with sign $\epsilon$. Suppose the sequence is modelled according to Remark \ref{rmk:twistnaturemodelling}, with a fixed twist collar $A_\gamma$. Let $m\coloneqq \rot_{\bm\tau}(\gamma,0,N)$.

Then there is another splitting sequence $\bm\tau'=(\tau_j)_{j=0}^{N''}$ with twist nature about $\gamma$; there are two indices $0<N'\leq N''$ and a factorization $N''-N'=mk$, such that:
\begin{itemize}
\item $\tau'_0=\tau_0$, and $\tau'_{N''}$ is obtained from $\tau_N$ with slides only;
\item $\rot_{\bm\tau'}(\gamma,0,N')=0$;
\item for any choice of indices $N'\leq j < j+k\leq N''$, $\tau'_{j+k}=D_\gamma^{\epsilon}(\tau'_j)$, where $D_\gamma$ is the Dehn twist about $\gamma$ in $S$. In particular, $\tau'_{N''}=D_\gamma^{\epsilon m}(\tau_{N'})$.
\end{itemize}
\end{lemma}
\begin{proof}
Let $X$ be a regular neighbourhood of $\gamma$, and $h$ be the twist modelling function associated with $\bm\tau(0,N)$.

\step{1} make sure that $h(x,0),0\geq x+2\pi m$, possibly operating further slides.

This step is devoted to defining recursively the entries of an extension $\bm\tau(N,N''')$ of the splitting sequence $\bm\tau$. This extension will also be of twist nature, \emph{with no splits}, but with the possibility that some $\tau_{j+1}$ is isotopic to $\tau_j$.

The extension will be associated with a twist modelling function $h'''$, such that $h'''(h(x,0),0)\geq x+2\pi m$ for all $x\in\R$. Actually, it will suffice to show that this inequality holds for every $x$ which is an upper obstacle for $\tau_0^X$. If this is true, one will be able to adapt $h'''$ so that the inequality holds for all $x\in\R$, without changing $\tau_{N'''}$.

We fix some notation for the entries $\bm\tau(N,N''')$, which we are about to build. For $0\leq j \leq N'''-N$, label $\beta_j^0,\ldots,\beta_j^r$ the segments of $(\tau_{N+j}|X).\gamma$ delimited by two switches, and such that their extremities are located at large branch ends. Due to Lemma \ref{lem:twistininduced}, these segments are exactly the ones which can be obtained as $\hs_\alpha(N+j)$ for some $\alpha\in V(\tau_{N+j}^X)$; of course, all the $\alpha$ defining the same segment have the same $\hl_\alpha(N+j)$, and we call it $\hl(\beta_j^i)$.

Note that, as $\bm\tau$ is a splitting sequence of twist nature, two elements $\alpha,\alpha'\in V(\tau_{N+j}^X)$ have $\hs_\alpha(N+j)=\hs_{\alpha'}(N+j)$ if and only if $\hs_\alpha(N+j')=\hs_{\alpha'}(N+j')$ for each $j'\leq j$: it is just another way of phrasing the behaviour of obstacles along a sequence which was noted in Remark \ref{rmk:twistnaturemodelling}. This observation establishes a natural bijection between the collections $(\beta_j^i)_i$, $(\beta_{j'}^i)_i$ for $j< j'$: i.e. two segments $\beta_j^i$ and $\beta_{j'}^{i'}$ may be supposed to comply with the condition $i=i' \Longleftrightarrow$ there is an element $\alpha\in V(\tau_{N+j'}^X)$ such that $\beta_j^i=\hs_\alpha(N+j)$ and $\beta_{j'}^{i'}=\hs_\alpha(N+j')$.

For each $i$, then, fix an $\alpha^i\in V(\tau_N^X)=V(\tau_{N+1}^X)=\ldots=V(\tau_{N'''}^X)$ such that $\hs_{\alpha^i}(\tau_{N+j})=\beta_j^i$ for all $j$.

Furthermore, we assign superscripts so that $\beta_j^i\subset \beta_j^{i'}\Rightarrow i<i'$. In order to get this condition one may impose, for instance, that the sequence $\left(\hl(\beta_0^i)\right)_i$ is increasing. The slides along the sequence will then force $\beta_j^i\subset \beta_j^{i'} \Rightarrow \beta_{j+1}^i\subset \beta_{j+1}^{i'}$ for all $0\leq j < N'''-N$.

During the course of our recursion, each almost track $\tau_{N+j}$ will have the following property: 
\begin{center}
$\hl(\beta_j^i)\leq \hl_{\alpha^i}(0)-2\pi m$ for all $i\leq j$.
\end{center}

The sequence will then stop at the index $N'''=N+r$. At that point, recall point \ref{itm:twistnaturehl} after Definition \ref{def:horizontallength}: given any upper obstacle $x$ for $\tau_0^X$, there is an $\alpha^i$ such that two consistent choices for interval stretches of $\alpha^i$ with respect to $\tau_0$ and $\tau_{N'''}$ turn out to be $[x,y]$ and $[h'''(h(x,0),0),y]$ respectively, for some $y\in\R$ lower obstacle for both $\tau_0^X,\tau_{N'''}^X$. This implies $\hl_{\alpha^i}(j)\leq \hl_{\alpha^i}(0)-2\pi m\Rightarrow h'''(h(x,0),0)>x+2\pi m$, as desired.

We now enter the recursive process: suppose that $\tau_{N+j}$ has been defined for some $0\leq j <N'''-N$: we construct $\tau_{N+j+1}$. If $\hl_{\alpha^{j+1}}(N+j)<\hl_{\alpha^{j+1}}(0)-2\pi m$, then it is fine to set $\tau_{N+j+1}=\tau_{N+j}$. Else, denote $[x_0,y]$, $[\xi,y]$ two consistent choices for interval stretches of $\alpha^{j+1}$ with respect to $\tau_0$ and $\tau_{N+j}$, respectively. By assumption, $\xi<x_0+2\pi m$.

We claim the following:
\begin{itemize}
\item there exists a twist modelling function $h_j$ such that $h_j(\xi+2\pi \zeta,1)=x_0+2\pi(m+\zeta)$ for all $\zeta\in\mathbb Z$, and $h(x,1)=x$ for any $x\not\in\xi+2\pi\mathbb Z$ which is an upper obstacle for $\tau_{N+j}^X$; 
\item given $H_j$ the self-map of $S$ derived from $h_j$ with the construction given in Remark \ref{rmk:twistnaturemodelling}, $\tau_{N+j+1}\coloneqq H_j(\tau_{N+j})$ is isotopic, or obtained with twist slides only, from $\tau_{N+j}$.
\end{itemize}

If it is possible to define $\tau_{N+j+1}$ this way, then $\hl(\beta_{j+1}^{j+1})=y-h_j(\xi,0)= y-x_0-2\pi m= \hl_{\alpha^j}(0)-2\pi m$; and for, $i\leq j$, $\hl(\beta_{j+1}^i)\leq \hl(\beta_j^i) \leq \hl_{\alpha^i}(0)-2\pi m$ as required by the property claimed above.

The claim in the first bullet is true if and only if the interval $(\xi,x_0+2\pi m]$ contains no upper obstacles for $\tau_{N+j}^X$ (see also Remark \ref{rmk:twistfunctionsflexibility}), and the one in the second bullet is true if and only if $(\xi,x_0+2\pi m]$ contains no lower obstacles for $\tau_{N+j}^X$.

Suppose that an upper or lower obstacle $\bar\xi$ exists in the specified segment, for a contradiction. Then one between $[\xi,\bar\xi]$ and $[\bar\xi, y]$ is a connected component of $q^{-1}(\beta_j^i)$, for some $i$; clearly $i\leq j$, because by construction $\beta_j^i\subseteq \beta_j^{j+1}$. Therefore $\hl(\beta_j^i) \leq \hl_{\alpha^i}(0)-2\pi m$.

If $\bar\xi$ is a lower obstacle, then $[\xi,\bar\xi]$ is the one of the two segments which makes a connected component of $q^{-1}(\beta_j^i)$. But then $\hl(\beta_j^i) \leq \hl_{\alpha^i}(0)-2\pi m$ implies $\xi\geq x_0+2\pi m$, contrary to the assumption. If $\bar\xi$ is an upper obstacle, then $[\bar\xi, y]$ is a connected component of $q^{-1}(\beta_j^i)$. Let $[\bar x_0,y]$ be an interval stretch for $\alpha^i$ in $\tau_0^X$, consistent with the fixed interval stretch $[\bar\xi, y]$ related to $\tau_{N+j}^X$. The hypothesis $\hl(\beta_j^i)=\hl_{\alpha^i}(N+j)\leq \hl_{\alpha^i}(0)-2\pi m$ translates into $\bar x_0 \leq \bar\xi -2\pi m$. On the other hand, since $\bar\xi\leq x_0+2\pi m$ by definition, we get $\bar x_0\leq x_0$; and this is impossible as it would imply $\beta_0^i\supseteq \beta_0^{j+1}$ while $\beta_{N+j}^i\subsetneq \beta_{N+j}^{j+1}$.

This concludes the recursion argument. The claimed twist modelling function for $\bm\tau(N,N''')$ is, of course, $h'''\coloneqq h_{N'''-N-1}\circ \cdots \circ h_0$.

\step{2} proof of the lemma.

We can now change the notation partially: remove from $\bm\tau(N,N''')$ any $\tau_j$ such that $\tau_{j+1}$ is isotopic to it (the value of $N'''$ decreases accordingly). Moreover, since the original $h$ will not be needed, for a simpler notation let $h$ be the twist modelling map associated with the entire sequence $\bm\tau(0,N''')$. Call $H$ the self-bijection of $S$ that one obtains from $h$. Furthermore, consider a twist modelling function $h_D$ (and, consequently, a map $H_D:S\rightarrow S$) defined with the condition that $h_D(x,0)=x+2\pi\eqqcolon \eta_D(x)$. According to what noted in Remark \ref{rmk:twistfunctionsflexibility}, such a map $h_D$ exists, and $H_D$ is a diffeomorphism of $S$, in the isotopy class of $D_\gamma^\epsilon$. 

Consider now the function $\eta':\R\rightarrow\R$ defined by $\eta'(x)=h(x,0)-2\pi m$. Clearly $\eta'$ is a strictly increasing function; also, $\eta(x)\geq x$ since, by construction, $h(x,0)\geq x+2\pi m$. So, again, it is possible to find a twist modelling function $h':\R\times[0,1]\rightarrow \R$ with $h'|_{\R\times\{0\}}=\eta'$ (with the relative map $H':S\rightarrow S$).

The composition $(\eta_D)^{m}\circ \eta'$ agrees with $h|_{\R\times\{0\}}$. So, according to what noted in Remark \ref{rmk:twistfunctionsflexibility}, $(H_D)^{m}\circ H'(\tau_0)$ is isotopic to $H(\tau)=\tau_{N'''}$.

The new sequence $\bm\tau'$ is then the concatenation of $m+1$ sequences, each obtained from an application of Lemma \ref{lem:functiongivestwist}. The first one, $\bm\tau'(0,N')$, is a splitting sequence of twist nature built from $\tau_0$ and $h'$: the Lemma's hypotheses are met, as by definition $H'(\tau_0)=D_\gamma^{-\epsilon m}\circ H(\tau_0)$ is a generic almost track.

Let $\alpha\in V(\tau_{N'''}^X)$ be an arc realizing the minimum in the definition of $\rot_{\bm\tau}(0,N''')$; then $\hl_\alpha(0)\in (2\pi m,2\pi(m+1))$, so if $[x_0,y]$ and $[x_{N'''},y]$ are interval stretches for $\alpha$ in $\tau_0$ and $\tau_{N'''}$ respectively, then $x_0\in (y-2\pi(m+1),y-2\pi m)$ and $x_{N'''}=h(x_0,0) \in (y-2\pi,y)$. An interval stretch for $\alpha$ in $\tau'_{N'}$ is instead $[\xi,y]$ where $\xi=\eta'(x_0)= h(x_0,0)-2\pi m \in (y-2\pi(m+1),y-2\pi m)$. So, in the sequence $\bm\tau'$, both $\hl_\alpha(0), \hl_\alpha(N')\in (2\pi m,2\pi (m+1))$ and this implies that $\rot_{\bm\tau'}(0,N')=0$.

Another subsequence $\bm\tau'(N',N'+k)$ will turn $\tau_{N'}$ into $H_D(\tau_{N'})$; and, for $1\leq j<m$, one can consistently define $\bm\tau'(N'+kj)= (H_D)^j\cdot\bm\tau'(N',N'+k)$ --- i.e. transforming each entry of $\bm\tau'(N',N'+k)$ under $(H_D)^j$. Set $N''\coloneqq N'+km$: by definition, $\tau'_{N''}=\tau_{N'''}$.
\end{proof}

\begin{lemma}\label{lem:twistsplitnumber}
Let $N_3=N_3(S)$ be a constant such that, given any generic almost track $\tau$ on $S$, and any twist curve $\gamma$ for $\tau$, for each side of $\tau^{\nei(\gamma)}.\gamma$ and any orientation on $\gamma$ there are at most $N_3$ branch ends sharing a switch with $\tau^{\nei(\gamma)}.\gamma$, located on the specified side and giving $\gamma$ the required orientation.

If $\bm\tau=(\tau_j)_{j=0}^N$ is a splitting sequence of almost tracks and $\bm\tau(k,l)$ has twist nature about $\gamma\in W(\tau_k)$, then the number of twist splits in $\bm\tau$ is bounded by $N_3^2\left(\rot(k,l)+3\right)$.
\end{lemma}

\begin{proof}
The existence of the constant $N_3$ is just a consequence of more general bounds on combinatorics of almost tracks (Lemma \ref{lem:vertexsetbounds}).

Model the sequence $\bm\tau(k,l)$ in accordance with Remark \ref{rmk:twistnaturemodelling}. Every time $\tau_{j+1}$ is obtained from $\tau_j$ with a twist split, then $\tau_{j+1}^X$ is obtained from $\tau_j^X$ with infinitely many splits, only one of which is a twist split along $\gamma$. So we count the number of those twist splits instead. For $k\leq j<j'\leq l$, let $h_j^{j'}$ be the twist modelling function associated with $\bm\tau(j,j')$.

Denote $\ldots y^{-1}, y^0, y^1,\ldots$ the ordered biinfinite sequence in $\R$ of lower obstacles for $\tau_k^X$ --- suppose for simplicity that $y^0=0$; and let $x_k^1,\ldots,x_k^r$ be the upper obstacles for $\tau_k^X$ lying within the interval $(0,2\pi)$; these have a natural bijection with the branch ends in $\tau_k^X$ hitting $A_\gamma$. 

Note that, according to the conventions as in Remark \ref{rmk:twistnaturemodelling}, $\ldots y^{-1}, y^0, y^1,\ldots$ are lower obstacles for $\tau_j^X$ for all $j\geq k$, too. For each $k<j\leq l$ we define instead $x_j^i\coloneqq h_k^j(x_k^i)$ for $i=1,\ldots,r$. Each $x_j^i$ is an upper obstacle for the respective $\tau_j^X$. So each sequence $(x_j^i)_j$ describes the alterations of a branch end in $\tau_k$, hitting $A_\gamma$, along the sequence $\bm\tau(k,l)$: the position of its endpoint along the carrying image of $\gamma$ changes along the sequence, moving long a direction specified by the $A_\gamma$-orientation on $\gamma$. As already specified in Remark \ref{rmk:twistnaturemodelling}, $j'>j\Rightarrow x_{j'}^i\geq x_j^i$.

The elementary move between $\tau_j$ and $\tau_{j+1}$, then, is a (twist) splitting if and only if there are two indices $1\leq a\leq r$, $b\geq 0$ such that $x_j^a<y^b<x_{j+1}^a$. In this case the choice for the indices $a, b$ is unique.

As a consequence of point \ref{itm:tmfbeyondrot} in Remark \ref{rmk:rotbasics}, for each $i=1,\ldots,r$ we have $x_l^i<x_k^i +2\pi(\rot(k,l)+3)$. In each interval $\left[2\pi \zeta,2\pi(\zeta+1)\right)$, $\zeta\in\mathbb N$, there can be at most $N_3$ lower obstacles of any fixed $\tau_j^X$, because they must be lifts of endpoints of distinct branch ends avoiding $A_\gamma$ and favourable. Hence, in the splitting sequence $\bm\tau$, between each $x_k^i$ and the corresponding $x_l^i$ there at most $N_3\left(\rot(k,l)+3\right)$ lower obstacles (of any of the tracks in the sequence). The total number of splits in $\bm\tau$ is bounded by total number of indices $j$ such that a choice of $a,b$ as above exists: since also $r\leq N_3$, a bound is $N_3^2\left(\rot(k,l)+3\right)$.
\end{proof}

\begin{lemma}\label{lem:vertexsetnontwist}
Suppose that, in a generic splitting sequence $\bm\tau$ of almost tracks on a surface $S$, a curve $\gamma$ is a twist curve for some $\tau_j,\tau_{j+1}$ and the elementary move between these two train tracks is \emph{not} a twist split about $\gamma$. Let $X$ be a regular neighbourhood of $\gamma$. Then $V(\tau_{j+1}|X)\subseteq V(\tau_{j}|X)$.
\end{lemma}
\begin{proof}
It suffices to prove the statement in the case the move occurring is a split which is spurious or far from $\gamma$. Suppose that $\alpha\in V(\tau_{j+1}|X)$ fails to be wide in $\tau_j|X$, i.e. $\hl_{\alpha}(j)>2\pi$ while $\hl_{\alpha}(j+1)<2\pi$: by point \ref{itm:farisininfluent} after Definition \ref{def:horizontallength}, we have that the split may only be spurious, because $\lfloor \hl_{\alpha}(j)/2\pi\rfloor$ and $\lfloor \hl_{\alpha}(j+1)/2\pi\rfloor$ are distinct quantities, and both well-defined up to isotopies. 

Let $b\subseteq \tau_j.\gamma$ be the branch that is about to be split. We use the conventions set up in Remark \ref{rmk:permanenceconventions} and repeat the notation used in the first paragraph there. Also, let $\tilde{\ul\alpha}_j, \tilde{\ul\alpha}_{j+1}$ be two lifts of $\ul\alpha_j, \ul\alpha_{j+1}$ via the maps $q_j,q_{j+1}$ respectively, with the same endpoints on $\R\times[-2,2]$. Let $[x_j,y_j]$ be the interval stretch for $\alpha$ in $\tau_j^X$ consistent with the choice of the lift $\tilde{\ul\alpha}_j$, and define $[x_{j+1},y_{j+1}]$ similarly, replacing all occurrences of the index $j$ with $j+1$. As a consequence of Remark \ref{rmk:sameorientationforarcs}, the points $x_j,x_{j+1}$ are upper obstacles while $y_j,y_{j+1}$ are lower ones.

Since the split is spurious, $\hat p^{-1}(R_b)$ has 1 or 2 components intersecting $\tau_j^X.\gamma$ (as many as the number of times $b$ is traversed by $\gamma$). 

In the rest of the proof, square brackets will enclose adaptations that apply for the case of $\gamma$ traversing $b$ twice. Let $b_1$[, $b_2$] be the [two] lift[s] of $b$ that lie along $\tau_j^X.\gamma$; and let $e_1$[, $e_2$] be the only favourable branch end[s] sharing a switch with $b_1$[, $b_2$, respectively]. 

As seen in point \ref{itm:horizontalstretch} of Remark \ref{rmk:annulusinducedbasics}, $\ul\alpha_{j+1}$ is the concatenation of train paths $\rho_{j+1}^h,\hs_\alpha(j+1),\rho_{j+1}^f$ where the first is an incoming ramp for $\tau_{j+1}^X$ hitting $A_\gamma$, while the last is an outgoing ramp avoiding $A_\gamma$ and favourable. A similar decomposition $\rho_j^h,\hs_\alpha(j),\rho_j^f$ for $\ul\alpha_j$ holds.

Using the conventions of Remark \ref{rmk:permanenceconventions}, $\tau_{j+1}$ is obtained from $\tau_j$ unzipping a single zipper, defined on an interval $[-\epsilon,t]$ for $1<t<2$. This lifts to an infinite family of zippers for $\tau_j^X$, and only one [two, resp.] of these intersects $\bar\nei(\tau^X.\gamma)$: we call it [them, resp.] $\kappa_1$[, $\kappa_2$ according to which of $b_1,b_2$ they traverse].

Note that, since unzipping $\kappa_1$ [and $\kappa_2$] does not realize a twist split, this [these two, resp.] zipper[s] cannot intersect $A_\gamma$. Therefore $\rho_j^h=c_{\tau_j^X}\left(\rho_{j+1}^h\right)$ i.e. the tie collapse cannot create a segment lying along $\tau_j^X.\gamma$, and in particular $x_j=x_{j+1}$. The only way to have $\hl_\alpha(j+1)<2\pi<\hl_\alpha(j)$, then, is that $y_{j+1}<y_j$.

If $\rho_j^f$ does not begin with $e_1$ [nor with $e_2$], then also $\rho_j^f=c_{\tau_j^X}\left(\rho_{j+1}^f\right)$. But this would imply that ${\ul\alpha}_j,{\ul\alpha}_{j+1}$ have the same horizontal stretch, so $\hl_\alpha(j)=\hl_\alpha(j+1)$ leading to a contradiction.

Thus, [without loss of generality] $\rho_j^f$ begins at $e_1$; and this means that there is a fake obstacle $x_j<w_j<y_j$ such that $[w_j,y_j]\times\{0\}$ is one of the connected components of $q_j^{-1}(b_1)$; in particular no obstacles lie in $(w_j,y_j)$. In order to have the hypothesized shortening in horizontal length, $\kappa_1$ shall begin along the component of $\partial_v\bar\nei(\tau_j^X)$ which lies close to $q_j(y_j,0)$. Then the conventions of Remark \ref{rmk:permanenceconventions} imply that $y_{j+1}<w_j$ and, for all obstacles $z<w_j$ for $\tau_j^X$, $y_{j+1}>z$. 

Note that the upper obstacle $x_j+2\pi\in (x_j,y_j)$ because $\hl_\alpha(j)>2\pi$. Necessarily, then $w_j>x_j+2\pi$ and, from the above paragraph, also $y_{j+1}>x_j+2\pi$. This implies $\hl_{\alpha}(j+1)=y_{j+1}-x_j>2\pi$, contrarily to our assumption.
\end{proof}

\begin{lemma}[Three ramps criterion]\label{lem:threeramps}
Let $\bm\tau$ be a generic almost track splitting sequence such that a curve $\gamma\in\cc(S)$ stays a twist curve in the subsequence $\bm\tau(k,l)$. Suppose the following:
\begin{itemize}
\item $\gamma$ is not combed in $\tau_l$;
\item $\bm\tau(k,l)$ consists of subsequences which have been alternatively modelled according to the conventions of Remark \ref{rmk:permanenceconventions} --- possibly including moves which consist of trivial unzips --- and of Remark \ref{rmk:twistnaturemodelling} (the latter subsequences, of course, must be of twist nature);
\item under this fixed model, there is an incoming ramp $\rho_l^h: (-\infty,0]\rightarrow \tau_l^X$, hitting $A_\gamma$, such that $\rho_k^h\coloneqq c_{\tau_k^X}\circ \rho_l^h$ also intersects $\tau_l.\gamma$ only in the point $c_{\tau_k^X}\circ \rho_l^h(0)$.
\end{itemize}

Then, if $\alpha\in\cc(\tau_l^X)$ has a train path realization which includes $\rho_l^h$ then $\hl_\alpha(k)> \hl_\alpha(l)-2\pi$; in particular $\rot_{\bm\tau}(\gamma;k,l)\leq 1$.
\end{lemma}
\begin{proof}
Given a subsequence $\bm\tau(r,r')$ of twist nature, and a model of it after the prescriptions given in Remark \ref{rmk:twistnaturemodelling}, the latter may easily be converted into a model $\bm\sigma$ after Remark \ref{rmk:permanenceconventions}, except that some unzips may be trivial. In order to get this alternative model, do \emph{not} apply the isotopies that, after each unzip in $\bm\tau(r,r')$, make sure that $\tau_j.\gamma$ stays unchanged; and keep trace of changes of $\tau_j.\gamma$ via the maps $\E_j$ instead, as explained in Remark \ref{rmk:permanenceconventions}. This means, in particular, that rather than having a single map $q:\R\times[-2,2]\rightarrow \ol{S^X}$, in the new model we have individual maps $q_j$ for each entry in $\bm\sigma$. Anyway, for each $r\leq j\leq r'$ one may find an index $j'$ such that $q^{-1}(\tau_j)=q_{j'}^{-1}(\sigma_{j'})$ (the indices do not correspond exactly because the trivial unzips appearing along $\bm\tau(r,r')$ have to be inserted as individual moves). Moreover, if $\bm\tau(r,r')$ is replaced with $\bm\sigma$, then the last bullet in the statement still holds.

So, ultimately, we can simplify the proof by supposing that the entire sequence $\bm\tau(k,l)$ is modelled after Remark \ref{rmk:permanenceconventions}, possibly with trivial moves.

Since $\gamma$ is not combed in $\tau_l$, $\tau_l^X$ features two branch ends $e^f,e^a$ which avoid $A_\gamma$ and are favourable and adverse, respectively. There are two outgoing ramps $\rho_l^f,\rho_l^a:[0,+\infty)\rightarrow \tau_l^X$ which begin at $e^f,e^a$ respectively.

Let $\tilde\rho_l^h$ be a lift of $\rho_l^h$ to $\R\times[-2,2]$ via $q_l$: $x_l^h\coloneqq \tilde\rho_l^h(0)$ is an upper obstacle for $\tau_l^X$. The hypothesis on $\rho_k^h$ implies that it is an incoming ramp for $\tau_k^X$ and, if one lifts it to a $\tilde\rho_k^h$ in $\R\times[-2,2]$ via $q_k$, so that $\tilde\rho_k^h,\tilde\rho_l^h$ depart from the same point of $\R\times\{-2,2\}$, then the upper obstacle $x_k^h\coloneqq\tilde\rho_k^h(0)=x_l^h$.

Let $\rho_k^f,\rho_k^a$ be outgoing ramps for $\tau_k^X$ obtained from $c_{\tau_k^X}\circ\rho_l^f,c_{\tau_k^X}\circ\rho_l^a$ respectively by trimming their initial subpaths lying along $\tau_k^X.\gamma$. Choose lifts $\tilde\rho_l^a,\tilde\rho_l^f$ of $\rho_l^a,\rho_l^f$ via $q_l$ so that, if $x_l^a\coloneqq\tilde\rho_l^a(0),x_l^f\coloneqq\tilde\rho_l^f(0)$, then $x_l^a<x_l^f<x_l^a+2\pi$ and $x_l^h<x_l^f$.

Let $\tilde\rho_k^a,\tilde\rho_k^f$ be lifts of $\rho_k^a,\rho_k^f$ whose endpoints on $\R\times\{-2\}$ are the same as $\tilde\rho_l^a,\tilde\rho_l^f$ respectively. Let $x_k^a\coloneqq\tilde\rho_k^a(0),x_k^f\coloneqq\tilde\rho_k^f(0)$.

Then, by Remark \ref{rmk:permanenceconventions}, $x_k^a\leq x_l^a$ and $x_k^f\geq x_l^f$. Also, note that the order of $x_l^a,x_l^f,x_l^a+2\pi$ along $\R\times\{0\}$ is the same as the order of the extremities of $\tilde\rho_l^a$, $\tilde\rho_l^f$, $\tilde\rho_l^a+(2\pi,0)$ along $\R\times\{-2\}$, which are also extremities of $\tilde\rho_k^a$, $\tilde\rho_k^f$, $\tilde\rho_k^a+(2\pi,0)$ (these paths are all disjoint); therefore also $x_k^a,x_k^f,x_k^a+2\pi$ have the same order.

So, joining the inequalities, $x_k^f< x_k^a+2\pi\leq x_l^a+2\pi < x_l^f+2\pi$. Let $\alpha\in\cc(\tau_l^X)$ be the arc defined by the train path $\ul\alpha$ having, among its lifts to $\R\times[-2,2]$, the concatenation of $\tilde\rho_l^h$, $[x_l^h,x_l^f]\times\{0\}$, $\tilde\rho_l^f$. In particular $\hl_{\alpha}(l)=x_l^f-x_l^h$ and, via a similar realization in $\tau_k^X$, $\hl_{\alpha}(k)=x_k^f-x_k^h < x_l^f+2\pi-x_l^h=\hl_{\alpha}(l)+2\pi$.

This construction covers all $\alpha\in\cc(\tau_l^X)$ whose realization includes $\rho_l^h$: it suffices to choose an appropriate $\rho_l^f$ at the beginning of this proof. In particular $\rot_{\bm\tau}(\gamma;k,l)\leq 1$, as claimed.
\end{proof}

\subsection{Twist split grouping}\label{sub:rearrang}
In our incoming estimates of distance in $\pa(S)$ induced by a train track splitting sequence, a big annoyance will be the potentially high number of twist splits that do not result in an any accordingly large contribution to the pants distance, as it is shown in Lemma \ref{lem:pantsboundunderdt}. This is the reason why we shall work to identify exactly how they alter the computation.

\begin{prop}[Split rearrangement for a twist curve]\label{prp:rearrang1}
Suppose that $\bm\tau=(\tau_j)_{j=0}^N$ is a generic train track splitting sequence on a surface $S$, such that a fixed curve $\gamma$ stays a twist curve, with one same twist collar $A_\gamma$ and sign $\epsilon$, throughout. Let $X$ be a regular neighbourhood of $\gamma$, and let $m\coloneqq\rot_{\bm\tau}(0,N)$.% and let $s_{far},s_{spu}$ be the numbers of split moves along $\bm\tau$ which are far from $\gamma$ and spurious, respectively.

Then there are another splitting sequence $\bm{\tau'}=(\tau'_j)_{j=0}^{N(5)}$ whose first and last entries are the same as $\bm\tau$, and three indices $0\leq N(1)\leq N(2)\leq N(3)\leq N(4) \leq N(5)$, with the following properties.
\begin{enumerate}
\item $\rot_{\bm\tau'}\left(0,N(2)\right)\leq 3$; there is an increasing map $f: [0,N]\rightarrow [0,N(2)]$ with $f(0)=0$, $f(N)=N(2)$ and such that, for all $0\leq j \leq N$, $\tau_j$ is obtained from $\tau'_{f(j)}$ with a splitting sequence of twist nature about $\gamma$ followed by a comb equivalence.
\item If $0<N(1)<N(2)$ then $\gamma$ is not combed in $\bm\tau'\left(0,N(1)-1\right)$ and it is in $\bm\tau'\left(N(1),N(2)\right)$; if these inequalities are not strict, instead, then $\gamma$ is either combed or not combed in the entire sequence $\bm\tau'\left(0,N(2)\right)$.
\item $\bm\tau'\left(N(2),N(3)\right)$ and $\bm\tau'\left(N(4),N(5)\right)$ have twist nature about $\gamma$, while\linebreak $\bm\tau'(N(3),N(4))$ consists of slides only.
\item Let $m'\coloneqq \rot_{\bm\tau'}(N(2),N(4))$ Then $m'\geq m-5$, while $\rot_{\bm\tau'}(N(2),N(3))=0$ and there is a factorization $N(5)-N(4)=m'k$ such that $\tau'_{k+j}=D_\gamma^\epsilon(\tau'_j)$ for all $N(4)\leq j\leq N(5)-k$.
\end{enumerate}
\end{prop}

A note of warning: the sequence $\bm\tau'$ will be built with the possibility that some entry is obtained from the previous one via isotopies only. It is clear, anyway, that one may delete the repeated entries in the sequence (forgetting about any model according to Remarks \ref{rmk:permanenceconventions} and \ref{rmk:twistnaturemodelling} that the sequence is given).

\begin{proof}
\step{1} replacement of $\bm\tau$ with a splitting sequence with control on the number of branches hitting $A_\gamma$.

\begin{figure}[h]
\centering
\includegraphics[width=.4\textwidth]{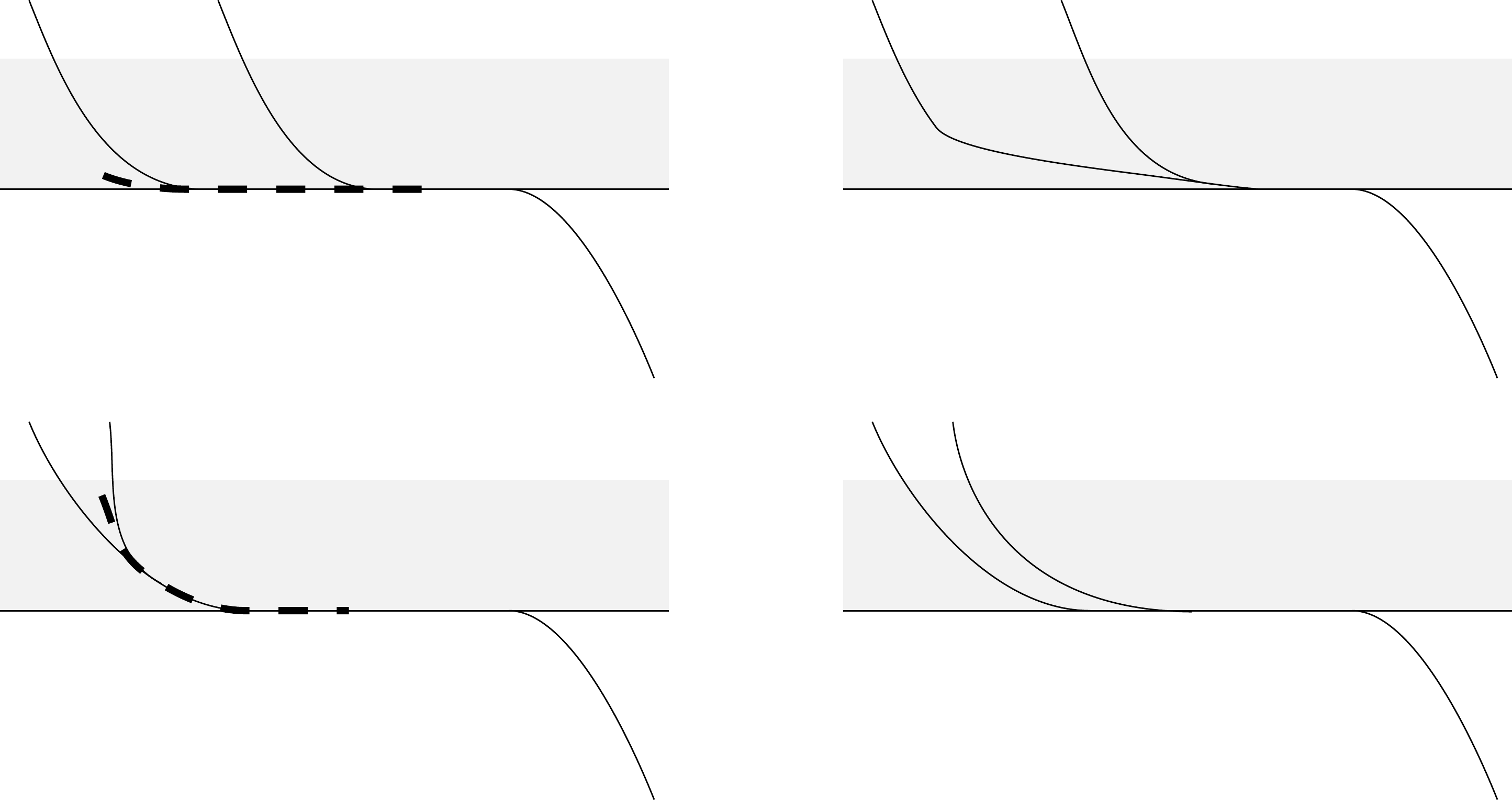}
\caption{\label{fig:pushpull}A pull (above) and a push (below) as the result of an unzip.}
\end{figure}

We define two particular kinds of slide around $\gamma$: a \emph{push} is a slide which increases the number of branch ends hitting $A_\gamma$; a \emph{pull} is one which decreases said number (see Figure \ref{fig:pushpull}). We also define a \emph{f-push} and a \emph{f-pull} accordingly, but replacing the words `hitting $A_\gamma$' with `avoiding $A_\gamma$ and favourable'.

In this step a new splitting sequence $\bm\sigma=(\sigma_j)_{j=0}^M$ will be constructed, with the same beginning track as $\bm\tau$ and with the elements of the sequence orderedly comb equivalent to the ones of $\bm\tau$ (in particular the last entries of the two sequences will be comb equivalent), but with the property that no pulls occur; and that, if $\bm\sigma(M',M)$ is the maximal subsequence of $\bm\sigma$ where $\gamma$ is combed, then no f-pull occurs there, either. Recall from the observation before Lemma \ref{lem:twistcurvebasics} that, if $\gamma$ is combed in a track along a splitting sequence, than $\gamma$ will remain combed in the sequence, as long as it is a twist curve.

Let $\bm\sigma^0$ be a wide splitting sequence obtained from $\bm\tau$ via Proposition \ref{prp:deleteslidings}; and let $\ul{\bm\sigma}^0$ be the conversion of it into another regular splitting sequence, by decomposing each wide split into elementary moves. 

We wish to define recursively a family of splitting sequences $\ul{\bm\sigma}^i$ whose entries keep orderedly comb equivalent to the ones of $\bm\tau$: in particular there is an $A_\gamma$-family of twist collars for all entries of all these sequences. We require that $\ul{\bm\sigma}^i$ is obtained from adjoining $\ul{\bm\sigma}^i_-*\ul{\bm\sigma}^i_+$, where $\ul{\bm\sigma}^i_-$ includes no pull, and no f-pull where $\gamma$ is combed; and $\ul{\bm\sigma}^i_+$ is the translation of a wide splitting sequence $\bm\sigma^i_+$ into a regular one --- for $i=0$, $\ul{\bm\sigma}^0_-$ is empty while $\ul{\bm\sigma}^0_+=\ul{\bm\sigma}^0$. The subsequence of $\ul{\bm\sigma}^i_+$ accounting for the $j$-th wide split in $\bm\sigma^i_+$ will be called $\ul{\bm\sigma}^i_+(j)$.

Also, we require $|\ul{\bm\sigma}^i|$ to be the same for all $i$ while $|\ul{\bm\sigma}^i_+|$ is strictly decreasing as $i$ increases: we stop the recursion when $|\ul{\bm\sigma}^i_+|=0$ (or earlier), so that $\ul{\bm\sigma}^i_-$ is the sequence $\bm\sigma$ that we desired to build. 

Suppose $\ul{\bm\sigma}^0,\ldots, \ul{\bm\sigma}^i$ have been defined; we proceed further to $\ul{\bm\sigma}^{i+1}$. Let $(\alpha^i_j)_{j=0}^{s(i)}$ be the sequence of the splitting arcs employed along the sequence $\bm\sigma^i_+$. If in $\ul{\bm\sigma}^i_+$ there is no pull nor f-pull where $\gamma$ is combed, then the recursion stops here with $\bm\sigma=\ul{\bm\sigma}^i$.

Else, the first pull in the sequence, or f-pull while $\gamma$ is combed, occurs as part of a $\ul{\bm\sigma}^i_+(j)$ such that the corresponding wide split in $\ul{\bm\sigma}^i_+$ is twist: far wide splits, indeed, fail to produce a pull or a f-pull, and a spurious wide split can only produce a f-pull, and only before $\gamma$ gets combed. Denote $t_1,t_2$ the indices such that $\ul{\bm\sigma}^i_+(j)=\ul{\bm\sigma}^i(t_1,t_2)$; but call $\xi_1\coloneqq \ul\sigma^i_{t_1}$, $\xi_2\coloneqq \ul\sigma^i_{t_2}$, for simplicity, the train tracks before and after this wide twist split.

Orient the splitting arc $\alpha^i_j$, embedded in $\bar\nei(\xi_1)$, so that it traverses branches of $\gamma$ according to the $A_\gamma$-orientation. 

Define $B_1,B_2\subseteq \br(\xi_1)$ as the two collections of branches such that $\bigcup B_1,\bigcup B_2$ give the two connected components of $\ol{\xi_1.\alpha^i_j\setminus\xi_1.\gamma}$ --- either, or both, may empty if the specified set is connected or empty. Let $R_1,R_2$ be the unions of the branch rectangles $R_b$ for $b\in B_1$, $b\in B_2$ respectively; and let $\alpha^i_{j1}=\alpha^i_j\cap R_1$, $\alpha^i_{j2}=\alpha^i_j\cap R_2$. They are the images of two disjoint small zippers (if either is empty, just ignore it): unzipping $\xi_1$ along $\alpha^i_{j1}$ and $\alpha^i_{j2}$, one gets a sequence of slides, none of which is a pull or a f-pull, turning $\xi_1$ into $\xi_1'$. In this latter almost track there is a splitting arc $\alpha^i_{j3}$ which corresponds to $\alpha^i_j$ (via Proposition \ref{prp:combpersistence}), and it traverses only branches contained in $\xi_1'.\gamma$.

With a series of twist splits, rearrange the branch ends with their endpoints lying along $\alpha^i_{j3}$: move all the ones hitting $A_\gamma$ past the ones that avoid it. The result, $\xi_2'$, is comb equivalent to the wide split along $\alpha^i_{j3}$, that is $\xi_2$.

We define $\ul{\bm\sigma}^{i+1}_-$ by adjoining $\ul{\bm\sigma}^i(0,t_1)$ to the elementary moves seen above which turn $\xi_1$ into $\xi_1'$ and then $\xi_2'$.

If $\xi_2$ is the last entry of the sequence $\bm\sigma^i_+$ then $\bm\sigma^{i+1}_+$ can be considered to be a trivial sequence with the entry $\xi_2'$ only, and the recursion ends here. If it is not then $\xi_2'$, being comb equivalent to $\xi_2$, carries a splitting arc which corresponds to $\alpha^i_{j+1}$; by Proposition \ref{prp:combpersistence} the wide split of it, with the same parity as $\alpha^i_{j+1}$ in $\xi_2$, gives an almost track which is comb equivalent to the entry of $\bm\sigma^i_+$ which succeeds $\xi_2$.

Continuing this way one gets a wide splitting sequence $\bm\sigma^{i+1}$ whose entries are comb equivalent to the entries of $\bm\sigma^i$ from $\xi_2$ on. As $\xi_2$ is not the first entry of $\bm\sigma^i_+$, $|\bm\sigma^i_+|>|\bm\sigma^{i+1}_+|$. This concludes the description of the recursive construction.

\step{2} grouping together the moves which are not twist; definition of the subsequence $\bm\tau'(0,N(2))$.

The argument used in this step is going to have some lines in common with the one of Lemma \ref{lem:postpone} about central split postponement. The underlying idea is: scan the sequence $\bm\sigma$ neglecting all twist moves, and performing all the other ones instead. At a later stage, the twist moves forgotten here will be performed altogether to give rise to the subsequence $\bm\tau'(N(2),N(5))$.

Consider the sequence of indices $0=j_0\leq j_1<\ldots <j_{2r-1}\leq j_{2r}=M$ such that, for all $i$ such that the following expressions make sense, the sequence $\bm\sigma(j_{2i},j_{2i+1})$ has twist nature whereas no split or slide in $\bm\sigma(j_{2i+1},j_{2i+2})$ is a twist one. We will have $j_1=0$ if and only if $\bm\sigma$ does \emph{not} begin with a twist split or slide, and similarly $j_{2r-1}=M$ if and only if $\bm\sigma$ ends with a twist split or slide.

While the sequences $\bm\sigma(j_{2i+1},j_{2i+2})$ will be modelled with the conventions explained in Remark \ref{rmk:permanenceconventions}, the sequences of twist nature $\bm\sigma(j_{2i},j_{2i+1})$ will be realized with the conventions of \ref{rmk:twistnaturemodelling}. In particular, along each sequence $\bm\sigma(j_{2i},j_{2i+1})$, the carrying images $\sigma_j.\gamma$ and the twist collar $A_\gamma(j)$, for $j_{2i}\leq j\leq j_{2i+1}$, are all the same subset of $S$; and each of these sequences shall be regarded as one only affecting the picture within $A_\gamma(j)$, associated to a twist modelling function $h_i$ and, correspondingly, to a bijection $H_i:S\rightarrow S$.

Let $q_j:\R\times[-2,2]\rightarrow S^X$ be the parametrization of $S^X$ which has been set up for $\sigma_j$, in compliance of Remark \ref{rmk:twistparam} and of either Remark \ref{rmk:permanenceconventions} or \ref{rmk:twistnaturemodelling} as specified above: it is handy to add to $q_j$ the request that all upper obstacles for $\sigma_j^X$ are points of $2\pi\mathbb Q\times\{0\}$; while all lower and fake obstacles are points of $2\pi(\R\setminus\mathbb Q)\times\{0\}$. With this request one may as well suppose that all $h_i$ have $h_i(2\pi\mathbb Q,0)=2\pi \mathbb Q$, and this will be helpful when removing the twist moves, as we will see.

In accordance with what seen in Remark \ref{rmk:generic_move_as_unzip}, for each index $j\in[j_{2i+1},j_{2i+2}-1]$ for some $i$, if the elementary move performed on $\sigma_j$ is not a central split, then it is the result of unzipping a zipper $\kappa_j$. For ease of notation, if the move is a central split, let $\kappa_j$ be a splitting arc traversing the branch that is being split. In order to comply with the further request on parametrizations specified above, necessarily the following statement shall hold, for any $j$ such that the $j$-th move is not a central split and for any $x\in\R$ such that $q_j(x,0)$ lies on the same tie as the last point of $\kappa_j$: $x\in 2\pi(\R\setminus\mathbb Q)$ if $\kappa_j$ intersects $A_\gamma$ (since twist moves have to be excluded, this means that its unzip describes necessarily a push); while $x\in 2\pi\mathbb Q$ otherwise.

Define, for all $i=0,\ldots r-1$, for all $j_{2i+1}< j \leq j_{2i+2}$, $\E_{(j)}\coloneqq  \E_{j_{2i+1}}\circ\ldots\circ\E_{j-1}: (S,\sigma_j)\rightarrow (S,\sigma_{j_{2i+1}})$. Here the maps $\E_j$ are relative to the sequence $\bm\sigma$, and defined as in Remark \ref{rmk:permanenceconventions}. Set then, for notational convenience, $\E_{(j_{2i+1})}=\mathrm{id}_S$, and $\E_{[i]}=\E_{(j_{2i+2})}$. Also, define $\F_{(j)}\coloneqq  \F_{j-1}\circ\ldots\circ\F_0$, which is defined on some subset of $S$ which we do not make explicit: for our purposes, it is important only that it includes the twist collar $A_\gamma(0)$. Note that there is a discrepancy between the definitions of $\F_{(j)}$ and $\E_{(j)}$.

Now, for all $j_{2i+1}\leq j < j_{2i+2}$, define $\phi_j:S\rightarrow S$ by
$$\phi_j|_{A_\gamma(j)}\coloneqq \F_{(j)}\circ H_0^{-1}\circ\E_{[0]}\circ\ldots\circ H_{i-1}^{-1}\circ \E_{[i-1]}\circ H_i^{-1}\circ \E_{(j)}\quad\text{and}\quad \phi_j|_{S\setminus A_\gamma(j)}=\mathrm{id}_{S\setminus A_\gamma(j)}.$$

This map is, in words: a bijection of $S$ which fixes $A_\gamma(j)$, is discontinuous only along $\sigma_j.\gamma$, and `undoes' the twist splits and slides performed in $\bm\sigma(0,j)$: while each such move pushes the branch ends hitting $A_\gamma$ forwards, this map moves them backwards by the same length. So define $\rho^i_j\coloneqq \phi_j(\sigma_j)$. We see that $\rho^i_j$ is a generic almost track, because we have kept upper obstacles in $2\pi\mathbb Q$ and lower and fake ones in $2\pi(\R\setminus\mathbb Q)$. Clearly, the existing parametrization $q_j$ can be used with respect to $(\rho_j^i)^X$, too, respecting the requests of Remark \ref{rmk:twistparam}.

Let then $\bm\rho^i=(\rho^i_j)_{j_{2i+1}}^{j_{2i+2}}$: this is a sequence resembling $\bm\sigma({j_{2i+1}},{j_{2i+2}})$ except that the effects of all the twist splits occurred earlier along $\bm\sigma$ have been cancelled. More precisely, the twist modelling function
\begin{equation}\label{eqn:h_tot}
h^*_i\coloneqq h_i(h_{i-1}(\cdots h_0(x,t)\cdots))
\end{equation}
induces a map $H^*_i$ such that $H^*_i(\rho^i_j)=\sigma_j$, for all $j_{2i+1}\leq j \leq j_{2i+2}$.

For each fixed ${j_{2i+1}}\leq j<{j_{2i+2}}$, the zipper/splitting arc $\kappa_j$ is still a zipper/multiple branch for $\rho^i_j$ and the unzip/central multiple split of it gives $\rho^i_{j+1}$. But, in general, if $\kappa_j$ is a zipper, its unzip may possibly describe \emph{not} a single elementary move, but more than one; or even just an isotopy. Similarly, if $\kappa_j$ is a splitting arc for $\sigma_j$, it may be the case that $\kappa_j$ is not a splitting arc in $\rho^i_j$. So, $\bm\rho^i$ in general is \emph{not} a splitting sequence, but each $\rho^i_{j+1}$ ($j_{2i+1}\leq j < j_{2i+2}$) is carried by the previous $\rho^i_j$. Also, note that $\rho^i_{j_{2i+2}}=\rho^{i+1}_{j_{2i+3}}$ for all $0\leq i\leq r-2$. Let $\bm\rho\coloneqq (\bm\rho^0)*\cdots*(\bm\rho^{r-1})$ (which is not a splitting sequence either, but the concatenation makes sense).

Each unzip may be subdivided into several unzips along shorter zippers, each of which give a single elementary move (or an isotopy). With this subdivision a splitting sequence $(\bm\rho^i)'$ is built, which touches orderedly all almost tracks in the sequence $\bm\rho^i$ and respects the conventions of Remark \ref{rmk:permanenceconventions}. Patch together all these pieces to get a new splitting sequence (possibly with trivial moves) $\bm\rho'\coloneqq (\bm\rho^0)'*\cdots*(\bm\rho^{r-1})'$. Note that, if $\xi$ is one of the almost tracks inserted between $\rho^i_j$ and $\rho^i_{j+1}$ (for suitable $i,j$) then $H^*_i(\xi)$ is isotopic to either $\sigma_j$ or $\sigma_{j+1}$; which, in turn, is comb equivalent to $\tau_{j'}$ for a suitable $j'$.

Let $(\bm\rho')^u$ be the maximal subsequence of $\bm\rho'$ where $\gamma$ is not combed, and let $(\bm\rho')^c$ the maximal subsequence where $\gamma$ is combed. Either of these may be empty or trivial. Index $\bm\rho'=(\rho'_j)_{j=0}^{N(2)}$ and let $N(1)\in [1,N(2)]$ be an index such that:
\begin{itemize}
\item either $(\bm\rho')^u$ is empty, or $(\bm\rho')^u=\bm\rho'(0,N(1)-1)$;
\item either $(\bm\rho')^c$ is empty, or $(\bm\rho')^u=\bm\rho'(N(1), N(2))$.
\end{itemize}

By construction, one has an increasing map $f_\sigma:[0,N]\rightarrow[0,M]$ such that $f_\sigma(0)=0$, $f_\sigma(N)=M$ and $\sigma_{f_{\sigma}(j)}$ comb equivalent to $\tau_j$ for all $j$. Also, for $0\leq j\leq M$, $\sigma_j$ is obtained from $\rho^i_j$ (for the correct $i$) via a splitting sequence of twist nature, as seen. This is enough to build the claimed map $f:[0,N]\rightarrow [0,N(2)]$.

\step{3} estimation of $\rot_{\bm\rho'}\left(0,N(2)\right)$.

Let $X$ be an annular neighbourhood of $\gamma$ in $S$. We only cover the case $0<N(1)<N(2)$, as the other ones are simple adaptations. 

We claim that there is an incoming, hitting ramp $\delta^h_{N(2)}$ in $(\rho'_{N(2)})^X$ such that $c_{(\rho'_0)^X}\circ \delta^h_{N(2)}$ is an hitting ramp for $(\rho'_0)^X$.

Recall, first of all, that each zipper involved in each of the subsequences\linebreak $\bm\sigma(j_{2i+1},j_{2i+2})$ induces a single elementary move; let $\sigma_j,\sigma_{j+1}$ be the almost tracks before and after the move.

If the move is a push: all the upper obstacles for $\sigma_j^X$ are also upper obstacles for $\sigma_{j+1}^X$. If $\delta:(-\infty,0]\rightarrow \sigma_{j+1}^X$ is a ramp hitting $A_\gamma(j+1)$, and $q_{j+1}^{-1}\left(\delta(0)\right)$ is a collection of upper obstacles for $\sigma_j^X$, too, then $c_{\sigma_j^X}\circ \delta$ is also a ramp hitting $A_\gamma(j)$ --- i.e. no segment of $c_{\sigma_j^X}\circ \delta$ lies along $\sigma_j^X.\gamma$. This can be understood from Figure \ref{fig:pushpull}, even though it is simplified.

The zipper $\kappa_j$ unzipped in this move is the same employed to turn $\rho^i_j$ into $\rho^i_{j+1}$; and, despite the result of this unzip not being necessarily a single elementary move, it is seen that the upper obstacles for $(\rho^i_j)^X$ are a subset of the ones for $(\rho^i_{j+1})^X$ and that a similar property as above holds: i.e. if $\delta:(-\infty,0]\rightarrow (\rho^i_{j+1})^X$ is an incoming ramp hitting $A_\gamma(j+1)$, with $q_{j+1}^{-1}\left(\delta(0)\right)$ a collection of upper obstacles for $(\rho^i_j)^X$, then $c_{(\rho^i_j)^X}\circ\delta$ is a ramp for $(\rho^i_j)^X$ hitting $A_\gamma(j)$.

If the elementary move from $\sigma_j$ to $\sigma_{j+1}$ is not a push, recall that it is not a pull or a twist move, either. The unzip that realizes it keeps $q_j^{-1}(\sigma_j^X)\cap\left(\R\times[0,1]\right)= q_{j+1}^{-1}(\sigma_{j+1}^X)\cap\left(\R\times[0,1]\right)$, and so also $q_j^{-1}(\rho_j^X)\cap\left(\R\times[0,1]\right)= q_{j+1}^{-1}(\rho_{j+1}^X)\cap\left(\R\times[0,1]\right)$. This implies that, for \emph{any} incoming ramp $\delta:(-\infty,0]\rightarrow (\rho^i_{j+1})^X$ hitting $A_\gamma(j+1)$, $c_{(\rho^i_j)^X}\circ\delta$ is a ramp for $(\rho^i_j)^X$ hitting $A_\gamma(j)$.

Considering that $\bm\rho'$ only inserts intermediate stages between the entries of $\bm\rho$, then, it is possible to pick $\delta^h_{N(2)}$ an incoming ramp for $(\rho'_{N(2)})^X$ with $(q'_{N(2)})^{-1}\left(\delta^h_{N(2)}(0)\right)$ a subset of the upper obstacles of $(\rho'_0)^X$. The argument seen above yields that both $c_{(\rho'_{N(1)})^X}\circ\delta^h_{N(2)}$ and $c_{(\rho'_0)^X}\circ\delta^h_{N(2)}$ are incoming ramps hitting $A_\gamma$ in $(\rho'_{N(1)})^X$, $(\rho'_0)^X$ respectively.

So the sequence $\bm\rho'\left(0,N(1)-1\right)$, if nonempty, complies with the hypotheses of Lemma \ref{lem:threeramps}; hence $\rot_{\bm\rho'}\left(\gamma;0,N(1)-1\right)\leq 1$.

As for the sequence $(\bm\rho')^c$, recall that no f-pulls occur here. So an argument entirely similar to the above shows that there is an outgoing, favourable ramp $\delta^f_{N(2)}$ for $\rho'_{N(2)}$ with $c_{(\rho'_{N(1)})^X}\circ\delta^h_{N(2)}$ a favourable ramp for $(\rho'_{N(1)})^X$. This means that, if $\alpha\in V(\rho'_{N(2)})$ has a train path realization which includes $\delta^h_{N(2)}$ and $\delta^f_{N(2)}$ then $\hl\left(\rho'_N(1),\alpha\right)=\hl\left(\rho'_N(2),\alpha\right)$ i.e. $\alpha\in V(\rho'_{N(1)})$. The move turning $\rho'_{N(1)-1}$ into $\rho'_{N(1)}$ makes $\gamma$ into a combed curve, so it must be a spurious split: hence, by Lemma \ref{lem:vertexsetnontwist}, $V(\rho'_{N(1)})\subseteq V(\rho'_{N(1)-1})$. The presence of the element $\alpha$ both in the latter set and in $V(\rho'_{N(2)})$ implies that $\rot_{\bm\rho'}\left(\gamma;N(1)-1, N(2)\right)=0$.

By point \ref{itm:concatrot_above} in Remark \ref{rmk:rotbasics}, $\rot_{\bm\rho'}\left(0,N(2)\right)\leq \rot_{\bm\rho'}\left(0,N(1)-1\right)+\linebreak \rot_{\bm\rho'}\left(N(1)-1,N(2)\right)+2 =3$.

\step{4} definition of the subsequence $\bm\tau'(N(2),N(5))$ and conclusion.

Let $h_{tot}\coloneqq h^*_{r-1}$ (see (\ref{eqn:h_tot}) above) and let $H_{tot}$ be the corresponding self-map of $S$. Note that, by construction, $\sigma_M=H_{tot}(\rho'_{N'})$. Applying Lemma \ref{lem:functiongivestwist} to the twist modelling function $h_{tot}$ is therefore possible to build a splitting sequence $\bm\omega$, having twist nature about $\gamma$, beginning with $\rho'_{N(2)}$ and ending with $\sigma_M$. Let $m'$ be the rotation number of the sequence $\bm\omega$.

As a consequence of Lemma \ref{lem:dehn+remainder}, $\bm\omega$ can be replaced by a new sequence $\bm\omega'=(\omega'_j)_{j=0}^{Q}$ of twist nature, with $\omega'_0=\rho'_{N(2)}$, $\omega'_Q$ comb equivalent to $\sigma_M$ (and to $\tau_N$) such that, for suitable integers $Q', k'$ we have $Q=Q'+k'm'$ and $\omega'_{j+k'}= D_\gamma^\epsilon(\omega'_j)$ for all $Q'\leq j\leq Q-k'$; while $\rot_{\bm\omega'}(\gamma;0,Q')=0$.

Let $\bm\xi$ be a series of slides that turns $\omega'_Q$ into $\tau_N$; so the sequence $D_\gamma^{-\epsilon m'}\cdot\bm\xi$ (i.e. the application of $D_\gamma^{-\epsilon m'}$ to all elements in the sequence $\bm\xi$) turns $\omega'_{Q'}$ into $D_\gamma^{-\epsilon m'}(\tau_N)$. Lemma \ref{lem:functiongivestwist} gives a new splitting sequence $\bm\omega''=(\omega''_j)_{j=0}^{m'k}$, with twist nature about $\gamma$, with $\omega''_{j+k}= D_\gamma^\epsilon(\omega''_j)$ for all $0\leq j\leq m'(k-1)$, and such that $\omega''_0=D_\gamma^{-\epsilon m'}(\tau_N)$ and $\omega''_{m'k}=\tau_N$.

Define finally $\bm\tau'=\bm\rho'*\bm\omega'(0,Q')*(D_\gamma^{-\epsilon m'}\cdot\bm\xi)*\bm\omega''$. This splitting sequence satisfies all requirements in the statement with $N(3)=N(2)+Q'$ and $N(4)=N(3)+ (\text{length of }\bm\xi)-1$. In particular, due to point \ref{itm:concatrot_above} in Remark \ref{rmk:rotbasics}, $m'\geq \rot_{\bm\tau'}(0,N(5))-\rot_{\bm\tau'}(0,N(2))-2 \geq m-5$.
\end{proof}

\begin{rmk}
The rearrangement procedure has a good behaviour with respect to the following properties --- meaning that if all tracks in $\bm\tau$ have the specified property then all tracks in $\bm\tau'$ have the same property, too.
\begin{itemize}
\item \textit{Recurrence and transverse recurrence.} In Remark \ref{rmk:recurrence_at_extremes} we have noted that all train tracks in a splitting sequence are recurrent if the last track in the sequence is. Moreover, they are all transversely recurrent if the first track in the sequence is. And $\bm\tau,\bm\tau'$ have the same endpoints.
\item \textit{Being cornered.} For each train track $\tau'_j$ there is a train track $\tau_i$ with a natural correspondence between the components of $S\setminus\nei(\tau'_j)$ and the ones of $S\setminus\nei(\tau_i)$, under which they are pairwise diffeomorphic.
\end{itemize}
\end{rmk}

\begin{rmk}\label{rmk:twist_disjointness}
We report a property of twist curves noted in \cite{mosher}, p. 215:
\begin{claim}
Let $\Gamma\subseteq \cc^0(S)$ be a family of essential curves which are pairwise disjoint up to isotopies, and are all twist curves for a given almost track $\tau$. Then, even if the carried images of these curves are not necessarily pairwise disjoint, it is possible to take a family of pairwise disjoint twist collars $A_\gamma$, for all $\gamma\in\Gamma$. For each of the $\gamma\in\Gamma$ which are combed in $\tau$, one may also choose what side of $\gamma$ the collar $A_\gamma$ must lie.
\end{claim}
\end{rmk}

\begin{lemma}[Small interference of twist curves]\label{lem:smallinterference}
Let $\bm\tau=(\tau_j)_{j=0}^N$ be a generic splitting sequence of almost tracks such that a curve $\gamma\in \cc^0(S)$ remains a twist curve throughout a subsequence $\bm\tau(k,l)$. If either of the following is true, then $\rot_{\bm\tau}(\gamma;k,l)=0$.
\begin{itemize}
\item There is a curve $\gamma_1$, intersecting $\gamma$ essentially, that also remains a twist curve throughout $\bm\tau(k,l)$.
\item There is a family of curves $\gamma_1,\ldots,\gamma_m$, all disjoint up to isotopy from $\gamma$ (not necessarily from each other), such that $\bm\tau(k,l)$ consists of subsequences of twist nature, each with respect to one of the curves $\gamma_j$.
\end{itemize}
\end{lemma}

\begin{proof}
Let $X$ be a regular neighbourhood of $\gamma$ in $S$.

In the first scenario: we claim that, for all $k\leq j\leq l$, $\pi_X(\gamma_1)\subseteq \cc(\tau_j^X)$ is actually a subset of $V(\tau_j^X)$. Note, first of all, that it is surely it is not empty.

If the claim is false, one of the branches in $\tau_j^{X}.\gamma$ is traversed twice, in the same direction, by an arc in the family $\pi_{X}(\gamma_1)$; and, if this is true, then also $\gamma_1$ traverses twice and in the same direction one of the branches in $\tau_j.\gamma$. But this would contradict the fact that, as a twist curve, $\gamma_1\in W(\tau_j)$.

This means that $V(\tau_k^{X})\cap V(\tau_l^{X})\supseteq \pi_{X}(\gamma_1)\not=\emptyset$. By point \ref{itm:rot_vs_dt_vertices} in Remark \ref{rmk:rotbasics}, this implies that $\rot_{\bm\tau}(\gamma,k,l)=0$.

In the second scenario, call $\bm\tau(k_i,l_i)$ the subsequence of $\bm\tau(k,l)$ which has twist nature with respect to $\gamma_i$; and model it according to Remark \ref{rmk:twistnaturemodelling}, applied to the twist curve $\gamma_i$. As the curves $\gamma_i$ are all disjoint from $\gamma$, the $j$-th elementary move in $\bm\tau(k,l)$ changes $\tau_{k+j-1}$ only within the relevant $A_{\gamma_i}(j-1)=A_{\gamma_i}(j)$, thus it does not affect $A_\gamma(j)$ nor $\tau_{k+j}.\gamma$, because of the disjointness property (Remark \ref{rmk:twist_disjointness} above). In other words, for $k_i\leq j < l_i$, $c_{\tau_j^X}|_{\tau_{j+1}^X}$ is the identity map out of $A_{\gamma_i}$. This means that we can employ the same parametrization $q:\R\times[-2,2]\rightarrow S^X$, in compliance with Remark \ref{rmk:twistparam}, and focused on the twist curve $\gamma$, for all almost tracks in the sequence $\bm\tau(k,l)$.

There are now two sub-cases to be considered. If none of the $\gamma_i$ traverses one same branch of the corresponding $\tau_{k_i}$ as $\gamma$, then all splits in the sequence $\bm\tau(k,l)$ are far from $\gamma$ and so $V(\tau_l^X)\subseteq V(\tau_k^X)$ by Lemma \ref{lem:vertexsetnontwist}. So $\rot_{\bm\tau}(\gamma,k,l)=0$ because the two vertex sets are not disjoint (point \ref{itm:rot_vs_dt_vertices} in Remark \ref{rmk:rotbasics}).

Now suppose that at least one of the curves $\gamma_i$ traverses a branch of $\tau_{k_i}.\gamma$. Then in $\tau_{k_i}^X$ there is necessarily an $A_\gamma$-adverse branch, also traversed by $\gamma_i$; and thus, there is one also in $\tau_k^{X}$. Let then $B\subseteq \tau_k^X.\gamma$ be a union of consecutive branches whose extremities are both small branch ends, whose endpoints are switches incident to a $A_\gamma$-favourable and an $A_\gamma$-adverse branch end, respectively, and such that all switches in $\inte(B)$ are incident to branch ends hitting $A_\gamma$. Then, if a twist curve $\delta$ for $\tau_k$ has no essential intersection with $\gamma$, then no lift of it to $S^X$ traverses any branch in $B$; nor any lift of a twist collar $A_\delta\subseteq S$ to $S^X$ may intersect $B$.

This implies, by recursion, that $B$ remains delimited by a pair of branch ends which are favourable and adverse, respectively, after each elementary move in the sequence $\bm\tau(k,l)$. Let $f$ be the favourable one of these two branch ends: by what has been said so far, $c_{\tau_j^X}(f)=f$ for all $k\leq j\leq l$.

Pick any $\alpha\in V(\tau_l^X)$ which traverses $f$. Due to the decomposition specified in point \ref{itm:horizontalstretch} of Remark \ref{rmk:annulusinducedbasics}, a train path realization $\ul\alpha_l$ of $\alpha$ will include an incoming, favourable ramp $\rho_l^f$ ending with $f$, followed by an embedded stretch $\hs(\tau_l,\alpha)$ along $\tau_l^X.\gamma$, and finally an outgoing ramp $\rho_l^h$ hitting $A_{\gamma}$, and starting at a branch end $e$.

The arc $\alpha$ also belongs to $\cc(\tau_k^X)$. A train path realization of $\alpha$ in $\tau_k^{X}$ is $c_{\tau_k^X}\circ\ul\alpha_l$ (possibly with a reparametrization). Note that both $c_{\tau_k^X}\circ\rho_l^f$, $c_{\tau_k^X}\circ\rho_l^h$ are ramps in $\tau_k^X$ because $c_{\tau_k^X}$ is the identity on both $f,e$. 

Therefore also $\hs(\tau_k,\alpha)=\hs(\tau_l,\alpha)$, hence $\alpha\in V(\tau_k^{X})$ and $\rot(\gamma;k,l)=0$.
\end{proof}

Now we are going to use some machinery that was already set up in \cite{mms}, so our hypotheses on the considered train track splitting sequences become more restrictive. Also, we will use shorthand notations like $\bm\tau(I)\coloneqq \bm\tau(\min I,\max I)$, where $I$ is an interval in $\mathbb Z$ and $\bm\tau$ is a splitting sequence indexed by a superinterval of $I$.

\begin{defin}
A train track splitting sequence $\bm\tau=(\tau_j)_{j=0}^N$ on a surface $S$ \nw{evolves firmly} in a possibly disconnected subsurface $S'$ of $S$ if, for each $0\leq j\leq N$, $V(\tau_j)$ fills exactly the subsurface $S'$.
\end{defin}

\begin{defin}\label{def:etc}
Let $\bm\tau=(\tau_j)_{j=0}^N$ be a generic train track splitting sequence on $S$, evolving firmly in a subsurface $S'$, not necessarily connected.

A curve $\gamma\in\cc(S)$, and essential in one of the non-annular connected components of $S'$, is an \nw{effective twist curve} for $\bm\tau$ if 
$$d_{\nei(\gamma)}(\tau_0,\tau_N)\geq 4\mathsf{K}_0+ 19,$$
where $\nei(\gamma)$ is a regular neighbourhood of $\gamma$ in $S$, and $\mathsf{K}_0$ is the constant defined in Theorem \ref{thm:mmsstructure}.
\end{defin}

Note that the given definition does not require that $\pi_{\nei(\gamma)}\left(V(\tau_0)\right), \pi_{\nei(\gamma)}\left(V(\tau_N)\right)\not=\emptyset$, because this is automatic by the request that all vertex cycles fill the same $S'$. Also, note that for an effective twist curve $\gamma$ in a splitting sequence $\bm\tau$ necessarily $I_\gamma\not=\emptyset$, by the first point of Theorem \ref{thm:mmsstructure}; in other words, an effective twist curve is, in particular, a twist curve for some tracks in the sequence $\bm\tau$. Moreover, $d_{\nei(\gamma)}(\tau_{\min I_\gamma},\tau_{\max I_\gamma})\geq 2\mathsf{K}_0+ 19$.

\begin{defin}\label{def:arranged}
Let $\bm\tau=(\tau_j)_{j=0}^N$ be a generic splitting sequence of cornered birecurrent train tracks on a surface $S$, evolving firmly in some subsurface $S'$, not necessarily connected. Let $\gamma_1,\ldots,\gamma_r\in\cc(\tau_0)$, and for each $1\leq t \leq r$ let $I_t$ be the accessible interval of $\nei_t\coloneqq \nei(\gamma_t)$ a regular neighbourhood of $\gamma_t$.

We say that the splitting sequence $\bm\tau$ is \nw{$(\gamma_1,\ldots,\gamma_r)$-arranged} if the following conditions hold. For each $t=1,\ldots,r$, there is a \nw{Dehn interval} for $\gamma_t$: a subinterval $DI_t\subset I_t$ such that $\bm\tau(DI_t)$ has twist nature with respect to $\gamma_t$ with $\rot_{\bm\tau}(\gamma_t;DI_t)\geq 2\mathsf{K}_0+4$, and is arranged into Dehn twists as prescribed in Lemma \ref{lem:dehn+remainder}, \emph{with no remainder}. Given any two intervals $DI_t$, for distinct values of $t$, they may intersect in at most one point; the curves are listed with the condition that the sequence $(\max DI_s)_{s=1}^r$ is increasing. Also, let $G_{t-}\coloneqq[0,\min I_t]$; $G_{t+}\coloneqq[\max I_t,N]$; $I_{t-}\coloneqq[\min I_t,\min DI_t]$; $I_{t+}\coloneqq [\max DI_t,\max I_t]$.

As an additional piece of notation, we subdivide each interval $DI_t$ into subintervals $DI_t(0), \ldots, DI_t(m_t-1)$, where $m_t=\rot_{\bm\tau}(\gamma_t;DI_t)$, dividing Dehn twists from one another. More precisely the maximum of each subinterval is also the minimum of the following one; if we call $a_t(i)=\min DI_t(i)$ for $i=0,\ldots,m_t-1$ and $a_t(m_t)=\max DI_t(m_t-1)$, then the sequence $a_t(0),\ldots,a_t(m_t)$ is an arithmetic progression and, for each $i,j$ such that $a_t(0) \leq a_t(i) +j \leq a_t(m_t)$, we have $\tau_{a_t(i)+j}=D_{\gamma_t}^{\epsilon i}\left(\tau_{a_t(0)+j}\right)$.

Let now $\gamma_1,\ldots,\gamma_r\in\cc(\tau_0)$ be the effective twist curves of $\bm\tau$. We say that $\bm\tau$ is \nw{effectively arranged} if it is $(\gamma_1,\ldots,\gamma_r)$-arranged and the following holds. For each $1\leq t \leq r$, if $i,j\in G_{t-}\cup I_{t-}$, then $d_{\nei_t}(\tau_i,\tau_j)\leq \mathsf{K}_0+2\mathsf{R}_0+9$; while, if $i,j\in I_{t+}\cup G_{t+}$, then $d_{\nei_t}(\tau_i,\tau_j)\leq \mathsf{K}_0+6$ --- and if $i,j\in I_{t+}$ then $d_{\nei_t}(\tau_i,\tau_j)\leq 6$. %The number of twist splits about $\gamma_t$ occurring in $R\bm\tau(I_{t-}^\rar)$ is bounded by $2N_4(2N_4+2+2\mathsf{R}_0)$.

Here $\mathsf{K}_0$ is as in Definition \ref{def:etc}, and $\mathsf{R}_0=\mathsf{R}_0(S,Q)$ is as defined in Lemma \ref{lem:reversetriangle}, where $Q$ is the quasi-isometry constant introduced in Theorem \ref{thm:mms_cc_geodicity}.
\end{defin}

\begin{prop}[Effective rearrangement]\label{prp:rearrang2}
Let $\bm\tau=(\tau_j)_{j=0}^N$ be a generic splitting sequence of cornered birecurrent train tracks, which evolves firmly in some (possibly disconnected) subsurface $S'$ of a surface $S$. Let $\gamma_1,\ldots,\gamma_r\in\cc(\tau_0)$ be the effective twist curves of $\bm\tau$, listed so that the sequence $(\max I_{\gamma_s})_{s=1}^r$ is increasing.

Then there is a $(\gamma_1,\ldots,\gamma_r)$-effectively arranged splitting sequence $\rar\bm\tau=(\rar\tau_j)_{j=0}^{N'}$ which begins and ends with the same train tracks as $\bm\tau$. The two splitting sequences, in particular, have the same family of effective twist curves. Moreover, if $I^\rar_s$ is the accessible interval of $\gamma_s$ in $\rar\bm\tau$, then the sequence $(\max I^\rar_s)_{s=1}^m$ is also increasing.
%The number of splits which are not twist splits about any of the effective twist curves $\gamma_1,\ldots,\gamma_q$ in the sequence $\bm\tau$ and in the sequence $\rar\bm\tau$ is the same.
\end{prop}

\begin{proof}
We will define recursively (decreasing the indices) a sequence of splitting sequences $\bm\tau=\bm\tau^{r+1},\bm\tau^r,\ldots,\bm\tau^1=\rar\rm\tau$ on the surface $S$ --- each of those will be indexed as $\bm\tau^s=(\tau^s_j)_{j=0}^{N^s}$ --- with the following properties. Their entries will all be cornered birecurrent train tracks; the first and last entries in each of these sequences will always be the same as in $\bm\tau$; each $\bm\tau^s$, informally speaking, is `partially' effectively arranged: it satisfies the requests in the definition of effectively arranged only for the curves $\gamma_t$ with $t\geq s$.

More precisely: for all $1\leq t \leq r$, denote $I^s_t$ the accessible interval relative to the curve $\gamma_t$ in the splitting sequence $\bm\tau^s$; and $G_{t-}^s\coloneqq[0,\min I_t^s]; G_{t+}^s\coloneqq[\max I_t^s,N^s]$. For the indices $s\leq t\leq r$ an interval $DI^s_t\subseteq I^s_t$ will be provided, together with $I^s_{t-}\coloneqq[\min I^s_t,\min DI^s_t]$; $I^s_{t+}\coloneqq [\max DI^s_t,\max I^s_t]$. The following claim will be true for each $s=r+1,r,\ldots,1$:
\begin{claim}
In the splitting sequence $\bm\tau^s$, for all $s\leq t\leq r$:
\begin{enumerate}
\item $i,j\in I^s_{t+}\cup G^s_{t+}\Longrightarrow d_{\nei_t}(\tau^s_i,\tau^s_j)\leq \mathsf{K}_0+6$; and $\rot_{\bm\tau^s}(\gamma_t;I^s_{t+})\leq 2$;
\item $d_{\nei_t}(\tau^s_0,\tau^s_{\min DI^s_t})\leq \mathsf{K}_0+9$, and $i,j\in G^s_{t-}\cup I^s_{t-}\Longrightarrow d_{\nei_t}(\tau^s_i,\tau^s_j)\leq \mathsf{K}_0+2\mathsf{R}_0+9$;
\item $DI^s_t$ is a Dehn interval for $\gamma_t$: $\bm\tau^s(DI^s_t)$ is a sequence of twist nature with respect to $\gamma_t$, with $\rot_{\bm\tau^s}(\gamma_t;DI^s_t)\geq 2\mathsf{K}_0+4$, arranged in Dehn twists with no remainder.
\end{enumerate}
\end{claim}
Eventually, it will suffice to define $\rar\bm\tau\coloneqq \bm\tau^1$. 

\step{1} recursive construction of the sequences $\bm\tau^s$.

Fix $1\leq s\leq r$. Suppose that all $\bm\tau^i$ for $i\geq s+1$ have been defined, together with intervals $DI^i_t, I^i_{t-}, I^i_{t+}\subseteq I^i_t$ for $t$ in the range $i,\ldots,r$. For each fixed $i$, the intervals $DI^i_t$ are pairwise disjoint except possibly for a common endpoint.

We now build $\bm\tau^s$ from $\bm\tau^{s+1}$ with an application of Proposition \ref{prp:rearrang1}, with respect to $\gamma_s$, on a suitable subsequence of $\bm\tau^{s+1}$. We define $I^{s+1}_{s+}$ as follows: the idea is that $I^{s+1}_{s+}$ is an interval we do not want to apply Proposition \ref{prp:rearrang1} on, because it is already structured in twists with respect to other curves.
\begin{itemize}
\item If $\max I^{s+1}_s$ is contained in $DI^{s+1}_t$ for a $t\geq s+1$, let $J$ be the maximal concatenation of intervals $DI^{s+1}_u$, $r\geq u \geq s+1$, such that the maximum of one is the minimum of another, and $DI^{s+1}_t$ is part of the union. Let then $I^{s+1}_{s+}= I^{s+1}_s\cap J$.
\item If $\max I^{s+1}_s$ is not contained in any $DI^{s+1}_t$, define $I^{s+1}_{s+}\coloneqq \{\max I^{s+1}_s\}$.
\end{itemize}

Set now $\bm\sigma\coloneqq\bm\tau^{s+1}(\min I^{s+1}_s,\min I^{s+1}_{s+})$. Let $\bm\sigma'$ be the splitting sequence obtained from $\bm\sigma$ by application of Proposition \ref{prp:rearrang1}; and define $\bm\tau^s\coloneqq \bm\tau^{s+1}(G^{s+1}_{s-})*\bm\sigma'*\bm\tau^{s+1}(I^{s+1}_{s+}\cup G^{s+1}_{s+})$. As entries in $\bm\tau^s$ are given indices in the interval $[0,N^s]$, the three sequences which compose it are indexed by subintervals which we call $[0,a^s]$, $[a^s,b^s]$, $[b^s,N^s]$, respectively. With this indexing, the subsequence $\bm\tau^{s+1}(I^{s+1}_{s+})$ is copied to a subsequence of $\bm\tau^s$ indexed by an interval which we call $I^s_{s+}$; and it has $\min I^s_{s+}=b^s$.

Note that, as a consequence of the construction, the sequence $(\max I^s_t)_{t=1}^r$ is increasing. Given an index $1\leq i\leq r+1$, let $P(i)$ be the following property: ``for all $i \leq t \leq r$, $\max DI^i_t$ coincides with either $\min DI^i_{u}$ for an index $u\geq t$, or with $\max I^i_t$; and the sequences $(\min DI^i_t)_{t=i}^r, (\max DI^i_t)_{t=i}^r$ are increasing''.

We prove the following: 
\begin{claim}
If $P(i)$ is true for a given $i> 1$, then $DI^i_t\subseteq I^i_{(i-1)+}\cup G^i_{(i-1)+}$ for all $i\leq t\leq r$.
\end{claim}
\begin{proof}
There are two cases to consider: if $\max DI^i_i=\max I^i_i$, which is $\geq \max I^i_{i-1}$, then clearly, by construction of $I^i_{(i-1)+}$, $DI^{i}_i\subseteq I^{i}_{(i-1)+}\cup G^{i}_{(i-1)+}$; and for $i \leq t \leq r$, $DI^{i}_t\subseteq I^{i}_{(i-1)+}\cup G^{i}_{(i-1)+}$ is true because of the last sentence in $P(i)$.

If $\max DI^i_i=\min DI^i_u$ for some $u\geq i$ then, by the monotonicity of $(\min DI^{i}_t)_{t=i}^r$, necessarily $u=i+1$. Let then $J'$ be the maximal interval which contains $DI^i_i$ and is obtained as a union of intervals $DI^i_t$, $i\leq t\leq r$. Let $i\leq t'\leq r$ be the index such that $\max J'=\max DI^i_{t'}$. Suppose for a contradiction that there exists an index $i\leq t''\leq r$ with $DI^i_{t''}\not\subseteq I^i_{(i-1)+}\cup G^i_{(i-1)+}$; then, again by monotonicity, also $DI^i_i\not\subseteq I^i_{(i-1)+}\cup G^i_{(i-1)+}$.

Considering the way $I^i_{(i-1)+}$ has been defined in the sequence $\bm\tau^i$, in order for this to happen it must be $\max I^i_{i-1}>\max J'$. But in this case, also $\max I^i_{t'}>\max J'$ for all $i\leq t \leq r$, and in particular $\max I^i_{t'}>\max DI^i_{t'}$. However, by definition of $J'$ and $t'$, there is no index $u$ such that $\max DI^i_{t'}=\min DI^i_u$. So the index $t'$ is a contradiction to $P(i)$. 
\end{proof}

Now, $P(r+1)$ is voidly true; we assume $P(s+1)$, and further ahead we prove that $P(s)$ follows: this is necessary to legitimate the recursive construction of $\bm\tau^s$.

As said above, $DI^{s+1}_t\subseteq I^{s+1}_{s+}\cup G^{s+1}_{s+}$ for all $s+1\leq t\leq r$. So the construction of $\bm\tau^s$ causes the subsequences $\bm\tau^{s+1}(I^{s+1}_{t+})$ and $\bm\tau^{s+1}(DI^{s+1}_t$) to be copied to subsequences inside $\bm\tau^s\left(b^s,N^s\right)$. They will be indexed by intervals which we call $I^s_{t+}, DI^s_t$ respectively.

In $\bm\sigma'$, following the notation given in Proposition \ref{prp:rearrang1}, there is a subsequence indexed by the subinterval $[N(4), N(5)]$, and we call it $\bm\sigma''$: $\bm\sigma''$ has twist nature with respect to $\gamma_s$ and is arranged into Dehn twists with no remainder. When inserting $\bm\sigma'$ as the subsequence $\bm\tau^s(a^s,b^s)$ of $\bm\tau^s$, $\bm\sigma''$ will be given indices in a subinterval of $[a^s,b^s]$: we call it $DI^s_s$. Define, for $t\geq s$, $I^s_{t-}\coloneqq [\min I^s_t,\min DI^s_t]$.

With these definitions, $P(s)$ is `almost true': it is true that ``for all $s+1 \leq t \leq r$, $\max DI^s_t$ coincides with either $\min DI^s_{u}$ for an index $u\geq t$, or with $\max I^s_t$; and the sequences $(\min DI^s_t)_{t=s+1}^r, (\max DI^s_t)_{t=s+1}^r$ are increasing''. This is because, since all the relevant intervals $DI^{s+1}_t$ ($s+1 \leq t \leq r$) are contained in $I^{s+1}_{s+}\cup G^{s+1}_{s+}$, the corresponding $DI^s_t$ are contained in $[b^s,N^s]$, and this family of subintervals is just a translation of the corresponding family $DI^{s+1}_t$ in $[0,N^{s+1}]$.

But the construction forces $\max DI^s_s$ to be a lower bound for all intervals $DI^s_t$, $s+1 \leq t \leq r$; and we have $\max DI^s_s=\max I^s_s$ if $I^s_s$ is disjoint from all $DI^s_t,t>s$, while $\max DI^s_s=\min_{t>s}\left(\min DI^s_t\right)$ otherwise. So $P(s)$ is true.

\step{2} the properties claimed above for $\bm\tau^s$ hold.

Those properties are empty for $s=r+1$. Supposing that they hold for $\bm\tau^{r+1},\ldots,\linebreak \bm\tau^{s+1}$, we prove that they hold for $\bm\tau^s$, establishing an inductive argument.

Note, first of all, that if $t\geq s$ then $i,j\in G^s_{t-}\Rightarrow d_{\nei_t}\left(\tau^s_i,\tau^s_j\right)\leq \mathsf{K}_0$ by point 1 of Theorem \ref{thm:mmsstructure}. Same for $i,j\in G^s_{t+}$. Hence $d_{\nei_t}\left(\tau^s_{\min I^s_t},\tau^s_{\max I^s_t}\right)\geq d_{\nei_t}(\tau^s_0,\tau^s_{N^s})-2\mathsf{K}_0=2\mathsf{K}_0+19$; and $\rot_{\bm\tau^s}(\gamma_t;I^s_t)\geq 2\mathsf{K}_0+15$, by point \ref{itm:concatrot_above} in Remark \ref{rmk:rotbasics}.

Properties 1 and 3 for $t>s$: since $DI^{s+1}_t, I^{s+1}_{t+}\subseteq I^{s+1}_{s+}\cup G^{s+1}_{s+}$, we have that $\bm\tau^s(DI^s_t),\bm\tau^s(I^s_{t+})$ are copies of $\bm\tau^{s+1}(DI^{s+1}_t),\bm\tau^{s+1}(I^{s+1}_{t+})$, respectively. Hence, by inductive hypothesis, they have $\rot_{\bm\tau^s}(\gamma_t; DI^s_t)\geq 2\mathsf{K}_0+ 4$; and $\rot_{\bm\tau^s}(\gamma_t,I^s_{t+})\leq 2$ which yields $d_{\nei_t}\left(\tau^s_{\min I^s_{t+}}, \tau^s_{\max I^s_{t+}}\right)\leq 6$ via point \ref{itm:rot_vs_dist} of Remark \ref{rmk:rotbasics}.

Property 1 for $t=s$: we claim that $\rot_{\bm\tau^{s+1}}(\gamma_s;I^{s+1}_{s+})\leq 2$ --- and $\rot_{\bm\tau^s}(\gamma_s;I^s_{s+})\leq 2$, because the two rotation numbers are computed on two copies of the same sequence --- therefore $d_{\nei_s}\left(\tau^s_{\min I^s_{s+}}, \tau^s_{\max I^s_{s+}}\right)\leq 6$.

The claim is obvious if $I^{s+1}_{s+}=\{\max I^{s+1}_s\}$. If $\max I^{s+1}_s\in DI^s_t$ for a fixed $t>s$, call $A=I^{s+1}_{s+}\cap DI^s_t; B=[\min I^{s+1}_{s+},\min A]$. Then $\rot_{\bm\tau^{s+1}}(\gamma_s;A)=0$ because of Lemma \ref{lem:smallinterference} (both in case $\gamma_s,\gamma_t$ intersect and in case they do not). According to the same lemma, when $B$ is not a single point, necessarily all curves $\gamma_u$ with $DI^{s+1}_u\subseteq B$ must be essentially disjoint from $\gamma_s$, because $\rot_{\bm\tau^{s+1}}(\gamma_u;DI^{s+1}_u)\geq 2\mathsf{K}_0+ 4$. But then the sequence $\bm\tau^s(B)$ falls into the case covered in the second point of Lemma \ref{lem:smallinterference}, which yields $\rot_{\bm\tau^{s+1}}(\gamma_s;B)=0$. So $\rot_{\bm\tau^{s+1}}(\gamma_s,I^{s+1}_{s+})\leq 2$ by point \ref{itm:concatrot_above} in Remark \ref{rmk:rotbasics}. 

Property 2 for $t>s$: again because $\max I^{s+1}_{t-}= \min DI^{s+1}_t \in I^{s+1}_{s+}\cup G^{s+1}_{s+}$, the sequence $\bm\tau^s(G^s_{t-}\cup I^s_{t-})$ begins and ends with the same train tracks as $\bm\tau^s(G^{s+1}_{t-}\cup I^{s+1}_{t-})$. So $d_{\nei_t}(\tau^s_0,\tau^s_{\min DI^s_t})\leq \mathsf{K}_0+9$ by inductive hypothesis. According to Theorem \ref{thm:mms_cc_geodicity} the sequence $\left(\pi_{\nei_t}(V(\tau^s_j))\right)_{j\in G^s_{t-}\cup I^s_{t-}}$ is a $Q$-unparametrized quasi-geodesic in $\cc(\nei_t)$, so the reverse triangle inequality in Lemma \ref{lem:reversetriangle} gives, for $i,j\in G^s_{t-}\cup I^s_{t-}$, $d_{\nei_t}(\tau^s_i,\tau^s_j)\leq \mathsf{K}_0+2\mathsf{R}_0+9$ as required.

Property 2 for $t=s$: Proposition \ref{prp:rearrang1} above guarantees that $\rot_{\bm\tau^s}(\gamma_s,I^s_{s-})\leq 5$, as $\bm\tau^s(I^s_{s-})$ indeed corresponds, using the notation given in that Proposition, to the subsequence $\bm\tau'\left(0,N(4)\right)$ of the output sequence $\bm\tau'$. So, for all pairs $i,j\in I^s_{s-}$, we have $d_{\nei_s}(\tau^s_i,\tau^s_j)\leq 9$ (see point \ref{itm:rot_vs_dist} in Remark \ref{rmk:rotbasics}). Combine it with the previously noted bounds for $i,j\in G^s_{s-}$ to complete the proof that property 2 holds.

Property 3 for $t=s$: by point \ref{itm:concatrot_above} in Remark \ref{rmk:rotbasics}, $\rot_{\bm\tau^s}(\gamma_s;DI^s_s)\geq \rot_{\bm\tau^s}(\gamma_s;I^s_s) - \rot_{\bm\tau^s}(\gamma_s;I^s_{s-})- \rot_{\bm\tau^s}(\gamma_s;I^s_{s+})-4\geq 2\mathsf{K}_0+4$.
\end{proof}

\subsection{A bound on the number of highly twisting curves}
\label{sub:twistcurvebound}

\ul{Note:} In this subsection we still deal mostly with generic almost tracks. However, we make use of the diagonal extension machinery from \S \ref{sub:diagext}; so we have to consider some semigeneric almost tracks as well. The adjective `generic' will be made explicit when appropriate, anyway.

Given a surface $S$, recall that in Remark \ref{rmk:pickparameters} we have fixed the parameters $k,\ell$ involved in the definition of $\pa(S)$ and of $\pa(X')$ for $X'$ a subsurface of $S$. Let $M\coloneqq \max_{X'\subseteq S} M_6(X',k(X',S),\ell(X'))$, where the maximum is taken over all $X'\subseteq S$ non-annular subsurfaces, and $M_6(X',k,\ell)$ is defined as in Lemma \ref{lem:pantsquasiisom}.

Given a non-annular subsurface $X'$ and two train tracks $\sigma,\tau$ on $S$, suppose that $\pi_Y\left(V(\sigma)\right),\pi_Y\left(V(\tau)\right)\not=\emptyset$ for all $Y\subseteq X'$ non-annular subsurfaces. In this case define
$$d'_{\pa(X')}(\sigma,\tau)\coloneqq \sum_{\substack{Y\subset X'\text{ essential} \\ \text{and non-annular}}} [d_Y(\sigma,\tau)]_M.$$

And, if non-emptyness holds also for all projections onto annuli $Y\subseteq X'$, we may also define
$$d'_{\ma(X')}(\sigma,\tau)\coloneqq \sum_{Y\subset X'\text{ essential}} [d_Y(\sigma,\tau)]_M.$$

Similarly as in Theorem \ref{thm:mmprojectiondist}, the summations shall be meant over $Y\subset X'$ subsurfaces, counting only one representative for each isotopy class in $S$. Restating that theorem, also in the light of Lemma \ref{lem:pantsquasiisom} one has
\begin{eqnarray*}
d'_{\pa(X')}(\sigma,\tau)  & =_{(e_0,e_1)} &
d_{\pa(X')}(\pi_{X'} V(\sigma), \pi_{X'} V(\tau)); \\
d'_{\ma(X')}(\sigma,\tau)  & =_{(e_0,e_1)} & 
d_{\ma(X')}(\pi_{X'} V(\sigma), \pi_{X'} V(\tau)).
\end{eqnarray*}
in the two respective cases, for suitable constants $e_0(X',M, k,\ell), e_1(X',M, k,\ell)$. 

Suppose now that $V(\sigma|X')$ and $V(\tau|X')$ are vertices of $\pa(X')$ (resp. $\ma(X)$). Then $\pi_{X'} V(\sigma)$ and $\pi_{X'} V(\tau)$ are vertices there, too, i.e. the above formulas make sense for them. This implies, via Lemma \ref{lem:induction_vertices_commute}, that
\begin{eqnarray*}
d'_{\pa(X')}(\sigma,\tau) & =_{(e_0,e_1+ C_1)} &
d_{\pa(X')}\left(\sigma, \tau\right); \\
d'_{\ma(X')}(\sigma,\tau) & =_{(e_0,e_1+ C_1)} &
d_{\ma(X')}\left(\sigma, \tau\right).
\end{eqnarray*}

The aim of this subsection is to prove the following
\begin{prop}\label{prp:tcbound}
Let $\bm\tau=(\tau_j)_{j=0}^N$ be a generic, recurrent train track splitting sequence on a surface $S$ which evolves firmly in some subsurface $S'$ --- not necessarily a connected one. Let $X$ be (another) non-annular subsurface of $S$; let $\gamma_1,\ldots,\gamma_q\subseteq \cc(\tau_0)$ be curves all contained, and essential, in $X$; and suppose that $\bm\tau$ is $(\gamma_1,\ldots,\gamma_q)$-arranged (see Definition \ref{def:arranged}; in particular the sequence $(\max DI_t)_{t=1}^q$ is increasing).

Fix $0\leq k\leq l\leq N$ with $V(\tau_l|X)\in\pa^0(X)$ and such that $DI_t \subseteq [k,l]$ for all $1\leq t \leq q$. 

Then there are constants $C_3, C_4$, only depending on $S$, such that
$$q\leq C_3 d'_{\pa(X)}(\tau_k,\tau_l) + C_4.$$
\end{prop}

Before we start, anyway, we prove a lemma which will be of use in the following sections, too.

\begin{lemma}\label{lem:pantsboundunderdt}
Let $S$ be a surface, $X$ be a non-annular subsurface of $S$ (possibly $X=S$), $\gamma\in \cc(X)$.

Let $\tau$ be a generic almost track and $\bm\tau=(\tau_j)_{j=0}^N$ be a generic splitting sequence of almost tracks on $S$, with $\bm\tau(k,l)$ a sequence of twist nature about $\gamma$.

\begin{enumerate}
\item $\left(D_\gamma(\tau)\right)|X$ and $D_\gamma(\tau|X)$ are isotopic (here $D_\gamma:X\rightarrow X$, so it can be extended trivially to both $S$ and $S^X$, where the almost tracks lie).
\item If $V(\tau_k|X)$ is a vertex of $\pa(X)$ (resp. of $\ma(X)$), then $V(\tau_j|X)$ is one, too, for all $k\leq j\leq l$.
\item In the sequence $(\tau_j|X)_{j=k}^l$, each entry is fully carried by the previous one.
\item There is a bound $C_2(S)$ such that, if $V(\tau_k|X)$ is a vertex of $\pa(X)$, then $d_{\pa(X)}(\tau_k|X,\tau_l|X)\leq C_2$, and $d_Y\left(\tau_k|X,\tau_l|X\right)\leq C_2$ for all $Y\subseteq X$ non-annular subsurfaces.
\end{enumerate}
\end{lemma}
\begin{proof}
To prove claim 1, note that the map $D_\gamma:S\rightarrow S$ has a lift $\hat D: S^X\rightarrow S^X$ whose restriction to $\core(X)=X\cap \core(S^X)$ concides with the restriction of $D_\gamma:X\rightarrow X$. Therefore $\hat D: S^X\rightarrow S^X$ is isotopic to $D_\gamma:S^X\rightarrow S^X$. The claim follows from $\hat D(\tau|X)=\left(D_\gamma(\tau)\right)|D_\gamma(X)=\left(D_\gamma(\tau)\right)|X$.

For claims 2 and 3, model $\bm\tau(k,l)$ according to Remark \ref{rmk:twistnaturemodelling}. If $h$ is a twist modelling function associated with $\bm\tau(k,l)$, then by point \ref{itm:tmfbeyondrot} in Remark \ref{rmk:rotbasics}, one may assume $h(x,0)<x+2\pi\left(\rot(k,l)+3\right)$ for all $x\in\R$ which means that there is a twist modelling function $h'$ such that $h'(h(x,0),0)=x+2\pi(\rot(k,l)+3)$. This is associated with a splitting sequence of twist nature, which we call $\bm\sigma$, turning $\tau_l$ into $D_\gamma^{\epsilon(\rot(k,l)+3)}(\tau_k)$, where $\epsilon$ is the sign of $\gamma$ as a twist curve.

It is clear that $V\left(D_\gamma^{\epsilon(\rot(k,l)+3)}(\tau_k|X)\right)= D_\gamma^{\epsilon(\rot(k,l)+3)}\cdot V(\tau_k|X)$: here, $D_\gamma$ shall be meant as the Dehn twist about $\gamma$ as a diffeomorphism $S^X\rightarrow S^X$ or $X\rightarrow X$. So $V\left(D_\gamma^{\epsilon(\rot(k,l)+3)}(\tau_k|X)\right)$ is a vertex of $\pa(S)$ (resp. of $\ma(X)$) if and only if $V(\tau_k|X)$ is one, too; but the remark following Lemma \ref{lem:decreasingfilling}, applied to $\bm\sigma$, yields that in this case $V(\tau_l|X)$ is also a vertex of $\pa(X)$ (resp. $\ma(X)$). Moreover in the sequence $\bm\tau*\bm\sigma$, induced on $X$, each entry carries the following one, and the first one $\tau_k|X$ fully carries the last one $D_\gamma^{\epsilon(\rot(k,l)+3)}(\tau_k|X)$. So the carrying must be a full one at any intermediate stage of the sequence.

For claim 4: define, from $\bm\tau(k,l)$, a splitting sequence $\bm\tau'=(\tau'_j)_{j=0}^{N''}$ arranged in Dehn twists plus remainder, as in Lemma \ref{lem:dehn+remainder}. We adopt the notation used in the statement of that lemma. In particular, $V(\tau'_0)=V(\tau_k)$ and $D_\gamma^{\epsilon m}\cdot V(\tau'_{N'})=V(\tau'_{N''})=V(\tau_l)$. Since $\rot_{\bm\tau'}(\gamma;0,N')=0$, at most $3N_3^2$ splits occur in $\bm\tau'(0,N')$ by Lemma \ref{lem:twistsplitnumber}.

But, once the surface $S$ is fixed, the possible pairs $(\tau_0,\gamma)$ as in the statement are finitely many up to the action of $\mcg(S)$ (cfr. Lemma \ref{lem:vertexsetbounds}). The finiteness of possible choices implies that there is a bound $k_1$, depending on $S$ only, on $i(\alpha,\beta)$, for $\alpha\in V(\tau'_0),\beta\in V(\tau'_{N'})$. 

Let $p(\tau_{N'})$ be a pants decomposition of $S$ including the curve $\gamma$, chosen so that $\max_{\alpha \in p(\tau_{N'}), \beta \in V(\tau_{N'})} i(\alpha,\beta)$ is minimal among all pants decompositions with this property. So, again by finiteness of possible configurations up to $\mcg(S)$, a constant $k_2=k_2(S)$ exists with $\max_{\alpha \in p(\tau_{N'}), \beta \in V(\tau_{N'})} i(\alpha,\beta)\leq k_2$. Using the fact that $D_\gamma^{\epsilon m}\cdot p(\tau_{N'})=p(\tau_{N'})$ and $D_\gamma^{\epsilon m}\cdot V(\tau'_{N'})=V(\tau'_{N''})$, also $\max_{\alpha \in p(\tau_{N'}), \beta \in V(\tau_{N''})} i(\alpha,\beta)\leq k_2$.

We now use arguments similar to the ones in Lemma \ref{lem:induction_vertices_commute}. For each $Y\subseteq X$ non-annular subsurface, $\pi_Y V(\tau'_0),\pi_Y V(\tau'_{N'}), \pi_Y\left(p(\tau'_{N'})\right), \pi_Y V(\tau'_{N'})$ are all non-empty. Using Remark \ref{rmk:subsurf_inters_bound}, given $\alpha\in \pi_Y V(\tau'_0)$, $\beta\in\pi_Y V(\tau'_{N'})$ one has $i(\alpha,\beta)\leq 4k_1+4$, and a similar bound $4k_2+4$ holds for $i(\alpha,\beta)$ if $\alpha\in \pi_Y\left(p(\tau'_{N'})\right)$ and $\beta\in \pi_Y V(\tau'_{N'})$ or $\in \pi_Y V(\tau'_{N''})$.

Appealing to Lemma \ref{lem:cc_distance},
$$
d_Y\left(V(\tau'_0), V(\tau'_{N''})\right)\leq F(4k_1+4)+ 2 F(4k_2+4).
$$
This implies also that $d_Y\left(\tau'_0|X, \tau'_{N''}|X\right)$ is bounded, by the first statement in Lemma \ref{lem:induction_vertices_commute}.

Let $M> \max\{M_6(S),F(4k_1+4)+ 2 F(4k_2+4)\}$. Then, by Theorem \ref{thm:mmprojectiondist} and Lemma \ref{lem:pantsquasiisom} applied with the specified value $M$,
$$d_{\pa(X)}\left(\pi_X V(\tau'_0),\pi_X V(\tau'_{N''})\right)\leq e_1(X,M,k(X,S),\ell(X)).$$
Lemma \ref{lem:induction_vertices_commute}, finally, gives our claim.
\end{proof}

Rather than proving Proposition \ref{prp:tcbound} directly in the form given, we will lean on the following definition and lemma.\footnote{Many thanks to Saul Schleimer for having suggested the proof of this lemma.}

\begin{defin}Given a surface $S$ and a sequence $\alpha_1,\ldots,\alpha_r$ of distinct isotopy classes of essential simple closed curves on $S$, we call an increasing subsequence $\alpha_{j_1},\ldots,\alpha_{j_s}$ a \nw{chain} if, for all $1\leq i < s$, $\alpha_{j_{i+1}}$ intersects $\alpha_{j_i}$ essentially.
\end{defin}

\begin{lemma}\label{lem:chainbound}
Let $S$ be a surface and let $\alpha_1,\ldots,\alpha_r$ be a sequence of distinct isotopy classes of essential simple closed curves on $S$. Suppose that any chain in this sequence has at most $c$ elements: then $r\leq \xi(S)c$.
\end{lemma}
\begin{proof}
We define a partition $A_1,\ldots,A_t$ of $\{1,\ldots, r\}$ inductively, as follows --- the number $t\geq 1$ will be determined by the construction. First, assign $1\in A_1$.

Suppose now that all indices $1,\ldots,i$ ($i<r$) have been assigned to some set in the partition. Let then $i+1\in A_{u+1}$, where $u$ is the highest of all $v\geq 0$ such that, for each $1\leq v'\leq v$, there exists $j\in A_{v'}$, $j\leq i$, such that $\alpha_{i+1}$ intersects $\alpha_j$ essentially. In particular, if $\alpha_{i+1}$ is disjoint from all curves $\alpha_j$ with $j\leq i$ and $j\in A_1$, then we assign $i+1\in A_1$, too. Eventually, we define $t$ to be the highest index $u$ such that some $1\leq i\leq r$ has been assigned to $A_u$.

As a consequence of the construction, if $j\in A_u$ for $u>1$ then there exists an index $1\leq l(j)< j$ such that $l(j)\in A_{u-1}$. Moreover, all indices in a specified set $A_u$ of the partition correspond to pairwise disjoint curves: as such, they are part of a pants decomposition for $S$, so there are at most $\xi(S)$ indices in $A_u$. Therefore $t\geq r/\xi(S)$.

Let now $j\in A_t$. The indices $l^{(t-1)}(j),l^{(t-2)}(j),\ldots,l(j),j$ belong to $A_1,A_2,$ \ldots, $A_{t-1},A_t$ respectively (here $l^{(i)}$ denotes the iteration of $l$ for $i$ times), and the subsequence of $\alpha_1,\ldots,\alpha_r$ specified by those indices is a chain. Therefore $t\leq c$; and $r\leq \xi(S)c$ as claimed.
\end{proof}

As our argument to prove Proposition \ref{prp:tcbound} is based on the diagonal extension machinery, we wish to make sure that diagonal extensions work smoothly with respect to Dehn twists.

\begin{lemma}\label{lem:weightsaftertwist}
Let $\bm\tau=(\tau_j)_{j=0}^N$ be a generic, recurrent train track splitting sequence. Consider the induced train tracks $\rho_j\coloneqq \tau_j|X$ on a non-annular essential subsurface $X$ of $S$. Suppose that a subsequence $\bm\tau(k,l)$ has twist nature with respect to a curve $\gamma\in\cc(X)$, and that $\rot_{\bm\tau}(\gamma;k,l)=m\geq 1$.

If $\alpha\in \cc(\rho_N)$ intersects $\gamma$ essentially, then $\alpha$ traverses each branch of $\rho_0$ contained in $\rho_0.\gamma$ at least $m$ times.

Suppose now that each $\rho_j$ fills $S^X$; let $\alpha$ be a curve essentially intersecting $\gamma$, with $\alpha\in \cc(\delta_N)$ for some $\delta_N\in \f(\rho_N)$. Then there is a $\delta_0\in \f(\rho_0)$ carrying $\alpha$, built from a recurrent subtrack of $\rho_0$ that fills $S^X$ and includes $\rho_0.\gamma$; and $\alpha$ traverses all branches in $\delta_0.\gamma$ at least $m-1$ times.
\end{lemma}
\begin{proof}
We prove the first statement. The set $\pi_{\nei(\gamma)}(\alpha)$ is nonempty and contained in $\cc(\rho_l^{\nei(\gamma)})$. By definition of rotation number, each arc $\alpha'\in \cc(\rho_l^{\nei(\gamma)})$ has $\hl_{\alpha'}(k)>2\pi m$ so it traverses each branch in $\rho_k^{\nei(\gamma)}.\gamma$ at least $m$ times (see point \ref{itm:hl_vs_multiplicity} after Definition \ref{def:horizontallength}). So also $\alpha$ shall traverse each branch of $\rho_k.\gamma$ at least $m$ times; and the same must be true in $\rho_0.\gamma$ (see Remark \ref{rmk:decreasingmeasures}). In particular this is true if one picks $\alpha\in \cc(\rho_N)\subseteq \cc(\rho_l)$.

We work now on the second of the two statements. Without altering the truth of the statement we can suppose, from this point on, that $\bm\tau(k,l)$ has been already replaced with the subdivision into Dehn twists with remainder guaranteed by Lemma \ref{lem:dehn+remainder}. In particular we are replacing the original $\tau_l$ with a train track which is comb equivalent to it; and therefore we are also operating some comb/uncomb moves after the new $\tau_l$, in order to continue smoothly with the original splitting sequence $\bm\tau(l,N)$.

This means that there is an index $k\leq r< l$ with $\tau_l=D_\gamma^{\epsilon m}(\tau_r)$ where $\epsilon$ is the sign of $\gamma$ as a twist curve. By Lemma \ref{lem:pantsboundunderdt}, also $\rho_l=D_{\gamma}^{\epsilon m}(\rho_r)$; and clearly $\f(\rho_l)=D_{\gamma}^{\epsilon m}\cdot \f(\rho_r)$, so similar relations hold for the sets of carried curves: $\cc(\rho_l)=D_{\gamma}^{\epsilon m}\cdot\cc(\rho_r)$; $\cf(\rho_l)=D_{\gamma}^{\epsilon m}\cdot \cf(\rho_r)$.

Rather than the original statement, we will prove this other one:
\begin{claim}
Let $\alpha$ be a curve essentially intersecting $\gamma$, carried by $\delta_l\in \f(\rho_l)$ which is a diagonal extension of a recurrent subtrack $\sigma_l$ of $\rho_l$ that fills $S^X$. Then there exists a $\delta'_r\in \f(\rho_r)$ which is a diagonal extension of the recurrent subtrack $\sigma_r=D_{\gamma}^{-\epsilon m}(\sigma_l)$ of $\rho_r$ and contains $\rho_r.\gamma$; and $\alpha$ traverses each branch in $\delta'_r.\gamma$ at least $m-1$ times.
\end{claim}

This implies the desired statement as follows. According to Lemma \ref{lem:cf_decreasing},\linebreak $\cf(\rho_N)\subseteq \cf(\rho_l)$ so the above statement is true for $\alpha\in\cf(\rho_N)$ in particular. Also, the last statement in that lemma ensures that not only $\cf(\rho_r)\subseteq \cf(\rho_0)$, but also there is a $\delta'_0\in \f(\rho_0)$ fully carrying $\delta'_r$. In particular $\delta'_0$ will carry both $\alpha$ and $\gamma$, and $\alpha$ traverses at least $m-1$ times any branch in $\delta'_0.\gamma$.

Let $\delta_r\coloneqq D_\gamma^{-\epsilon m}(\delta_l)$, and let $\alpha_-\coloneqq D_{\gamma}^{-\epsilon m}(\alpha)\in \cc(\delta_r)$. Without loss of generality, we may suppose that $\sigma_r$ complies with the following maximality property: any almost track $\sigma_r\subsetneq \xi\subseteq \rho_r$ has the property that a branch of $\xi$ intersects one of $\delta_r$ at a point that is interior for both branches; and this property remains true when changing $\xi$ up to isotopies fixing $\sigma_r$.

Our $\delta'_r$ is going to be a diagonal extension of $\sigma'_r\coloneqq \sigma_r\cup\rho_r.\gamma$ --- which is a generic almost track, is recurrent as $\sigma_r$ is, and fills $S^X$. When $\sigma_r=\sigma'_r$, one just takes $\delta'_r\coloneqq \delta_r$ and proves the desired claim using the first claim of the present lemma. When $\sigma_r\not=\sigma'_r$ instead, $\sigma'_r$ and $\delta_r$ are not subtracks of a common almost track, hence $\alpha_-$ may not be carried by any diagonal extension of $\sigma'_r$; nevertheless we will show that $D_\gamma^\epsilon(\alpha_-)$ is carried by a suitable one, and this will allow us to conclude.

\step{1} finding an `efficient position' for the extra branches in $\br(\delta_r)\setminus\br(\sigma_r)$ (similar to what Definition \ref{def:efficientposition} requires for curves): this will be necessary in order to avoid monogons in the construction of $\delta'_r$. The construction has some points in common with \S 4.2 in \cite{mms}.

Let $a_1,\ldots,a_s$ be an enumeration of the branches in $\br(\delta_r)\setminus\br(\sigma_r)$ that are traversed by $\alpha_-$. For each $i=1,\ldots,s$, let $Q_i$ be the closure in $S$ of the connected component of $S\setminus \sigma_r$ which contains $a_i$. Also, if $\ul\gamma$ is a train path realization of $\gamma$ in $\rho_r$ (and in $\sigma'_r$), let $\beta_i^1,\ldots,\beta_i^{u(i)}$ be the maximal segments of this path which lie in $Q_i$ and are not contained entirely in its boundary. Any two of them are not necessarily disjoint but, since $\gamma$ is wide in $\rho_i$, each branch of $\sigma'_r$ is not traversed more than twice in total, with multiplicities, by this collection.

For each of the $\beta_i^j$, let $\ul\beta_i^j$ be a smooth embedded path in $\bar\nei(\sigma'_r)$ which is transverse to all ties and has its extremes at two cusps in $\partial Q_i$: $\ul\beta_i^j$ shall be constructed to be parallel to $\beta_i^j$ but, at each extremity of $\beta_i^j$, if it is not a corner of $\partial Q_i$, $\ul\beta_i^j$ shall continue traversing branches in $\partial Q_i$, until a corner is reached. Also, we choose $\ul\beta_i^j$ to sit inside $\inte(Q_i)$, except for its endpoints; and all $\ul\beta_i^j$ to be disjoint from each other except possibly for their endpoints.

Fix a branch $b\in\br(\sigma'_r)$ with $b\subseteq Q_i$, such at least one of the following is true:
\begin{itemize}
\item $b\subseteq\partial Q_i$ and $b$ is traversed by some $\ul\beta_i^j$ at least once;
\item $b$ is traversed at least twice in total by the family $\{\ul\beta_i^j\}_{j=1}^{u(i)}$.
\end{itemize}
If $b\subseteq \partial Q_i$, define the \emph{fold} at $b$ as the union of all sub-ties which are contained in $R_b([-1,1]\times[-1,1])$ and have one endpoint along $b$ and the other along some $\ul\beta_i^j$. Otherwise, define the fold at $b$ as the union of all sub-ties which are contained in $R_b([-1,1]\times[-1,1])$ and have both endpoints along some $\ul\beta_i^j$. Informally, the fold at $b$ is a triangle or rectangle containing all chunks of the arcs $\ul\beta_i^j$ which traverse $b$: the ones that one could sensibly fold together or fold to $b$.

Define a \emph{folding area} as a maximal union of folds whose interior is connected. For each folding area $A$, the boundary $\partial A$ includes open subsets of a finite number of ties of $\nei(\sigma'_r)$. Cut $A$ along each tie that has an open subset in $\partial A$, to get a collection of \emph{folding rectangles} and \emph{triangles}. For $A'$ a folding rectangle or triangle, call $\partial_h A'=\mathrm{int}\left(\partial A'\cap(\sigma_r\cup \ul\beta_i^1\ldots \cup\ul\beta_i^{u(i)})\right)$ and $\partial_v A'=\inte(\partial A'\setminus \partial_h A')\cup \left(A'\cap (\text{corners of }\partial Q_i)\right)$. The two components of $\partial_h A'$ are two tie-transverse arcs in $\nei(\sigma'_r)$. They have the same image under the tie collapse $c:\bar\nei_0(\sigma'_r)\rightarrow \sigma'_r$: we call it the \emph{crush} of $\partial_h A'$; its closure is a bounded train path in $\sigma'_r$ (up to reparametrization). Figure \ref{fig:diagext_efficientpos}, left, explains this construction.

We need to adjust each branch $a_i$ with respect to $\sigma'_r$ via isotopies which leave their endpoints fixed and keep the inclusion $a_i\subseteq Q_i$, in order to attain an analogous condition to efficient position for curves (Definition \ref{def:efficientposition}); in particular, we wish that the newly obtained smooth paths $a''_i$ are still embedded in $S^X$ individually; but they need not be pairwise disjoint.

More specifically, we require that, for each $i$:
\begin{enumerate}
\item $a_i''\setminus \sigma_r$ is connected and $a''_i\cap\sigma'_r$ is a union of branches of $\sigma'_r$ plus isolated points;
\item if $i'\not=i$ but $Q_{i'}=Q_i$, then $a''_i\cap a''_{i'}$ shall be exactly one of the following: empty, one common endpoint, or a union of branches of $\sigma'_r$; moreover, $a''_{i'}$ cannot intersect more than 1 component of $Q_i\setminus a''_i$;
\item the association of each connected component of $a''_i\cap\sigma'_r$ with the connected component of $a''_i\cap\bar\nei(\sigma'_r)$ which contains it is a bijection; 
\item given a connected component of $a''_i\cap\bar\nei(\sigma'_r)$, it is either a single tie, or it is transverse to all ties of $\nei(\sigma'_r)$ it encounters; in this second case, each of its endpoints either is a corner of $\partial Q_i$ or lies along $\partial_v\bar\nei(\sigma'_r)$;
\item if a connected component of $a''_i\cap\sigma'_r$ is an isolated point, then either
\begin{itemize}
\item it is contained in $\partial Q_i$ and there is a component of $a''_i\cap\bar\nei_0(\sigma'_r)$ which consists of exactly that point,
\item or it is contained in a component of $a''_i\cap\bar\nei(\sigma'_r)$ which is a tie of $\bar\nei(\sigma'_r)$;
\end{itemize}
\item if a given connected component of $a''_i\cap\sigma'_r$ is a union $b$ of branches of $\sigma'_r$, let $B$ be the corresponding connected component of $a''_i\cap\bar\nei(\sigma'_r)$: then $B\setminus b$ is entirely contained in $\bar\nei(\sigma'_r)\setminus \bar\nei_0(\sigma'_r)$;
\item however one chooses a finite-length smooth, embedded, path $\rho$ along $\sigma'_r\cap Q_i$ (not necessarily a train path), and a subarc $a$ of $a''_i$, which intersect exactly at their endpoints, the region they bound together has negative index.
\end{enumerate}

Note that, if one can choose two paths $\rho$ and $a$ as in condition 7 and they satisfy it, then they can only delimit a 1-punctured bigon. If they exist but do not satify the condition then, even admitting that $a$ may degenerate to a single point, the region they bound cannot be a zero-gon, monogon, or 1-punctured zero-gon/monogon.

\begin{figure}
\centering
\begin{minipage}[c]{.5\textwidth}
\def\svgwidth{\textwidth}
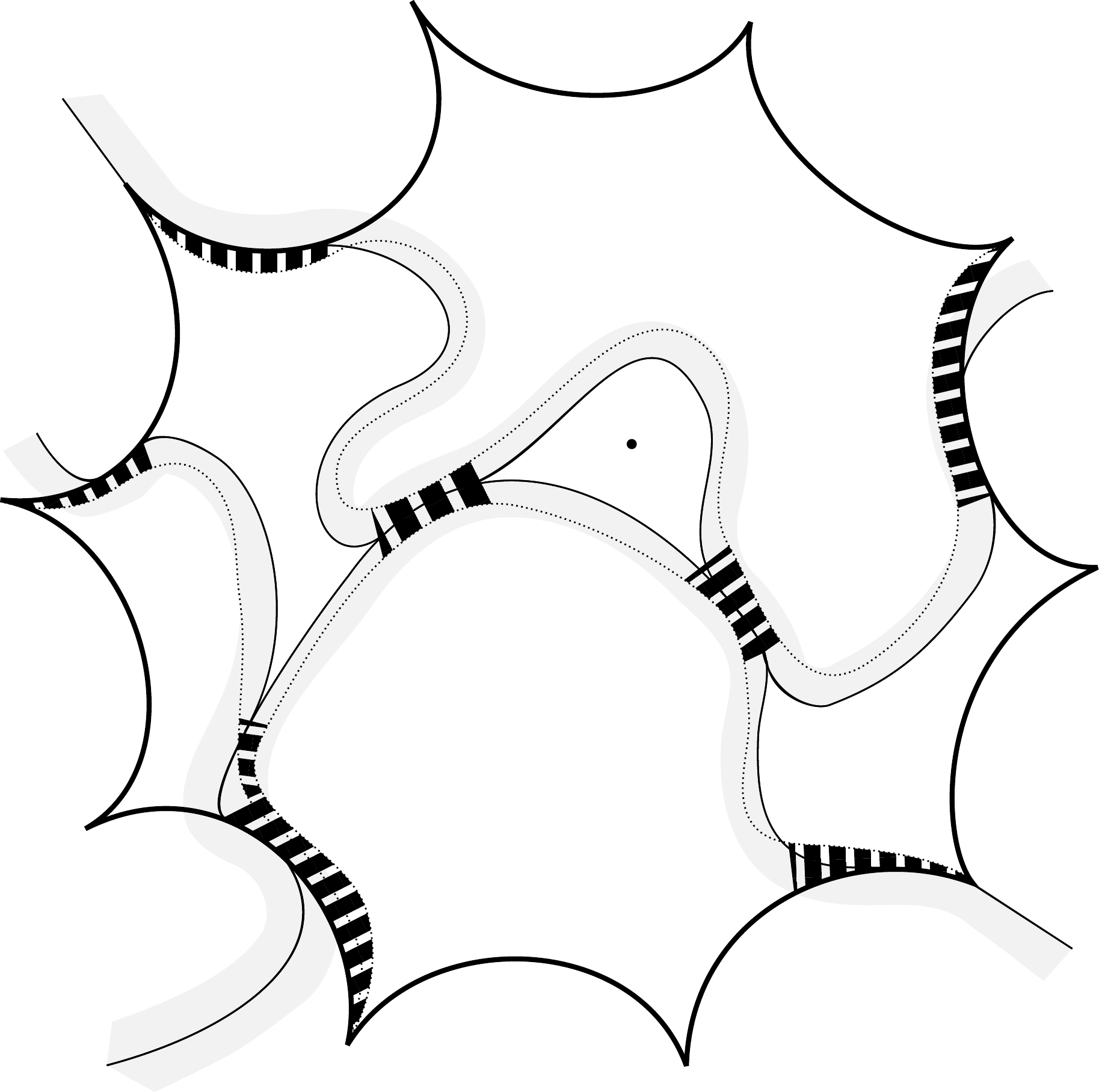
\end{minipage}\hspace{5em}
\begin{minipage}[c]{.35\textwidth}
\def\svgwidth{\textwidth}
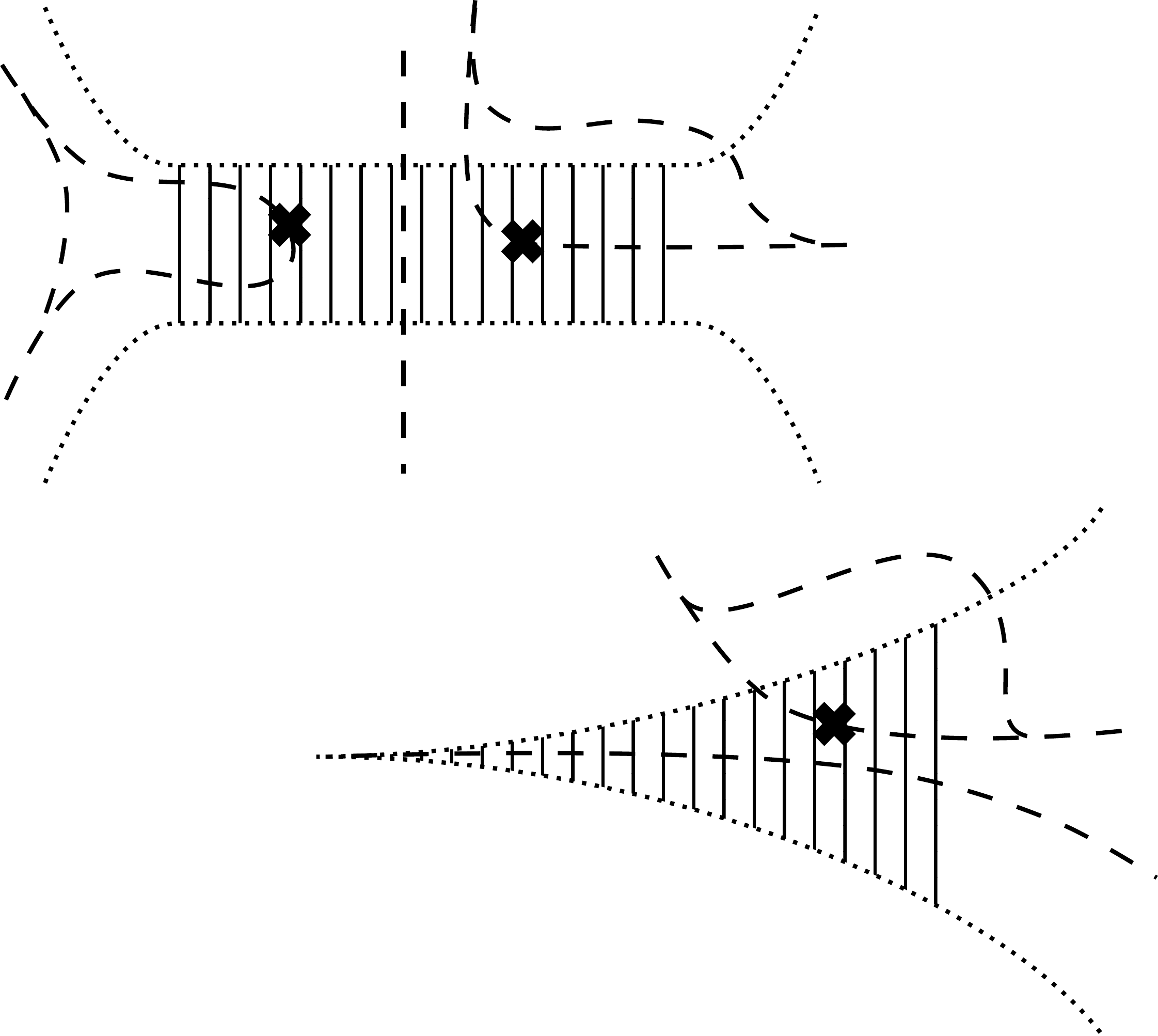
\end{minipage} 
\caption{\label{fig:diagext_efficientpos}\textit{Left:} The construction of the arcs $\ul\beta_i^j$ given $\beta_i^j$ in a complementary region $Q_i$ of $\sigma_r$ (dotted), and the related folding rectangles and triangles (marked with black stripes). The twist collar $A_\gamma$ is marked in grey. \textit{Right:} Adaptation of the arcs $a_i$ (dashed) with respect to a folding rectangle or triangle (here shown with its foliation into ties). In short: every time $a_i$ does \emph{not} enter and exit the folding triangle/rectangle from opposite edges, we move it off.}
\end{figure}
\begin{figure}
\centering
\def\svgwidth{.8\textwidth}
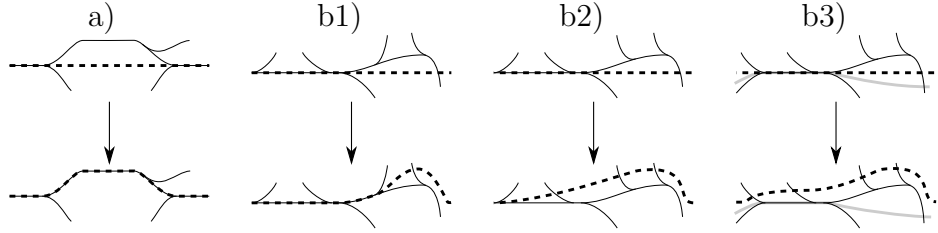
\caption{\label{fig:diagext_furtheradjust}Ways to remove one of the innermost bigons between $\bigcup_{i=1}^{s} \phi(a_i)$ and $\sigma'_r$ --- say one bounded by $(\rho,a)$. The picture a) depicts the case of a bigon whose corners are both tangential intersections of some $a_i$ (dashed) with a switch of $\sigma'_r$: the bigon can be crushed with $a$ into $\sigma'_r$. The pictures b1), b2), b3) depict the case of a bigon with one corner being a transverse intersection between $\phi(a_i)$ and $\sigma'_r$, and the other a tangential one. The idea is to move $a$ across $\rho$ in order to remove completely the transverse intersection point. In b1) a new tangential intersection point is introduced at a switch along the $\sigma'_r$-edge of the cancelled bigon; in b2) a tangential intersection point is introduced at a switch not along that edge; in b3) none is introduced at all. None of these moves increases the number of innermost bigons, but it may leave it unvaried: in b3) another arc $\phi(a_{i'})$ is marked in grey to show how it may happen. In any case the number $x$ decreases. The arc $\phi(a_{i'})$ depicted also shows how condition 2 is not violated when moving $\phi(a_i)$.}
\end{figure}

\begin{figure}
\centering
\def\svgwidth{.85\textwidth}
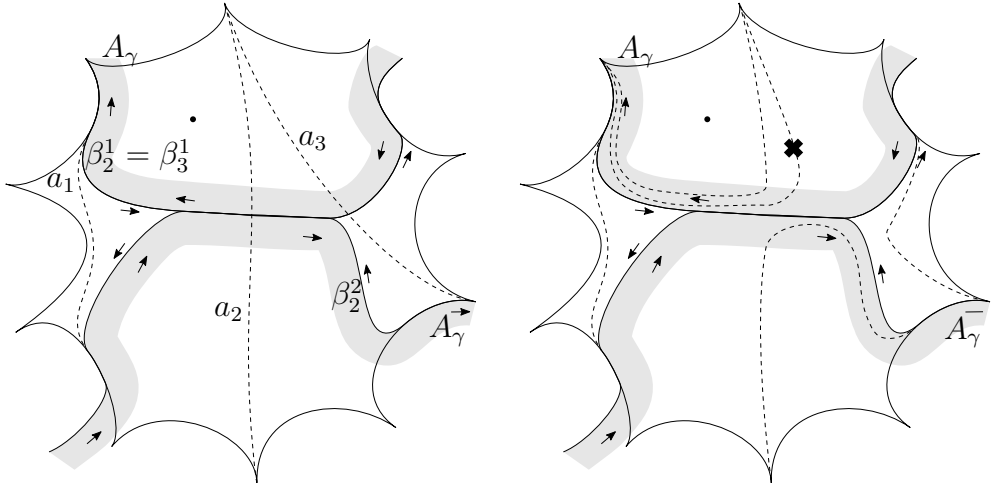
\caption{\label{fig:diagextconstruction}An example of the way the construction of $\delta'_r$ works in one of the complementary regions of $\sigma_r$, in this case $Q_1=Q_2=Q_3$. The arcs $\beta_\cdot^\cdot$ are part of $\rho_r.\gamma$ (in this picture they share a branch traversed twice by $\gamma$), and a twist collar for the latter is marked in grey. Branches $a_1,a_2,a_3\in\br(\delta_r)\setminus\br(\sigma'_r)$ (picture to the left) are drawn with a dashed line. The arrows around the $\beta_\cdot^\cdot$ show the instructions to `bend' the segments $e_j$ to get the corresponding $e'_j$ in $\nei(\sigma'_r.\gamma)$. The picture to the right shows the new branches that are added to $\sigma'_r$ to get $\sigma''_r$. A cross marks a new branch that forms a bigon with another one, thus is to be excluded from the definition of $\delta'_r$.}
\end{figure}

In practice, the position described above can be attained with the process we are about to describe. Start with isotoping the arcs $a_i$ so that their interiors remain contained in the respective $\mathrm{int}(Q_i)$, pairwise disjoint, and each intersects each arc $\ul\beta_i^j$ transversally, without forming bigons. Then, via slight perturbations, for each folding triangle/rectangle $A$, make sure that no $a_i$ includes points of $\partial A\setminus (\partial_h A\cup\partial_v A)$ (i.e. the corners of the folding triangle/rectangle which are not cusps in $\partial Q_i$).

Let $\mathcal S$ be the set of all connected components of the intersections of the form $a_i\cap A$, for some $1\leq i\leq s$ and $A$ a folding rectangle or triangle. First, we make sure that, for every $a_i$ and every folding rectangle or triangle $A$, every connected component of $a_i\cap A$ has its endpoints either on the two opposite components of $\partial_h A$, or the two opposite components of $\partial_v A$ (remember: one of the latter may be a single point). This is done as specified in Figure \ref{fig:diagext_efficientpos}: the process comes to an end, because moving one or more portions of the arcs $a_i$ off a folding rectangle/triangle results into making the cardinality of $\mathcal S$ strictly lower. 

When performing each of these moves, we keep the interiors of the $a_i$ pairwise disjoint, thus they comply with a stricter version of condition 2 in the list above. When nothing is left to move off, they can also be supposed to intersect $\sigma'_r$ only in isolated points, thus complying with a stricter version of condition 1. We can also suppose that conditions 3, 4 are satisfied by all $a_i$ provided that the tie neighbourhood is chosen wisely. 

Let now $\phi:S^X\rightarrow S^X$ be a smooth map complying with the following requests. On the points belonging $\sigma'_r$, to any of the $\ul\beta_i^j$, or to any folding area --- denote $Y$ the set of all these points --- $\phi$ coincides with the tie collapse $c:\bar\nei_0(\sigma'_r)\rightarrow \sigma'_r$. $\phi$ is then defined to be homotopic to $\mathrm{id}_{S^X}$, with $\phi|_{S^X\setminus Y}$ injective, $\phi|_{\bar\nei(\sigma'_r)}$ keeping each point along the same tie, and $\phi(S^X\setminus Y)\cap \sigma'_r=\emptyset$.

The family $\{\phi(a_i)\}$ is seen to comply with all conditions 1--6 that the family $\{a''_i\}$ is required to satisfy (possibly adapting the tie neighbourhood $\bar\nei(\sigma'_r)$), but not necessarily condition 7. However, it is impossible that a segment of one of the $\phi(a_i)$ bounds, together with a smooth embedded path along $\sigma'_r$, a bigon whose corners are each a transverse intersection point between some $\phi(a_i)$ and some branch of $\sigma'_r$. If there were one such bigon, then necessarily there would be a bigon bounded by the corresponding $a_i$ and one of the $\ul\beta_i^j$.

Now, let $x$ be the number of pairs $(\rho,a)$ which transgress condition $7$: $\rho$ is a smooth path along $\sigma'_r$, $a$ is a subarc of one of the $\phi(a_i)$, $\rho$ and $a$ intersect exactly at their endpoints, and they bound a bigon. In Figure \ref{fig:diagext_furtheradjust} it is shown how to set up a procedure that recursively considers an innermost bigon bounded by some $(\rho,a)$, and moves $\phi(a_i)$ off the bigon to decrease $x$ strictly. At each stage of the recursion, one shall probably perform isotopies to make sure that the position with respect to $\bar\nei(\sigma'_r)$ keeps respecting conditions 3--6. When it is impossible to proceed with the recursion any further, condition 7 will also be satisfied; while conditions 1 and 2 will be still satisfied, too. Define the family $\{a''_i\}$ to be the set of arcs eventually attained.

Finally, define $\delta''_r\coloneqq \sigma_r\cup\left(\bigcup_{i=1}^s a''_i\right)$. It is clear that $\delta''_r$ carries $\ul\alpha_-$.

\step{2} Construction of $\delta'_r$.

Before we start with this step, we announce that the construction of $\delta'_r$ will consist of removing any transverse intersection between the family $\{a''_i\}$ and $\sigma'_r$: the arcs $a''_i$ will be `broken' at each transverse intersection point, and then bent as $D_\gamma^\epsilon$ would do, in order to get attached to a switch of $\sigma'_r$. In Step 3 we will see why $\delta'_r$ thus constructed carries $\alpha$.

The curve $\gamma$ is a twist curve in $\sigma'_r$: the first condition in Definition \ref{def:twistcurve} is clearly verified; the second one is, too, because $\sigma'_r$ is recurrent. In particular $\gamma$ inherits the twist collar $A_\gamma$ from $\rho_r$. As it was already done previously, we suppose that the component of $\Lambda$ of $\partial\bar A_\gamma$ which is not part of $\sigma'_r.\gamma$ coincides with a component of $\partial\bar\nei(\sigma'_r.\gamma)$.

Enumerate (in any way) $e_1,\ldots,e_t$ the `half-ties' contained in $\left(\bigcup_{i=1}^s a''_i\right)\cap\bar\nei(\sigma'_r.\gamma)$, i.e., for each connected component $e$ of this intersection which is a tie, the two opposite segments of it, from $e\cap\sigma'_r.\gamma$ to $\partial\bar\nei(\sigma'_r)$, are two of the elements of the list. Let then $P_j$ be the endpoint of the corresponding $e_j$ that lies along $\partial\bar\nei(\sigma'_r)$; and let $i(j)$ be the only value of $1\leq i\leq s$ such that $e_j\subset a''_i$. We wish to define, for each $j$, a new arc $e'_j$ contained in $\bar\nei_0(\sigma'_r)$ whose endpoints are $P_j$ and a suitable switch of $\sigma'_r$. 

To do so, we perform the following construction for each $e_j$, one after another according to their order --- see also Figure \ref{fig:diagextconstruction}. Supposing that $e'_1,\ldots,e'_{j-1}$ have been constructed, consider the closures of the connected components of $(\bar\nei_0(\sigma'_r)\cap Q_{i(j)})\setminus\sigma'_r$: each of them is bounded by a smooth component of $\partial\bar\nei_0(\sigma'_r)$ and a bounded train path along $\sigma'_r.\gamma$, and therefore is a bigon, with its edges transverse to the ties of $\sigma'_r$, and corners (cusps) at switches of $\sigma'_r$. Let then $E_j\subseteq Q_{i(j)}$ be the region among these which contains $e_j$.

Define $e'_j$ to be an arc in $E_j$ with the following properties:
\begin{itemize}
\item $e'_j\cap\partial E_j$ consists exactly of the endpoints of $e'_j$, which shall be $P_j$ and one of the cusps of $\partial E_j$;
\item $e'_j$ is transverse to all ties it meets;
\item $e'_j$ runs in parallel with (part) of one of the arcs $\ul\beta_{i(j)}^{\ldots}$; moreover it proceeds in the verse consistent with the $A_\gamma$-orientation if $e_j\subseteq \ol{A_\gamma}$, and in the opposite verse otherwise;
\item $e'_j$ is part of a smooth path that begins at a point of $\left(\bigcup_{i=1}^s a''_i\right)\setminus \bar\nei(\sigma'_r.\gamma)$, enters $e'_j$ from $P_j$ and, after leaving it, continues entering a branch in $\sigma'_r$ and following it;
\item $e'_j$ is disjoint from all $e'_{j'}$ for $j'<j$.
\end{itemize}
Note that these conditions determine uniquely what is the endpoint of $e'_j$ other than $P_j$.

We claim that $e'_j\subseteq\bar\nei(\sigma'_r.\gamma)$. This is clear if $e_j\subseteq\ol{A_\gamma}$: as $e'_j$ proceeds in the $A_\gamma$-orientation, eventually it must travel between $\sigma'_r.\gamma$ and a branch end hitting $A_\gamma$ (there is one necessarily, because $\sigma'_r.\gamma$ intersects more than one of the regions $Q_i$), and is forced to reach the switch between the two. Since this switch belongs to a branch end hitting $A_\gamma$, it must necessarily be located along $\partial Q_{i(j)}$.

In case $e_j\cap A_\gamma=\emptyset$, the analysis is slightly more involved. Suppose that $e'_j$, starting at $P_j$, and proceeding along an arc $\ul\beta$ among the $\ul\beta_{i(j)}^{\ldots}$, oppositely to the $A_\gamma$-orientation, leaves $\nei(\sigma'_r.\gamma)$. Then this event must occur just after $e'_j$ has traversed a large branch end in $\sigma'_r.\gamma$. The first branch end $\eta_1$ traversed which is not in $\sigma'_r.\gamma$ is then, necessarily, one that avoids $A_\gamma$; and it is adverse, due to the orientation of $e'_j$. Also, $\eta_1$ is included in some smooth edge $\lambda$ of $\partial Q_i$. Necessarily, $\ul\beta$ does not travel along the entire length of $\lambda$ but only a portion.

This means that $\gamma$ traverses a large branch $b$ of $\sigma'_r$, contained in the \emph{interior} of $\lambda$, and then gets out of $\lambda$ on the side opposite to $Q_{i(j)}$. Let $\eta_2$ be the branch end of $\sigma'_r$, contained in $\lambda$, which shares with $b$ the switch other than the one $b$ shares with $\eta_1$. Then $\eta_1,\eta_2$ are branch ends giving $\gamma$ opposite orientations; so, as $\eta_2$ is a branch end hitting $A_\gamma$, $\eta_1$ is necessarily a favourable one, a contradiction.

When this process is done, define $\sigma''_r\coloneqq \sigma'_r\cup \left(\bigcup_{i=1}^s \left(a''_i\setminus \nei_0(\sigma'_r)\right)\right)\cup\left(\bigcup_{j=1}^t e'_j\right)$. In general $\sigma''_r$ is not a train track as there may be bigons among the complementary regions of $S\setminus \sigma''_r$. Monogons, nullgons and 1-punctured nullgons are to be excluded instead, as they would require either $a''_i$ to bound a bigon with $\sigma'_r$, or $\sigma''_r$ to have more switches than $\sigma'_r$.

Define $\delta'_r$ to be a maximal subtrack of $\sigma''_r$ which includes $\sigma'_r$ and is a train track. Practically speaking, define $\delta'_r$ by deleting any branch in $\br(\sigma''_r)\setminus\br(\sigma'_r)$ which bounds a bigon together with a train path along $\sigma'_r$; and furthermore, every time there are two branches in $\br(\sigma''_r)\setminus\br(\sigma'_r)$ surviving after this operation, and bounding a bigon together, delete one of them. As each of these operations affects only one of the regions $\{Q_i\}$, and any innermost bigon has to be entirely contained in one of these regions, the suggested procedure is able to remove the extra branches without leaving any bigon in $\delta'_r$. 

The curve $\gamma$ is carried by $\delta'_r$ and is again a twist curve there.

We consider a tie neighbourhood for $\delta''_r$ built with the following constraints: $\bar\nei(\delta''_r)$ and $\bar\nei(\sigma'_r)$ induce one same tie neighbourhood $\nei(\sigma_r)$ for their common subtrack $\sigma_r$; each transverse intersection point between $\delta''_r$ and $\sigma'_r$ is contained in a connected component of $\bar\nei(\delta''_r)\cap\bar\nei(\sigma'_r)$ which is diffeomorphic to a square, and the tie foliations given by the two tie neighbourhoods are each parallel to one of the two pairs of opposite edges of the square.

From this, we build a tie neighbourhood for $\sigma''_r$ too, with the property that $\bar\nei(\sigma''_r)\subseteq \bar\nei(\sigma'_r.\gamma)\cup \bar\nei(\delta''_r)$; $\bar\nei(\sigma''_r)\setminus \bar\nei(\sigma'_r.\gamma)= \bar\nei(\delta''_r)\setminus \bar\nei(\sigma'_r.\gamma)$; and all ties in each rectangle $R_b$, for $b$ a branch of $\sigma''_r$ which is part of $\sigma''_r.\gamma$, are sub-ties of $R_{b'}$, for $b'\in\br(\sigma'_r)$ a branch contained in $\sigma'_r.\gamma$ ($b,b'$ are actually the same branch, regarded as part of the two different almost tracks). This means that the branch end rectangles $R_{e'_j}$ have their image contained in $\bar\nei(\sigma'_r.\gamma)$, albeit with a new specification for their ties. This definition of tie neighbourhood induces a tie neighbourhood for $\delta'_r$, and again one for $\sigma_r$, which \emph{does not} coincide with the $\nei(\sigma_r)$ previously defined. An example of these constructions is given in Figure \ref{fig:diagonalnbhs}.

\begin{figure}
\centering{\includegraphics[width=.9\textwidth]{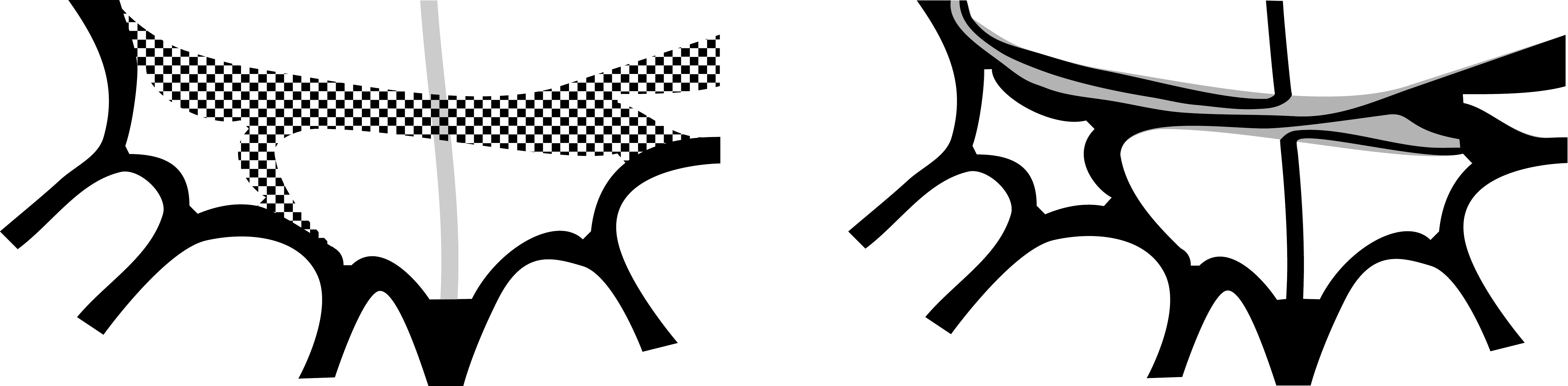}}
\caption{\label{fig:diagonalnbhs}A choice of tie neighbourhoods related to a portion of the previous Figure \ref{fig:diagextconstruction}. The picture on the left shows how $\nei(\sigma'_r)$ (chequerboard-coloured) and $\nei(\delta''_r)$ (grey) intersect and define on their common subtrack $\sigma_r$ the same tie neighbourhood $\nei(\sigma_r)$ (black). The picture on the right shows a choice of $\nei(\sigma''_r)$, contained in $\sigma'_r$ (painted in grey here) and coinciding with it outside $\bar\nei(\sigma'_r.\gamma)$.}
\end{figure}

We will now use, several times, the notation $\nei(\sigma'_r.\gamma)$ to mean specifically the tie neighbourhood induced by $\nei(\sigma'_r)$: it is necessary to clarify this, because $\sigma''_r.\gamma$ (or $\delta'_r.\gamma$), albeit coinciding with $\sigma'_r.\gamma$ setwise, inherits a different tie neighbourhood from $\sigma''_r$ (or $\delta'_r$). Similarly we will use $\nei(\sigma''_r.\sigma'_r)$, $\nei(\sigma''_r.\sigma_r)$ to mean the tie neighbourhoods inherited from $\sigma''_r$, as opposed to $\nei(\sigma'_r)$ the `native' tie neighbourhood and $\nei(\sigma_r)$ the tie neighbourhood that $\sigma_r$ inherits from $\sigma'_r$ or, equivalently, from $\delta''_r$. Note that not only $\bar\nei(\sigma''_r.\sigma'_r)\subseteq \bar\nei(\sigma'_r)$ but there is a diffeomorphism $f:\bar\nei(\sigma'_r)\rightarrow \bar\nei(\sigma''_r.\sigma'_r)$ sending each tie of $\bar\nei(\sigma'_r)$ to a subset of it.

For each $\beta_i^j$, let $\mathcal X_i^j$ be the connected component of $\bar\nei(\sigma'_r)\setminus \nei(\sigma_r)$ which intersects it. 

\step{3} $D_\gamma^\epsilon(\alpha_-)$ is carried by $\delta'_r$, and conclusion.

Let $\ul\alpha_-$ be a carried realization of $\alpha_-$ in $\bar\nei(\delta''_r)$ and $\ul\gamma$ be a carried realization of $\gamma$ in $\bar\nei(\sigma_r)$ (therefore it is one in $\bar\nei(\sigma'_r)$ and in $\bar\nei(\delta''_r)$, too), with the properties that: they intersect transversely, minimally among the pairs of curves in the respective homotopy classes; each of the connected components of $\ul\alpha_-\cap\bar\nei(\sigma'_r)$ is either of the following:
\begin{itemize}
\item transverse to all ties, with its endpoints on $\partial_v\bar\nei(\sigma'_r)$;
\item a single tie of $\bar\nei(\sigma'_r)$.
\end{itemize}
This property may be required because of the previously specified intersection pattern between $\bar\nei(\sigma'_r)$ and $\bar\nei(\delta''_r)$. Similarly, we require that each of the connected components of $\ul\gamma\cap\bar\nei(\delta''_r)$ is either:
\begin{itemize}
\item transverse to all ties of $\bar\nei(\delta''_r)$ (but this time its endpoints may lie on $\partial_h\bar\nei(\delta''_r)$, too);
\item a single tie of $\bar\nei(\delta''_r)$.
\end{itemize}

We are about to construct a carried realization of $\alpha_+\coloneqq D_\gamma^\epsilon(\alpha_-)$. What we will do, informally, is realize the Dehn twist on $\ul\alpha_-$ by bending each transverse intersection of $\alpha_-$ in $\bar\nei(\delta''_r)$ to follow the branches of $\sigma''_r$, which have been introduced as a bending of the branches of $\delta''_r$, specifically for this idea to work; and by making the carried portions of $\ul\alpha_-$ in $\sigma'_r$ wind once more about $\sigma'_r.\gamma$. Once a carried realization of $D_\gamma^\epsilon(\alpha_-)$ in $\sigma''_r$ is proved to exist, it will be clear that there exists one in $\delta'_r$, too.

Let $\mathcal T$ be a (narrow) regular neighbourhood of $\ul\gamma$ in $\bar\nei(\sigma'_r.\gamma)$, and let $D_{\mathcal T}$ be a concrete realization of the Dehn twist about $\ul\gamma$, chosen to be the identity map outside $\mathcal T$.

Let $\Xi\coloneqq \ul\alpha_-\cap\bar\nei(\sigma'_r.\gamma)$. It is legitimate to suppose that there is a bijection between the connected components of $\Xi\cap \mathcal T$ and the points of $\Xi\cap\ul\gamma$.

Recall that $\Lambda$ is the connected component of $\partial\bar A_\gamma$ which is also a connected component of $\partial\bar\nei(\sigma'_r.\gamma)$. Given a connected component $\xi$ of $\Xi$, we have that $\xi\cap \ul\gamma\not=\emptyset$ if and only if $\xi$ has one endpoint along $\Lambda$, and the other endpoint along some other component of $\partial\bar\nei(\sigma'_r.\gamma)$. If this is not the case, any intersection point between $\ul\gamma$ and $\ul\alpha_-$, located along $\xi$, could be deleted by moving $\ul\gamma$ via isotopies, and this would reduce their total amount, a contradiction.

So, if a connected component $\xi$ intersects $\ul\gamma$ and is not a tie, orient it from its endpoint along $\Lambda$ towards the other one: then we have that $\xi$ traverses the ties according to the $A_\gamma$-orientation. This is the case because, at the point where $\xi$ begins, either $\ul\alpha_-$ is entering $\bar\nei(\sigma'_r.\gamma)$ by traversing a branch end of $\sigma'_r$ which hits $A_\gamma$; or $\ul\alpha_-$ has just entered $\bar\nei(\sigma'_r)$ via a component of $\partial_v\bar\nei(\sigma'_r)$ which has an endpoint along $\Lambda$.

Define $\Xi'=D_{\mathcal T}^\epsilon(\Xi)$: then the curve $\ul\alpha_+\coloneqq (\ul\alpha_-\setminus\Xi)\cup\Xi'$ is a loop in $S$ in the isotopy class of $\alpha_+$. Moreover $\ul\alpha_+\setminus \nei(\sigma'_r.\gamma)$ is transverse to all ties of $\bar\nei(\delta''_r)$, and therefore of $\bar\nei(\sigma''_r)$, it encounters. We wish to isotope $\Xi'$ so that it is entirely contained in $\bar\nei(\sigma''_r)\cap\bar\nei(\sigma'_r.\gamma)$, and is transverse to the ties of $\bar\nei(\sigma''_r)$. After the end of this modification, the resulting $\ul\alpha_+$ is a carried realization of $\alpha_+$ in $\bar\nei(\sigma''_r)$.

Regardless of what are the properties of the single connected components $\Xi'$, after a small perturbation leaving fixed the endpoints of each connected component and not altering the intersection pattern, is transverse to all ties of $\bar\nei(\sigma'_r.\gamma)$. Moreover, given any component of $\Xi'$ with an endpoint along $\Lambda$, if one orients it from this endpoint towards the other one, then one actually gets the $A_\gamma$-orientation on it.

For each $\mathcal X_i^j$, each connected component $\zeta$ of $\Xi'\cap \ol{\mathcal X}_i^j$ falls into one of the two following cases:
\begin{itemize}
\item $\zeta$ has one endpoint along $\partial_h\bar\nei(\sigma'_r)$. In this case $\zeta$ is part of $D_{\mathcal T}^\epsilon(\xi)$, for $\xi$ a connected component of $\Xi$ which is a tie of $\bar\nei(\sigma'_r.\gamma)$. The other endpoint of $\zeta$ lies necessarily along $\partial\bar\nei(\sigma_r)$. Furthermore, there is one of the branch ends $e'_u$ defined in Step 2 such that the first endpoint of $\zeta$ lies actually along the tie $R_{e'_u}(\{0\}\times[-1,1])$ of $\bar\nei(\sigma''_r)$, and such that $R_{e'_u}$ intersects the same connected component of $\ol{\mathcal X}_i^j\cap \partial\bar\nei(\sigma_r)$ on which the second endpoint of $\zeta$ lies. Informally, $\zeta$ traverses, in the same order, the ties of $\bar\nei(\sigma'.\gamma)$ that $e'_u$ traverses.
\item both endpoints of $\zeta$ are away from $\partial_h\bar\nei(\sigma'_r)$. This includes the case of $\zeta$ being part of $D_{\mathcal T}^\epsilon(\xi)$, for $\xi$ a connected component of $\Xi$ which does not intersect $\ul\gamma$ --- so that, actually, $D_{\mathcal T}^\epsilon(\xi)=\xi$.
\end{itemize}

Recall the diffeomorphism $f:\bar\nei(\sigma'_r)\rightarrow \bar\nei(\sigma''_r.\sigma'_r)$ defined above, and define $\Xi'_f$ as follows. Given any connected component $\xi'$ of $\Xi'$, attach to each extreme point of $f(\xi')$ not lying along $\partial_h\bar\nei(\sigma'_r.\gamma)$ a segment of tie in $\bar\nei(\sigma'_r.\gamma)$, so as to obtain a path in $\bar\nei(\sigma'_r.\gamma)$ with both endpoints lying along $\partial\bar\nei(\sigma'_r.\gamma)$. Let then $\Xi'_f$ be the union of all paths thus obtained: $(\ul\alpha_-\setminus\Xi)\cup\Xi'_f$ is again a representative of $\alpha_+$. Now the components $\zeta$ of $\Xi'_f\cap \ol{\mathcal X}_i^j$ behave again as in one of the bullets above, but the ones in the second bullet are entirely contained in $\bar\nei(\sigma''_r)$, and the ones in the first bullet may be isotoped, leaving their endpoints fixed, so that they end up entirely contained in the relevant $R_{e'_u}$, transversely to its ties.

This gives the announced, desired realization of $\Xi'$; and we have proved that $\alpha_+$ is carried by $\sigma''_r$ and, as it was anticipated above, this is enough to say that it is carried by $\delta'_r$, too: if $\ul\alpha_+$ traverses any branch of $\sigma''_r$ which is not found in $\delta'_r$, there is another branch or union of branches with the same endpoints, which is is subset of $\delta'_r$ instead: so the carried realization of $\alpha_+$ may be adjusted to traverse this other branch instead.

Note that $\delta'_r$ is recurrent, as each branch is traversed by either $\alpha_+$ or by an element of $\cc(\sigma'_r)$. Let $\eta$ be a generic almost track which is comb equivalent to $\delta'_r$: then $\alpha_+\in\cc(\eta)$ and there is a generic splitting sequence of twist nature which turns $\eta$ into $D_\gamma^{\epsilon(m-1)}(\eta)$; its rotation number is $m-1$ (see Lemma \ref{lem:functiongivestwist}, and point \ref{itm:rotofdehn} of Remark \ref{rmk:rotbasics}). Note that, by definition, $\alpha=D_\gamma^{\epsilon(m-1)}(\alpha_+)\in\cc\left(D_\gamma^{\epsilon(m-1)}(\eta)\right)$: so, by the first statement of the present lemma, $\alpha$ traverses each branch in $\eta.\gamma$ at least $m-1$ times. This property is not affected by comb equivalences: so the same is true for the carrying image of $\alpha$ in $\delta'_r$.
\end{proof}

Given an almost track $\rho$ (in $S^X$, say), we give a generalized version of the definitions specified in \S \ref{sub:diagext}. If $\alpha_1,\ldots,\alpha_s\in\cc(\rho)$ and $k\geq 1$, denote:
$$
\cc_k(\rho;\alpha_1,\ldots,\alpha_s)\coloneqq\left\{\alpha\in \cc(\rho)\,\left|\,\parbox{.5\textwidth}{$\alpha$ traverses at least $k$ times each branch contained in one of the $\rho.\alpha_j$}\right.\right\}.
$$
Denote subsequently $\ce_k(\rho;\alpha_1,\ldots,\alpha_s) \coloneqq \bigcup_{\delta\in \e(\rho)}  \cc_k(\delta;\alpha_1,\ldots,\alpha_s)$; and\linebreak $\cf_k(\rho;\alpha_1,\ldots,\alpha_s)\coloneqq \bigcup_\omega \ce_k(\omega;\alpha_1,\ldots,\alpha_s)$, where the union is performed over all $\omega$ subtracks of $\rho$ which fill $S^X$, and include $\rho.\alpha_j$ for all $j$.

\begin{lemma}\label{lem:tcboundsimplest}
In the setting of Proposition \ref{prp:tcbound}, suppose a subsequence $(\gamma_{t_j})_{j=1}^r$ of the sequence of curves $(\gamma_j)$ fills a subsurface $X\subset S$ which is homeomorphic to $S_{0,4}$ or $S_{1,1}$.

Fix $a_-\leq\min DI_{t_1}$, and $a_+\geq\max DI_{t_r}$ such that $V(\tau_{a_+}|X)\not=\emptyset$. Then, for any $\alpha_-\in V(\tau_{a_-}|X), \alpha_+\in V(\tau_{a_+}|X)$,
$$
d_X\left(\alpha_-,\alpha_+\right)\geq \lfloor (r-1)/3\rfloor.
$$
\end{lemma}
\begin{proof}
This proof is based on the ideas used to prove Theorem 1.3 in \cite{masurminskyq} (stated as Theorem \ref{thm:mm_cc_geodicity} in the present work). We will use the notation $\rho_j=\tau_j|X$ and, in order to avoid introducing further notation, we denote our $\gamma_{t_j}$'s simply as $\gamma_j$'s, forgetting about the other effective twist curves. We do the same with the notation for the intervals $DI_j$ and similar ones. For each $j$, let $m_j\coloneqq \rot(\gamma_j;DI_j)$. Also, for $1\leq j\leq r-1$ set $a_j\coloneqq\min DI_j$ and, asymmetrically, set also $a_r\coloneqq\max DI_{r-1}$.

Note that each pair of distinct curves $\gamma_j,\gamma_{j'}$ intersect and fill $X$ --- and $S^X$ --- because $X$ is a 4-holed sphere or a 1-holed torus. So, for all $1\leq j\leq r$, $\rho_{a_j}$ fills $S^X$ as it carries two of these curves.

Suppose $3\leq j \leq r$. Then $a_{j-2}<\max DI_{j-2}\leq a_{j-1}<\max DI_{j-1}\leq a_j$. Lemma \ref{lem:weightsaftertwist}, applied to $\bm\tau(a_{j-2},a_j)$ with respect to the twist curve $\gamma_{j-1}$, yields that $\gamma_j\in \cc(\rho_{\alpha_j})$ traverses at least $2\mathsf{K}_0+4$ times each branch of $\rho_{a_{j-2}}$ contained in $\rho_{a_{j-2}}.\gamma_{j-1}$. The same lemma, applied with respect to the twist curve $\gamma_{j-2}$, yields also that $\gamma_j$ traverses at least $2\mathsf{K}_0+4$ times each branch of $\rho_{a_{j-2}}$ contained in $\rho_{a_{j-2}}.\gamma_{j-2}$.

Let now $4 \leq j \leq r$, and fix $\alpha\in \cf(\rho_{a_j})$: we want to show that $\alpha\in \cf_3(\rho_{a_{j-3}})$. If $\alpha=\gamma_{j-1}$ a completely similar argument to the one just applied, considered for the sequence $\bm\tau(a_{j-3}, a_j)$ with respect to the two twist curves $\gamma_{j-3},\gamma_{j-2}$, yields $\gamma_{j-1}\in \cc_3(\rho_{a_{j-3}}; \gamma_{j-2},\gamma_{j-3})\subset\cf_3(\rho_{a_{j-3}})$. (Actually we might as well conclude that $\gamma_{j-1}\in \cc_{2\mathsf{K}_0+4}(\rho_{a_{j-3}}; \gamma_{j-2},\gamma_{j-3})$, but we do not need it; also in the following inclusions, we will only care about branches being traversed thrice.) The last inclusion is due to the fact that, as $\gamma_{j-2},\gamma_{j-3}$ fill $S^X$, also $\rho_{a_{j-3}}.\gamma_{j-2}\cup \rho_{a_{j-3}}.\gamma_{j-3}$ is an almost track filling $S^X$, and $\cc_3(\rho_{a_{j-3}}; \gamma_{j-2},\gamma_{j-3})\subseteq \ce_3(\rho_{a_{j-3}}.\gamma_{j-2}\cup \rho_{a_{j-3}}.\gamma_{j-3})$. 

If $\alpha$ is any other curve, it will intersect $\gamma_{j-1}$. The following chain of inclusions holds:
\begin{itemize}
\item $\alpha\in \cf_3(\rho_{a_{j-1}};\gamma_{j-1})$, by Lemma \ref{lem:weightsaftertwist} applied to the sequence $\bm\tau(a_{j-1},a_j)$ with respect to the twist curve $\gamma_{j-1}$.
\item $\cf_3(\rho_{a_{j-1}};\gamma_{j-1})\subset \cf_3(\rho_{a_{j-3}};\gamma_{j-1})$, because of Lemma \ref{lem:cf_decreasing}: first of all, clearly $\rho_{a_{j-1}}$ is carried by $\rho_{a_{j-3}}$ and, as noted above, both almost tracks fill $S^X$. If a curve $\beta$ is carried by $\delta'$, diagonal extension of $\sigma'$ which is a subtrack of $\rho_{a_{j-1}}$ which fills $S^X$ and contains $\rho_{a_{j-1}}.\gamma_{j-1}$, and $\beta$ traverses all branches in $\sigma'.\gamma_{j-1}$ at least thrice, then consider the almost tracks $\delta$ and $\sigma$ given by the last statement of said Lemma, with $\sigma$ a subtrack of $\rho_{a_{j-3}}$: necessarily $\sigma\supset \rho_{a_{j-3}}.\gamma_{j-1}$ and, by composition of carrying maps, $\beta$ in $\delta$ traverses each branch in $\rho_{a_{j-3}}.\gamma_{j-1}$ at least thrice.
\item $\cf_3(\rho_{a_{j-3}};\gamma_{j-1}) \subset\cf_3(\rho_{a_{j-3}};\gamma_{j-1},\gamma_{j-2},\gamma_{j-3})\subset \cf_3(\rho_{a_{j-3}})$: the first of these two inclusions is just due to the argument above (with translated indices) that $\rho_{a_{j-3}}.\gamma_{j-1}$ will traverse all branches carrying $\gamma_{j-2},\gamma_{j-3}$. The second one is due to the fact that $\cf_3(\rho_{a_{j-3}};\gamma_{j-1},\gamma_{j-2},\gamma_{j-3})=\ce_3(\rho_{a_{j-3}}.\gamma_{j-1}\cup \rho_{a_{j-3}}.\gamma_{j-2}\cup \rho_{a_{j-3}}.\gamma_{j-3})$.
\end{itemize}

Lemma \ref{lem:ccnesting}, together with this chain of inclusions, yields that $\nei_1(\cf(\rho_{a_j}))\subseteq \nei_1(\cf_3(\rho_{a_{j-3}}))\subseteq \cf(\rho_{a_{j-3}})$. And, nesting these last found inclusions for different values of $j$, for any pair of indices $1\leq i<i'\leq r$ such that $3|(i'-i)$ we get
$$
\mathcal N_{(i'-i)/3}\left(\cf(\rho_{a_{i'}})\right)\subseteq \cf(\rho_{a_i}).
$$

Denote $\hat r\coloneqq \lfloor(r-1)/3\rfloor$. If, for any $\alpha_+\in V(\rho_{a+})\subseteq \cf(\rho_{a_{3\hat r+1}}),\alpha_-\in V(\rho_{a-})$, we have $d_X(\alpha_-,\alpha_+)< \hat r$, then
$$\alpha_- \in \nei_{\hat r - 1}(\cf(\rho_{a_{3\hat r+1}}))\subseteq \cf(\rho_{a_4}) \subseteq \cf_3(\rho_{a_1}),$$
as a consequence of the inclusions proved above. Also, $\cf_3(\rho_{a_1})\subseteq \cf_3(\rho_{a_-})$, by an argument entirely similar to the inclusion shown in the second bullet above.

But $\alpha_- \in \cf_3(\rho_{a_-})$ cannot be true because of Lemma \ref{lem:vertexnotinterior}. So $d_X(\alpha_-,\alpha_+)\geq \hat r$, and this proves the claim.
\end{proof}

\begin{lemma}\label{lem:subsurfacesdontrepeat}
Let $\bm\tau=(\tau_j)_{j=0}^N$ be a splitting sequence of generic, recurrent almost tracks on a surface $S$, such that $\bm\tau(k,l)$ has twist nature about a curve $\gamma$, with $\rot_{\bm\tau}(\gamma;k,l)\geq 2\mathsf{K}_0+4$. Let $Y\subset S$ be a subsurface with $\gamma$ essentially intersecting $\partial Y$.

Let $0\leq a_-\leq k<l\leq a_+\leq N$, and let $\alpha_-\in W(\tau_{a-}),\alpha_+\in W(\tau_{a+})$ with the properties that
\begin{itemize}
\item $\alpha_-\subset Y$ (when choosing appropriate representatives in their isotopy classes);
\item $\alpha_+\cap Y\not=\emptyset$ (however the representative of $\alpha_+$ is chosen);
\item both $\alpha_-,\alpha_+$ intersect $\gamma$ essentially.
\end{itemize}

Then $\alpha_+$ intersects $\partial Y$ essentially.
\end{lemma}
\begin{proof}
Let $X$ be a regular neighbourhood of $\gamma$, let $m\coloneqq \rot_{\bm\tau}(\gamma;k,l)$ and let $c_-=\max\{\min I_\gamma,a_-\}$, $c_+=\min\{\max I_\gamma,a_+\}$. Then $c_-\leq k < l \leq c_+$.

As $\alpha_-\in W(\tau_{a-})$, then also $\pi_X(\alpha_-)\subseteq V(\tau_{a_-}^X)$: if any arc $\pi_X(\alpha_-)$ traverses a branch of $\tau_{a_-}^X$ twice in the same direction (this is the only way it may not be wide: see point \ref{itm:windaboutgamma} in Remark \ref{rmk:annulusinducedbasics}), then also $\alpha_-$ traverses a branch of $\tau_{a_-}$ twice in the same direction. Similarly, $\pi_X(\alpha_+)\subseteq V(\tau_{a_+}^X)$.%\subseteq \cc(\tau_l^X)$. Lemma \ref{lem:onerollingdirection} and point \ref{itm:rot_vs_dt_vertices} of Remark \ref{rmk:rotbasics} yield that $\cc(\tau_l^X)\subseteq D_{\gamma}^{\epsilon m}\cdot \cc(\tau_k^X)\subseteq D_{\gamma}^{\epsilon m}\cdot \cc(\tau_{a_-}^X)$.

From the statement 1 of Theorem \ref{thm:mmsstructure} we have $d_{\cc(X)}\left(\tau_{a_-}|X, \tau_{c_-}|X\right)\leq \mathsf{K}_0$ and $d_{\cc(X)}\left(V(\tau_{c_+}|X), \pi_X(\alpha_+)\right)\leq \mathsf{K}_0.$

Note that $\pi_X(\partial Y)\not=\emptyset$, and as $\partial Y,\alpha_-$ do not intersect, $d_X(\partial Y,\alpha_-)=1$. The triangle inequality holds for $d_{\cc(X)}$ even when its arguments are \emph{sets}, so\linebreak $d_{\cc(X)}\left(\pi_X(\partial Y), V(\tau_{c_-}|X)\right)\leq \mathsf{K}_0+1$. Subsequently,
\begin{eqnarray*}
 & d_X\left(\partial Y,\alpha_+\right) \geq d_{\cc(X)}\left(V(\tau_{c_-}|X),V(\tau_{c_+}|X)\right)- d_{\cc(X)}\left(\pi_X(\partial Y), V(\tau_{c_-}|X)\right) + & \\
 & - d_{\cc(X)}\left(V(\tau_{c_+}|X), \pi_X(\alpha_+)\right) \geq d_{\cc(X)}\left(V(\tau_{c_-}|X),V(\tau_{c_+}|X)\right)-2\mathsf{K}_0-1 & 
\end{eqnarray*}
again by the triangle inequality; and $d_{\cc(X)}\left(V(\tau_{c_-}|X),V(\tau_{c_+}|X)\right)\geq \rot_{\bm\tau}(\gamma;c_-,c_+)\geq m$ as seen in point \ref{itm:rot_vs_dist} of Remark \ref{rmk:rotbasics}.

Since $m\geq 2\mathsf{K}_0+4$, we have $d_X\left(\partial Y,\alpha_+\right)\geq 3$, meaning in particular that $\pi_X(\alpha_+)$ intersects $\pi_X(\partial Y)$. Thus also $\alpha_+$ intersects $\partial Y$ as required.
\end{proof}

\begin{coroll}\label{cor:subsurfacesdontrepeat}
In the setting of Proposition \ref{prp:tcbound}, let $1\leq t_1<t_2<t_3\leq q$ be three indices, and let $Y\subseteq S$ be a subsurface, with the following properties: $\gamma_{t_2}$ essentially intersects both $\gamma_{t_1},\gamma_{t_3}$; $\gamma_{t_1}\subset Y$, $\gamma_{t_2}$ essentially intersects $\partial Y$; and $\gamma_{t_3}$ cannot be realized disjointly from $Y$ as a curve on $S$.

Then $\gamma_{t_3}$ intersects $\partial Y$ essentially, too.
\end{coroll}
This statement may be read as follows: once a chain subsequence breaks out of a given subsurface $Y$, no subsequent entry will enter it again.
\begin{proof}
Apply the previous lemma with $\gamma=\gamma_{t_2}$,  $[k,l]=DI_{t_2}$, $a_-=\max DI_{t_1}$, $a_+=\min DI_{t_3}$, $\alpha_-=\gamma_{t_1}$, $\alpha_+=\gamma_{t_3}$. We have $\rot_{\bm\tau}(\gamma_{t_2};DI_{t_2})\geq 2\mathsf{K}_0+4$ by definition of Dehn interval.
\end{proof}

Proposition \ref{prp:tcbound} will be proved employing an auxiliary statement, by induction on the complexity $\xi(X')$ of subsurfaces $X'\subseteq X\subseteq S$. Its proof will be based on the ideas used to prove Theorem 1.3 in \cite{masurminskyq}.

Fix any chain subsequence $\bm\delta=(\delta_1,\ldots,\delta_r)$ of $(\gamma_1,\ldots,\gamma_q)$; let $X'\subseteq S$ be the subsurface which is filled by this chain subsequence. We say that $X'$ is \nw{$0$-good} with respect to the given chain. Note that, as $\bm\delta$ is a chain, $X'$ is connected.

We give a recursive definition of two functions $\tl,\tr:[1,r]\rightarrow [1,r]$, which are auxiliary to $\bm\delta$. The idea is that, for each index $1\leq i\leq r$, the following holds: $i\in [\tl i,\tr i]$; the curves in $\bm\delta$ indexed by $[\tl i,\tr i]$ fill a proper subsurface of $X'$; and it is impossible to pick a larger interval with the same property. However, we would also like that, every time two indices $i\not=i''$ have $[\tl i,\tr i]\not= [\tl i'',\tr i'']$ and these two intervals both give families of curves not filling the entire $X'$, there is an index $i< i'< i''$ such that the curves indexed by $[\tl i',\tr i']$ do fill $X'$ instead.

Start by setting $\tl 1\coloneqq 1$, and let $\tr 1$ be the highest index $L$ such that the curves $\delta_1,\ldots,\delta_L$ do not fill $X'$. Now, given an index $1<i_0\leq r$, suppose we have defined $\tl i,\tr i$ for all $i< i_0$: if $\tr(i_0-1)\geq i_0$, set $\tl i_0\coloneqq\tl(i_0-1), \tr i_0\coloneqq\tr(i_0-1)$. Else let $\tl i_0\coloneqq\tl(i_0-1)$ and $\tr i_0\coloneqq i_0$, so that the curves $\delta_{\tl i_0},\ldots,\delta_{\tr i_0}$ fill $X'$; furthermore, let $\tl(i_0+1)\coloneqq i_0+1$, and let $\tr(i_0+1)$ be the highest index $L$ such that $\delta_{i_0+1},\ldots,\delta_L$ do not fill $X'$.

For each $i$, let $\bm\delta(i)=(\delta_{\tl i},\ldots,\delta_{\tr i})$, and let $Y(i)$ be the subsurface of $X'$ filled by the curves in $\bm\delta(i)$. We call elements of the family $\{Y(i)|1\leq i\leq r\}\setminus\{X'\}$ the \nw{$1$-good} subsurfaces.

Again with a recursive definition, for $k>1$ we say that $Z\subseteq S$ is a \nw{$k$-good} subsurface with respect to $\bm\delta$ if there is an index $1\leq i\leq r$ such that $Y(i)\not=X'$ and $Z$ is a $(k-1)$-good subsurface of $Y(i)$, computed with respect to the chain $\bm\delta(i)$.

Also, a subsurface of $S$ is \nw{good} with respect to $\bm\delta$ if it is $k$-good for some $k\in\mathbb N$. We prove a fact about good subsurfaces:

\begin{claim}
Given $Z\subsetneq X'$ a good subsurface (with respect to $\bm\delta$), there is at most one $1$-good subsurface $Y$ such that $Z\subseteq Y$. Moreover two subsurfaces $Y(i),Y(i')$ which are $\subsetneq X'$, but such that $Y(i'')=X'$ for some $i<i''<i'$, are certainly distinct.
\end{claim}

\begin{proof}
We prove the two claims together, with an an argument by contradiction that works for both. Supposing that either of the two is false may be restated as the following hypothesis: there is a good subsurface $Z$ (for the second statement, $Z\coloneqq Y(i)=Y(i')$), contained in two $1$-good subsurfaces $Y(i),Y(i')$, not necessarily distinct but with an index $i<i''<i'$ with $Y(i'')=X'$. 

This means that $Z$, when appearing as a good subsurface of $Y(i)$ with respect to the chain $\bm\delta(i)$, gets filled by some curves $\delta_{\iota_-},\ldots,\delta_{\iota_+}$, for $\tl i\leq \iota_-< \iota_+ \leq \tr i$. Similarly with respect to the index $i'$: $Z$ is filled by $\delta_{\iota'_-},\ldots,\delta_{\iota'_+}$, for $\tl i'\leq \iota'_-< \iota'_+ \leq \tr i'$.

But, by definition, $\{\delta_{\iota_-},\ldots,\delta_{\iota'_+}\}\supseteq \{\delta_{\tl i''},\ldots,\delta_{\tr i''}\}$: so these families both fill $X'$. Therefore there is some $\delta_{\iota(2)}$ (for $\iota_+< \iota(2)< \iota'_-$) intersecting $\partial Z$; moreover there are $\iota_-\leq\iota(1)\leq\iota_+$, $\iota'_-\leq\iota(3)\leq\iota'_+$ such that both $\delta_{\iota(1)}$ and $\delta_{\iota(3)}$ intersect $\delta_{\iota(2)}$: this is because we are picking from two families of curves that fill $Z$. Since the curves $\delta_{\iota(1)}, \delta_{\iota(2)}, \delta_{\iota(3)}$, in the given order, are taken from the sequence $(\gamma_1,\ldots,\gamma_q)$ with respect to which $\bm\tau$ is arranged, this contradicts Corollary \ref{cor:subsurfacesdontrepeat} above.
\end{proof}

The definition of $\tl$, $\tr$ and the second statement of this claim together imply what follows. Let $1<x_1<\ldots<x_\eta\leq r$ are the indices $i$ such that $Y(i)=X'$ (in particular, for these indices, $\tr i=i$); note that this collection includes no two consecutive values of $i$. If an interval $I$ of consecutive indices $i$ does not include any of these ones, then the corresponding $Y(i)$ are always the same surface. If $1\leq i< x_j < i' \leq r$ for some $1\leq j\leq\eta$, and $Y(i),Y(i')\not=X'$, then $Y(i)\not= Y(i')$. So the number $\eta'$ of $1$-good subsurfaces of $X'$ equals either $\eta$ or $\eta+1$. The quantities $\eta,\eta'$ will be recalled below.

Let $Y_1,\ldots,Y_{\eta'}$ be an enumeration of these subsurfaces, in the same order as the sequence $\left(Y(i)\right)_i$ but avoiding repetitions and occurrences of $X'$. And, for each $1\leq u\leq \eta'$, if $Y_u=Y(i)$, let $\bm\delta_u=\bm\delta(i)$ be the chain sequence, subsequence of $\bm\delta$, certifying that $Y_u$ is a good subsurface. This is independent of the choice of $i$.

Now define, for $0\leq j,j'\leq N$,
$$d''_{\bm\delta}(\tau_j,\tau_{j'})\coloneqq \sum_{Y \subseteq X' \text{ good w.r.t. } \bm\delta} [d_Y(\tau_j,\tau_{j'})]_M.$$

Clearly, $d''_{\bm\delta}(\tau_j,\tau_{j'})\leq d'_{\pa(X')}(\tau_j,\tau_{j'})$. We prove the following claim: 
\begin{claim}
Given the arranged sequence $\bm\tau=(\tau_j)_{j=0}^N$ as in the statement of Proposition \ref{prp:tcbound}, let $\bm\delta=(\delta_1\coloneqq\gamma_{t_1},\ldots,\delta_r\coloneqq\gamma_{t_r})$ be a chain subsequence of the $(\gamma_t)_{t=1}^q$ ($t_1<\ldots<t_r$), and call $X' \subseteq S$ ($X'=S$ is allowed here) the essential, non-annular subsurface filled by the curves in $\bm\delta$. Let $a_-\leq \min DI_{t_1}, a_+\geq \max DI_{t_r}$ be two indices along the sequence $\bm\tau$, with $\pi_{X'}\left(V(\tau_{a_+})\right) \in \pa^0(X')$.

Then there are two constants $c_3, c_4$, only depending on the topological types of $S$ and $X'$, such that $r\leq c_3 d''_{\bm\delta}(\tau_{a_-},\tau_{a_+})+c_4$ (which is in turn $\leq c_3 d'_{\pa(X')}(\tau_{a_-},\tau_{a_+})+c_4$).
\end{claim}
\begin{proof}\label{prf:tcbound}

\step{1} Set-up of the induction on $\xi(X')$.

As previously noted (Lemma \ref{lem:decreasingfilling} and subsequent observations), the condition $\pi_{X'}\left(V(\tau_{a_+})\right) \in \pa^0(X')$ implies that, for all $a_-\leq j\leq a_+$ and all subsurfaces $Y\subseteq X'$, $\pi_Y\left(V(\tau_j)\right) \in \pa^0(Y)$.

For ease of notation, for $i=1,\ldots,r$ let $a_i \coloneqq\min DI_{t_i}$, %$a_{i+}\coloneqq\max DI_{t_i}$;
and $m_i\coloneqq\rot_{\bm\tau}(\delta_i,DI_{t_i})$. Let moreover $\rho_j\coloneqq \tau_j|X'$ for all $j$.

The induction basis is for $X'\cong S_{0,4}$ or $\cong S_{1,1}$. In that case 
$$d''_{\pa(X')}\left(\tau_{a_-},\tau_{a_+}\right)=d_{X'}\left(V(\tau_{a_-}),V(\tau_{a_+})\right)\geq d_{\cc(X')}\left(V(\rho_{a_-}),V(\rho_{a_+})\right)-2\tilde n_1$$
where $\tilde n_1\coloneqq 2F(8 N_1(S))$ (see Lemma \ref{lem:induction_vertices_commute}). Lemma \ref{lem:tcboundsimplest}, applied with reference to the family $\delta_1,\ldots,\delta_r$, yields that $d_{\cc(X')}\left(V(\rho_{a_-}),V(\rho_{a_+})\right) \geq r/3-1$. So the claim is proved with $c_3(X')= 3, c_4(X') = 3(1+2\tilde n_1)$.

We get now to the induction step, i.e. prove that the statement holds for $X'$, provided that $X'$ is not homeomorphic to $S_{0,4}$ or $S_{1,1}$ and that the statement holds for any essential, non-annular $Y\subsetneq X'$.

\step{2} For any $\alpha_-\in V(\tau_{a_-})$ and $\alpha_+\in V(\tau_{a_+})$, $d_{\cc({X'})}(\alpha_-,\alpha_+)\geq \lfloor (\eta-1)/2\rfloor$.

For $1\leq j \leq \eta$ consider the previously defined $x_j$ and let $\sigma_j\coloneqq\rho_{a_{\tl x_j}}$. As already noted, for any fixed $1\leq j \leq \eta$, the chain sequence $\bm\delta(x_j)= \{\delta_{\tl x_j},\ldots,\delta_{x_j}\}$ fills $X'$. As all curves in each sequence $\bm\delta(x_j)$ are carried by the respective $\sigma_j$, each $\sigma_j$ fills $X'$; and, by Lemma \ref{lem:cf_decreasing}, $\cf(\sigma_j)\supseteq \cf(\sigma_{j+1})$.

If a curve $\delta_k$ belongs to the sequence $\bm\delta(x_{j+1})$ (for $1\leq j<r$), then it belongs to the set $\cc_1\left(\sigma_j;\bm\delta(x_j)\right)$: this is because, for each $i \in [(\tl x_j)+1, k]$, $\delta_i\in \cc(\rho_{a_i})$ intersects the previous $\delta_{i-1}$; so, according to the first statement of Lemma \ref{lem:weightsaftertwist} applied to the sequence $\bm\tau(a_{\tl x_j}, a_i)$, we have $\sigma_j.\delta_i\supseteq \sigma_j.\delta_{i-1}$. Therefore $\sigma_j.\delta_k$ includes all $\sigma_j.\delta$ for $\delta\in \bm\delta(x_j)$, and this is exactly what has just been claimed.

Now, for $1\leq j\leq \eta-2$ consider any curve $\alpha\in \cf(\sigma_{j+2})$: $\alpha$ will intersect some $\delta\in \bm\delta(x_{j+1})$ because that collection fills $X'$. So Lemma \ref{lem:weightsaftertwist}, applied along the sequence $\bm\tau(a_{\tl x_j},a_{\tl x_{j+2}})$ with respect to the twist curve $\delta$, yields that there is a diagonal extension $\omega_j$ of a subtrack of $\sigma_j$ which fills $X'$, where $\alpha$ is carried with $\omega_j.\alpha\supseteq \omega_j.\delta=\sigma_j.\delta$; and, as $\sigma_j.\delta$ contains all the carrying images in $\sigma_j$ of the curves in $\bm\delta(x_j)$, it is a subtrack filling $X'$: $\alpha\in \cf_1(\sigma_j)$.

Lemma \ref{lem:ccnesting} gives then $\cf(\sigma_{j+2})\subseteq \cf_1(\sigma_j)\subseteq \nei_1\left(\cf(\sigma_j)\right)$. Now the argument goes as in Lemma \ref{lem:tcboundsimplest}. Nesting these inclusions, for all pairs of indices $1\leq j<j'\leq \eta$ such that $2|(j'-j)$, we get
$$
\mathcal N_{(j'-j)/2}(\cf(\sigma_{j'}))\subseteq \cf(\sigma_j).
$$

Denote $\hat \eta\coloneqq \lfloor(\eta-1)/2\rfloor$. If, for any $\alpha_+\in V(\rho_{a_+})\subseteq \cf(\sigma_{a_{2\hat \eta}+1})$, $\alpha_-\in V(\rho_{a_-})$, we have $d_{X'}(\alpha_-,\alpha_+)< \hat\eta$, then
$$\alpha_- \in \mathcal N_{\hat\eta - 1}(\cf(\sigma_{a_{2\hat \eta}+1}))\subseteq \cf(\sigma_{3}) \subseteq \cf_1(\sigma_1),$$
as a consequence of the inclusions proved above. Also, $\cf_1(\sigma_1)\subseteq \cf_1(\rho_{a_-})$, due to an argument already employed, based on the last statement of Lemma \ref{lem:cf_decreasing}: fix any $\xi\in \cf_1(\sigma_1)$. Then there is a $\omega_1\in \e(\sigma'_1)\subseteq \f(\sigma_1)$, where $\sigma'_1$ is a subtrack of $\sigma_1$ that fills $X'$; and $\xi$ is carried by $\omega_1$ and traverses all branches of $\sigma'_1$. By said Lemma, there is a $\omega_-\in \e(\rho_{a_-}')\subseteq \f(\rho_{a_-})$, with $\rho_{a_-}'$ a subtrack of $\rho_{a_-}$ filling $X'$, such that $\omega_-$ fully carries $\omega_1$, which carries $\xi$; moreover $\rho_{a_-}'$ fully carries $\sigma'_1$ so, by composition of carrying maps, $\xi$ will traverse in $\omega_-$ all branches belonging to $\rho_{a_-}'$.

But $\alpha_- \in \cf_1(\rho_{a_-})$ cannot be true because of Lemma \ref{lem:vertexnotinterior}. So $d_{X'}(\alpha_-,\alpha_+)\geq \hat \eta$, proving the claim of this step.

\step{3} Proof of the key claim for Proposition \ref{prp:tcbound}.

Write
$$d''_{\bm\delta}(\tau_{a_-},\tau_{a_+})= [d_{X'}(\tau_{a_-},\tau_{a_+})]_M + \sum_{\substack{Y\subset {X'}\text{ }k \text{-good w.r.t. }\bm\delta \\ (k\geq 1)}} [d_Y(\tau_{a_-},\tau_{a_+})]_M.$$

The fact proved above the beginning of the present proof implies that the set of the $k$-good subsurfaces for ${X'}$, $k\geq 1$, can be partitioned into families: one for each $1\leq u\leq \eta'$, consisting of the good subsurfaces of $Y_u$ which are good with respect to $\bm\delta_u$. So the above summation can be split accordingly. As for the first term instead, from Step 2 we get that $[d_{X'}(\tau_{a_-},\tau_{a_+})]_M\geq d_{X'}(\tau_{a_-},\tau_{a_+}) - M \geq \eta/2 - 1 -M$. So
\begin{eqnarray*}
 & d''_{\bm\delta}(\tau_{a_-},\tau_{a_+}) \geq \left(\sum_{u=1}^{\eta'} d''_{\bm\delta_u}(\tau_{a_-},\tau_{a_+})\right) + \eta/2 - 1 -M & \\
 & \geq \sum_{u=1}^{\eta'} \left(1/2 + d''_{\bm\delta_u}(\tau_{a_-},\tau_{a_+})\right) - 3/2 -M & 
\end{eqnarray*}
(where we are also using the fact that $\eta\leq \eta'\leq \eta+1$). Now we apply the induction hypothesis and get that, if $r(u)$ is the length of the chain $\bm\delta_u$, the last expression is
$$
\geq \sum_{u=1}^{\eta'} \left(1/2 + [r(u)/\hat c_3 - \hat c_4]_0\right) - 3/2 -M,
$$
where $\hat c_3, \hat c_4$ are upper bounds for the constants $c_3(Y), c_4(Y)$ over all topological types of subsurfaces $Y$ of ${X'}$; and the notation $[\cdot]_0$ indicates that we consider this summand only if it is positive. There will be a constant $c'$ (depending on ${X'}$) such that $1/2 + [r(u)/\hat c_3 - \hat c_4]_0\geq 1/3 +c' r(u)$, so the expression is 
$$
\geq \eta'/3 + c'\sum_{u=1}^{\eta'} r(u) - 3/2 -M \geq \min\{1/3,c'\}\left(\eta'+\sum_{u=1}^{\eta'} r(u)\right) -3/2 -M \geq r -3/2 -M.
$$

The last inequality is due to the following argument: the sequence $\bm\delta$ consists of the junction of the sequences, and elements, $\bm\delta_1,\delta_{x_1},\bm\delta_2,\ldots,\bm\delta_\eta,\delta_{x_\eta},[\bm\delta_{\eta+1}]$ (the last sequence may not exist). So its length is $r=\eta+\sum_{u=1}^{\eta'} r(u)$.

This concludes the proof of the key claim.
\end{proof}

We now prove Proposition \ref{prp:tcbound}. Let $\bm\delta=(\delta_1\coloneqq\gamma_{t_1},\ldots,\delta_r\coloneqq\gamma_{t_r})$ be a chain subsequence \emph{with maximal length} of $\gamma_1,\ldots,\gamma_q$ as defined in the statement. Let $r$ be the length of this chain, and let $X'$ be the subsurface of $X$ (and $S$) filled by the curves in $\bm\delta$. Then
$$r\leq c_3(X') d'_{\pa(X')}(\tau_k,\tau_l)+c_4(X')\leq c_3(X') d'_{\pa(X)}(\tau_k,\tau_l)+c_4(X')$$
($d'_{\pa(X)}$ is indeed a summation involving all summands already present in $d'_{\pa(X')}$). Lemma \ref{lem:chainbound} gives then
$$
q\leq \xi(X)r \leq \xi(X)\left( c_3(X') d'_{\pa(X)}(\tau_k,\tau_l)+c_4(X') \right).
$$

Let now $\tilde c_3(S)\coloneqq \max_{Y\subseteq S} c_3(Y)$, $\tilde c_4(S)\coloneqq \max_{Y\subseteq S} c_4(Y)$. Since\linebreak $\xi(S) = \max_{Y\subseteq S} \xi(Y)$ we have
$$
q \leq \xi(S)\left( \tilde c_3(S) d'_{\pa(X)}(\tau_k,\tau_l)+\tilde c_4(S) \right).
$$
which defines the required constants $C_3(S),C_4(S)$. This completes the proof.
\section{Untwisted sequences. Proof of the main statement}\label{sec:traintrackconclusion}

\ul{Note:} Consistently with \S \ref{sec:twistcurves}, we work in the setting of \emph{generic} train tracks only. This is not restrictive, however, since a semigeneric splitting sequence can always be converted to a generic one, to which the main result (Theorem \ref{thm:main_full}) applies.

In this section we will explain how to deprive a splitting sequence of a high number of Dehn twist, and then we show that this sequence is suitable for application of the same techniques of proof of quasi-geodicity as in \cite{mms}, Theorem 6.12.

\subsection{The untwisted sequence}\label{sub:untwistedsequence}
\begin{defin}\label{def:untwistedsequence}
Let $S$ be a surface, and let $\bm\tau=(\tau_j)_{j=0}^N$ be a generic splitting sequence of train tracks evolving firmly in a subsurface $S'$, not necessarily connected. Let $\gamma_1,\ldots,\gamma_r\subseteq \cc(\tau_0)$ be curves all contained in $S'$, and suppose that $\bm\tau$ is $(\gamma_1,\ldots,\gamma_r)$-arranged. In particular the sequence $(\max DI_t)_{t=1}^r$ is increasing.

Let $m_t\coloneqq \rot_{\bm\tau}(\gamma_t;DI_t)\geq 2\mathsf{K}_0+4$, let $\epsilon_t$ be the sign of $\gamma_t$ as a twist curve in $\bm\tau$, and let $g_t\coloneqq \max DI_t(0)-\min DI_t(0)$ be the `period length' in the sequence $DI_t$.

Define recursively $\phi_0\coloneqq \mathrm{id}_S:S\rightarrow S$; and, for $1\leq t \leq r$, $\phi_t\coloneqq D_{\gamma_t}^{\epsilon_t(2\mathsf{K}_0+4-m_t)}\circ \phi_{t-1}$.

Let $NI_0\coloneqq[0,\min DI_1]$; $NI_t\coloneqq [\max DI_t,\min DI_{t+1}]$ for $1\leq t\leq r-1$; $NI_r\coloneqq [\max DI_r,N]$; $DK_t\coloneqq DI_t(0)\cup \ldots \cup DI_t(2\mathsf{K}_0+3)$ and $DL_t\coloneqq DI_t(m_t-2\mathsf{K}_0-4)\cup \ldots \cup DI_t(m_t-1)$ for $1\leq t\leq r$. Define the \nw{untwisting} of $\bm\tau$ as
\begin{align*}
\utw\bm\tau \coloneqq\  & \bm\tau(NI_0)*\\
 &  * \left(\phi_1\cdot \bm\tau(DL_1)\right) * \left(\phi_1\cdot\bm\tau(NI_1)\right) *\\
 & \ldots  \\
 & * \left(\phi_{r-1}\cdot\bm\tau(DL_{r-1})\right) * \left(\phi_{r-1}\cdot\bm\tau(NI_{r-1})\right) * \\
 & * \left(\phi_r\cdot\bm\tau(DL_r)\right) * \left(\phi_r\cdot\bm\tau(NI_r)\right).
\end{align*}

In the above notation, $\phi_t\cdot \bm\tau(\ldots)$ means the splitting sequence obtained from $\bm\tau(\ldots)$ via application of $\phi_t$ to all its entries.
\end{defin}

If $\utw\bm\tau=(\utw\tau_j)_{j=0}^{N'}$, the above definition provides a natural subdivision of $[0,N']$ into subintervals
$$
NI_0^\utw, DI_1^\utw,NI_2^\utw,DI_2^\utw,\ldots,NI_{r-1}^\utw,DI_r^\utw, NI_r^\utw
$$
where the maximum of each subinterval is the minimum of the following one. Each interval $DI_t^\utw$ has twist nature with respect to $\phi_t(\gamma_t)$, and may be subdivided into $DI_t^\utw(0),\ldots, DI_t^\utw(2\mathsf{K}_0+3)$. For all $1\leq t\leq r+1$, there is a natural bijection between $NI_t$ and $NI_t^\utw$; and, if $t\not=0$, between $DI_t(m_t-2\mathsf{K}_0-4+s)$ and $DI_t^\utw(s)$ for all $0\leq s \leq 2\mathsf{K}_0+3$. In order to keep the employed notation simple, let $DI_0=DI_0^\utw=DK_0=DL_0=\{0\}$. If $j\in NI_t$ (resp. $DL_t$) let $\dn j$ be the corresponding index in $NI_t^\utw$ (resp. $DI_t^\utw$). Extend $\dn:[0,N]\rightarrow [0,N']$ in the only way that gives a monotonic map. For $j'\in [0,N']$ let $\up j'$ be the least of the indices $j\in[0,N]$ such that $\dn j=j'$. Note that there is always a $t$ such that $\up j'\in DL_t\cup NI_t$.

For $j\in [0,N]$, let $t(j)$ be the least index $t$ such that $j\in DI_t$ or $j\in NI_t$. For $j'\in [0,N']$, let $\utw t(j')$ be the least index $t$ such that $j\in DI_t^\utw$ or $j\in NI_t^\utw$. Then, for all $j\in \bigcup_{t=0}^{r} (DL_t\cup NI_t)$, $\utw\tau_{\dn j}=\phi_{t(j)}(\tau_j)$; and for all $j'\in[0,N']$, $\utw\tau_{j'}=\phi_{\utw t(j')}(\tau_{\up j'})$.

Anyway note that each interval $DI^\utw_t$ can be made to correspond not only to $DL_t$, but to $DK_t$ as well: for all $1\leq t \leq r$, if $j\in DK_t$, and $j'\coloneqq \min DI_t^\utw + (j - \min DI_t)$, then $\utw\tau_{j'}=\phi_{t-1}(\tau_j)$. This means that, for all $0\leq t \leq r-1$,
$$
\utw\bm\tau(DI_t^\utw\cup NI_t^\utw\cup DI_{t+1}^\utw) = \phi_t\cdot  \bm\tau\left(DL_t\cup NI_t\cup DK_{t+1}\right).
$$

Also, for $0\leq t\leq r-1$, let $[t]\dn:[0,N]\rightarrow [0,N']$ be defined as a correspondence that exploits this identity: $[t]\dn j\coloneqq j-\min DK_{t+1}+\min DK_{t+1}^\utw$ if $j\in DK_{t+1}$; $[t]\dn j\coloneqq \max DI_{t+1}^\utw$ if $j\in DI_{t+1}\setminus DK_{t+1}$; and $[t]\dn j\coloneqq \dn j$ otherwise. With this correspondence, $\utw\tau_{[t]\dn j}=\phi_t(\tau_j)$ for all $j\in DL_t\cup NI_t\cup DK_{t+1}$.

Define also $[t]\up:[0,N']\rightarrow [0,N]$ by setting $[t]\up j'$ to be the least $j\in[0,N]$ such that $[t]\dn j=j'$. Define $[r]\dn\coloneqq \dn$ and $[r]\up\coloneqq\up$.

For $X\subseteq S$ a subsurface, denote $I_X^\utw$ the accessible interval of $X$ in $\utw\bm\tau$.

\ul{Note:} most of the time in this section we will deal with an arranged splitting sequence $\bm\tau$ and its respective untwisted sequence $\utw\bm\tau$, as above. In order to simplify notations for distances along these two splitting sequences we will adopt similar ones as in \cite{mms}.
\begin{itemize}
\item If $i,j\in [0,N]$ and $Y\subseteq S$ is a subsurface, $d_Y(i,j)\coloneqq d_Y(\tau_i,\tau_j)$, and similarly for $d_{\pa(Y)}$, $d'_{\pa(Y)}$ and other distances in graphs. If $I\subseteq [0,N]$ is a subinterval, $d_Y(I)\coloneqq d_Y(\min I, \max I)$ etc.
\item If $i,j\in [0,N']$ and $X\subseteq S$ is a subsurface, $d_X(i,j)^\utw\coloneqq d_X(\utw\tau_i,\utw\tau_j)$, and similarly for $d_{\pa(X)}(\ldots,\ldots)^\utw$, $d'_{\pa(X)}(\ldots,\ldots)^\utw$ and other distances in graphs. If $I\subseteq [0,N]$ is a subinterval, $d_Y(I)^\utw\coloneqq d_Y(\min I, \max I)^\utw$ etc.
\end{itemize}

In a bit we will need a version of untwisting for train track splitting sequences which \emph{do not} evolve firmly in any subsurface of $X$. Let $\bm\tau=(\tau_j)_{j=0}^N$ be a generic splitting sequence of cornered birecurrent train tracks: $\bm\tau$ can be seen as a concatenation $\bm\tau^1*\bm\epsilon^2*\bm\tau^2*\ldots*\bm\epsilon^w*\bm\tau^w$ where each $\bm\tau^i$ evolves firmly in a fixed subsurface of $S$ and each $\bm\epsilon^u$ is a single split, say from a track $\tau_j$ to $\tau_{j+1}$, such that $V(\tau_j)$ fills a surface strictly and essentially containing the one filled by $\tau_{j+1}$ --- thus the complexity of the former is higher than the one of the latter. This decomposition is possible because, as it has been pointed out after Lemma \ref{lem:decreasingfilling}, the subsurfaces filled by $V(\tau_j)$, for $0\leq j\leq N$, are a decreasing family with respect to the inclusion. Moreover, the number $w\leq \xi(S)$; and the single split in each $\bm\epsilon_u$ may induce at most $d_{\pa(X)}(\rar\tau_j,\rar\tau_{j+1})\leq 1$.

\begin{defin}\label{def:not_firmly}
For $\bm\tau$ as above, we define
$$
\rar\bm\tau\coloneqq \rar\bm\tau^1*\bm\epsilon^1*\rar\bm\tau^2*\ldots*\bm\epsilon^{w-1}*\rar\bm\tau^w.
$$

Suppose now that each $\bm\tau^u$ in the subdivision above is $(\gamma_1^u,\ldots,\gamma_{r(u)}^u)$-arranged for a suitable family of curves. The construction of each $\utw\bm\tau^u$ as in Definition \ref{def:untwistedsequence} above would give, in particular, a diffeomorphism $\phi_{r(u)}^u:S\rightarrow S$ (called simply $\phi_r$ there); and we denote here $\psi_{u+1}\coloneqq \phi_{r(u)}^u\circ\ldots\circ \phi_{r(1)}^1$ for all $1\leq u \leq w-1$; while $\psi_1\coloneqq \mathrm{id}_S$. Let then
\begin{align*}
\utw\bm\tau \coloneqq\  & \utw\bm\tau^1*\\
 &  * \left(\psi_2\cdot \bm\epsilon^2\right) * \left(\psi_2\cdot\utw\bm\tau^2\right) *\\
 & \ldots  \\
 & * \left(\psi_{w-1}\cdot \bm\epsilon^{w-1}\right) * \left(\psi_{w-1}\cdot\utw\bm\tau^{w-1}\right) * \\
 & * \left(\psi_w \cdot \bm\epsilon^w\right) * \left(\psi_w\cdot\utw\bm\tau^w\right).
\end{align*}
\end{defin}

\begin{lemma}[Unbroken accessible intervals]\label{lem:untwistedsubsurfaces}
Let $\bm\tau=(\tau_j)_{j=0}^N$ be a $(\gamma_1,\ldots,\gamma_r)$-arranged train track splitting sequence, evolving firmly in a subsurface $S'$, not necessarily connected. Let $X\subseteq S'$ be a subsurface of $S$. Then the following properties hold.
\begin{enumerate}
\item If $\gamma_t$ cuts $\partial X$, then $DI_t\not\subseteq I_X$; more precisely, $DI_t\cap I_X$ contains at most $2g_t$ indices;
\item If $\gamma_t$ does not cut $\partial X$, then either $DI_t\subseteq I_X$ or $DI_t\cap I_X=\emptyset$.
\item Let $t_-\coloneqq t(\min I_X)$; let $t_+$ be the highest $t$ such that $\max I_X\geq \max DI_t$, and suppose that there is no $t$ with $I_X\subsetneq DI_t$. Then, for all $t_-\leq t\leq t_+$, $\phi_t^{-1}\circ \phi_{t_-}$ fixes $X$  and each component of $\partial X$, up to isotopies of $S$.
\item In the setting specified above, $[t_+]\dn I_X=I_{\phi_{t_-}(X)}^\utw$.
\item Let $\utw t_- \coloneqq \utw t(\min I_X^\utw)$; let $\utw t_+$ be the highest $t$ such that $\max I_X^\utw \geq \max DI_t^\utw$, and suppose that there is no $t$ with $I_X^\utw\subsetneq DI_t^\utw$. Then
$$I_{\phi_{\utw t_-}^{-1}(X)}= \left[[\utw t_+]\up \min I_X^\utw, \max DI_{\utw t_+}\right]
\text{ or }
I_{\phi_{\utw t_-}^{-1}(X)}= \left[[\utw t_+]\up \min I_X^\utw,[\utw t_+]\up \max I_X^\utw\right]
$$ depending on whether $\max I_X^\utw=\max DI_{\utw t_+}^\utw$ or not.
\item If $X$ is not an annulus, $\gamma_t$ does not intersect $X$ essentially, and $j, j+g_t\in DI_t\subseteq I_X$, then $\tau_j|X=\tau_{j+g_t}|X$ (up to isotopy).

\item If $\bigcup_{t=t_0}^{t_1} (DI_t\cup NI_t)\subseteq (DI_{t_0}\cup NI_{t_0})\cup I_X$, and $\gamma_t$ does not intersect $X$ for $t_0+1\leq t \leq t_1$, then all the respective maps $\phi_t\circ \phi_{t_0}^{-1}$ have their restriction to $\phi_{t_0}(X)$ isotopic to the inclusion $\phi_{t_0}(X)\hookrightarrow S$; and all the $\hat\phi_t: S^X\rightarrow S^{\phi_t(X)}=S^{\phi_{t_0}(X)}$, lift of the respective $\phi_t$, are isotopic to $\hat\phi_{t_0}$.

As a consequence, $\utw\tau_{\dn j}|\phi_{t_0}(X)=\hat\phi_{t_0}(\tau_j|X)$ for all $j\in \bigcup_{t=t_0}^{t_1} (DI_t\cup NI_t)$. If, in addition, there is an interval $I\subseteq I_X$ such that $I\subseteq \left(\bigcup_{t=t_0}^{t_1} (DI_t\cup NI_t)\right)\cup DK_{t_1+1}$ (setting $DK_{r+1}=\{N\}$ for simplicity), then $\utw\tau_{[t_1]\dn j}|\phi_{t_0}(X)=\hat\phi_{t_0}(\tau_j|X)$ for all $j\in I$.
\end{enumerate}
\end{lemma}
\begin{proof}
When $X$ is an annulus, claim 1 is a direct consequence of Lemma \ref{lem:smallinterference}, according to which one must have $\rot_{\bm\tau}(\gamma_t; DI_t\cap I_X)=0$. For $X$ not an annulus, suppose for a contradiction that $\#\left(DI_t\cap I_X\right)\geq 2g_t+1$, which means that there are two indices $k,l\in DI_t\cap I_X$ such that $\tau_l=D_{\gamma_t}^{2\epsilon_t}(\tau_k)$. Since $k,l\in I_X$, $\cc(\tau_k|X),\cc(\tau_l|X)$ both have diameter $\geq 3$ in $\cc(X)$, so they fill $X$, and both sets must include a curve which essentially intersects $\gamma_t$ in $S$, as $\gamma_t$ essentially intersects $X$: i.e. $\pi_{\nei_t}\cc(\tau_k|X),\pi_{\nei_t}\cc(\tau_l|X)\not=\emptyset$ where $\nei_t$ is a regular neighbourhood of $\gamma_t$.

From Theorem \ref{thm:mmsstructure} and the subsequent remark, we know that any efficient position for $\partial X$ in $\tau_k$ or $\tau_l$ is wide. Therefore, when lifting an efficiently positioned $\partial X$ to $S^{\nei_t}$, it will not traverse the same branch of $\tau_k^{\nei_t}$, or of $\tau_l^{\nei_t}$, twice in the same verse.

We claim that $\pi_{\nei_t}\cc(\tau_k|X)\subseteq V(\tau_k^{\nei_t})\cup D_{\nei_t}^{\epsilon_t}\cdot V(\tau_k^{\nei_t})$, for $D_{\nei_t}$ denoting the Dehn twist in $S^{\nei_t}$ about its core. It is certainly true, to start with, that $\pi_{\nei_t}\cc(\tau_k|X)\subseteq \cc(\tau_k^{\nei_t})$; if there is an $\alpha\in \pi_{\nei_t}\cc(\tau_k|X) \cap \left(D_{\nei_t}^{\epsilon_t i}\cdot V(\tau_k^{\nei_t})\right)$ for $i\not=0,1$, then $i>1$ by Lemma \ref{lem:onerollingdirection}, implying that $\alpha$, when embedded in $\tau_k^{\nei_t}$, traverses thrice a branch $b$ contained in $\tau_k^{\nei_t}.\gamma_t$ (see point \ref{itm:hl_vs_multiplicity} after Definition \ref{def:horizontallength}). Now, $\alpha$ shall be essentially disjoint from all components of $\pi_{\nei_t}(\partial X)$: but then, any component of $\pi_{\nei_t}(\partial X)$, assuming $\partial X$ in efficient position, shall traverse $b$ at least twice; and this contradicts the fact that it is wide.

However, $\cc(\tau_l^{\nei_t}) \subseteq D_{\nei_t}^{2\epsilon_t} \cdot \cc(\tau_k^{\nei_t})= \bigcup_{i=2}^\infty D_{\nei_t}^{\epsilon_t i}\cdot V(\tau_k^{\nei_t})$ and this, together with the fact just proved, implies in particular $\pi_{\nei_t}\cc(\tau_k|X)\cap \pi_{\nei_t}\cc(\tau_l|X)=\emptyset$. On the other hand $\pi_{\nei_t}\cc(\tau_l|X)\subseteq \pi_{\nei_t}\cc(\tau_k|X)$ because $\tau_k|X$ carries $\tau_l|X$, and they are both nonempty. This is a contradiction.

As for claim 2: let now $[k,l]\coloneqq DI_t$, so that $\tau_l=D_{\gamma_t}^{\epsilon_t m_t}(\tau_k)$. There is a lift $\hat D: S^X\rightarrow S^X$ of $D_{\gamma_t}^{\epsilon_t m_t}$ which fixes $X$ and each component $\partial X$ up to isotopy in $S^X$. So $\tau_l^X= \hat D(\tau_k^X)$ and $\cc(\tau_l^X)= \hat D\cdot \cc(\tau_k^X)$, yielding that also $\tau_l|X=\hat D(\tau_k|X)$. 

Now, when $X$ is an annulus, $j\in I_X$ means that $\tau_j|X$ is combed, i.e. its core $\gamma$ is a twist curve of $\tau_j$ (Lemma \ref{lem:twistininduced}); and it is clear from the above that $\gamma$ is a twist curve for $\tau_k$ if and only if it is one for $\tau_l=D_{\gamma_t}^{\epsilon_t m_t}(\tau_k)$ and for all tracks in between, proving the claim in this case. When $X$ is not an annulus, $\cc(\tau_l|X)=\hat D^{\epsilon_t m_t}\cdot\cc(\tau_k|X)$ and $\cc^*(\tau_l|X)=\hat D^{\epsilon_t m_t}\cdot\cc^*(\tau_k|X)$. So the sets $\cc(\tau_k|X), \cc(\tau_l|X)$ have the same diameter in $\cc(X)$, and the same is true of $\cc^*(\tau_k|X), \cc^*(\tau_l|X)$. Again, this means that $k\in I_X$ if and only if $l\in I_X$.

For claim 3 note, first of all, that it were $t_->t_+$ then $\max DI_{t_-}> \max I_X$ leading to $I_X\subsetneq DI_{t_-}$, contradicting the hypothesis.  Then $t_-\leq t_+$. The claim is proved if one proves that, for each $t_-\leq t<t+1\leq t_+$, the diffeomorphism $\phi_{t+1}^{-1}\circ\phi_t=D_{\gamma_{t+1}}^{-\epsilon_{t+1}(2\mathsf{K}_0+4-m_{t+1})}$ fixes $X$ and and each component of $\partial X$, up to isotopies of $S$. For this value of $t$, $\min DI_{t+1}\geq \min I_X$ and $\max DI_{t+1}\leq \max I_X$, so claims 1 and 2 apply, and imply that $D_{\gamma_{t+1}}$ fixes $X$ and each of the connected components of $\partial X$ (again, up to isotopies of $S$): the same is true of any power of it.

To prove claim 4: let $I_X=[k,l]$. Suppose first that $X$ is not an annulus. 

Note that if $t_+<r$ then, from claims 1 and 2 above, it can be derived that not only $l<\max DI_{t_+ +1}$, but also $l<\max DK_{t_+ +1}$. Therefore (even for $t_+=r$), if $j\in I_X$ then $j'\coloneqq [t_+]\dn j$ has the property that $\utw\tau_{j'}=\phi_t(\tau_j)$ for a suitable $t-\leq t\leq t_+$: this is a consequence of the basic remarks and definitions given above. So $\mathrm{diam}_{\cc\left(\phi_t(X)\right)}\left(\cc(\utw\tau_{j'}|\phi_t(X)\right)$, $\mathrm{diam}_{\cc\left(\phi_t(X)\right)}\left(\cc^*(\utw\tau_{j'}|\phi_t(X)\right)\geq 3$ just by application of $\phi_t$ to $\tau_j$. But it has been just proved above that, for the considered values of $t$, $\phi_t(X)=\phi_{t_-}(X)$. This yields $j'\in I_{\phi_{t_-}(X)}^\utw$ i.e. one inclusion is proved for $X$ not an annulus.

For the opposite inclusion, suppose $l<N$: this implies that, if $t_+<r$, then $l+1\leq \max DK_{t_+ +1}$. Thus (even for $t_+=r$) not only $\utw\tau_{[t_+]\dn l}=\phi_{t_+}(\tau_l)$, but also $\utw\tau_{[t_+]\dn (l+1)}=\phi_{t_+}(\tau_{l+1})$. Since $\mathrm{diam}_{\cc(X)}\left(\cc(\tau_{l+1}|X)\right)<3$, also $\mathrm{diam}_{\cc(\phi_{t_+}(X))}\left(\cc(\utw\tau_{[t_+]\dn(l+1)}|\phi_{t_+}(X))\right)<3$. As $\phi_{t_+}(X)=\phi_{t_-}(X)$, this yields $[t_+]\dn(l+1)\not\in I_{\phi_{t_-}(X)}^\utw$.

Similarly, suppose that $k>0$: then $k-1\geq \min DI_{t_-}$, thus $\utw\tau_{[t_+]\dn k}=\phi_{t_-}(\tau_k)$ and $\utw\tau_{[t_+]\dn (k-1)}=\phi_{t_-}(\tau_{k-1})$. Since $\mathrm{diam}_{\cc(X)}\left(\cc^*(\tau_{k-1}|X)\right)<3$, also\linebreak $\mathrm{diam}_{\cc(\phi_{t_-}(X))}\left(\cc^*(\utw\tau_{[t_+]\dn(k-1)}|\phi_{t_-}(X))\right)<3$, and this means that $[t_+]\dn(k-1)\not\in I_{\phi_{t_-}(X)}^\utw$. To sum up, considering that the cases $l=N$ or $k=0$ are trivial, $I_{\phi_{t_-}(X)}^\utw\subseteq [t_+]\dn I_X$.

If $X$ is an annulus a similar argument applies: but, rather than looking at the preservation of the diameter of the considered sets under the respective diffeomorphisms, the preserved property is whether the core curve of $X$, which turns into the core of $\phi_{t_-}(X)$, is twist.

In order to prove claim 5, note that similarly as above one has $\utw t_-\leq \utw t_+$. Let $k\coloneqq [\utw t_-]\up(\min I_X^\utw)$, and let $Y\coloneqq \phi_{\utw t_-}^{-1}(X)$. Necessarily $k\in I_Y$, since $[\utw t_-]\dn k=\min I_X^\utw$ and $\utw\tau_{\min I_X^\utw}=\phi_{\utw t_-}(\tau_k)$. Suppose there is a $t$ such that $I_Y\subsetneq DI_t$. There are two cases to discern.

If $k=\max NI_{\utw t_-}=\min DI_{\utw t_-+1}$, then $t=\utw t_-+1$ and $\max I_Y< \max DI_t$; also $\max I_Y< \max DK_t$ by claim 1. But then $\utw\tau_{\max DI_t}=\phi_{\utw t_-}(\tau_{\max DK_t})$ implying, with arguments similar to the ones already seen above, that $\max DI_t^\utw$ and so that $I_X\subseteq \left[[\utw t_-]\dn k, \max DI_t^\utw-1\right]$. But necessarily $\dn k=\min DI_t^\utw$ resulting in $I_X\subsetneq DI_t^\utw$, a contradiction.

If $k$ is any other index, then $t=\utw t_-$, and $\min DI_t\not \in I_Y$. This implies, by definition of $k$ as $[\utw t_+]\up(\min I_X^\utw)$, which by $t_-\leq t_+$ is the same as $\up(\min I_X^\utw)$, that $k\in DL_t$, $k\not=\min DL_t$. Either $\max DI_t=N$, or $\max DI_t+1\not\in I_Y$ and $\utw\tau_{\max DI_t^\utw+1}=\phi_{\utw t_-}(\tau_{\max DI_t+1})$. Both scenarios imply $\max I_X\leq \max DI_t^\utw$, but on the other hand also $\min I_X=[\utw t_+]\dn k > \min DI_t^\utw$ by assumption. Hence the same contradiction as above: $I_X\subsetneq DI_t^\utw$.

Therefore there is no $t$ with $I_Y\subsetneq DI_t$ and claim 4 applies: $I_X=[\utw t_+]\dn I_Y$. Now if, for any fixed value of $t$, $I_Y\supseteq DK_t$ or $DL_t$, then $I_Y\supseteq DI_t$, by claims 1 and 2 above. This completes the proof of the claim.

In claim 6, note that $\tau_{j+g_t}=D_{\gamma_t}^{\epsilon_t}(\tau_j)$. Since $\gamma_t$ does not intersect $X$ essentially, $D_{\gamma_t}|_X:X\rightarrow S$ is isotopic to the inclusion map $X\hookrightarrow S$. The map has a lift $\hat D: S^X\rightarrow S^X$ whose restriction to $X=\core(S^X)$ is again isotopic to the inclusion map. This means that $\hat D$ is itself isotopic to $\mathrm{id}_{S^X}$. Hence $\tau_j^X,\tau_{j+g_t}^X=\hat D(\tau_j^X)$ are pretracks isotopic in $S$, and $\tau_j|X$,$\tau_{j+g_t}|X$ are, too.

For the first statement in claim 7, note that $\phi_t\circ \phi_{t_0}^{-1}$ is a number of Dehn twist about curves essentially disjoint from $X$, so it follows as an application of what has been said previously. As for the second statement, we know that for all $j\in \bigcup_{t=t_0}^{t_1} (DI_t\cup NI_t)$ we have $\utw\tau_{\dn j}=\phi_{t(j)}(\tau_j)$ and $\dn j=[t_1]\dn j$ while, for all $j\in I\setminus \bigcup_{t=t_0}^{t_1} (DI_t\cup NI_t)\subseteq DK_{t_1+1}$, we have $\utw\tau_{[t_1]\dn j}=\phi_{t_1}(\tau_j)$. Lifting, we find that $\utw\tau_{\dn j}|\phi_{t(j)}(X)=\hat\phi_{t(j)}(\tau_j|X)$ and $\utw\tau_{[t_1]\dn j}|\phi_{t_1}(X)=\hat\phi_{t_1}(\tau_j|X)$, respectively, in the two specified cases; but $\phi_{t(j)}(X)$ and $\phi_{t_1}(X)$ are isotopic to $\phi_{t_0}(X)$ in both scenarios.
\end{proof}

\begin{rmk}\label{rmk:pantsboundunderdt_cutting}
There is a constant $C'_2=C'_2(S)$ such that the following are true. In the setting of the above lemma, for $X$ not an annulus, suppose that $\gamma_t$ cuts $\partial X$ essentially, and let $k,l\in DI_t\cap I_X$. Then $d_{\pa(X)}(\tau_k|X,\tau_l|X)\leq C'_2$, and\linebreak $d_Y\left(V(\tau_k|X),V(\tau_l|X)\right)\leq C'_2$ for all $Y\subseteq X$ non-annular subsurfaces. Incidentally, it has been already noted in the remark after Lemma \ref{lem:decreasingfilling} that $V(\tau_k|X),V(\tau_l|X)$ are both vertices of $\pa(X)$.

\ul{Note:} the constants $C_2,C'_2$ will be merged into a single one $C_2$ after the present remark.

A straightforward consequence of claim 1 in the above lemma is that $\rot_{\bm\tau}(\gamma_t; k,l)\leq 2$ so, by Lemma \ref{lem:twistsplitnumber}, at most $5N_3^2$ splits --- which are all twist ones about $\gamma_t$ --- occur in $\bm\tau(k,l)$. A combinatorial finiteness argument, entirely similar to the one in the proof of Lemma \ref{lem:pantsboundunderdt}, gives an upper bound $K$, depending on $S$ only, for $i(\alpha,\beta)$ where $\alpha\in V(\tau_k)$ and $\beta\in V(\tau_l)$.

For each $Y\subseteq X$ non-annular subsurface, $\pi_Y V(\tau_k),\pi_Y V(\tau_l)$ are both nonempty; and the intersection number between any pair of elements, one from each set, is not greater than $4K+4$ (Remark \ref{rmk:subsurf_inters_bound}). Hence there is a bound on $d_Y\left(V(\tau_k), V(\tau_l)\right)$ (Lemma \ref{lem:cc_distance}), and one on $d_{\pa(X)}\left(V(\tau_k), V(\tau_l)\right)$ (Theorem \ref{thm:mmprojectiondist}, Lemma \ref{lem:pantsquasiisom}).

Lemma \ref{lem:induction_vertices_commute} commutes these bounds into ones for the required distances.
\end{rmk}

\begin{coroll}\label{cor:subsurface_bijection}
Let $\bm\tau=(\tau_j)_{j=0}^N$ be a $(\gamma_1,\ldots,\gamma_r)$-arranged train track splitting sequence, evolving firmly in a subsurface $S'$, not necessarily connected, of $S$. Let $\Sigma_1(\bm\tau)$ be the family of all connected components of $S'$, and let $\Sigma_2(\bm\tau)$ be the family of all (isotopy classes of) subsurfaces $X\subseteq S'$ of $S$ such that $I_X$ is not empty and not strictly contained in a single $DI_t$, $1\leq t \leq r$. Define $\Sigma_1(\utw\bm\tau)=\Sigma_1(\bm\tau)$, and $\Sigma_2(\utw\bm\tau)$ similarly as above (adding superscripts $^\utw$).

Also let $\Sigma(\bm\tau)\coloneqq \Sigma_1(\bm\tau)\cup \Sigma_2(\bm\tau)$ and $\Sigma(\utw\bm\tau)\coloneqq \Sigma_1(\utw\bm\tau)\cup \Sigma_2(\utw\bm\tau)$.

For each $X\in \Sigma_2(\bm\tau)$, let $\phi_X\coloneqq \phi_{t_-}$ for $t_-=t_-(X)$ defined as in claim 3 of Lemma \ref{lem:untwistedsubsurfaces} above.

Then the map $\Sigma(\bm\tau)\rightarrow \Sigma(\utw\bm\tau)$, defined by $X\mapsto\utw X\coloneqq X$ for $X\in \Sigma_1(S)$, and by $X\mapsto \utw X\coloneqq \phi_X(X)$ (up to isotopies in $S$) for $X\in\Sigma_2(S)$, is a bijection.
\end{coroll}
\begin{proof}
The two definitions are easily seen to agree for $X\in \Sigma_1(\bm\tau)\cap \Sigma_2(\bm\tau)$. Also, the fact that the map is bijective easily follows from its restriction $\Sigma_2(\bm\tau)\rightarrow\Sigma_2(\utw\bm\tau)$ being so.

We first prove injectivity. Given $X_1,X_2\in \Sigma_2(\bm\tau)$, they may have:
\begin{itemize}
\item $t_-(X_1)=t_-(X_2)$. In this case $\phi_{X_1}=\phi_{X_2}$; so $\utw X_1,\utw X_2$ are isotopic surfaces if and only if $X_1,X_2$ are.
\item $t_-(X_1)<t_-(X_2)$ (without loss of generality). In this case $\dn(\min I_{X_1})<\dn(\min I_{X_2})$ so $I_{X_1}^\utw=[t_+(X_1)]\dn I_{X_1}\not=[t_+(X_2)]\dn I_{X_2}=I_{X_2}^\utw$. This implies, in particular, $X_1,X_2$ not isotopic in $S$.
\end{itemize}

Surjectivity is a straightforward consequence of claim 5 in the lemma above.
\end{proof}

Note that, in particular, all regular neighbourhoods of curves $\nei(\gamma_t)\in \Sigma(\bm\tau)$, so there is a bijection between the family $\nei(\gamma_1),\ldots, \nei(\gamma_r)$ and the corresponding $\utw\left(\nei(\gamma_1)\right),\ldots, \utw\left(\nei(\gamma_r)\right)$. This means, in particular, that their core curves $\utw\gamma_1\coloneqq \phi_{\nei(\gamma_1)}(\gamma_1),\ldots, \utw\gamma_r\coloneqq \phi_{\nei(\gamma_r)}(\gamma_r)$ are all distinct, meaning that $\utw\bm\tau$ is $(\utw\gamma_1,\ldots,\utw\gamma_r)$-arranged and that $\utw(\utw\bm\tau)=\utw\bm\tau$.

\begin{prop}\label{prp:locallyfinite}
Let $S$ be a surface. There is a constant $C_5(S)$ such that the following is true.

Let $\bm\tau=(\tau_j)_{j=0}^N$ be a $(\gamma_1,\ldots,\gamma_r)$-arranged splitting sequence, evolving firmly in a subsurface $S'$, not necessarily connected, of $S$. Let $X\in \Sigma(\bm\tau)$, $X$ not an annulus, and let $[k,l]$ be an interval of indices for $\bm\tau$, with the condition that $[k,l]\subseteq I_X$ if $X$ is not a connected component of $S'$. Then
$$
d_{\pa(\utw X)}(\dn k,\dn l)^\utw \leq C_5 \left(d_{\pa(X)}(k,l)\right)^2
\text{ and }
d_{\pa(X)}(k,l) \leq C_5 \left(d_{\pa(\utw X)}(\dn k,\dn l)^\utw\right)^2.
$$

There are two increasing functions $\Psi_S, \Psi'_S:[0,+\infty)\rightarrow [0,+\infty)$ such that, if $\gamma_1,\ldots,\gamma_r$ are the effective twist curves of $\bm\tau$ and $\bm\tau$ is effectively arranged, then
$$
d_{\ma(\utw X)}(\dn k,\dn l)^\utw \leq \Psi_S\left(d_{\pa(\utw X)}(\dn k,\dn l)^\utw\right) \leq \Psi'_S\left( d_{\pa(X)}(k,l) \right).
$$
\end{prop}
\begin{proof}

\step{1} preparation.

The curves $\gamma_1,\ldots,\gamma_r$ come with the usual condition that $(\min DI_t)_t$ is an increasing sequence. Let $1\leq t_1<\ldots<t_q\leq r$ be the indices such that $DI_{t_s}\cap [k,l] \not=\emptyset$ and $\gamma_t$ is essentially \emph{not} disjoint from $X$. So each of these $\gamma_{t_s}$ may be essentially contained in $X$ or intersect $\partial X$ essentially; but, according to claims 1 and 2 in Lemma \ref{lem:untwistedsubsurfaces}, the latter case may occur only for $s=1,q$ because all the other $DI_{t_s}\subseteq [k,l]$.

We establish a parallel notation for intervals of indices of $\bm\tau$, one that resembles the one used so far but focuses only on $X$ and on the family $\gamma_{t_1},\ldots,\gamma_{t_q}$. First of all rename these curves as $\delta_1,\ldots,\delta_q$. Then, for $s=1,\ldots,q$, denote $XDI_s\coloneqq DI_{t_s}\cap [k,l]$, and $XDL_s \coloneqq DL_{t_s}\cap [k,l]$. For $s=1,\ldots, q-1$, moreover, define $XNI_s\coloneqq [\max XDI_s,\min XDI_{s+1}]$. Define also $XNI_0\coloneqq [k, \min XDI_1]$, $XNI_q\coloneqq [\max XDI_q,l]$. Each $XNI_s$ is a union of intervals $NI_t$, and of intervals $DI_t$ for some $\gamma_t$ not intersecting $X$ essentially. In order to improve compatibility with the formulas, let also $XDI_0\coloneqq XDL_0\coloneqq \{k\}$.

We define also intervals related to the sequence $\utw\bm\tau$: $XDI_s^\utw\coloneqq \dn XDI_s$ for all $s=1,\ldots, q$, $XNI_s^\utw\coloneqq \dn XNI_s$ for all $s=0,\ldots, q$.

The collection $V(\tau_l|X)$ is a vertex of $\pa(X)$: if $X$ is strictly contained in a connected component of $S'$, this derives from the hypothesis $l\in I_X$. Otherwise we know from Lemma \ref{lem:decreasingfilling} that $V(\tau_l|X)$ consists of all elements of $V(\tau_l)$ which are essentially contained in $X$ and not isotopic to connected components of $\partial X$: these elements fill $X$ and abide by the mutual intersection bound established in accordance with Remark \ref{rmk:pickparameters}.

Now, $\bm\tau$ is, in particular, $(\delta_1,\ldots,\delta_q)$-arranged: so, by Proposition \ref{prp:tcbound}, simplifying its statement, there is a bound $q\leq c\cdot d_{\pa(X)}(k,l)$ for a $c=c(S)$. On the other hand: $\utw\bm\tau$ is $(\utw\gamma_{t_1},\ldots,\utw\gamma_{t_q})$-arranged; these curves are all contained in $\utw X$ except for possibly $\utw\gamma_{t_1}, \utw\gamma_{t_q}$; and either $\dn l \in I_{\utw X}^\utw$ or $\utw X=X$ is a connected component of $S'$. This implies that $V(\utw\tau_{\dn l}|\utw X)$ is a vertex of $\pa(\utw X)$ so, again according to Proposition \ref{prp:tcbound}, it is also true that $q\leq c \cdot d_{\pa(X)}(\dn k,\dn l)^\utw$.

Denote, for simplicity, $t_0\coloneqq 0$ and $t_{q+1}\coloneqq r+1$. The following argument is straightforward if we are working on $X\in \Sigma_1(\bm\tau)$, and is motivated by claim 7 in Lemma \ref{lem:untwistedsubsurfaces} if $X\in \Sigma_2(\bm\tau)$. For $0\leq s \leq q$, if $t_s\leq t < t_{s+1}$ then $\phi_t|_X=\phi_{t_s}|_X$; also, if $\hat \phi_t: S^X\rightarrow S^{\phi_t(X)}$ is the lift of $\phi_t$, then the sequence $\left(\utw\tau_j|\utw X\right)_{j\in XDI_s^\utw \cup XNI_s^\utw}$ is obtained from $(\tau_j|X)_{j\in XDL_s\cup XNI_s}$ applying $\hat\phi_{t_s}$ to each entry, and removing some repetitions of tracks in the Dehn intervals $DI_t\subset XNI_s$, where $\gamma_t$ does not intersect $X$ essentially. So $d_{\pa(\utw X)}(XDI_s^\utw\cup XNI_s^\utw)^\utw=d_{\pa(X)}(XDL_s\cup XNI_s)$. 

\step{2} reciprocal bounds for distances in the pants graph.

Fix an index $s$, and denote simply $a\coloneqq \min XDL_s$; $b\coloneqq \max XDI_s$; $\phi\coloneqq \phi_{t(\min XDL_s)}$. Take $\mathsf{R}_0=\max_{Y\subseteq S \text{ subsurface}} \mathsf{R}_0(Y,Q)$, for $Q$ the quasi-isometry constant introduced in Theorem \ref{thm:mms_cc_geodicity}, as defined in Lemma \ref{lem:reversetriangle}. By those theorem and lemma, we have that for any non-annular subsurface $Y\subseteq X$, $d_Y\left(\tau_a|X,\tau_b|X\right) \leq  d_Y(\tau_k|X,\tau_l|X)+ 2\mathsf{R}_0$.

Let $M'=M+2\mathsf{R}_0$ --- for the definition of $M$, see the beginning of \S \ref{sub:twistcurvebound} --- we have
$$
\sum_{\substack{Y\subset X\text{ essential} \\ \text{and non-annular}}} [d_Y(\tau_a|X,\tau_b|X)]_{M'} =
\sum_{\substack{Y\subset X\text{ essential} \\ \text{and non-annular}}} \left([d_Y(\tau_k|X,\tau_l|X)]_M+ \mathsf{r}(Y)\right)
$$

where $\mathsf{r}(Y)=2\mathsf{R}_0$ if it comes alongside a nonzero summand, and $0$ otherwise. So $[d_Y(\tau_k,\tau_l)]_M+ \mathsf{r}(Y)\leq \frac{M'}{M} [d_Y(\tau_k|X,\tau_l|X)]_M$. Therefore, by Theorem \ref{thm:mmprojectiondist} and Lemma \ref{lem:pantsquasiisom} there are constants $c_0, c_1$ depending on $S$ (and $M$) such that $d_{\pa(X)}(\tau_a,\tau_b)\leq c_0 d_{\pa(X)}(\tau_k,\tau_l) + c_1$.

Similarly, we get $d_{\pa(\utw X)}(\utw\tau_{\dn a},\utw\tau_{\dn b})\leq c_0 d_{\pa(\utw X)}(\utw\tau_{\dn k},\utw\tau_{\dn l}) + c_1$.

So, on the one hand,
\begin{eqnarray*}
 & d_{\pa(\utw X)}(\dn k,\dn l)^\utw \leq \sum_{s=0}^q d_{\pa(\utw X)}(XDI_s^\utw\cup XNI_s^\utw)^\utw = & \\
 &  \sum_{s=0}^q d_{\pa(X)}(XDL_s\cup XNI_s) \leq (q+1) \left(c_0 d_{\pa(X)}(k,l)+c_1\right) & 
\end{eqnarray*}
and, as already pointed out, $q\leq c\cdot d_{\pa(X)}(k,l)$. This, together the considerations at the beginning of \S \ref{sub:twistcurvebound}, proves the existence of a $C_5(S)$ such that $d_{\pa(\utw X)}(\dn k,\dn l)^\utw\leq C_5 \left(d_{\pa(X)}(k,l)\right)^2$.

On the other hand,
\begin{eqnarray*}
 & d_{\pa(X)}(k, l) \leq \sum_{s=0}^q \left(d_{\pa(X)}(\min XDI_s,\min XDL_s) + d_{\pa(X)}(XDL_s\cup XNI_s) \right). & 
\end{eqnarray*}

For $s=1,\ldots, q$ (neglect the dummy index $s=0$), $\bm\tau(\min XDI_s,\min XDL_s)$ has twist nature about $\delta_s$ which is either contained in $X$ or cutting $\partial X$ essentially. So the sequence falls in the kind treated in either the last bullet of Lemma \ref{lem:pantsboundunderdt}, or Remark \ref{rmk:pantsboundunderdt_cutting}, implying that $d_{\pa(X)}(\min XDI_s,\min XDL_s)\leq C_2(S)$. Continuing the chain of inequalities:
\begin{eqnarray*}
\ldots & \leq \sum_{s=0}^q \left(C_2 + c_0 d_{\pa(\utw X)}(XDI_s^\utw\cup XNI_s^\utw)^\utw +c_1\right) \leq & \\
 & \leq (q+1)\left(C_2+c_0 d_{\pa(\utw X)}(\dn k,\dn l)^\utw + c_1\right) & 
\end{eqnarray*}

Again as already pointed out, $q\leq c \cdot d_{\pa(\utw X)}(\dn k,\dn l)^\utw$ therefore $C_5(S)$ can be taken so that $d_{\pa(X)}(k,l)\leq C_5 \left(d_{\pa(\utw X)}(\dn k,\dn l)^\utw\right)^2$.

\step{3} bounds for distances in annular subsurfaces, when $\bm\tau$ is effectively arranged.

Let $\alpha\in \cc(X)$ be a curve and, for ease of notation, let simply $\nei=\nei(\alpha)$ be a regular neighbourhood of $\alpha$ in $X$. We wish to prove the existence of constants depending on $S$ such that, for all $0\leq s\leq q$, and all $j,j'\in XDL_s\cup XNI_s$, $d_\nei(\tau_j|X,\tau_{j'}|X)$ is bounded by this constant. Note that each $V\left((\tau_j|X)^\nei\right)$ is a subset of the respective $V(\tau_j^\nei)$.

\begin{itemize}
\item If $\alpha$ is none of the $\delta_s$ then, as it follows immediately from the definition of effective twist curve, $d_\nei(\tau_0|X,\tau_N|X)< 4\mathsf{K}_0+19$; and consequently, $d_\nei(\tau_j|X,\tau_{j'}|X)< 4\mathsf{K}_0+2\mathsf{R}_0+19$.
\item If $\alpha=\delta_u$ for $u\not=s$, then $[j,j']\subseteq G_{t_u-}\cup I_{t_u-}$ or $[j,j']\subseteq I_{t_u+}\cup G_{t_u+}$ --- see Definition \ref{def:arranged}. The first case implies that $d_\nei(\tau_j|X,\tau_{j'}|X)\leq \mathsf{\mathsf{K}_0}+2\mathsf{R}_0+9$, and the second one that $d_\nei(\tau_j|X,\tau_{j'}|X)\leq \mathsf{K}_0+6$.
\item If $\alpha=\delta_s$ then $[j,j']\subseteq DL_s\cup I_{s+}\cup G_{s+}$. If $[j,j']\subseteq DL_s$ then $\rot_{\bm\tau}(\delta_s;j,j')\leq 2\mathsf{K}_0+4$ hence $d_\nei(\tau_j|X,\tau_j'|X)\leq 2\mathsf{K}_0+8$ by point \ref{itm:rot_vs_dist} in Remark \ref{rmk:rotbasics}. If $[j,j']\subseteq I_{s+}\cup G_{s+}$ then same as above applies, and if $j,j'$ range in the union of the three intervals then $d_\nei(\tau_j|X,\tau_{j'}|X)$ is bounded by the sum of the two previous bounds.
\end{itemize}

This proves the existence of the desired bound, which we call $c'$. As for the splitting sequence $\utw\bm\tau$, let $\alpha\in\cc(\utw X)$.  

Fix a value $0\leq s \leq q$, and recall the conclusions of Step 1. If $t_s\leq t < t_{s+1}$ then $\phi_t|X=\phi_{t_s}|X$ and in particular there is a curve $\beta_s\in \cc(X)$ such that $\phi_t(\beta_s)=\phi_{t_s}(\beta_s)=\alpha$. If $\hat{\hat\phi}_t: S^{\nei(\beta_s)}\rightarrow S^{\nei(\alpha)}$ is the lift of $\phi_t$, then the sequence $\left((\utw\tau_j|X)^{\nei(\alpha)}\right)_{j\in XDI_s^\utw \cup XNI_s^\utw}$ is obtained from $\left((\utw\tau_j|X)^{\nei(\beta_s)}\right)_{j\in XDL_s\cup XNI_s}$ applying $\hat{\hat\phi}_{t_s}(\min XDL_s)$ to each entry, and removing some repetitions of entries in the intervals $DI_t\subset XNI_s$, where $\gamma_t$ does not intersect $X$ essentially. So if $[a,b]=XDL_s\cup XNI_s$ then $[\dn a,\dn b]=XDI_s^\utw\cup XNI_s^\utw$ and  $d_{\nei(\alpha)}(\utw\tau_{\dn a}|\utw X,\utw\tau_{\dn b}|\utw X)=d_{\nei(\beta_s)}(\tau_a|X,\tau_b|X)\leq c'$.

\step{4} bound for $d_{\ma(\utw X)}(k,l)^\utw$, when $\bm\tau$ is effectively arranged.

The distance bounds seen above prove immediately that\linebreak $d_{\nei(\alpha)}(\utw\tau_{\dn k}|\utw X,\tau_{\dn l}|\utw X)\leq c'(q+1)$ for all $\alpha\in \cc(\utw X)$.

Let then $M''\coloneqq \max \{M,c'(q+1)+1\}$. We have
$$
\sum_{Y\subset \utw X\text{ essential}} [d_Y(\utw\tau_{\dn k}|X,\tau_{\dn l}|X)]_{M''} =
\sum_{\substack{Y\subset \utw X\text{ essential} \\ \text{and non-annular}}} [d_Y(\utw\tau_{\dn k}|X,\tau_{\dn l}|X)]_{M''}
$$
and this implies, via the usual Theorem \ref{thm:mmprojectiondist} and Lemma \ref{lem:pantsquasiisom}, that\linebreak $d_{\ma(\utw X)}(\dn k,\dn l)^\utw\leq f_0 d_{\pa(\utw X)}(\dn k,\dn l)^\utw + f_1$. Here, $f_0$ and $f_1$ derive from $e_0,e_1$ cited in Theorem \ref{thm:mmprojectiondist} and may be supposed, by taking looser bounds, to depend on $S$ and $M''$ only, and be increasing functions in the variable $M''$.

Using a simplified notation $d_\utw\coloneqq d_{\pa(\utw X)}(\dn k,\dn l)^\utw$ and $d\coloneqq d_{\pa(X)}(k,l)$:
\begin{eqnarray*}
 & d_{\ma(\utw X)}(\dn k,\dn l)^\utw\leq f_0\left(S,c'' d_\utw\right) d_\utw + f_1\left(S,c'' d_\utw\right) & \\
 &\leq C_5 f_0\left(S,c'' C_5 d^2\right) d^2 + f_1\left(S, c'' C_5 d^2\right) & 
\end{eqnarray*}
where $c''$ is a further constant, deriving from the previous ones $c,c'$; and we are using the fact that $q\leq c d_\utw$, $d_\utw\leq d^2$. This proves our claim.
\end{proof}

\subsection{Proof of the main statement}\label{sub:conclusion}

In this subsection we prove Theorem \ref{thm:core}.

\begin{defin}
Let $\bm\tau=(\tau_j)_{j=0}^N$ be a splitting sequence of birecurrent train tracks on $S$, and let $X\subseteq S$ be a non-annular subsurface. Let $0\leq j <N$ be an index such that $\tau_j$ splits to $\tau_{j+1}$. We say that this split move is \nw{visible} in $X$ if, in order to turn $\tau_j|X$ into $\tau_{j+1}|X$, a split is required (i.e. it does not suffice to apply comb equivalences and/or take subtracks).

We denote with $|\bm\tau|_X$ the number of splits in $\bm\tau$ which are visible in $X$.
\end{defin}

%The last part of the proof of Theorem 5.3 in \cite{mms} (which is quoted in this work as Theorem \ref{thm:mmsstructure}) deals with the setting described in the first sentence of the above Definition. In that setting, we have that $\tau_{j+1}|_X$ is comb equivalent to $\tau_j|X$, a subtrack of it, or obtained from it with a single split.

Paraphrasing the statement of Theorem \ref{thm:core}, we want to show that the distance induced in the pants graph by a splitting sequence is, up to constants, the number of splits in the untwisted sequence: this is easy in one direction.

\begin{prop}\label{prp:easyttbound}
There is a constant $C_6=C_6(S)$ such that the following is true. Let $\bm\tau=(\tau_j)_{j=0}^N$ be a generic splitting sequence of recurrent, cornered train tracks on $S$. Let $X$ be a subsurface of $S$ with $X\supseteq S'$ the subsurface, not necessarily connected, filled by $V(\tau_0)$, and suppose that $V(\tau_N)$ is a vertex of $\pa(S)$. Then
$$d_{\pa(X)}\left(\tau_0,\tau_N\right)\leq C_6 |\utw(\rar\bm\tau)|.$$
\end{prop}
\begin{proof}
Suppose first that $\bm\tau$ evolves firmly in a (possibly disconnected) subsurface $S'$. Define a strictly increasing sequence of indices $(j_i)_{i=0}^{K}$ for the sequence $\rar\bm\tau=(\rar\tau_j)_{j=0}^{N'}$ such that one of the following holds --- using the standard notation of Definition \ref{def:arranged} and of Proposition \ref{prp:rearrang2}.
\begin{itemize}
\item $[j_i,j_{i+1}]=DI^\rar_s$, or consists of $DI^\rar_s$ plus some slide moves, for some $s$; in this case $d_{\pa(S)}(\rar\tau_{i_j},\rar\tau_{i_{j+1}})\leq C_2$ by Lemma \ref{lem:pantsboundunderdt}.
\item $[j_i,j_{i+1}]\cap DI^\rar_s$ is not more than a single index for each $s$, and there is exactly one split move in $(\rar\bm\tau)(j_i,j_{i+1})$. Then $d_{\pa(S)}(\rar\tau_{i_j},\rar\tau_{i_{j+1}})\leq 1$ by the choice of parameters in Remark \ref{rmk:pickparameters} and, in particular, $i\left(V(\rar\tau_{i_j}),V(\rar\tau_{i_{j+1}})\right)\leq \ell$. This bounds uniformly $i\left(\pi_Y V(\rar\tau_{i_j}),\pi_Y V(\rar\tau_{i_{j+1}})\right)$ and then $d_Y\left(\rar\tau_{i_j}|X,\rar\tau_{i_{j+1}}|X\right)$, for each non-annular subsurface $Y\subseteq X$ (with an application of Remark \ref{rmk:subsurf_inters_bound}, Lemma \ref{lem:cc_distance}, Lemma \ref{lem:induction_vertices_commute} successively). Theorem \ref{thm:mmprojectiondist} with Lemma \ref{lem:pantsquasiisom}, then, gives a uniform bound $C_6$ for $d_{\pa(X)}(\rar\tau_{i_j},\rar\tau_{i_{j+1}})$. Suppose $C_6\geq C_2$, in order to make notation simpler in the following chain of inequalities.
\end{itemize}

Now, $|\utw(\rar\bm\tau)|\geq K\geq d_{\pa(S)}(\rar\tau_0,\rar\tau_{N'})/C_6$. The second inequality follows directly from concatenating the ones above, while the first one is due to the fact that each $\left(\utw(\rar\bm\tau)\right)(\dn j_i, \dn j_{i+1})$ contains at least one split. This proves the claim in the specified special case.

More generally, view $\bm\tau$ as a concatenation $\bm\tau^1*\bm\epsilon^1*\bm\tau^2*\ldots*\bm\epsilon^{w-1}*\bm\tau^w$ as pointed out just before Definition \ref{def:not_firmly}. For each $1\leq u\leq w$, let $J^u$ be the interval of indices for $\bm\tau$ coming from $\bm\tau^u$. So
\begin{eqnarray*}
 & |\utw(\rar\bm\tau)|= \sum_{u=1}^w \left(|\psi_u\cdot\utw(\rar\bm\tau^u)|+1\right)-1 \geq & \\
 & \sum_{u=1}^w \left(d_{\pa(X)}\left(\psi_u(\rar\tau_{\min J^u})|X, \psi_u(\rar\tau_{\max J^u}|X)\right)/C_6+1\right) -1 = & \\
 & \sum_{u=1}^w \left(d_{\pa(X)}\left(\rar\tau_{\min J^u}|X, \rar\tau_{\max J^u}|X\right)/C_6+1\right) -1\geq d_{\pa(X)}(\rar\tau_0|X,\rar\tau_{N'}|X)/C_6 &
\end{eqnarray*} 
using the special case proved above combined with the triangle inequality.
\end{proof}

We state Claim 6.14 from \cite{mms} in a slightly more general setting:
\begin{lemma}\label{lem:mms614}
Let $\bm\tau=(\tau_j)_{j=0}^N$ be a splitting sequence of birecurrent train tracks on $S$, and let $X\subseteq S$ be a subsurface. Then there is a constant $N_4=N_4(X)$ such that, if $V(\tau_k|X)=V(\tau_l|X)$, then $|\bm\tau(k,l)|_X\leq N_4$.
\end{lemma}

The proof is the same as the original claim: in that setting, the splitting sequence had some more properties than in the above statement, but they are not actually employed.

\begin{prop}\label{prp:hardttbound}
There is a constant $C_7(S)$ such that the following is true. Let $\bm\tau=(\tau_j)_{j=0}^N$ be a generic splitting sequence of birecurrent, cornered train tracks, which evolves firmly in a subsurface $S'$ of $S$, not necessarily connected; then, for each non-annular connected component $T$ of $S'$,
$$d_{\pa(T)}(V(\tau_0),V(\tau_N))\geq_{C_7}|\utw(\rar\bm\tau)|_T.$$
\end{prop}

The proof of this statement will be very similar to the one of Theorem 6.1 in \cite{mms}, quoted in this work as Theorem \ref{thm:mms_main}. Large portions of our proof, actually, would be verbatim repetitions of pieces of the proof of the mentioned result, up to a replacement of a few names of objects.

The original splitting sequence $\bm\tau$ is never needed in this proof: so we may well identify $\rar\bm\tau=\bm\tau$ and, in particular, suppose that $\bm\tau$ is effectively arranged. Index $\utw\bm\tau=(\utw\tau_j)_{j=0}^{N'}$. In this proof, in order to avoid introducing \emph{ad hoc} notations and arguments, when $T$ is a connected component of $S'$ redefine $I_T\coloneqq[0,N]$ and $I_T^\utw\coloneqq [0,N']$. 

If $X\in \Sigma(\bm\tau)$ is not an annulus and $[p,q]=I\subset I_X$, let $\mathcal T_X(p,q)=\mathcal T_X(I)$ be the set of indices $\dn p\leq j\leq (\dn q)-1$ (in particular $j\in \dn I\subseteq I_{\utw X}^\utw$) such that $\utw\tau_{j+1}$ is obtained from $\utw\tau_j$ via a split which is visible in $\utw X$. Note that for $X\not= T$ Theorem \ref{thm:mmsstructure} holds, and yields that in this case only one split is required. 

The main step in our proof is a tailored version of Proposition 6.9 from \cite{mms}, to be proved by induction on the complexity of subsurfaces of $S'$. We will refer to it as the key claim:

\begin{claim}
Let $X\in \Sigma(\bm\tau)$. There is a constant $C_7(X,S)$ such that for any $J_X\subset I_X$ we have
$$|\mathcal T_X(J_X)|\leq_{C_7} d_{\pa(X)}(J_X).$$
\end{claim}

Proposition \ref{prp:hardttbound} will then follow taking $X=T$ and $J_T=[0,N]$.

Before we start the proof, we make some observations. First of all, there is a difference between our hypotheses and the ones of \cite{mms}: the latter include the request that the vertex cycles of each entry in $\bm\tau$ fill the entire surface, whereas here we restrict to a connected component of the subsurface $S'$ they fill. 

The proof of our result, just as the one given in \cite{mms}, works by induction on $\xi(X)$. In general, however, our argument will exclude annular subsurfaces while they are explicitly considered in \cite{mms}. In particular, annuli cannot serve as an induction basis: this role will be covered by those subsurfaces which are homeomorphic to $S_{0,4}$ or $S_{1,1}$; but it is not necessary to distinguish the induction basis from the rest. It is worth noting that the re-definition of $I_T$ given above invalidates Theorem \ref{thm:mmsstructure}. This causes no harm, since that theorem, during the present proof, will be used only on \emph{proper} subsurfaces of the $X$ fixed in the above claim. The statements of Lemma \ref{lem:untwistedsubsurfaces} are still true with the new definition instead, and their proofs are either easier or trivial.

To start the proof, fix two constants $\mathsf{T}_0(X),\mathsf{T}_1(X)$ with the following conditions:
\begin{eqnarray*}
\max\{6K_1+2K_2+2\mathsf{K}_0(X)+2,2\mathsf{R}_0, M_6(X), & & \\
 C_2+2F(32N_1(S^X)+4)+2\mathsf{K}_0(X)+1\} & \leq & \mathsf{T}_0(X) \\
\max\{\mathsf{T}_0(X)+2\mathsf{R}_0, & & \\
N_1K_2 + C_2 + 2 F\left(8 N_1(S^X)\right)+2\mathsf{K}_0\} & \leq & \mathsf{T}_1(X)
\end{eqnarray*}

Here, $K_1$ is the constant denoted $\mathsf{N}_1$ in \cite{mms}: an upper bound for $d_Y(\alpha,\beta)$, for $Y$ a subsurface of $S$ and $\alpha,\beta$ wide curves for one same almost track on $S$. Its existence is proved with a bound on $i(\pi_Y(\alpha),\pi_Y(\beta))$, similarly as in Lemma \ref{lem:induction_vertices_commute} and Remark \ref{rmk:subsurf_inters_bound}. $K_2$ is the constant denoted $\mathsf{N}_2$ in \cite{mms}: an upper bound for $d_Y(V(\tau|Z),V(\tau'|Z))$, where $Y,Z$ are subsurfaces of $S$; $\tau,\tau'$ are almost tracks with the latter obtained from the former as a subtrack, or with a split. $K_2$ exists again by a similar argument as in Lemma \ref{lem:induction_vertices_commute}. Again we take $\mathsf{R}_0=\max_{Y\subseteq S \text{ subsurface}} \mathsf{R}_0(Y,Q)$, for $Q$ the quasi-isometry constant introduced in Theorem \ref{thm:mms_cc_geodicity}, as defined in Lemma \ref{lem:reversetriangle}.

The constant $M_6(X)$ is introduced in Theorem 6.12 of \cite{masurminskyii}, also stated as Theorem \ref{thm:mmprojectiondist} here, in the light of \S 8 in that work, which allows to ignore annular subsurfaces. Finally, $C_2$ has been defined in Lemma \ref{lem:pantsboundunderdt} and re-defined in Remark \ref{rmk:pantsboundunderdt_cutting}; $N_1$ is defined in Lemma \ref{lem:vertexsetbounds}; $F$ is the function introduced in Lemma \ref{lem:induction_vertices_commute}.

\begin{defin}
A proper, non-annular subsurface $Y\subsetneq X$ is an \nw{inductive subsurface} if $d_Y(J_X)\geq \mathsf{T}_0(X)$. The related \nw{inductive subinterval} is then $J_Y\coloneqq I_Y\cap J_X$. For any non-inductive $Y$, just set $J_Y\coloneqq \emptyset$.

A subinterval $I\subseteq J_X$ is \nw{straight} if, for all non-annular, proper subsurfaces $Y\subsetneq X$, we have $\mathrm{diam}_Y(I)\leq\mathsf{T}_1(X)$. Here $\mathrm{diam}_Y(I)\coloneqq \mathrm{diam}_{\cc(Y)}\left(\bigcup_{j\in I}\pi_Y\left(V(\tau_j)\right)\right)$.
\end{defin}
Note that if $Y$ is an inductive subsurface then $Y\in \Sigma_2(\bm\tau)$. This is because, for any $Y\not\in \Sigma_2(\bm\tau)$ with $I_Y\not=\emptyset$, there is an index $s$ such that $I_Y\subsetneq DI_s$, necessarily for a $\gamma_s$ intersecting $\partial Y$ (claims 1, 2 in Lemma \ref{lem:untwistedsubsurfaces}) and then, for any $i,j\in I_Y$, one has $d_Y(i,j)\leq C_2+2F(32N_1(S^X)+4)$ by Remark \ref{rmk:pantsboundunderdt_cutting} combined with Lemma \ref{lem:induction_vertices_commute}. Hence one may subdivide $J_X=[\min J_X, \min I_Y]\cup (J_X\cap I_Y) \cup [\max I_Y, \max J_X]$ (one or more intervals may be empty, as we may have defined an upper bound lower than the lower bound), and get $d_Y(J_X)\leq 2\mathsf{K_0} + C_2+F(32 N_1(S^X)+4)< \mathsf{T}_0(X)$ with an application of Theorem \ref{thm:mmsstructure}.

\begin{lemma}\label{lem:mms612} \emph{(Plays the role of \cite{mms}, Lemma 6.12)}
If $I\subseteq J_X$ is disjoint from all inductive subintervals of $J_X$, then $I$ is straight for $X$.
\end{lemma}
The proof is the one of Lemma 6.12 in \cite{mms}, verbatim.

In our setting Claim 6.14 from \cite{mms} (i.e. Lemma \ref{lem:mms614} above) reads:
\begin{claim}
If for a pair of indices $k,l\in J_X$ we have $V(\utw\tau_{\dn k}|\utw X)=V(\utw\tau_{\dn l}|\utw X)$, then $\mathcal T_X(k,l)\leq N_4$.
\end{claim}

Similarly as noted at the beginning of \S \ref{sub:twistcurvebound},  for all $C\geq M_6(X)$ there are constants $e_0(S,C)$, $e_1(S,C)$ such that, for all $k,l\in J_X$,
\begin{equation}\label{eqn:mmsubsurfaceproj_mixed}
e_0^{-1} d_{\pa(X)}\left(\tau_k, \tau_l\right) - e_1 -C_1 \leq
\sum_{\substack{Y\subset X\text{ essential} \\ \text{and non-annular}}} [d_Y(\tau_k,\tau_l)]_C \leq
e_0 d_{\pa(X)}\left(\tau_k, \tau_l\right) + e_1 + C_1
\end{equation}
where $C_1=C_1(X)$.

\begin{lemma}\label{lem:mms613}\emph{(Plays the role of \cite{mms}, Lemma 6.13)}
Let $I\subseteq J_X$ be a straight subinterval. Then $|\mathcal T_X(I)|\leq_A d_X(I)$, for a constant $A=A(X)$.
\end{lemma}
\begin{proof}
Let $[p,q]=I$. Fix $C\coloneqq 1+\max \{M_6(X),\mathsf{T}_1(X)\}$ and let\linebreak $\mathsf{R}_1\coloneqq \max\{e_0(X,C),e_1(X,C)+C_1(X)\}+1$. Define a map $\rho:[0,M]\rightarrow I$, starting with $\rho(0)\coloneqq p$, and recursively setting $\rho(n+1)$ to be the smallest element in $[\rho(n),q]$ with $d_{\pa(X)}\left(\rho(n),\rho(n+1)\right)=\mathsf{R}_1$, until this is no longer defined: when this stage is reached, set $M\coloneqq n+1$ and $\rho(M)=q$.

Then for all $0\leq n \leq M-1$, and all $\rho(n)\leq j\leq \rho(n+1)$, we get that $d_{\pa(X)}(\tau_{\rho(n)},\tau_j)\leq \mathsf{R}_1$, so $d_{\ma(\utw X)}(\dn \rho(n), \dn j)^\utw \leq \Psi'_S(\mathsf R_1)$ according to Proposition \ref{prp:locallyfinite}.

We define $\mathsf V$ to be the maximum cardinality for a ball of radius $\Psi'_S(\mathsf R_1)$ in $\ma(X)$: then the sequence $\utw\bm\tau(\dn\rho(n),\dn\rho(n+1))$ has all the respective entries $V(\tau_j|X)$ lying within one of these balls. So, via the above claim, $|\mathcal T_X(\rho(n),\rho(n+1)|\leq N_4\mathsf V$ and therefore $|\mathcal T_X(I)|\leq N_4 \mathsf V \cdot M$.

Claim 6.15 in \cite{mms} holds for us (just replace $\mathcal M(X)$ with $\pa(X)$ and employ formula \ref{eqn:mmsubsurfaceproj_mixed} above; in particular, perform the summation over non-annular subsurfaces $Y$ only):
\begin{claim}
For $0\leq n\leq M-2$, $d_X\left(\rho(n),\rho(n+1)\right)\geq \mathsf{R}_0+1$.
\end{claim}

The remainder of the proof of the present lemma is the same the final part in the proof of Lemma 6.13 in \cite{mms}, where $M \leq d_X(I)+1$ is obtained. Hence the bound on the number of splits.
\end{proof}

When $X\cong S_{1,1}$ or $S_{0,4}$, the proof of the key claim ends here because, as $X$ has no subsurfaces, all intervals can be supposed to be straight. Theorem \ref{thm:mmprojectiondist} and Lemma \ref{lem:pantsquasiisom}, applied for these subsurfaces, assert that $\pa(X)$ is quasi-isometric to $\cc(X)$. So the key claim, for these subsurfaces, coincides with the statement of the above lemma. We now continue for the induction step.

All that is said in \cite{mms} from Definition 6.17 to Lemma 6.20 applies completely in our setting. Our proof continues with that chunk of proof, replacing all occurrences of $\mathcal S_X$ there with $\mathcal T_X$; the constants $\mathsf{N}_1,\mathsf{N}_2$ there with $K_1,K_2$ respectively; references to Lemmas 6.12 and 6.13 there with our Lemmas \ref{lem:mms612}, \ref{lem:mms613} respectively.

\begin{defin}
\nw{Assign} an index $r\in \mathcal T_X(J_X)$ to a subsurface $Y\in \mathsf{Ind}'$ if $r\in J_Y$, if the split from $\utw\tau_r$ to $\utw \tau_{r+1}$ is visible in $\utw Y$, and there is no other $Z\in \mathsf{Ind}$, $Z\subsetneq Y$, with these properties.

For $I\subseteq J_X$, denote the set of indices in $\mathcal T_X(I)$ which are assigned to $Y$ by $\mathsf{AI}_Y(I)\subseteq \mathcal T_Y(J_Y)$.
\end{defin}

\begin{lemma}\label{lem:mms616}
Let $Y\in\mathsf{Ind}$. There is a constant $B=B(X)$ such that, if $I\subseteq J_Y$ has all indices in $\mathcal T_X(I)$ assigned to $X$, then $|\mathcal T_X(I)|\leq B$.
\end{lemma}
\begin{proof}
Let $[p,q]=I$. First we follow the approach of Lemma 6.16 in \cite{mms} to prove the existence of a bound, depending on $X$ only, for $d_{\cc(X)}\left(V(\tau_p|X),V(\tau_q|X)\right)$. Since $p,q\in J_Y\subseteq I_Y$, Theorem \ref{thm:mmsstructure} asserts that $\partial Y$ is wide when put in efficient position with respect to $\tau_p$ or to $\tau_q$; and, via lift, also when put in efficient position with respect to $\tau_p^X$ or $\tau_q^X$. Let $\alpha\in V(\tau_p|X)$ and identify it with a carried realization, which we may suppose to realize the intersection number of $\alpha$ with $\partial Y$ --- if it does not, one may fix this with the bigon removal technique employed in the proof of Corollary 4.3 from \cite{mms}. 

As $\partial Y$ is wide in $\tau_p$, it crosses each of the branch rectangles in $\bar\nei(\tau_p)$ --- which are at most $N_1$: see Lemma \ref{lem:vertexsetbounds} --- in at most two connected components each. So it intersects at most $2N_1$ branch rectangles of $\tau_p^X$. For each branch rectangle $R_b$ of $\tau_p^X$, the intersections $\alpha\cap R_b$ and $\partial Y\cap R_b$ consist of at most two arcs each; so $\alpha\cap\partial Y\cap R_b$ is at most $4$ points, else there is a bigon. This bounds $i\left(\alpha,\partial Y\right)$; and, via Lemma \ref{lem:cc_distance}, also $d_X\left(\alpha,\partial Y\right)$.

The same bound applies for all $\alpha\in V(\tau_q|X)$. So, via triangle equality, we also have a bound for $d_{\cc(X)}\left(V(\tau_p| X),V(\tau_q|X)\right)$; and, via Lemma \ref{lem:induction_vertices_commute}, for $d_X(p, q)$.

Now we prove that $I$ is straight i.e. we look for a bound for $d_Z\left(J\right)$, for $Z\in \mathsf{Ind}$ and $\emptyset\not=J\subseteq I$ a subinterval. First we find a bound for $d_Z(J\cap I_Z)$, provided that the given interval $J\cap I_Z=[k,l]$ is nonempty. For ease of notation, let also $I'\coloneqq I\cap I_Z$.

Since $\dn [k,l]$ contains only indices of $\mathcal T_X(J_X)$ assigned to $X$, and not to any proper, inductive subsurface, $(\utw\tau_{j'}|\utw Z)_{j'=\dn k}^{\dn l}$ includes no split move. Define $\tau_0\coloneqq t(\min I')$, and $t_1$ as the highest $t$ such that $\max I' \geq \max DI_t$. Then $\bigcup_{t=t_0}^{t_1} (DI_t\cup NI_t)\subseteq DI_{t_0}\cup NI_{t_0}\cup I'$ and $I'\subseteq \left(\bigcup_{t=t_0}^{t_1} (DI_t\cup NI_t)\right)\cup DK_{t_1+1}$ (by convention, $DK_{r+1}=\{N\}$).

We know from claim 1 in Lemma \ref{lem:untwistedsubsurfaces} that, for all $t_0<t\leq t_1$: $\gamma_t$ does not cut $\partial Z$; moreover, the splits in $\utw\bm\tau$, indexed by $DI_t^\utw$, shall be invisible when inducing the tracks to $\utw Z$. This means that $\utw\gamma_t$ does not intersect $\utw Z$ at all, and that $\gamma_t$ does not intersect $Z$.

Claim 7 in Lemma \ref{lem:untwistedsubsurfaces} applies for $I'$, and therefore for $J$. The entries of $\bm\tau(I')$ which do not reflect in an entry of $\utw\bm\tau(\dn I')$ are contained in some $DI_t\setminus DL_t$ for $t_0+1\leq t\leq t_1$ or in $DK_{t_1+1}$. Set $J''\coloneqq \left(DK_{t_1+1}\cap [k,l]\right)\cup \{l\}$ (so that $l\in J''$ anyway). By said claim, the sequence $(\utw\tau_{j'}|\utw Z)_{j'\in \dn J}$ is obtained from $(\tau_j|Z)_{j=k}^{\min J''}$ by application of $\hat \phi_{t_0}:S^Z\rightarrow S^{\utw Z}$ to each entry, and possibly removal of some repeated ones. If $\tau_j|Z$ splits to $\tau_{j+1}|Z$ for $j,j+1\in [k,l]$, then neither of the two belongs to $DI_t\setminus DL_t$, for $t_0+1\leq t\leq t_1$, otherwise disjointness of $\gamma_t$ from $Z$ would be contradicted. So in that case $j,j+1\in DK_{t_1+1}\cap [k,l]$.

When one induces to $Z$ the subsequence $\bm\tau(k, \min J'')$, only comb equivalences and subtrack extractions are seen. While a subtrack always has less non-mixed branches than the almost track it is taken from, comb equivalences are unable to alter the count of these branches. So $d_{\cc(Z)}\left(V(\tau_k|Z),V(\tau_{\min J''}|Z)\right)\leq N_1K_2$.

And the distance spanned in $J''$, when this interval is not trivial, is cared after with Lemma \ref{lem:pantsboundunderdt} and Remark \ref{rmk:pantsboundunderdt_cutting}\footnote{These two exclude the case of $\gamma_{t_1+1}$ disjoint from $Z$, but the way to deal with it is obvious.}: $d_{\cc(Z)}\left(V(\tau_{\min J''}|Z),V(\tau_l|Z)\right)\leq C_2$. Combining the two estimates, and using Lemma \ref{lem:induction_vertices_commute}, we have $d_Z(k,l)\leq N_1K_2 + C_2 + 2 F\left(8N_1(S^X)\right)$.

In general, we may need to add a contribution for the distance spanned in $[\min J,k]$ and in $[l,\max J]$ (if either, or both, are nonempty): so $d_Z(J)\leq N_1K_2 + C_2 + 2 F\left(8N_1(S^X)\right)+2\mathsf{K_0}$, by Theorem \ref{thm:mmsstructure}. This last bound holds also if $J\cap I_Z$ is empty.

This proves that $I$ is a straight interval, and there is a bound for $d_X(I)$: but then, Lemma \ref{lem:mms613} ensures that there is a bound $|\mathcal T_X(I)|\leq B$, too.
\end{proof}

\begin{lemma}\emph{(Plays the role of \cite{mms}, Lemma 6.22)}
There is a constant $A=A(X)$ such that
$$
\left|\bigcup_{Y\in \mathsf{Ind}}\mathcal T_X(J_Y)\right|\leq_A |\mathsf{Ind}|+ \sum_{Y\in \mathsf{Ind}} d_Y(J_X).
$$
\end{lemma}
\begin{proof}
Given any maximal interval $I$ among the ones treated in the lemma above, we have just shown that $|\mathcal T_X(I)|\leq B$. This implies that, for a suitable constant $A'$,
$$\left|\bigcup_{Y\in \mathsf{Ind}}\mathcal T_X(J_Y)\right|\leq_{A'}  \left|\bigcup_{Y\in \mathsf{Ind}}\mathsf{AI}_Y(J_Y)\right|\leq \sum_{Y\in\mathsf{Ind}} |\mathcal T_Y(J_Y)|.$$
We now apply the inductive hypothesis given by the key claim: $|\mathcal T_Y(J_Y)|\leq_A d_{\pa(Y)}(J_Y)$ for a constant $A=A(X)$.

Let $C\coloneqq 1+\max\{ M_6(Y), \mathsf{T}_0(X)+2\mathsf{R}_0\}$. Then, as shown in formula \ref{eqn:mmsubsurfaceproj_mixed} above,
$$
d_{\pa(Y)}(J_Y)\leq_{\mathsf{E}} \sum_{Z\subseteq Y\text{ non-annular}} [d_Z(J_Y)]_C
$$
for a suitable constant $\mathsf E(S,C)>1$. The proof ends as the one of Lemma 6.22 of \cite{mms}, verbatim.
\end{proof}

The proof ends for us the same way as the proof of Proposition 6.9 in \cite{mms} after proving Lemma 6.22. It just suffices to replace occurrences of $\mathcal S_X$ there with $\mathcal T_X$; the ones of $\mathcal M(X)$ with $\pa(S)$; and references to lemmas previously proved there with the lemmas above. 

It is just worth marking that the final estimate $|\mathsf{Ind}'|\leq_A d_{\pa(X)}(J_X)$ may be proved simply by saying that
$$
|\mathsf{Ind}'|\leq \frac{1}{\mathsf T_0}\sum_{Y\in\mathsf{Ind'}} [d_Y(J_X)]_{\mathsf T_0} \leq_{\mathsf E(X;\mathsf T_0)} \frac{1}{\mathsf T_0}d_{\pa(X)}(J_X)
$$
again by application of formula \ref{eqn:mmsubsurfaceproj_mixed}.
\cvd

\begin{coroll}\label{cor:hardttbound}
There is a constant $C_8=C_8(S)$ such that the following is true. Let $\bm\tau=(\tau_j)_{j=0}^N$ be a generic splitting sequence of birecurrent, cornered train tracks on $S$. Let $X$ be a subsurface of $S$ with $X\supseteq S'$ the subsurface, not necessarily connected, filled by $V(\tau_0)$, and suppose that $V(\tau_N|X)$ is a vertex of $\pa(X)$: then
$$d_{\pa(X)}\left(V(\tau_0|X),V(\tau_N|X)\right)\geq_{C_8} |\utw(\rar\bm\tau)|.$$
\end{coroll}

\begin{proof}
Subdivide $\bm\tau=\bm\tau^1*\bm\epsilon^2*\bm\tau^2*\ldots*\bm\epsilon^w*\bm\tau^w$ as it is done before Definition \ref{def:not_firmly}. To simplify notations, we may replace each $\bm\tau^u$ with the respective $\rar\bm\tau^u$. For $1\leq u\leq w$, let $J^u$ be the interval of indices in $\bm\tau$ supplied by $\bm\tau^u$. The subsequence $\bm\tau^u$ will evolve firmly in a (possibly disconnected) subsurface $S^u\subset S$ and so will do $\utw\tau^u$; let $S^u=T_1^u\sqcup\ldots\sqcup T_{k(u)}^u$ be the decomposition into connected components.

\step{1} for each $1\leq u\leq w$, $|\utw\bm\tau^u|\leq \frac{1}{N_0N_4+1} \sum_{i=1}^{k(u)} |\utw\bm\tau^u|_{T_i^u}$. When using this notation, we agree that $|\utw\bm\tau^u|_{T_i^u}=0$ when $T_i^u$ is an annulus.

Let $p,q\in \dn J^u$ be two indices such that no split in $\utw\bm\tau(p,q)$ is visible in any of the non-annular connected components of $S^u$. Let $T^u_i$ be non-annular. Our assumption on the interval of indices $[p,q]$ implies that, within this interval, the induced tracks $\utw\tau_j|T^u_i$ may change only under comb equivalences and subtrack extractions. Consequently, if $j'>j$, then $V(\utw\tau_{j'}|T^u_i)\subseteq V(\utw\tau_j|T^u_i)$. In Lemma \ref{lem:decreasingfilling} we have shown that $V(\utw\tau_j|T^u_i)=V(\utw\tau_j)\cap \cc(T^u_i)$.

So we have $V(\utw\tau_j)=\Xi^u\cup\bigcup_{T^u_i\text{ non-annular}} V(\utw\tau_j|T^u_i)$, where $\Xi^u$ is the collection of all core curves of the $T^u_i$ which are annuli; therefore this set is also decreasing as $j$ increases in $[p,q]$. As a consequence of Lemma \ref{lem:vertexsetbounds}, then, $V(\utw\tau_j)$ changes at most $N_0$ times within the interval $[p,q]$. By Lemma \ref{lem:mms614} applied on the entire surface $S$, then, $|\utw\bm\tau^u(p,q)|\leq N_4N_0$.

In other words, among every $N_0N_4+1$ consecutive splits in $\utw\bm\tau^u$, at least one of them must be visible in one of the connected components of $S^u$.

\step{2} we claim that there is a constant $A=A(S)$ such that, for each $1\leq u\leq w$,
$$
|\utw\bm\tau^u|\leq_A \sum_{\substack{Y\subset X\text{ essential} \\ \text{and non-annular}}} [d_Y(J^u)]_M
$$
where $M\geq \max \left(\{M_6(X)|X\text{ subsurface of }S\}\cup\{2\}\right)$ is fixed.

According to formula \ref{eqn:mmsubsurfaceproj_mixed} above, for all $Z\subseteq X$, $Z$ subsurface of $S$, there is a constant $e=e(S,M)$ such that:
$$
d_{\pa(Z)}(J^u)=_{e} \sum_{\substack{Y\subset Z\text{ essential} \\ \text{and non-annular}}} [d_Y(J^u)]_M \eqqcolon \mathrm{sum}(Z,M,u)
$$
for a suitable constant $M$ which we may suppose to be depending only on $S$. The formula may be given sense also for $Z$ an annulus: in that case the summation is empty, and the $d_{\pa(Z)}(J^u)$ can also be set to $0$.

Clearly (see the Remark following Lemma \ref{lem:decreasingfilling}), $S^u\subseteq S'\subseteq X$. The family of all non-annular subsurfaces $Y\subset X$ can be partitioned into:
\begin{itemize}
\item the sub-families of subsurfaces $Y\subset T_i^u$ --- one for each $1\leq i\leq k(u)$;
\item the sub-family of all subsurfaces $Y$ which are cut by $\partial S^u$;
\item the sub-family of all subsurfaces $Y$ essentially disjoint from $S^u$.
\end{itemize}
If $Y$ is any of the subsurfaces as in the second bullet and $j,j'\in J^u$, then any $\alpha\in \pi_Y(V(\tau_j)), \beta\in \pi_Y(V(\tau_{j'}))$  will not intersect $\pi_Y(\partial S^u)\not=\emptyset$. So $d_Y(\alpha,\beta)\leq 2\leq M$. The surfaces as in the third bullet, instead, just do not exist, because $V(\tau_N|X)\in \pa^0(X)$ implies $V(\tau_j|X)\in \pa^0(X)$ for all $j\in J^u$, too. So $\ol{X\setminus S^u}$ consists of discs, punctured discs, annuli and pairs of pants.

This yields that only the subsurfaces as in the first bullet count in the summation $\mathrm{sum}(X,M,u)$, which is thus equal to $\sum_{i=1}^{k(u)} \mathrm{sum}(T_i^u,M,u)$.

From Step 1 above and Proposition \ref{prp:hardttbound} we have

\begin{center}
$|\utw\bm\tau^u| \leq \frac{1}{N_0 N_4+1} \sum_{i=1}^{k(u)} |\utw\bm\tau^u|_{T_i} \leq_{C_7(N_0 N_4+1)} \sum_{i=1}^{k(u)} d_{\pa(T_i)}(J^u).$
\end{center}

According to equalities we have established previously, then, there is a constant $A=A(S)$ such that

\begin{center}
$|\utw\bm\tau^u|\leq_A \sum_{i=1}^{k(u)}\mathrm{sum}(T_i^u,M,u) = \mathrm{sum}(X,M,u).$
\end{center}

\step{3} proof of the statement.

Recall the constant $K_2$ introduced in the proof of Proposition \ref{prp:hardttbound}. Fix\linebreak $M\geq \max\left( \{M_6(X)|X\text{ subsurface of }S\}\cup\{2\}\right) + 2(\xi(S)-1)\mathsf{R}_0$.

Note that, for each subsurface $Y\subset X$, $d_Y(\tau_0,\tau_N)\geq \sum_{u=1}^w d_Y(J^u)-2(w-1)\mathsf{R}_0$ by repeated application of Lemma \ref{lem:reversetriangle}.

Let $E\coloneqq \max_{0\leq j\leq\xi(S)}e\left(S,M-2j\mathsf{R}_0\right)$; and let $M'\coloneqq M-2(w-1)\mathsf{R}_0$. Then
\begin{eqnarray*}
 & d_{\pa(X)}(0,N)=_E \sum_{\substack{Y\subset X\text{ essential} \\ \text{and non-annular}}} [d_Y(0,N)]_{M'} & \\
 & \geq \sum_{\substack{Y\subset S\text{ essential} \\ \text{and non-annular}}} \left([d_Y(J^1) + \ldots + d_Y(J^w)]_M-\mathsf r(X))\right)
\end{eqnarray*}
where $\mathsf r(X)=2(w-1)\mathsf{R}_0$ if the other summand is nonzero, and $0$ otherwise. A simple computation shows that $[x]_M-\mathsf r(X)\geq \left(1-\frac{2(\xi(S)-1)\mathsf{R}_0}{M}\right)[x]_M$ (the bracketed term is positive). So, if $E'=E\left(1-\frac{2(\xi(S)-1)\mathsf{R}_0}{M}\right)^{-1}$, then
\begin{eqnarray*}
 & d_{\pa(X)}(0,N) \geq_{E'} \sum_{\substack{Y\subset X\text{ essential} \\ \text{and non-annular}}} [d_Y(J^1) + \ldots + d_Y(J^w)]_M & \\
 & \geq \sum_Y \left([d_Y(J^1)]_M + \ldots + \sum_X [d_Y(J^w)]_M\right) =\sum_{u=1}^w \mathrm{sum}(X,M,u). &
\end{eqnarray*}

Step 1 above gives $\sum(X,M,u)\geq_A |\utw\bm\tau^u|=|\psi_u\cdot \utw\bm\tau^u|$: therefore\linebreak $\sum_{u=1}^w \mathrm{sum}(X,M,u)\geq_A |\utw\bm\tau|-\xi(S)+1$. This concludes the proof.
\end{proof}

From Proposition \ref{prp:easyttbound} and Corollary \ref{prp:hardttbound} we easily derive:
\begin{coroll}\label{cor:ttsamelength}
Let $\bm\tau=(\tau_j)_{j=0}^N$ be a generic splitting sequence of birecurrent, cornered train tracks on $S$. Let $X$ be a subsurface of $S$ with $X\supseteq S'$ the subsurface, not necessarily connected, filled by $V(\tau_0)$, and suppose that $V(\tau_N|X)$ is a vertex of $\pa(X)$. Let $\bm\tau'$ be another splitting sequence, beginning and ending with the same train tracks as $\bm\tau$. Then
$$
|\utw(\rar\bm\tau)| =_{C_6C_8} |\utw(\rar\bm\tau')|.
$$
\end{coroll}

We may now prove a slightly generalized statement for Theorem \ref{thm:core}:

\begin{theo}\label{thm:main_full}
Let $\bm\tau=(\tau_j)_{j=0}^N$ be a generic splitting sequence of birecurrent train tracks on $S$, \emph{not necessarily cornered ones}. Let $X$ be a subsurface of $S$ with $X\supseteq S'$ the subsurface, not necessarily connected, filled by $V(\tau_0)$, and suppose that $V(\tau_N|X)$ is a vertex of $\pa(X)$. Then, for a constant $C_9=C_9(S)$ independent of $\bm\tau$,
$$d_{\pa(X)}(0,N)=_{C_9} |\utw(\rar(\cnr\bm\tau))|.$$
\end{theo}

\begin{proof}
Consider a cornerization $\cnr\bm\tau=(\cnr\tau_j)_{j=0}^M$ of the splitting sequence $\bm\tau$. Point 4 in Lemma \ref{lem:ctauproperties} gives, in particular, that $\cnr\tau_0$ is a cornerization of $\tau_0$ and $\cnr\tau_M$ is a cornerization of $\tau_N$. In particular, as $V(\tau_0|X)\subseteq \cc(\cnr\tau_0|X)$ and $V(\tau_N|X)\subseteq \cc(\cnr\tau_M|X)$, an application of Lemma \ref{lem:decreasingfilling} and following observations gives that $V(\cnr\tau_0|X)$, $V(\cnr\tau_M|X)$ are vertices of $\pa(X)$.

By Remark \ref{rmk:centralsplitbound} the number of splits turning $\cnr\tau_0$ into $\tau_0$, and $\cnr\tau_M$ into $\tau_N$, is bounded in terms of the topology of $S$. Therefore $d_\pa(\tau_0,\cnr\tau_0)$ and $d_\pa(\cnr\tau_M,\tau_N)$ are also bounded.

Finally,
$$
d_{\pa(X)}(\cnr\tau_0,\cnr\tau_M)=_{\max\{C_6,C_8\}} |\utw(\rar(\cnr\bm\tau))|,
$$
by a combination of Proposition \ref{prp:easyttbound} and of Corollary \ref{cor:subsurface_bijection}.

The triangle inequality completes the proof.
\end{proof}

\begin{coroll}
Let $\bm\tau=(\tau_j)_{j=0}^N$ be a generic splitting sequence of birecurrent train tracks on $S$. Let $X$ be a subsurface of $S$ with $X\supseteq S'$ the subsurface, not necessarily connected, filled by $V(\tau_0)$, and suppose that $V(\tau_N|X)$ is a vertex of $\pa(X)$. Then $(\sigma_j)_{j=0}^M=\bm\sigma\coloneqq \rar(\cnr\bm\tau)$ describes an unparametrized quasi-geodesic in the pants graph.

If $J$ is the sequence of indices $0\leq j< M$ such that $j,j+1 \in DL_t\cup NI_t$ for some $0\leq t \leq q$ (see Definition \ref{def:untwistedsequence} and following constructions) and $\sigma_j$ splits to $\sigma_{j+1}$, then $\left(V(\sigma_j)\right)_{j\in J}$ describes a $\max\{C_6,C_8\}$-quasi-geodesic in $\pa(X)$.
\end{coroll}
\begin{proof}
With minor adaptations to Proposition \ref{prp:easyttbound}, Proposition \ref{prp:hardttbound}, Corollary \ref{cor:hardttbound}, one proves that for any two indices $0\leq k\leq l\leq M$,
$$
|(\utw\bm\sigma)(\dn k,\dn l)|\leq_{C_8} d_{\pa(X)}\left(V(\sigma_k),V(\sigma_l)\right)\leq_{C_6} |(\utw\bm\sigma)(\dn k,\dn l)|
$$
which is just a restatement of our second claim.

In order to prove the first claim, it is sufficient to prove that\linebreak $d_{\pa(X)}\left(V(\sigma_a),V(\sigma_b)\right)\leq C_2$ for all $a,b$ comprised between two consecutive indices in $J$. For any two indices as such, there is a $t$ such that, for all $j\in [a,b]\setminus DI_t$, $\sigma_{j+1}$ is obtained from $\sigma_j$ with a slide. We conclude with Lemma \ref{lem:pantsboundunderdt} and Remark \ref{rmk:pantsboundunderdt_cutting}, combined.
\end{proof}
\chapter{Hyperbolic volume estimates}

\section{Pseudo-Anosov mapping tori}

In all this section, $S$ is a \emph{closed} surface and $\psi:S\rightarrow S$ is a pseudo-Anosov diffeomorphism. We give an application of our main theorem and of Theorem \ref{thm:brockmappingtori} proved by Brock. Recall the introductory notions given in \S \ref{sub:mappingtori}.

Given a measured lamination $(\lambda,\mu)$ and a train track $\tau$ on $S$, one may say that $\lambda$ is \nw{carried} by $\tau$ if $\lambda$ can be ambient-isotoped to lie in $\nei(\tau)$, in a way that is transverse to all ties: we are just repeating Definition \ref{def:carried} to suit this setting. Uniqueness of carrying (Lemma 1.7.11 in \cite{penner}, previously simplified in Proposition \ref{prp:carryingunique} here) still holds. The bijection between rational transverse measures on a train track and weighted multicurves, described in Proposition \ref{prp:measurecurvecorresp} here, is a simplified version of Theorem 1.7.12 in \cite{penner}, establishing a similar bijection between \emph{real} transverse measures and carried measured laminations.

We say that $\tau$ is \nw{suited} to $(\lambda,\mu)$ if $\lambda$ is fully carried by $\tau$ and $\bar\nei(\tau)$ does not admit any carried, properly embedded arc which is disjoint from some carried realization of $\lambda$ --- this is equivalent to the definition given in \cite{papadopoulos} in terms of foliations, via the correspondence explained in \cite{levitt}. Note that, when such an arc exists instead, it is impossible that its ends lie along the same component of $\partial_v\bar\nei(\tau)$. We have that
\begin{claim}
Every measured lamination $(\lambda,\mu)$ has a birecurrent train track $\tau$ which is suited to it.
\end{claim}
This is a consequence of Corollary 1.7.6 in \cite{penner}: its statement only guarantees the existence of a birecurrent $\tau$ which fully carries $(\lambda,\mu)$; but, if $\tau$ is not suited to $\lambda$, it is sufficient to perform a multiple central split (see Definition \ref{def:multiplesplit}) along any of the carried arcs of $\tau$ which are disjoint from the carried realization of $\lambda$: this keeps the train track transversely recurrent, and also recurrent because the new track still fully carries $(\lambda,\mu)$ and we may apply Proposition 1.3.1 in \cite{penner}. Repeat until there are no more carried arcs disjoint from the carried realization of $\lambda$: this takes a finite number of steps.

In \cite{agol_pa} it described how, starting from $(\lambda,\mu)$ and a generic track $\tau_0$ suited to it, one may define the \nw{maximal splitting sequence} $\bm\tau=(\tau_j)_{j=0}^{+\infty}$: if $\tau_j$ has been defined, split simultaneously all branches of $\tau_j$ which are given maximal weight by $(\lambda,\mu)$, with the only parity which keeps $(\lambda,\mu)$ carried by the new track $\tau_{j+1}$. This sequence will not feature any central split. And in Theorem 3.5 of that work, which improves Theorem 4.1 in \cite{papadopoulos}, it is proved that
\begin{theo}
Let $(\lambda^s,\mu^s)$ be the stable lamination of $\psi$, and let $\tau_0$ be a train track suited to $(\lambda^s,\mu^s)$. The maximal splitting sequence $\bm\tau$ built from $\tau_0$ has two associated numbers $m,n$ such that, for all $j\geq m$,
$$
\tau_{j+n}=\psi(\tau_j).
$$
Moreover, if $\mu_j\in\mathcal M(\tau)$ is the transverse measure induced by $(\lambda^s,\mu^s)$ on $\tau_j$, then $\mu_{j+n}=c^{-1}\psi_*(\mu_j)$ where $c>1$ is the constant associated to the pseudo-Anosov diffeomorphism $\psi$.
\end{theo}

We prove the following
\begin{theo}\label{thm:agol_volume}
Let $(\lambda^s,\mu^s)$ be the stable lamination of $\psi$, and let $\tau_0$ be a birecurrent train track suited to $(\lambda^s,\mu^s)$. Let $\bm\rho\coloneqq \bm\tau(m,m+n)$, using the notation of the above theorem. Let $M\coloneqq \faktor{S\times [0,1]}{\sim_\psi}$ be the mapping torus built from $S$ and $\psi$ --- which is hyperbolic. Then there is a constant $C_{10}$, only depending on $S$, such that
$$
\vol(M)=_{C_{10}} |\utw(\rar(\bm\rho))|.
$$
\end{theo}

Note: for the purposes of this theorem and in order to connect it with the previously developed theory, $\bm\rho$ may be considered as a sequence where exactly one split occurs at each move: if more are split simultaneously, we just insert more intermediate steps.

\begin{lemma}
All train tracks $\tau_j$ in a splitting sequence $\bm\tau$ as in Theorem \ref{thm:agol_volume} are cornered and birecurrent, and $\bm\tau$ evolves firmly in $S$.
\end{lemma}
\begin{proof}
Each $\tau_j$ is cornered: if $\partial\bar\nei(\tau_j)$ includes a smooth component, then necessarily $\lambda^s$ includes a component which is a closed geodesic; but this contradicts minimality of the stable lamination.

Since $\tau_0$ is transversely recurrent, all $\tau_j$ are. Also, they are recurrent because they all fully carry the measured lamination $(\lambda^s,\mu^s)$, so Proposition 1.3.1 in \cite{penner} applies.

Let $S'$ be the subsurface of $S$, possibly a disconnected one, which is filled by $\cc(\tau_m)$ --- and by $V(\tau_m)$, by Lemma \ref{lem:decreasingfilling}. Then $\cc(\tau_{m+n})=\psi\cdot\cc(\tau_m)$ fills $\psi(S')$; moreover, up to isotopies, $\psi(S')\subseteq S'$, because of the decreasing filling properties of splitting sequences, stated after Lemma \ref{lem:decreasingfilling}. On the other hand, $\xi(S')=\xi\left(\psi(S')\right)$\footnote{We may just define $\xi(S')$ as the sum of the complexities of each connected component.}, so $S',\psi(S')$ are isotopic in $S$. If $S'\subsetneq S$, then $\psi$ fixes $\partial S'$: and this contradicts the fact that $\psi$, being pseudo-Anosov, is in particular irreducible.

So each set $\cc(\tau_{m+in})$ or $V(\tau_{m+in})$ for $i\geq 0$ fills $S$. Since the filled surface decreases along a splitting sequence, all $V(\tau_j)$, $j\geq 0$, fill $S$.
\end{proof}

\begin{proof}[of Theorem \ref{thm:agol_volume}]
Note that $M$ is hyperbolic because of Theorem \ref{thm:mappingtorushyperbolic}.

\step{1} we prove that is sufficient to show the existence of a constant $A=A(S)$ such that, for all $z\in \mathbb Z_{>0}$,\footnote{Here we use again the notation $d_{\pa(S)}(m,m+zn)$ when the splitting sequence is understood, as it was introduced in \S \ref{sub:untwistedsequence}.}
$$
d_{\pa(S)}(m,m+zn)\geq_A z|\utw(\rar(\bm\rho))|.
$$

Suppose that this condition is true, and note that actually $d_{\pa(S)}(m,m+zn)=\linebreak d_{\pa(S)}\left(V(\tau_m),\psi^z\cdot V(\tau_m)\right)$. Lemma \ref{lem:pantsquasiisom} proves that the inclusion $\pc(S)\hookrightarrow \pa(S)$ is a quasi-isometry, and this implies that there exists a pants decomposition $p$ of $S$ such that $d_{\pc(S)}(p, \psi^z(p))\geq_{A'} z|\utw(\rar(\bm\rho))|$ for another constant $A'=A'(S)$, and for all $z\in \mathbb Z_{>0}$. So the stable translation distance in $\pc(S)$ is $|\psi|^{st} \geq_{A'} |\utw(\rar(\bm\rho))|$.

On the other hand, Proposition \ref{prp:easyttbound} gives that $d_{\pa(S)}(m,m+n)\leq_{A''} |\utw(\rar(\bm\rho))|$, and this proves that also $|\psi|\leq_{A'''} |\utw(\rar(\bm\rho))|$ (again here $A''$, $A'''$ depend on $S$ only). Remark \ref{rmk:stable_dist_pc} and Brock's Theorem \ref{thm:brockmappingtori} conclude the argument.

Let $\bm\omega\coloneqq\rar\bm\rho$, indexed as $(\omega_j)_{j=0}^N$; for $z\in \mathbb Z_{>0}$, let $\bm\omega^{*z}\coloneqq \bm\omega*(\psi\cdot \bm\omega)*\ldots*(\psi^{z-1}\cdot \bm\omega)$: $\bm\omega^{*z}$ begins and ends with the same train tracks as $\bm\tau(m,m+zn)$.

\step{2} given any curve $\gamma\in \cc(S)$, the indices $0\leq i \leq z-1$ such that $\gamma$ is a twist curve for at least one entry in the splitting sequence $\bm\omega^i\coloneqq \psi^i\cdot\bm\omega =\bm\omega^{*z}\left(iN,(i+1)N\right)$ are at most $N_0+1$, and they are all consecutive. $N_0$ was introduced in Lemma \ref{lem:vertexsetbounds}.

Let $i$, $i'$ be two indices such that $\gamma$ is an effective twist curve in $\bm\rho^i$, $\bm\rho^{i'}$. Suppose, for a contradiction, that $i'-i\geq N_0+2$. Then $\gamma$ is a twist curve for at least one entry of $\bm\rho^i$ and one of $\bm\rho^{i'}$. Necessarily (Lemma \ref{lem:twistcurvebasics} applied to a sufficiently long initial segment of the sequence $\bm\tau$), it is a twist curve in all entries in $\bm\rho^j$ for all $i<j<i'$; and in particular $\gamma$ is a twist curve in $\tau_{m+jn}=\psi^j(\tau_m)$ for all $i<j\leq i'$.

Therefore $\psi^{-j}(\gamma)$ is a twist curve in $\tau_m$ for all $i<j\leq i'$. No two curves in this family are isotopic, because all $\psi^{-j}$ are pseudo-Anosov (see Remark \ref{rmk:power_pa}). They also all belong to $W(\tau_m)$, whose size is $\leq N_0$. This is a contradiction.

\step{3} we define a generalized version of the untwisted sequence, fitted to our scenario, and prove some basic properties.

If $\gamma_1,\ldots,\gamma_r$ are the effective twist curves of $\bm\omega$ (and of $\bm\rho$), it is possible to subdivide $[0,N]$ into $NI_0,DI_1,NI_1,\ldots, DI_r, NI_r$ as shown in \S \ref{sub:untwistedsequence}. Since $\bm\omega^{*z}$ is the concatenation of $z$ `copies' of $\bm\omega$, each transformed under a suitable power of $\psi$, one may similarly subdivide $[0,zN]$ into
$$NI^*_0,DI^*_1,NI^*_1,\ldots, DI^*_{zr}, NI^*_{zr}$$
with the maximum of each interval coinciding with the minimum of the following one. Note that each $NI^*_{ir}$, $0< i < z$, is the concatenation of a `copy' of $NI_r$ and a `copy' of $NI_0$. One defines also $DK^*_t, DL^*_t\subseteq DI^*_t$ just as seen in \S \ref{sub:untwistedsequence}.

For $1\leq t \leq r$, $1\leq i \leq z-1$, define $\gamma_{ri+t}=\psi^i(\gamma_t)$: then $\bm\omega^{*z}(DI^*_t)$ has twist nature about $\gamma_t$ for all $1\leq t\leq zr$. The curves in $(\gamma_t)_{t=1}^{zr}$ are \emph{not} necessarily all distinct, but we have proved in Step 2 that each of them occurs at most $N_0+1$ times in this enumeration.

If they were all distinct, then $\bm\omega^{*z}$ would be $(\gamma_t)_{t=1}^{zr}$-arranged: so we may say that the above notation is a `variation' of the notation used to describe arranged sequences, in a more general case. We will see now how to generalize the constructions developed in \S \ref{sec:traintrackconclusion}.

For notational convenience, let $\eta_1,\ldots, \eta_Q$ be an enumeration of the $(\gamma_t)_{t=1}^{zr}$ such that no curve occurs twice. The splitting sequence $\bm\omega^{*z}$ is $(\eta_1,\ldots, \eta_Q)$-arranged, even if in general each $\eta_u$ admits more than one choice for a Dehn interval; and one may need to change the order in which these curves are listed, in order to have $\left(\min DI_{\eta_u}\right)_{u=1}^Q$ increasing.

The definition of $\utw\bm\omega=(\utw\omega_j)_{j=0}^{N'}$ from $\bm\omega$ in Definition \ref{def:untwistedsequence} involves application of diffeomorphisms $\phi_t$ for $1 \leq t \leq r$. Let $\Psi\coloneqq \phi_r^{-1}\circ \psi$, and let
$$
\utw^{*z}\bm\omega\coloneqq \utw\bm\omega*(\Psi\cdot \utw\bm\omega)*\ldots*(\Psi^{z-1}\cdot \utw\bm\omega).
$$

$\utw^{*z}\bm\omega=(\utw^{*z}\omega_j)_{j=0}^{zN'}$ serves as a generalization of the concept of untwisted sequence for $\bm\omega^{*z}$: informally, the latter sequence is `arranged except that the sequence $(\gamma_t)$ may include repetitions of the same curve'. The construction to get $\utw^{*z}\bm\omega$ from $\bm\omega^{*z}$ is indeed exactly the same as in \S \ref{sub:untwistedsequence}, except that here we do not require the curves $(\gamma_t)_{t=1}^{zr}$ to be distinct. A sequence of subintervals in $[0, zN']$:
$$NI^{\utw*}_0,DI^{\utw*}_1,NI^{\utw*}_1,\ldots, DI^{\utw*}_{zq}, NI^{\utw*}_{zr}$$
is naturally defined. The function $\dn:[0,zN]\rightarrow [0,zN']$, which, for all $1\leq t \leq zr$, maps $DL^*_t$ onto the respective $DI^{\utw*}_t$ and $NI^*_t$ onto the respective $NI^{\utw*}_t$, is defined just as in \S \ref{sub:untwistedsequence}, and so are the maps $\up$, $[t]\dn$, $[t]\up$. For $X\subseteq S$ a subsurface, denote $I^*_X, I^{\utw *}_X$ its accessible interval with respect to the splitting sequences $\bm\omega^{*z}$, $\utw^{*z}\bm\omega$ respectively.

For $j\in [0,zN]$, let $t(j)$ be the least index $t$ such that $j\in DI^*_t$ or $j\in NI^*_t$. For $j'\in [0,zN']$, let $\utw^* t(j')$ be the least index $t$ such that $j\in DI_t^{\utw *}$ or $j\in NI_t^{\utw *}$. There are diffeomorphisms $\phi^*_t$ for $1\leq t\leq zq$ such that, for all $j\in \bigcup_{t=0}^{zr} (DL^*_t\cup NI^*_t)$, $\utw^{*z}\omega_{\dn j}=\phi_{t(j)}(\omega^{*z}_j)$; and for all $j\in[0,N']$, $\utw^{*z}\omega_{j'}=\phi_{\utw^* t(j')}(\omega^{*z}_{\up j'})$. 

All claims in Lemma \ref{lem:untwistedsubsurfaces} work also in this setting (replace $\bm\tau$ with $\bm\omega^{*z}$, $\utw\bm\tau$ with $\utw^{*z}\bm\omega$, all index intervals with the starred versions defined here): this is because the proof of that lemma makes no use of the fact that the curves $\gamma_t$ were distinct in its original setting. Similarly, the terminology introduced in Corollary \ref{cor:subsurface_bijection} is immediately adapted so that it will work here.

And using that terminology, if one defines the sequence $\utw^*\gamma_1\coloneqq \phi^*_{\nei(\gamma_1)}(\gamma_1)$, \ldots, $\utw^*\gamma_{zr}\coloneqq \phi^*_{\nei(\gamma_{zr})}(\gamma_{zr})$, each $\utw^{*z}\bm\omega(DI^{\utw*}_t)$ has twist nature with respect to $\utw^*\gamma_t$. Two $\gamma_t,\gamma_{t'}$ coincide (up to isotopy) if and only if $\utw^*\gamma_t,\utw^*\gamma_{t'}$ do. So $\utw^{*z}\bm\omega$ is $(\utw^*\eta_1,\ldots,\utw^*\eta_Q)$-arranged.

\step{4} we claim a modified version of Proposition \ref{prp:locallyfinite}:
\begin{claim}
There is a constant $C'_5(S)$ such that the following is true.

Let $X\in \Sigma(\bm\omega^{*z})$, $X$ not an annulus, and let $[k,l]\subseteq [0,zN]$, with $[k,l]\subseteq I^*_X$ if $X\not=S$. Then
$$
d_{\pa(\utw^* X)}(\dn k,\dn l)^{\utw*} \leq C'_5 \left(d_{\pa(X)}(k,l)\right)^2
\text{ and }
d_{\pa(X)}(k,l) \leq C'_5 \left(d_{\pa(\utw^* X)}(\dn k,\dn l)^{\utw*}\right)^2.
$$

There are two increasing functions $\Psi''_S, \Psi'''_S:[0,+\infty)\rightarrow [0,+\infty)$ such that
$$
d_{\ma(\utw^* X)}(\dn k,\dn l)^{\utw*} \leq \Psi''_S\left(d_{\pa(\utw^* X)}(\dn k,\dn l)^{\utw*}\right) \leq \Psi'''_S\left( d_{\pa(X)}(k,l) \right).
$$
\end{claim}

With the notation $d_{\pa(X)}(\cdot,\cdot)$, this time, we measure pants distances along the sequence $\bm\omega^{*z}$; and with $d_{\pa(\utw^* X)}(\cdot,\cdot)^{\utw*}$, distances along $\utw^{*z}\bm\omega$.

The proof of this fact follows the proof of Proposition \ref{prp:locallyfinite}, with few modifications. As explained above, the intervals which index the sequences $\bm\omega^{*z}$ and $\utw^{*z}\bm\omega$ admit a subdivision with respect to the sequences of curves $(\gamma_t)_{t_1}^{zr}$, $(\utw^*\gamma_t)_{t=1}^{zr}$, respectively, similarly to arranged splitting sequences.

In Step 1 of the original proof several definition were given, which we repeat without substantial modifications here. The only thing that needs to be clarified is why the number $q$ of curves to be considered is bounded by the pants distance, since those curve may appear multiple times in the current setting.

For each $1\leq u\leq Q$ let $\langle u\rangle$ be a choice of $t$ such that $\gamma_t=\eta_u$ and, if possible, such that $DI^*_{\langle u\rangle}\subseteq [k,l]$. Consider $\bm\omega$ as $(\eta_1,\ldots,\eta_Q)$-arranged (possibly reindexing these curves). 
As a consequence of Proposition \ref{prp:tcbound}, then, the number of indices $u$ such that $\eta_u\subseteq X$ and $DI^*_{\langle u\rangle}\subseteq [k,l]$ is bounded from above by $C_3 d_{\pa(X)}(k,l)+ C_4$.

So, as we define $\delta_1=\gamma_{t_1},\ldots,\delta_q=\gamma_{t_q}$ to be the curves (here listed with repetitions allowed) such that $DI^*_{t_s}\cap [k,l] \not=\emptyset$ and $\gamma_t$ intersects $X$ essentially, we find that the ones such that $DI^*_{t_s}\subseteq [k,l]$ are disjoint from $\partial X$ (by claim 1 in Lemma \ref{lem:untwistedsubsurfaces} adapted to this setting). Therefore they are exactly the curves $\eta_u$ with $DI^*_{\langle u\rangle}\subset[k,l]$, each counted at most $N_0+1$ times. Moreover, it is only for $s=1,q$ one may have simultaneously $DI^*_{t_s} \setminus [k,l], DI^*_{t_s}\cap [k,l]\not=\emptyset$. So $q\leq c \cdot d_{\pa(X)}(k,l)$ for a suitable $c=c(S)$.

The same ideas may be applied to bound $q$ in terms of pants distance in $\utw^{*z}\bm\omega$. Under the correspondence $\utw^*$ given by Corollary \ref{cor:subsurface_bijection} in this modified setting, the indices $t_s$, $1 \leq s \leq q$, turn out to be almost exactly the values of $t$ such that $\utw^*\gamma_t$ intersects $\utw^* X$, $DI^{\utw*}_{t_s}\cap [\dn k,\dn l] \not=\emptyset$. We say `almost', because the only exception to this last sentence is that $\utw^*\gamma_{t_q}$ may have $DI^{\utw*}_{t_q}\cap [\dn k,\dn l] =\emptyset$. So it is also true that $q\leq c \cdot d_{\pa(\utw^* X)}(\dn k,\dn l)^{\utw*}$, if $c=c(S)$ is chosen suitably.

After these modifications, reconstructing Step 2 in the proof of Proposition \ref{prp:locallyfinite} is straightforward.

Following Step 3 of that proof, we wish to prove the existence of constants depending on $S$ such that, however one picks $\alpha\in\cc(X)$ with $\nei=\nei(\alpha)$ a regular neighbourhood, for all $0\leq s\leq q$, and all $j,j'\in XDL^*_s\cup XNI^*_s$, $d_\nei(\omega^{*z}_j|X,\omega^{*z}_{j'}|X)$ is bounded by this constant. Let $I^*_\nei$ be the accessible interval for $\nei$ in the sequence $\bm\omega^{*z}$.

For each $0\leq i\leq z$, let $[a(i),b(i)]=H_i\coloneqq I^*_\nei\cap [j,j']\cap [iN, (i+1)N]$. Step 2 of this proof implies immediately that there are at most $N_0+1$ values of $i$ such that $H_i\not=\emptyset$, and they are consecutive --- let $\Lambda$ be their family. Then one may write $[j,j']=J_-\cup \left(\bigcup_{i\in\Lambda}H_i\right)\cup J_+$ where $J_-, J_+$, if nonempty, are intervals sharing only one element with $I^*_\nei$.

As all $\bm\omega^i$ are effectively arranged, the original proof of Proposition \ref{prp:locallyfinite} provides bounds for $d_\nei(\omega^{*z}_{a(i)}|X,\omega^{*z}_{b(i)}|X)$, for all $i\in \Lambda$. For what concerns $J_-\eqqcolon[j,b]$, Theorem \ref{thm:mmsstructure} gives that $d_\nei(\omega^{*z}_{j},\omega^{*z}_{b})\leq \mathsf{K}_0$. Similarly for $J_+$. So we have a bound for the distance covered in each of the pieces in which we have split $[j,j']$, and these pieces are at most $N_0+3$: hence the existence of a constant, $c'$, bounding $d_\nei(\omega^{*z}_j|X,\omega^{*z}_{j'}|X)$ from above.

The remainder of Step 3, and Step 4, of Proposition \ref{prp:locallyfinite} may be applied here with no substantial modifications, thus completing the proof of our claim.

\step{5} The sufficient condition declared in Step 1 above holds.

Step 3 above acts in place of Proposition \ref{prp:locallyfinite} to prove the following version of Proposition \ref{prp:hardttbound}, using exactly the same line of proof.
\begin{claim}
There is a constant $C'_7(S)$ such that
$$d_{\pa(S)}(V(\omega^{*z}_0),V(\omega^{*z}_0)))\geq_{C'_7}|\utw^{*z}\bm\omega|.$$
\end{claim}

Finally, just note that $V(\omega^{*z}_0)=V(\rho_0)=V(\tau_m)$; $V(\omega^{*z}_0)=V(\tau_{m+zn})$;\linebreak $|\utw^{*z}\bm\omega|=z|\utw\bm\omega|= z|\utw(\rar\bm\rho)|$.
\end{proof}

\section{Braids in the solid torus}\label{sec:dw}

Among mapping tori, the ones coming from \emph{braids} in the way we are about to describe admit an application of the train track machinery different from the one given in the previous section. See \cite{farb}, \S 9.1 for details about braids and interpretations of braid groups.

Recall that the \nw{braid group} on $n$ strands,
\begin{equation}\label{eqn:braidgroup}
B_n\coloneqq \left\langle\sigma_1,\ldots,\sigma_{n-1}\left|
\begin{array}{lr}
\sigma_i\sigma_j=\sigma_j\sigma_i & \text{for }|i-j|\geq2;\\
\sigma_i\sigma_{i+1}\sigma_i=\sigma_{i+1}\sigma_i\sigma_{i+1} & \text{for }1\leq i\leq n-1
\end{array}
\right.\right\rangle
\end{equation}
has a natural identification with $\mcg(D^2_n)$, where with $D^2_n$ we mean the closed disk, punctured $n$ times. Here we identify $D^2_n$ with 
$$
\left\{z\in \R^2\left| \|z\|\leq 1\right.\right\}\setminus \left\{\left.\left(-1+\frac{2}{n+1}j,0\right) \right|j=1,\ldots,n\right\}.
$$
This model, in particular, places all punctures along the horizontal axis $\R\times\{0\}$; and, for $1\leq i\leq n-1$, the generator $\sigma_i$ of $B_n$ corresponds to a \emph{half-twist} in $D^2_n$, swapping the $i$-th and the $(i+1)$-th punctures counting from the left.

From now on, it will always be assumed that $n\geq 3$, which implies $\xi\left(\inte(D^2_n)\right)\geq 4$ i.e. $\inte(D^2_n)$ is a surface in the sense we have stuck with in all the previous work. For simplicity, we will identify $\R$ with $\R\times\{0\}$ and similarly for points and intervals in the two sets.

If $\psi\in B_n\cong \mcg(D^2_n)$ has a restriction to $\inte(D^2_n)$ which is pseudo-Anosov, then we know that the mapping torus $M\coloneqq \faktor{\inte(D^2_n)\times [0,1]}{\sim_\psi}$ is hyperbolic (see Theorem \ref{thm:mappingtorushyperbolic}). This mapping torus actually admits a (diffeomorphic) embedding in $\R^3$ as follows: if $w$ is any word representing $\psi$ in $B_n$, let $\ul{\ul w}$ be a braid representation of the word $w$, embedded in $\inte(D^2)\times [0,1]$ with the property that $\ul{\ul w}\cap \left(\inte(D^2)\times\{0,1\}\right)= \left\{\left.\left(-1+\frac{2}{n+1}j,0\right) \right|j=1,\ldots,n\right\}\times\{0,1\}$.

The identification of each point of $\inte(D^2)\times\{0\}$ with the corresponding point of $\inte(D^2)\times\{1\}$ (via the identity map $\inte(D^2)\times\{0\}\rightarrow \inte(D^2)\times\{1\}$) produces a solid torus $T\cong \inte(D^2)\times\mathbb S^1$, which admits an embedding $T\hookrightarrow \R^3$ (so we identify it with a chosen embedding). Through this quotient, $\ul{\ul w}$ projects to a \nw{closed braid} $\ul w$ in $\R^3$ which is fitted to $T$: each strand of $\ul{\ul w}$ will project to a path or loop which intersects the image of $\inte(D^2)\times\{t\}$ only once for each $0<t<1$. Our mapping torus $M$ is seen to be diffeomorphic to $T\setminus\ul w$.

\subsection{Strip decompositions}

We give here a construction very similar to the one introduced in \cite{dynnikovwiest}, even if our definitions will have a more `visual' flavour which will make it easier to relate them with train track splitting sequences (see Figure \ref{fig:stripdecomposition}).

\begin{figure}
\def\svgwidth{.65\textwidth}
\centering{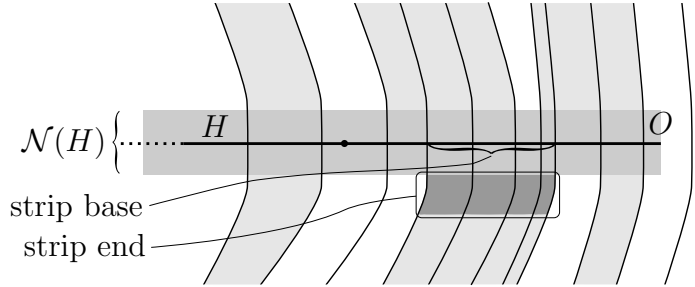}
\caption{\label{fig:stripdecomposition}The basic construction of a strip decomposition, with a marker $O$ and a cutter $H$. A puncture of $D^2_n$ lies along the cutter. Strips are filled in light grey.}
\end{figure}

Let $p$ be (the union of all curves in) a pants decomposition on $D^2_n$; suppose that the curves of $p$ are realized in a way that intersects $\R$ transversely, and such that $p\cup\R$ bounds no bigons. Fix $O\in \R$, which we call a \nw{marker}; and let $H\coloneqq (-\infty,O]$ be the corresponding \nw{cutter}. Fix $\nei(\R)\coloneqq\left(\R \times (-\epsilon,\epsilon)\right)$, where $\epsilon>0$ is chosen so that $p\cap \R \times (-2\epsilon,2\epsilon)$ consists of a set of arcs joining the two opposite boundary components $\R\times\{-\epsilon\}$, $\R\times\{\epsilon\}$. We can suppose, up to isotopies, that these arcs are all vertical. Define also $\nei(H)\coloneqq H \times (-\epsilon, \epsilon)$.

The \nw{strip decomposition} $\beta(p,O)$ is then defined as follows.

Let $AC(p, H)$ be the set of all connected components of $p\setminus\nei(H)$. In general, $AC(p,O)$ will consist of arcs and loops: we subdivide $AC(p,O)=A(p,O)\sqcup C(p,O)$ accordingly. We say that two arcs $\alpha_1,\alpha_2\in A(p,O)$ are \nw{consecutive} if there is a closed region $R=R(\alpha_1,\alpha_2)\subseteq D^2_n$, diffeomorphic to a rectangle, such that $\alpha_1,\alpha_2$ are two opposite sides of $\partial R$, the other two sides are two intervals along $\partial\bar\nei(\R)$, and $\inte(R)\cap p=\emptyset$.

We say that two arcs in $A(p,O)$ are \nw{parallel} if they are in the same class under the equivalence relation generated by consecutiveness: an equivalence class for parallelism will be called a \nw{strip}. The \nw{width} of a strip is its size as a set. The strip decomposition $\beta(p,O)$ is then the disjoint union of $C(p,0)$ with the set of all strips in $A(p,O)$.

If $s\in \beta(p,O)$ is a strip, then 
$$
R(s)\coloneqq \left(\bigcup_{\alpha\in s}\alpha\right)\cup \left(\bigcup_{\substack{\alpha_1,\alpha_2\in s\\\text{consecutive}}}R(\alpha_1,\alpha_2)\right)
$$
is again diffeomorphic to a rectangle, and the elements of $s$ are then uniquely defined from $R(s)$ as the connected components of $R(s)\cap p$. With this in mind, we may confuse a strip $s$ with the corresponding $R(s)$.

In particular, each strip $s$ has two \nw{bases}, i.e. the two intervals $I_1,I_2\in \R$ such that $I_1\times\{\epsilon_1\}$ and $I_2\times\{\epsilon_2\}$ are the two connected components of $R(s)\cap \partial\bar\nei(H)$ for a suitable (unique) choice of $\epsilon_1,\epsilon_2\in\{\pm\epsilon\}$, and two \nw{ends}, i.e. the two connected components of $R(s)\cap (H\times\left(-2\epsilon,2\epsilon)\right)$. 

We say that two strip ends $e_1,e_2$, belonging to either the same strip in $\beta(p,O)$ or distinct ones, \nw{overlap} if the corresponding strip bases $I_1,I_2$ have $I_1\cap I_2\not=\emptyset$. The two ends are said to \nw{completely overlap} if $I_1=I_2$. If $e_1,e_2$ completely overlap and belong to the same strip $s$, then $s$ consists of a single arc which, together with a vertical arc in $\nei(H)$, makes up a curve in $p$: we cannot have more than one arc, else $p$ contains a pair of isotopic curves. In all other cases, the bases of $s$ can be distinguished into a \nw{left} and a \nw{right} one, according to the relative order of their minima in $\R$. The same terminology applies to the two ends of $s$.

A \nw{strip cut} is an elementary move on a strip decomposition, defined as follows.

There are exactly two strip ends $e_1,e_2$ which are `closest' to the marker $O$, i.e. such that the corresponding bases $I_1,I_2$ have $\max I_1=\max I_2$ and $p\cap\R \cap (\max I_1,O]=\emptyset$. Place the indices so that $I_1$ is shorter, or equal, to $I_2$. Let $O'\coloneqq \max \left((\R\cap p)\setminus I_1\right)$: then $O'<O$. We define the strip cut of $\beta(p,O)$ to be the strip decomposition $\beta(p,O')$. A \nw{strip cutting sequence} is a sequence of strip decompositions, each obtained from the previous with a strip cut: an example is given in Figure \ref{fig:stripcut}.

\begin{figure}
\includegraphics[width=\textwidth]{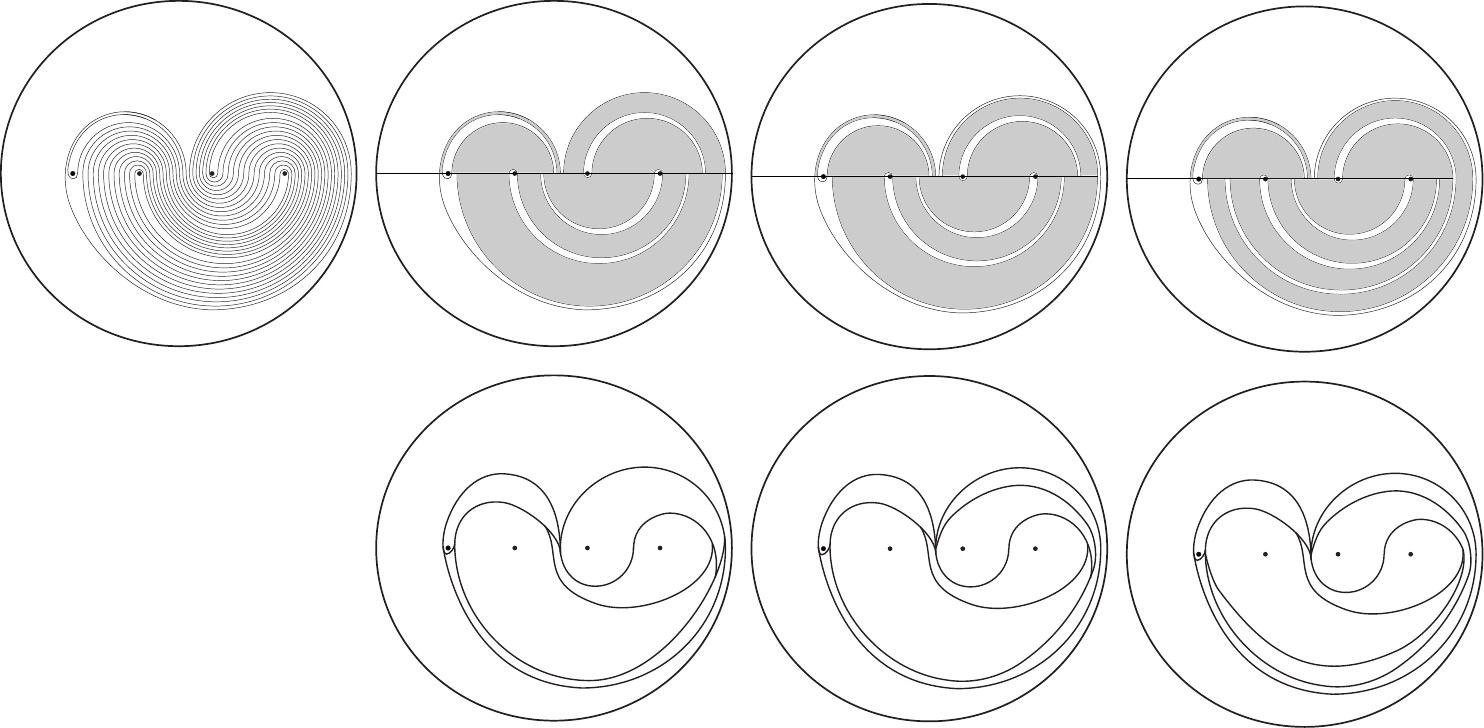}
\caption{\label{fig:stripcut}In the upper line is an example of strip cutting sequence obtained from the pants decomposition drawn in the leftmost picture. In the lower line are the train tracks obtained accordingly. Note how the cutter (the horizontal line) in the upper pictures gets shorter and shorter. (This figure has been derived from Figure 1 in \cite{dynnikovwiest}. Many thanks to Bert Wiest for having kindly agreed to its reuse.)}
\end{figure}

It is convenient to have a closer look at what happens with a strip cut. When $e_1,e_2$ belong to one same strip $s$, then necessarily $I_1=I_2$ and they must consist of a single point, else $p$ includes two isotopic curves --- similarly to what has been noted above. In this case, $\beta(p,O)$ differs from $\beta(p,O')$ only in that a strip in the former, consisting of a single arc, has been replaced with a loop.

When the strips $s_1,s_2$ to which $e_1,e_2$ belong are different, it may still be the case that $I_1=I_2$ i.e. $s_1,s_2$ contain the same number of arcs. Then the move's effect is that $\beta(p,O)\setminus\{s_1,s_2\}=\beta(p,O')\setminus\{s'\}$ for a strip $s'$ such that $R(s)=R(s_1)\cup R(s_2)\cup \left(I_1\times (-\epsilon,\epsilon)\right)$. We say that $s_1$ and $s_2$ \nw{merge} to $s'$.

If $I_1\subsetneq I_2$ instead, there are two (unique) strips $s'_1$, $s'_2$ such that $\beta(p,O)\setminus\{s_1,s_2\}=\beta(p,O')\setminus\{s'_1,s'_2\}$ and the following is true. Assign each arc $\alpha\in s_2$ to a family $s_{21}$ or $s_{22}$ according to whether $\alpha\cap e_2$ is contained in $I_1\times (-2\epsilon,2\epsilon)$ or not, respectively. Define $R(s_{21}),R(s_{22})$ exactly at it has been done above for a strip: they are again two rectangles. Then $R(s'_1)=R(s_1)\cup R(s_{21})\cup \left(I_1\times (-\epsilon,\epsilon)\right)$ while $s'_2=s_{22}$ --- and in particular $R(s'_2)=R(s_{22})$. We say that $s_1$ \nw{stretches} to $s'_1$ while $s_2$ \nw{shrinks} to $s'_2$.

%In a strip cutting sequence $\bm\beta=(\beta_j)_{j=0}^N$, a selection of strips $(s_j\in\beta_j)_{j=k}^l$ for $0\leq k\leq l\leq N$ is a \nw{line of descent} if, for all $k\leq j < l$, $s_j$ is equal, stretches, shrinks or merges with another strip to $\beta_{j+1}$.

%If $J=[k,l]\subseteq [0,N]$, for a strip cutting sequence $\bm\beta$ we adopt the notation $\bm\beta(k,l)=\bm\beta(I)=(\beta_k,\ldots,\beta_l)$ as done for splitting sequences.

%A \nw{transmission} in $\bm\beta$ is a subsequence $\bm\beta(k,l)$ if it features a line of descent $(s_j)_{j=k}^l$ such that either:
%\begin{itemize}
%\item $s_j$ shrinks to $s_{j+1}$ for all $k\leq j <l$;
%\item $s_j$ shrinks to $s_{j+1}$ for all $k\leq j <l-1$ and $s_{l-1}$ merges, with another strip, to $s_l$
%\end{itemize}
%and it is maximal i.e. there is no subsequence $\bm\beta(k',l')$, for $0\leq k'\leq k \leq l\leq l'\leq N$ and $l-k<l'-k'$, with the same property. This notion of transmission coincides with the one given in \cite{dynnikovwiest}.

Given a pants decomposition $p$, there is a \nw{canonical} strip cutting sequence $\bm\beta^p$ defined from $\beta(p,1)$ and performing strip cuts until it is no longer possible (i.e. there is no strip left).

\subsection{Strip decompositions turn into train tracks}

Strip cutting sequences are, morally, a particular case of train track splitting sequences. Given a strip decomposition $\beta=\beta(p,O)$, we define a semigeneric train track $\trk\beta$ as follows: let $G\subseteq AC(p,O)$ be a collection consisting of all elements of $C(p,O)$, and exactly one arc $a(s)$ for each strip $s\in\beta(p,O)$. If $e$ is a strip end, let $a(e)$ be the only endpoint of $a(s)$ which is contained in $e$. Let $E$ be the family of all pairs of overlapping strip ends. Finally, let $\Gamma\subset D^2_n$ be a $1$-complex obtained as the union of the elements of $G$, plus a straight segment $a(e_1,e_2)$ joining $a(e_1)$ to $a(e_2)$, for each pair $(e_1,e_2)\in E$. If two of these new segments intersect each other, or one intersects an element of $G$, then the point they share is an endpoint for both arcs. Finally, homotope $\Gamma$ to make sure that, for each strip $s$, each end $e$ of $s$, and each pair $(e,e')\in E$ for $e'$ another strip end, $a(s)\cup a(e,e')$ is smoothly embedded.

Call $\trk\beta$ the result of this operation: it is a pretrack. 

\begin{lemma}\label{lem:birecurrent}
If $\beta=\beta(p,O)$ is a strip decomposition, then $\trk\beta$ is a birecurrent train track.
\end{lemma}
\begin{proof}[Sketch]
We only give a sketch of the argument proving that $\trk\beta$ is a train track. One has to make sure that $S\setminus \nei_0(\trk\beta)$ includes no connected component which is a:
\begin{itemize}
\item disc with smooth boundary: the existence of one would imply that $p$ includes a homotopically trivial curve;
\item 1-punctured disc with smooth boundary: the existence of one would imply that $p$ includes a curve homotopic into a puncture;
\item monogon: the existence of one would imply that $p$ forms a bigon with $\R$, which was excluded at the beginning of the construction;
\item bigon: the existence of one would imply that $p$ has two distinct, isotopic components.
\end{itemize}

$\trk\beta$ is recurrent because each branch is traversed by a connected component of $p$. As for transverse recurrence, it suffices to exhibit a collection of curves in $\trk\beta$ such that, for each branch $\in\br(\trk\beta)$, there is a curve in the collection which can be put in dual position with respect to $\trk\beta$, intersecting $b$. Recall the notation $a(s), a(e_1,e_2)$ used above for smooth segments in $\trk\beta$ --- which are parts of branches, but not necessarily \emph{entire} ones.

\begin{itemize}
\item For each pair of consecutive punctures $\left(-1+\frac{2}{n+1}j,0\right), \left(-1+\frac{2}{n+1}(j+1),0\right)$, $0\leq j\leq n-1$, include in the collection the round curve encircling them. There is a realization of this curve, dual to $\trk\beta$, which intersect all segments $a(e_1,e_2)$ for $e_1$, $e_2$ strip ends located between the two punctures. And there is a different realization which will intersect all segments $a(s)$ for $s$ a strip with an end located between the two punctures. Both realizations intersect the elements of $C(p,O)$ which pass between the two given punctures.
\item Include the curve encircling all punctures but the leftmost one $\left(-1+\frac{2}{n+1},0\right)$. More precisely, realize it in a way that encircles all $D^2_n$ except for a small neighbourhood of $\partial D^2_n$ and of $\left[-1, -1+\frac{2}{n+1}+\epsilon\right]\times \{0\}$. Depending on what particular realization of the curve has been chosen, the curve will be dual to $\trk\beta$ and will intersect all segments $a(e_1,e_2)$ for $e_1$, $e_2$ strip ends located to the left of the leftmost puncture; or all segments $a(s)$ for $s$ a strip with an end in the same locations. Both realizations all the elements of $C(p,O)$ which pass to the left of the leftmost puncture.
\item Similarly, include the curve encircling all punctures but the rightmost one $\left(-1+\frac{2}{n+1}n,0\right)$.
\end{itemize}

Figure \ref{fig:dualcurves} gives a local picture for these curves. In order to show that these curves are actually in efficient position with respect to $\trk\beta$, one may apply an argument similar as the one applied above to show that $\trk\beta$ is a train track.
\end{proof}

\begin{figure}[h]
\centering
\includegraphics[width=.7\textwidth]{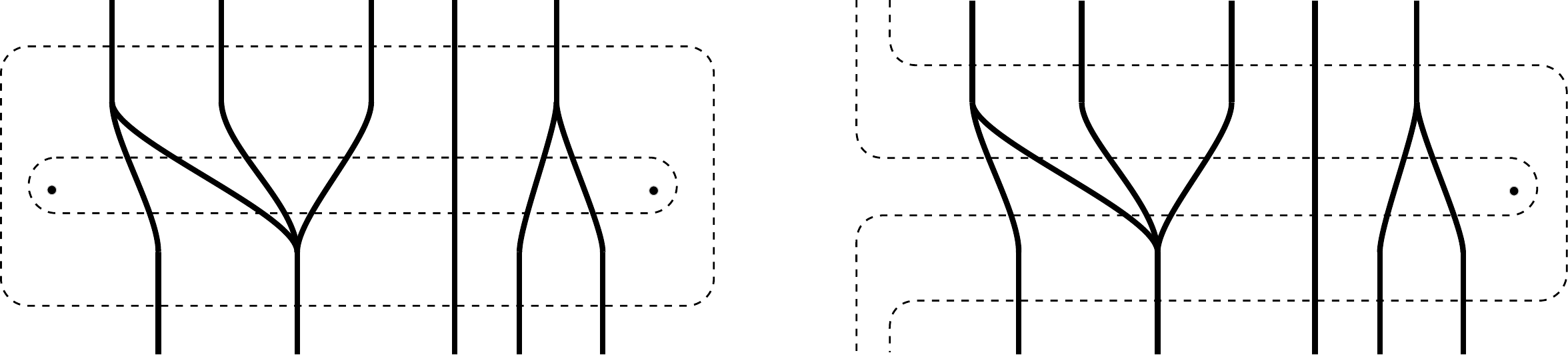}
\caption{\label{fig:dualcurves}The curves showing transverse recurrence of $\trk\beta$ as in Lemma \ref{lem:birecurrent}. The picture to the left shows, dashed, a round curve encircling two consecutive punctures of $D^2_n$, and how it can be isotoped to be dual to $\trk\beta$, and intersect any of the branches located `between' the two punctures. The picture to the right shows that a similar property holds for the round curve encircling all punctures of $D^2_n$  but the leftmost one --- embedded in two different ways, as prescribed in the proof of that lemma.}
\end{figure}

\begin{figure}
\centering
\def\svgwidth{.8\textwidth}
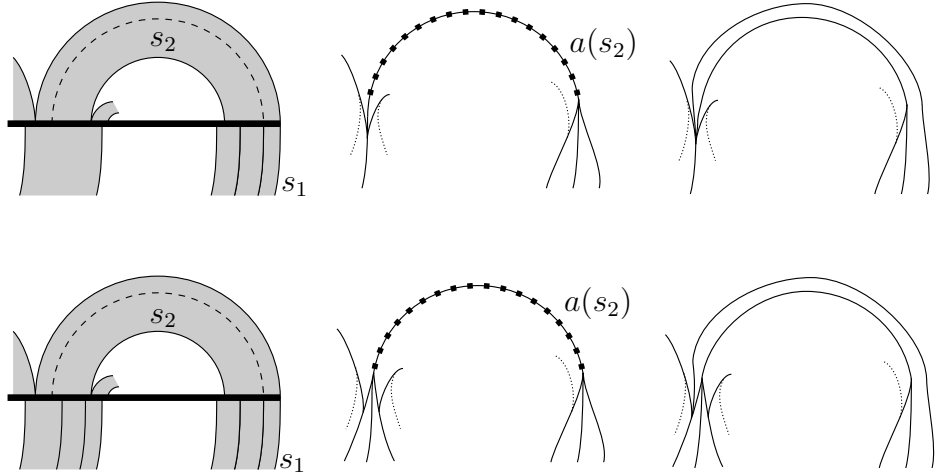
\caption{\label{fig:cutissplit}When $\beta$ strip cuts to $\beta'$ and the strip cut causes a strip $s_1$ to stretch and another one, $s_2$ to shrink, either $\trk\beta'$ is comb equivalent to $\trk\beta$ (line above), or there is a sequence of elementary moves, which are all comb/uncomb moves except for one split, turning $\trk\beta$ to $\trk\beta'$. The first case occurs when the left base of $s_2$ is entirely contained in the base of another strip end, and the same case otherwise. The split parity, in this latter case, is determined by the strip widths before and after the cut.}
\end{figure}

How do $\trk\beta$ and $\trk\beta'$ relate, when $\beta'$ is a strip decomposition obtained from $\beta$ with a split?

When the strip cut causes the replacement of a strip containing a single arc with a loop, or causes two strips to merge, $\trk\beta=\trk\beta'$. Otherwise, there is a strip $s_1$ which stretches and another one, $s_2$, which shrinks. Let $e_1,e_2$ be the right ends of the two. The strip cut reflects on $\trk\beta$ the following way: first of all, let $\kappa$ be a zipper in $\trk\beta$ such that $\kappa_P$ begins at $a(e_2)$ and runs along $a(s_2)$. Unzip $\trk\beta$ along $\kappa$: this gives a track which is comb equivalent to $\trk\beta$. Now, in order to get $\trk\beta'$, another elementary move is needed, which is either a comb or a split. Some more detail is given in Figure \ref{fig:cutissplit}.

This means that, for a strip cutting sequence $\bm\beta=(\beta_j)_{j=0}^N$, the corresponding sequence $(\trk\beta_j)_{j=0}^N$ may be completed to a splitting sequence by deleting any redundant entry and then possibly inserting, immediately after each $\trk\beta_j$, a train track with is obtained from it with an uncomb move. We call $\trk\bm\beta$ the splitting sequence thus obtained.

\begin{defin}
A pants decomposition $p$ of $D^2_n$ is \nw{round} if each connected component in the pants decomposition intersects $\R$ in exactly 2 points.
\end{defin}

\begin{coroll}\label{cor:dwgivesdistance}
There is a constant $A_1=A_1(n)$ such that, if $r$ is a round pants decomposition in $D^2_n$ and $\psi\in B_n\cong \mcg(D^2_n)$, then
$$d_{\pa(D^2_n)}\left(r,\psi(r)\right)=_{A_1}|\utw\left(\rar\left(\cnr(\trk\bm\beta^{\psi(r)})\right)\right)|.$$
\end{coroll}

Note that, in this statement, we are letting it be understood that $\cnr(\trk\bm\beta^{\psi(r)})$ must be turned into a generic splitting sequence in order for $\rar$ and $\utw$ to make sense (at least if we wish to rely on the approach of \S \ref{sec:twistcurves}, \ref{sec:traintrackconclusion}, which was meant for generic splitting sequences only).

\begin{proof}
For simplicity, let $\trk\bm\beta^{\psi(r)}\eqqcolon \bm\tau =(\tau_j)_{j=0}^N$. Theorem \ref{thm:main_full} gives
$$|\utw\left(\rar\left(\cnr\bm\tau\right)\right)|=_{C_9} d_{\pa(D^2_n)}\left(V(\tau_0),\psi(r)\right).$$
Note that the number of strips in $\beta\left(\psi(r),1\right)$ is bounded in terms of $n$ --- a strip $s$ is uniquely determined by whether it lies above or below $\R$, and by the punctures of $D^2_n$ the two ends of $s$ lie between. This implies that the number of distinct possibilities for $\tau_0$ is bounded in terms of $n$, so there is a bound for $d_{\pa(D^2_n)}\left(r,V(\tau_0)\right)$. This proves our claim.
\end{proof}

\begin{coroll}\label{cor:dwgivesvolume}
There is a constant $A_2=A_2(n)$ such that the following is true. Let $r$ be a round pants decomposition in $D^2_n$ and $\psi\in B_n\cong \mcg(D^2_n)$ such that $\psi$ defines a pseudo-Anosov mapping class on $\inte(D^2_n)$. Let $M\coloneqq \faktor{\inte(D^2_n)\times [0,1]}{\sim_\psi}$ be the related mapping torus. Then, defining $r(m)\coloneqq \psi^m(r)$,
$$
\vol(M) =_{A_2}\limsup_{m\rightarrow+\infty} \frac{1}{m} |\utw\left(\rar\left(\cnr(\trk\bm\beta^{r(m)})\right)\right)|
$$
and also
$$
\vol(M) =_{A_2} \min_{\phi\in \mathrm{Conj}(\psi)} |\utw\left(\rar\left(\cnr(\trk\bm\beta^{\phi(r)})\right)\right)|
$$
where $\mathrm{Conj}(\psi)$ is the conjugacy class of $\psi$ in $B_n$.
\end{coroll}
The same considerations as the ones after the statement of Corollary \ref{cor:dwgivesdistance} apply.
\begin{proof}
We use a simplified notation $s(\nu)\coloneqq |\utw\left(\rar\left(\cnr(\trk\bm\beta^{\nu(r)})\right)\right)|$ for $\nu\in\mcg(D^2_n)$.

As a consequence of Corollary \ref{cor:dwgivesdistance} plus Lemma \ref{lem:pantsquasiisom}, there is a constant $A'_1=A'_1(n)$ such that, for each $m>0$, one has $d_{\pc(S)}\left(r,r(m)\right)=_{A'_1}s(\psi^m)$.

So $|\psi|^{st}=_{A'_1} \limsup_{m\rightarrow+\infty} \frac{1}{m} s(\psi^m)$ and, via Remark \ref{rmk:stable_dist_pc}, there is a further constant $A''_1=A''_1(n)$ such that $|\psi|=_{A''_1} \limsup_{m\rightarrow+\infty} \frac{1}{m} s(\psi^m)$.

Moreover, let $p$ be a pants decomposition such that $d_{\pc(S)}\left(p,\psi(p)\right)=|\psi|$. There is a $\lambda\in \mcg(D^2_n)$ such that $\lambda(p)$ is round. The pants decompositions $\lambda(p),r$ both belong to the finite set in $\pc(D^2_n)$ consisting of round pants decompositions, so there is an upper bound $d_{\pc(S)}\left(\psi(p),\psi\circ \lambda^{-1}(r)\right)=d_{\pc(S)}\left(p,\lambda^{-1}(r)\right)=d_{\pc(S)}\left(\lambda(p),r\right)\leq c$ depending only on the number of strands $n$.

Thus $d_{\pc(S)}\left(r,\lambda\circ \psi \circ \lambda^{-1} (r)\right)= d_{\pc(S)}\left(\lambda^{-1}(r), \psi \circ \lambda^{-1} (r)\right)\leq d_{\pc(S)}\left(\lambda^{-1}(r), p\right) + \linebreak d_{\pc(S)}\left(p, \psi(p)\right) + d_{\pc(S)}\left(\psi(p), \psi \circ \lambda^{-1} (r)\right) \leq |\psi|+2c$, while $d_{\pc(S)}\left(r,\lambda\circ \psi \circ \lambda^{-1} (r)\right) =_{A'_1}
s(\lambda\circ \psi \circ \lambda^{-1})$ from Corollary \ref{cor:dwgivesdistance} above.

On the other hand, $|\psi|\leq d_{\pc(S)}\left(\nu^{-1}(r), \psi \circ \nu^{-1} (r)\right)$ however $\nu\in \mcg(D^2_n)$ is chosen. Combining the two bounds, $|\psi|=_{A''_1} \min_{\phi\in \mathrm{Conj}(\psi)} s(\phi)$. Here we choose, for simplicity, an $A''_1=A''_1(n)$ such that both this coarse equality and the one proved above hold.

Theorem \ref{thm:brockmappingtori} concludes.
\end{proof}

\subsection{Bounds without train tracks (sketch)}

The splitting sequences arising as $\trk\bm\beta$ have several good properties, which we state in Proposition \ref{prp:dwcutnumber}. Due to these properties, it is not really necessary to switch to the train track formalism and employ a derived splitting sequence such as\linebreak $|\utw\left(\rar\left(\cnr(\trk\bm\beta^{\psi(r)})\right)\right)|$ to have Corollaries \ref{cor:dwgivesdistance} and \ref{cor:dwgivesvolume}. As a complete proof of Proposition \ref{prp:dwcutnumber} would go partly beyond the scope of this thesis, we only outline it.

We need a definition. Let a strip decomposition $\beta(p,M)$ be given, with the property that the strip cut on this decomposition will result in a strip $s_1$ to stretch, and another one $s_2$ to shrink to a strip $s'_2$. Let $e$ be the right end of $s_2$, $e'$ be the left one. If the two strip ends of $s_2$ partially overlap and so do the ones of $s'_2$, then the move is a \nw{spiralling} of $s_2$. Let $I$ be the left base of $s_2$, and let $M'\coloneqq \max I$. Successive strip cuts will turn $\beta(p,M)$ into $\beta(p,M')$, and generate a strip cutting sequence $\bm\varsigma$ with a sequence $(s_2^{(j)})_{j=1}^k$ of strips, one for each decomposition in $\bm\varsigma$, such that $s_2^{(j)}$ shrinks to $s_2^{(j+1)}$ for all $1\leq j<k$. If each cut in $\bm\varsigma$ is a spiralling of the respective $s_2^{(j)}$ --- i.e. if the ends of $s_2^{(k)}$ still overlap --- then we say that $\bm\varsigma$ is a \nw{complete spiralling} of $s_2$. See Figure \ref{fig:complete_spiralling}.

Let $\gamma$ be the only curve of $V(\trk\beta(p,M))$ which traverses only $a(s_2)$ and $a(e,e')$, both exactly once. Then $\gamma$ is a twist curve for $\trk\beta(p,M)$. A spiralling move reflects into an elementary move which is twist\footnote{We have not defined what this means in a semigeneric setting. We may say that an elementary move on a track $\tau$ with respect to a twist curve $\gamma$ is twist if, given $A_\gamma$ a twist collar, the move is the result of unzipping along a zipper $\kappa:[-\epsilon, t]\rightarrow\bar\nei(\tau)$, which does not intersect $\tau.\gamma$ and such that $\kappa([0,t])\subseteq A_\gamma$. Also, we may define the sign of $\gamma$ as the sign $\gamma$ has in any generic track which is comb equivalent to $\tau$.} with respect to $\gamma$, and if $\bm\varsigma$ is a complete spiralling turning $\beta(p,M)$ into $\beta(p,M')$, then $\trk\beta(p,M')=D_\gamma^\epsilon(\trk\beta(p,M))$, with $\epsilon$ the sign of $\gamma$ as a twist curve. We call $\gamma$ the \nw{spiralled curve}.

Given a strip cutting sequence $\bm\beta$, the subsequence $\bm\beta(k,l)$ is a \nw{maximal spiralling} if there is a sequence $(s_2^{(j)})_{j=k}^l$ of strips, one for each decomposition in $\bm\beta(k,l)$, such that $s_2^{(j)}$ shrinks to $s_2^{(j+1)}$ for all $k\leq j<l$, each strip cut in $\beta(k,l)$ is a spiralling of the respective $s_2^{(j)}$, and $\bm\beta(k,l)$ is not contained in any longer subsequence of $\bm\beta$ with the same properties. It may not be possible to subdivide a maximal spiralling into complete spirallings, but note that $\trk\left(\bm\beta(k,l)\right)$ has twist nature about one same spiralled curve $\gamma$.

\begin{figure}
\def\svgwidth{.5\textwidth}\centering
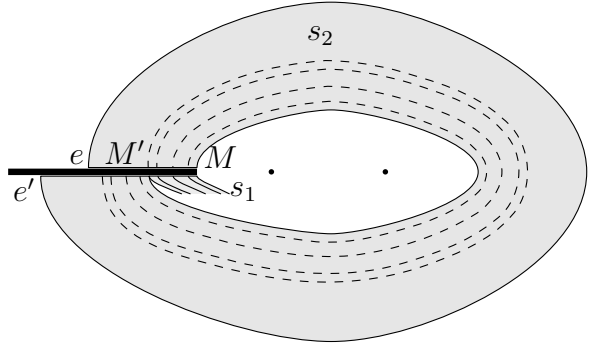
\caption{\label{fig:complete_spiralling}The successive cuts occurring in a complete spiralling.}
\end{figure}

We will sketch the proof of the following Proposition.
\begin{prop}\label{prp:dwcutnumber}
Let $p$ be a pants decomposition of $D^2_n$, and let $(\tau_j)_{j=0}^N=\bm\tau\coloneqq \trk\bm\beta^p$: in particular, $V(\tau_N)=p$. Then there is a constant $A_3=A_3(S)$ such that
$$d_{\pa(D^2_n)}(\cnr\tau_0,\cnr\tau_N),\ d_{\pa(D^2_n)}(\tau_0,\tau_N),\ d_{\pc}(r,p)=_{A_3}\|\bm\beta^p\|,$$
where $\|\bm\beta^p\|$ is the number of strip cuts in the sequence $\bm\beta^p$, renormalized so that each maximal spiralling is counted as only $1$; and $r$ is any round pants decomposition of $D^2_n$.
\end{prop}
A restatement of Corollary \ref{cor:dwgivesvolume} follows from the above Proposition:
\begin{coroll}\label{cor:dwgivesvolume_cutnumber}
There is a constant $A_4=A_4(n)$ such that the following is true. Let $r$ be a round pants decomposition in $D^2_n$ and $\psi\in B_n\cong \mcg(D^2_n)$ be such that $\psi$ defines a pseudo-Anosov mapping class on $\inte(D^2_n)$. Let $M\coloneqq \faktor{\inte(D^2_n)\times [0,1]}{\sim_\psi}$ be the related mapping torus. Then, defining $r(m)\coloneqq \psi^m(r)$,
$$
\vol(M) =_{A_4}\limsup_{m\rightarrow+\infty} \frac{1}{m} \|\bm\beta^{r(m)})\|
$$
and also
$$
\vol(M) =_{A_4} \min_{\phi\in \mathrm{Conj}(\psi)} \|\bm\beta^{\phi(r)})\|
$$
where $\mathrm{Conj}(\psi)$ is the conjugacy class of $\psi$ in $B_n$.
\end{coroll}

The important property of spiralling moves is that they encode almost all elementary moves affecting twist curves:

\begin{lemma}\label{lem:nonspiralbound}
Let $\bm\beta$ be a strip cutting sequence $D^2_n$, and let $\bm\tau\coloneqq \trk\bm\beta$. Let $\gamma\in\cc(D^2_n)$ be a twist curve at some stage along the splitting sequence $\bm\tau$; more precisely let $I_\gamma$ be the accessible interval of $\nei(\gamma)$, and suppose that $k,l$ are the indices such that $\bm\tau(I_\gamma)=\trk\left(\bm\beta(k,l)\right)$.

Then there is a number $A_5=A_5(n)\geq 3$ such that, if more than $A_5$ strip cuts in $\bm\beta(k,l)$ reflect into moves on a train track level which are \emph{not} far from $\gamma$, then the ones after the $A_5$-th are all spirallings with $\gamma$ their spiralled curve (in particular they all occur consecutively).
\end{lemma}

The proof of this lemma is based on the fact that, every time a strip is cut, at least one branch end of the respective $\tau_j$ disappears, or is replaced with a new one located to the left of the old one. Whilst $j$ increases within the interval $[k,l]$, and branch ends get concentrated to the left hand side of $D^2_n$, the number of strips $s$ such that $a(s)\subseteq \tau_j.\gamma$ decreases. When there is only one such $s$, $\gamma$ is certainly a combed curve in the respective $\tau_j$. And, provided that $\gamma$ stays carried, it takes a bounded number of cuts, affecting the respective $\tau_j.\gamma$'s, before we reach that stage.

Some further properties hold:
\begin{itemize}
\item After a maximal spiralling with $\gamma$ its spiralled curve, the next strip cut reflects into $\gamma$ not being carried any longer by the new train track;
\item A spiralling cannot change the subsurface filled by the set of vertex cycles of the related train tracks.
\item Let $\bm\tau=\trk\bm\beta$ for $\bm\beta$ a strip cutting sequence. Then there is a number $A_6$ such that, if a curve $\gamma\in\cc(D^2_n)$ and two indices $j,j'$ are such that $d_{\nei(\gamma)}(\tau_j,\tau_{j'})\geq A_6$, then not only $\gamma$ needs become a twist curve at some stage between $j$ and $j'$, but it undergoes elementary moves which are produced by spiralling.
\item If $\trk\bm\beta(k,l)$ consists of $m>1$ complete spiralling with $\gamma$ the spiralled curve, and $\bm\sigma(k',l')=\cnr\left(\trk\bm\beta(k,l)\right)$, then $\sigma_{l'}=D_\gamma^{\epsilon (m-1)}\sigma_i$ for an index $k'\leq i<l'$ such that $\gamma$ is a twist curve in $\sigma_i$.
\end{itemize}

All this means that:
\begin{lemma}
There is a constant $A_7=A_7(n)$ such that the following is true.

Let $\bm\beta$ be a strip cutting sequence in $D^2_n$, let $(\tau_j)_{j=0}^N=\bm\tau\coloneqq \cnr\left(\trk\bm\beta\right)$, and let $\gamma\in\cc(D^2_n)$ be a curve such that two indices $0\leq k<l\leq N$ exist with $d_{\nei(\gamma)}(\tau_k,\tau_l)\geq A_7$. Then there is an interval $DI_\gamma=[p,q]\subseteq [0,N]$ such that $\gamma$ is a twist curve for all $\tau_j$, $j\in DI_\gamma$, with some sign $\epsilon$, $\tau_q=D_\gamma^{\epsilon m}(\tau_p)$ for some $m\geq 2\mathsf K_0+4$, and the moves in $\bm\tau(DI_\gamma)$ all derive from spirallings where $\gamma$ is the spiralled curve.

Moreover, if $[j,j']\subset [0,N]$ intersects $DI_\gamma$ at an endpoint at most, then\linebreak $d_{\nei(\gamma)}(\tau_j,\tau_{j'})\leq A_6$.

A similar property is true for $\trk\bm\beta$ instead of $\bm\tau$.
\end{lemma}

This means, in particular, that \emph{any subsequence of $\bm\tau$ which evolves firmly in some (not necessarily connected) subsurface of $D^2_n$ is effectively arranged} (see Definition \ref{def:arranged}), possibly changing the constants involved in the definition of effectively arranged sequence.

Subdivide $\bm\tau= \bm\tau^1*\bm\epsilon^2*\bm\tau^2*\ldots*\bm\epsilon^w*\bm\tau^w$ as seen in \S \ref{sub:untwistedsequence}: one proves, with no substantial changes from \S \ref{sub:twistcurvebound} and \S \ref{sec:traintrackconclusion}, that
\begin{claim}
There is a constants $A_8=A_8(n)$, such that the following is true.

Let $p$ be a pants decomposition on $D^2_n$, and let $r$ be a round pants decomposition. Let $(\tau_j)_{j=0}^N=\bm\tau\coloneqq\cnr(\trk\bm\beta^p)$. Then
$$
d_{\pc}(r,p),\ d_{\pa}(r,p),\ d_{\pa}(\tau_0,\tau_N)=_{A_8} \left|\utw\bm\tau\right|.
$$

Here, $\utw$ is defined piece-by-piece as in Definition \ref{def:not_firmly}.
\end{claim}

Note that counting the number of splits in $\utw\bm\tau$ gives roughly the same number as counting their number in $\bm\tau$, but assigning a fixed weight to each interval $DI_\gamma$, \emph{no matter how many splits it includes}. And if one counts splits in this latter way in $\trk\bm\beta$ obtains again the same number, roughly (see Lemma \ref{lem:ctauproperties} and the last bullet above).

The only passage left to prove that this last number is roughly $\|\bm\beta^p\|$ and complete the proof of Proposition \ref{prp:dwcutnumber} is the following Lemma:

\begin{lemma}
There is a bound $A_9(n)$ such that, if $\bm\beta$ is a strip cutting sequence on $D^2_n$, and $\trk\left(\bm\beta(k,l)\right)$ consists of comb equivalences only, then $l-k\leq A_9$.
\end{lemma}

As a conclusion to this work, we note that Corollary \ref{cor:dwgivesvolume_cutnumber} implies a closed formula for hyperbolic volume, of which David Futer has an independent proof, unpublished at the time of writing. A set of generators for the $B_n$, larger than the standard one given in the presentation (\ref{eqn:braidgroup}), is given by
$$\bm\Delta\coloneqq \{\Delta_{ij}\coloneqq (\sigma_i\cdot\cdots\cdot\sigma_{j-1})(\sigma_i\cdot\cdots\cdot\sigma_{j-2})\ldots\sigma_i|0\leq i<j\leq n\}$$
which, rather than representing a half-twist switching two consecutive punctures of $D^2_n$, give a half twist reversing the position of all punctures from the $i$-th to the $j$-th. These generators are the ones used in \cite{dynnikovwiest}.

For $\psi\in B_n$, denote
$$g_\Delta(\psi) \coloneqq \min \{l|\exists\ a_1,\ldots,a_l \mbox{ so that } \psi=\delta_1^{a_1}\ldots\delta_l^{a_l}\mbox{ for some choice of }\delta_k\in \bm\Delta\}.$$

\begin{prop}\label{prp:futervolume}
There is a constant $A=A(n)$ such that the following is true. Let $\psi\in B_n\cong \mcg(D^2_n)$ be such that $\psi$ defines a pseudo-Anosov mapping class on $\inte(D^2_n)$. Let $M\coloneqq \faktor{\inte(D^2_n)\times [0,1]}{\sim_\psi}$ be the related mapping torus. Then
$$
\vol(M) =_{A}\limsup_{m\rightarrow+\infty} \frac{1}{m} g_\Delta(\psi^m)
$$
and
$$
\vol(M) =_{A} \min_{\phi\in \mathrm{Conj}(\psi)} g_\Delta(\phi)
$$
where $\mathrm{Conj}(\psi)$ is the conjugacy class of $\psi$ in $B_n$.
\end{prop}

Given Corollary \ref{cor:dwgivesvolume_cutnumber}, one produces the $\geq_A$ part of these two statements by constructing a braid word $w=\delta_1^{a_1}\ldots\delta_l^{a_l}$ for $\psi^m$ (resp. for $\phi\in \mathrm{Conj}(\psi)$), with $\|\bm\beta^{r(m)}\|$ (resp. $\|\bm\beta^{\phi(r)}\|$) $\geq l/3-(2n+1)$. More precisely, one constructs $w^{-1}$ adapting the process described in \cite{dynnikovwiest}, \S 2.4 to \emph{relax} the strips of the strip decomposition. Given the strip cutting sequence $\bm\beta^{r(m)}$ (resp. $\bm\beta^{\phi(r)}$) for $r$ a round pants decomposition, the process describes how to apply, to each entry $\beta_j$ of the sequence, an element of $\mcg(D^2_n)$ which will turn it into a strip decomposition $\beta'_j$ for a different pants decomposition, with the property that if $\alpha\in \beta'_j$ is a loop then $\#(\alpha\cap\R)\leq 2$ while, if $s\in \beta'_j$ is a strip, then $R(s)\cap\R$ is empty or connected. The diffeomorphism $\lambda$ to be applied to the last entry of $\bm\beta^{r(m)}$ (resp. of $\bm\beta^{\phi(r)}$), which it just $r(m)$ (resp. $\phi(r)$), will make it round. This does not mean that $\lambda=\psi^{-m}$ (resp. that $\lambda=\phi^{-1}$), but note that the process gives a word $w'$ for $\lambda$ having $w=\delta_1^{a_1}\ldots\delta_{l'}^{a_{l'}}$ for $l' \leq 3\|\bm\beta^{r(m)}\|$ (resp. $3\|\bm\beta^{\phi(r)}\|$). Moreover, it is easy to realize that there is a further $\nu\in\mcg(D^2_n)$, with $g_\Delta(\nu)\leq 2n+1$, such that $\nu\circ\lambda=\psi^{-m}$ (resp. $\nu\circ\lambda=\phi^{-1}$).

The $\leq_A$ inequality does not depend on Corollary \ref{cor:dwgivesvolume_cutnumber}, but rather on a Gromov norm argument (see \cite{thurstonnotes}, \S 6 and in particular \S 6.5), based on a construction very similar to the one in \cite{lackenby}, \S 2. Given any $\phi\in\mathrm{Conj}(\psi)$, one has $\faktor{\inte(D^2_n)\times [0,1]}{\sim_\psi}\cong \faktor{\inte(D^2_n)\times [0,1]}{\sim_\phi}$; while $\faktor{\inte(D^2_n)\times [0,1]}{\sim_{\psi^m}}$ is an $m$-fold cyclic cover of $\faktor{\inte(D^2_n)\times [0,1]}{\sim_\psi}$, hence the ratio between the respective volumes is $m$. So one may prove that, if $\faktor{\inte(D^2_n)\times [0,1]}{\sim_\nu}$ for $\nu$ pseudo-Anosov, then $\vol(M) \leq_{A} g_\Delta(\nu)$. Both statements of the Corollary will follow, varying $\nu$ suitably.

Let $w=\delta_1^{a_1}\ldots \delta_l^{a_l}$ be a word in $B_n$ which realizes $g_\Delta(\nu)$. Note that $M\cong T\setminus\ul w$ is also diffeomorphic to $\mathbb S^3\setminus (\ul w\cup L_0)$, for $L_0$ any meridian circle of $\partial\bar T$. One may regard $\ul w$ as subdivided into a number of braids, each corresponding to a factor $\delta_j^{a_j}$ ($1\leq j\leq l$): encircle each of these braids with a further \emph{augmenting} loop $L_j$, similarly to what is done in \cite{lackenby}.

Given a disc $D_j$ bounded by $L_j$ in $\mathbb S^3$, cut $\mathbb S^3\setminus (\ul w\cup L_0\cup\ldots \cup L_l)$ along $D_j$, twist one of the two copies of $D_j$ by a multiple of $2\pi$ and attach it back onto the other: this is a homeomorphism. This property gives, in particular, $\mathbb S^3\setminus (\ul w\cup L_0\cup\ldots \cup L_l)\cong \mathbb S^3\setminus (\ul{w'}\cup L_0\cup\ldots \cup L_l)$, where $w'=\delta_1^{\epsilon_1}\ldots \delta_l^{\epsilon_l}$ with each $\epsilon_j\in\{0,1\}$.

When $W\coloneqq \ul{w'}\cup L_0\cup\ldots \cup L_l$ is isotoped close to the equator of $\mathbb S^3$ to give a link diagram, the decomposition of $\mathbb S^3\setminus W$ into two 3-cells with ideal vertices, as shown in \cite{benedetti}, \S E.5-iv, may be refined, without adding any new vertices or ideal vertices, into a triangulation of this manifold, as it is done in \cite{lackenby}: the number of triangles employed is linear in $l$, therefore the Gromov norm of $\mathbb S^3\setminus W$ is also at most linear in $l$. But then so is the Gromov norm of $\mathbb S^3\setminus (\ul w \cup L_0)$, which is obtained from the former manifold by Dehn surgery (this is immediately implied by Proposition 6.5.2, Corollary 6.5.3 and Lemma 6.5.4 of \cite{thurstonnotes}): and the latter is proportional to $\vol\left(\mathbb S^3\setminus (\ul w \cup L_0)\right)$. This ends the sketch of proof.

%\backmatter
%\fancyhead[CO]{}
%\fancyhead[CE]{}
\newpage
\addcontentsline{toc}{chapter}{Bibliography}
\bibliographystyle{alpha}
{\small
\bibliography{Biblio}}
\end{document}